\documentclass{article}
\usepackage{amsmath, amsfonts, latexsym, amsrefs, amssymb, amsthm}
\usepackage{mathtools}

\usepackage{bm}

\usepackage[font={small}]{caption}

\usepackage{tocloft}

\usepackage{pstricks}

\usepackage{subcaption}

\allowdisplaybreaks

\usepackage{tikz}
\usepackage{tikz,fullpage}
\usetikzlibrary{shapes,arrows, petri, topaths, automata}

\usepackage{tabulary}
\usepackage{booktabs}

\usepackage{graphicx}
\usepackage{grffile}
\graphicspath{{../Figure/}}

\usepackage{enumerate}
\usepackage{mathrsfs}
\usepackage{wrapfig}

\usepackage{color}

\numberwithin{equation}{section}

\newcommand{\rd}{{\rm d}}

\newcommand{\bZ}{{\mathbb Z}}

\newcommand{\beq}{\begin{equation}}
\newcommand{\bEq}{\end{equation}}

\newcommand{\bu}{{\bf{u}}}
\newcommand{\bv}{{\bf{v}}}
\newcommand{\bw}{{\bf{w}}}

\newcommand{\ord}{{\rm{ord}}}

\newcommand{\al}{\alpha}

\newcommand{\be}{\begin{equation}}
\newcommand{\ee}{\end{equation}}

\newcommand{\e}{{\varepsilon}}

\newcommand{\PP}{{\mathbb P}}

\newcommand{\fa}{{\mathfrak a}}
\newcommand{\fb}{{\mathfrak b}}

\renewcommand{\Dot}{\mathbf {Dot}}

\renewcommand{\cal}{\mathcal}
\newcommand{\wh}{\widehat}
\newcommand{\wt}{\widetilde}

\newcommand{\ii}{\mathrm{i}} 
\newcommand{\dd}{\mathrm{d}}

\renewcommand{\epsilon}{\varepsilon}
\renewcommand{\leq}{\leqslant}
\renewcommand{\geq}{\geqslant}

\renewcommand{\le}{\leq}
\renewcommand{\ge}{\geq}

\renewcommand{\P}{\mathbb{P}}
\newcommand{\E}{\mathbb{E}}
\newcommand{\R}{\mathbb{R}}
\newcommand{\C}{\mathbb{C}}
\newcommand{\N}{\mathbb{N}}
\newcommand{\Z}{\mathbb{Z}}

\newcommand{\dashed}{\text{diffusive}}
\newcommand{\Dashed}{\text{Diffusive}}
\newcommand{\self}{\text{self-energy}}
\newcommand{\selfs}{\text{self-energies}}

\newcommand{\Selfs}{\text{Self-energies}}

\newcommand{\abs}{{\rm abs}}

\DeclareMathOperator{\tr}{Tr}

\DeclareMathOperator{\re}{Re}
\DeclareMathOperator{\im}{Im}

\DeclareMathOperator{\OO}{O}
\DeclareMathOperator{\oo}{o}

\theoremstyle{plain} 
\newtheorem{theorem}{Theorem}[section]
\newtheorem*{theorem*}{Theorem}
\newtheorem{lemma}[theorem]{Lemma}
\newtheorem{assumption}[theorem]{Assumption}
\newtheorem*{lemma*}{Lemma}
\newtheorem{corollary}[theorem]{Corollary}
\newtheorem*{corollary*}{Corollary}

\newtheorem*{proposition*}{Proposition}
\newtheorem{claim}[theorem]{Claim}
\newtheorem*{claim*}{Claim}

\newtheorem{definition}[theorem]{Definition}
\newtheorem*{definition*}{Definition}
\theoremstyle{remark}
\newtheorem{example}[theorem]{Example}
\newtheorem*{example*}{Example}
\newtheorem{remark}[theorem]{Remark}

\newtheorem*{remark*}{Remark}
\newtheorem*{remarks*}{Remarks}
\newtheorem{strategy}[theorem]{Strategy}

\usepackage{dsfont}
\usepackage{stmaryrd}

\newcommand{\cor}{\color{red}}

\newcommand{\Sdelta}{{\Sigma}_T}
\newcommand{\Sdeltak}{{\Sigma}_{T,k}}
\newcommand{\wtSdelta}{\Sigma}
\newcommand{\wtSdeltan}{\Sigma^{(n)}}

\newcommand{\Sele}{{\mathcal E}}

\newcommand{\Selek}{{\mathcal E}}

\newcommand{\Seleinf}{{\mathcal E}^{\infty}}

\newcommand{\Sint}{\mathfrak S}

\newcommand{\Incomp}{{\text{$T$-equation}}}
\newcommand{\incomp}{{\text{$T$-equation}}}

\newcommand{\PT}{\mathcal R_T}
\newcommand{\wtPT}{\widetilde{\mathcal R}_T}
\newcommand{\PTn}{\mathcal R_T^{(n)}}
\newcommand{\PTk}{{\mathcal R_{T,k}}}
\newcommand{\PIT}{\mathcal R_{IT}}

\newcommand{\PITn}{{\mathcal R_{IT}^{(n)}}}
\newcommand{\PITk}{{\mathcal R_{IT,k}}}

\newcommand{\PGn}{\mathcal R^{(n)}}
\newcommand{\AGn}{\mathcal A^{(>n)}_{ho}}
\newcommand{\QGn}{\mathcal Q^{(n)}}

\newcommand{\QT}{\mathcal Q_T}
\newcommand{\QTn}{\mathcal Q_T^{(n)}}
\newcommand{\QTk}{{\mathcal Q_{T,k}}}

\newcommand{\QITn}{{\mathcal Q_{IT}^{(n)}}}
\newcommand{\QITk}{{\mathcal Q_{IT,k}}}

\newcommand{\AT}{\mathcal A_T}
\newcommand{\ATn}{\mathcal A_T^{(>n)}}

\newcommand{\AIT}{\mathcal A_{IT}}
\newcommand{\AITn}{{\mathcal A_{IT}^{(>n)}}}

\newcommand{\Err}{{\mathcal Err}}

\def\bZ{{\mathbb Z}}

\def\@empty{}

\def\author#1{\par
    {\centering{\authorfont#1}\par\vspace*{0.05in}}
}

\def\titlefont{\fontsize{13}{15}\bfseries\boldmath\selectfont\centering{}}
\def\authorfont{\fontsize{13}{15}}
\def\abstractfont{\fontsize{8}{10}}

\let\affiliationfont\rhfont

\def\address#1{\par
    {\centering{\affiliationfont#1\par}}\par\vspace*{11pt}
}

\def\body{
\setcounter{footnote}{0}
\def\thefootnote{\alph{footnote}}
\def\@makefnmark{{$^{\rm \@thefnmark}$}}
}

\def\title#1{
    \thispagestyle{plain}
    \vspace*{-14pt}
    \vskip 79pt
    {\centering{\titlefont #1\par}}%
    \vskip 1em
}

\renewenvironment{abstract}{\par%
    \vspace*{6pt}\noindent 
    \abstractfont
    \noindent\leftskip18pt\rightskip18pt
}{%
  \par}

\usepackage{tocloft}

\makeatletter
\renewcommand{\section}{\@startsection
{section}
{1}
{0mm}
{-2\baselineskip}
{2\baselineskip}
{\normalfont\large\scshape\centering}} 
\makeatother

\makeatletter
\renewcommand{\subsection}{\@startsection
{subsection}
{2}
{0mm}
{-\baselineskip}
{0.5\baselineskip}
{\normalfont\bf\itshape}} 
\makeatother

\newcommand{\tnorm}[1]{{\left\vert\kern-0.25ex\left\vert\kern-0.25ex\left\vert #1 
    \right\vert\kern-0.25ex\right\vert\kern-0.25ex\right\vert}}

\newcommand{\hty}[1]{{\cor (ht: #1) }}

\begin{document}

\title{Delocalization and quantum diffusion of random band matrices in high dimensions I: Self-energy renormalization}
{\let\thefootnote\relax\footnotetext{\noindent 
The work of F.Y. is partially supported by the Wharton Dean’s Fund for Postdoctoral Research. 
The work of H.-T. Y. is partially supported by the NSF grant DMS-1855509 and a Simons Investigator award. 
The work of J.Y. is partially supported by the NSF grant DMS-1802861. 
}}
\vspace{1cm}
\noindent \begin{minipage}[b]{0.32\textwidth}
\author{Fan Yang }
\address{University of Pennsylvania\\
   fyang75@wharton.upenn.edu}
 \end{minipage}
\begin{minipage}[b]{0.32\textwidth}
 \author{Horng-Tzer Yau}
\address{Harvard University\\
  htyau@math.harvard.edu}
 \end{minipage}
\begin{minipage}[b]{0.32\textwidth}
 \author{Jun Yin}
\address{University of California, Los Angeles\\
    jyin@math.ucla.edu}
 \end{minipage}

\begin{abstract}
We consider Hermitian random band matrices $H=(h_{xy})$ on the $d$-dimensional lattice $(\Z/L\Z)^d$. The entries  $h_{xy}$ are  independent (up to Hermitian conditions) centered complex Gaussian random variables with variances $s_{xy}=\mathbb E|h_{xy}|^2$. The variance matrix $S=(s_{xy})$ has a banded structure so that  $s_{xy}$ is negligible if $|x-y|$ exceeds the band width $W$. 
In dimensions $d\ge 8$, we prove that, as long as $W\ge L^\e$ for a small constant $\e>0$, with high probability  most  bulk eigenvectors of $H$ are delocalized in the sense that their localization lengths are comparable to $L$. Denote by $G(z)=(H-z)^{-1}$ the Green's function of $H$. For $\im z\gg W^2/L^2$, we also prove a widely used criterion in physics for quantum diffusion of this model, namely, the leading term in the Fourier transform of $\E|G_{xy}(z)|^2$ with respect to $x-y$ is of the form $(\im z + a(p))^{-1}$ for some $a(p)$  quadratic in $p$, where $p$ is the Fourier variable. 
Our method is based on an  expansion of  $T_{xy}=|m|^2 \sum_{\alpha}s_{x\alpha}|G_{\alpha y}|^2$ and it   requires  a
self-energy renormalization up to error $W^{-K}$ for any large constant  $K$ independent of $W$ and $L$. 
We expect that this method can be extended to non-Gaussian band matrices. 
\end{abstract}

\tableofcontents

\newpage

\section{Introduction}

\subsection{Random band matrices}

Wigner envisioned \cite{Wigner} that spectral properties of quantum systems of high complexity can be modeled by Gaussian random matrices such as GOE (Gaussian orthogonal ensemble) or GUE (Gaussian unitary ensemble). Although many non-rigorous arguments and numerical simulations support his thesis, rigorous works have been mostly restricted to mean-field models such as Wigner matrices or the adjacency matrices of random graphs of various types. For non-mean-field models, the understanding of their spectral properties is much more limited. One important non-mean-field model is the random Schr{\"o}dinger operator or more specifically, the Anderson model \cite{Anderson}. More precisely, the $d$-dimensional Anderson model is defined  by a Hamiltonian $H=-\Delta + \lambda V$, where $\Delta$ is the graph Laplacian on $\Z^d$, $V$ is a random potential with i.i.d. entries, and $\lambda$ is a small coupling strength. This model is highly non-mean-field because the off-diagonal elements consist of only 2$d$ entries of constant value in each row or column, while all the randomness is in the diagonal elements. In the strong disorder regime, i.e., when $\lambda $ is large, the eigenvectors of the Anderson model are expected to be localized and the local eigenvalue statistics converge to a Poisson process;  in the weak disorder regime, the eigenvectors are expected to be delocalized and the local eigenvalue statistics coincide with those of a GOE or GUE. 
The localization was first proved rigorously by Fr{\"o}hlich and Spencer \cite{FroSpen_1983} using a multi-scale analysis;   an alternative  proof was given years  later by Aizenman and Molchanov \cite{Aizenman1993} using a fractional moment method. Many spectacular results have been proved regarding the localization of the Anderson model (see, e.g., \cite{FroSpen_1985,Bourgain2005,Carmona1987,DingSmart2020,Damanik2002,Germinet2013,LiZhang2019}). The existence of the delocalized regime for the Anderson model has only been proved for the Bethe lattice \cite{Bethe_PRL,Bethe_JEMS}, but not for any finite-dimensional integer lattice $\Z^d$.

A model that is more tractable than the Anderson model but still preserves its key non-mean-field property is the following random band matrix ensemble. Let $\Z_L^d:=\{1,2, \cdots, L\}^d$ be a lattice of linear size $L$, and $N\equiv L^d$ be the total number of lattice sites. A $d$-dimensional random band matrix ensemble consists of $N\times N$ random Hermitian matrices  \smash{$H=(h_{xy})_{x,y\in \Z_L^d}$}, whose entries  $h_{xy}$ are  centered random variables  that are independent up to the Hermitian condition $h_{xy}=\overline h_{yx}$. 
%
%
In this paper, we require that $s_{xy}$ be negligible when $|x-y|\gg W$ for some length scale $1\ll W \ll L$ and satisfy the normalization condition  
\be\label{fxy}
\sum_{x}s_{xy}=\sum_{y}s_{xy}=1.
\ee
It is well-known that under the condition \eqref{fxy}, the global eigenvalue distribution of $H$ converges weakly to the Wigner's semicircle law supported in $[-2,2]$. 
As $W$ varies, the random band matrices naturally interpolate between the random Schr\"odinger operator  \cite{Anderson}
and the mean-field Wigner  ensemble \cite{Wigner}. 

A key physical quantity for both the Anderson model and random band matrices is the \emph{localization length} $\ell$, which, roughly speaking, is the length scale of the region in which most weight of an eigenvector resides. There are different ways to define the localization length depending on how the eigenvector decays outside the localized region (e.g. polynomial decay, exponential decay, etc.). 
For the Anderson model in infinite volume, an eigenvector is localized if its localization length is finite, and delocalized otherwise. 
For random band matrices, one can define an eigenvector to be \emph{delocalized} if its localization length $\ell$ is comparable with the linear size $L$ of the system, and \emph{localized} otherwise. It should be remarked that localization and delocalization in general depend on the energy levels. 
In this paper, we will restrict ourselves to the bulk eigenvectors, that is, eigenvectors with   eigenvalues in $(-2+ \kappa, 2-\kappa)$ for some small constant $\kappa>0$ independent of $L$.

We assume for the moment that the majority of the bulk eigenvectors have similar localization lengths so that we can refer to the localization length of a random band matrix. The localization length $\ell\equiv \ell(d,W)$ is expected to increase with $W$ and an (almost) sharp localization-delocalization transition occurs at some critical band width $W_c\equiv W_c(d,L)$ when the localization length  $\ell$ becomes comparable to  the system size $L$, i.e., 
\begin{itemize}
	\item for  $W \gg W_c$, the bulk eigenvectors are delocalized, i.e. $\ell$ is of order similar to  $L$; 
	\item for  $W \ll W_c$, the bulk eigenvectors are localized, i.e. $\ell$ is much smaller than $L$.
\end{itemize}

Heuristically, the random band matrices and the Anderson model are expected to have  same qualitative properties with $\lambda \sim W^{-1}$. 
In dimension $d=1$, the localization length of the Anderson model 
is known to be of order $\ell\sim \lambda^{-2}$. 
By simulations \cite{ConJ-Ref1, ConJ-Ref2, ConJ-Ref4, ConJ-Ref6} and non-rigorous supersymmetric arguments \cite{fy}, the localization length of one-dimensional random band matrices is conjectured to be of order $\ell \sim W^2$,  leading  to the critical band width $W_c\sim \sqrt L$. The localization length of the two-dimensional Anderson model is conjectured to be exponentially large in $\lambda^{-2}$ \cite{PRL_Anderson} (although this conjecture is not universally accepted). Correspondingly, it is conjectured that the localization length of 
two-dimensional random band matrices also grows exponentially fast in $W^{2}$, leading to the critical band width $W_c\sim \sqrt{\log L}$. In dimensions $d\ge 3$, it is conjectured that there is a threshold energy in the Anderson model,  the mobility edge, that separates the delocalized and localized states. For random band matrices with $d\ge 3$, 
the bulk eigenvectors are conjectured to be delocalized. 
More precisely, the localization length of the bulk eigenvectors is expected to be of macroscopic scale $\ell\sim L$ independently of the band width $W$, and the critical band width $W_c$ is a large number independent of $L$. 
The previous summary on the localization-delocalization conjecture is mainly focused on the random band matrices, and we refer the reader to \cite{Spencer1,Spencer2,Spencer3, PB_review} for more details. There are extensive works concerning this problem for random Schr\"odinger operators in the past several  decades;  
they are beyond the scope of this paper and we refer the reader to \cite{Kirsch2007,CarLa1990,Spencer_Anderson} for extensive reviews. 

There have been  many partial results concerning these conjectures for  random band matrices in dimension $d=1$ \cite{BaoErd2015,Semicircle, ErdKno2013,ErdKno2011,delocal,HeMa2018,BouErdYauYin2017,PartI,PartII,PartIII,Sch2009,PelSchShaSod,Sod2010,SchMT,Sch1,Sch2,Sch3,1Dchara,Sch2014}. Key results include that $\ell < W^7$ if the entries of $H$ are Gaussian \cite{PelSchShaSod}, and $\ell > W^{4/3}$ for general random band matrices without Gaussian assumption \cite{PartI, PartII, PartIII}. For a class of complex Hermitian Gaussian random band matrices with  certain special variance profiles, supersymmetry techniques can be used \cite{Efe1997,DisPinSpe2002,BaoErd2015,Sch2014,SchMT,Sch1,Sch2,Sch3,1Dchara}, and a transition in the two-point correlation function for the bulk eigenvalues at $W_c\sim L^{1/2}$ was proved in \cite{1Dchara}. It is still not clear if the supersymmetry method can be adapted to prove  localization or delocalization of the random band matrices treated in \cite{1Dchara}.

The understanding of the delocalization of random band matrices in dimensions $d\ge 2$, however, is much more limited.   
Based on studying the unitary operator $e^{\ii t H}$, it was shown \cite{ErdKno2013,ErdKno2011} that the localization length for 
$d$-dimensional random band matrices satisfies $\ell > W^{1+d/6}$. The delocalization used in these papers is defined in a weak sense which we will explain  later on. 
With a Green's function method, it was proved  \cite{PartIII} that $\ell > W^{1+d/2}$, improving the earlier results obtained in \cite{delocal, HeMa2018}.

In this paper, we prove that with high probability in dimensions $d\ge 8$, 
the bulk eigenvectors of random band matrices are (weakly) delocalized in the sense defined in \cite{ErdKno2013,ErdKno2011} provided that $W\ge L^\e$ for a small constant $\e>0$. Recall that the delocalization conjecture asserts that the random band matrices are delocalized in dimensions $d \ge 3$ as long as $W\ge C$ for a large enough constant $C>0$. Our result gives a positive answer to  this conjecture for $d \ge 8$  in the weak delocalization sense under the slightly  stronger assumption $W\ge L^\e$ (vs. $W\ge C$). The definition of the delocalization used in this paper, following \cite{ErdKno2013,ErdKno2011}, is still far from the strong delocalization used for Wigner matrices \cite{ESY_local,ESY1}. Major works remain to be done to prove the strong delocalization even under the conditions $d \ge 8$ and $W\ge L^\e$. We will discuss some of these problems after stating the main results.

\subsection{Delocalization and local law}\label{subsec_mainresult}
 
 In this subsection, we define our model and state the first two main results, Theorem \ref{comp_delocal} and Theorem \ref{main thm0}, of this paper. 
 We will consider  $d$-dimensional random band matrices indexed by a cube  of linear size \(L\) in \(\mathbb{Z}^{d} \), i.e., 
 \be\label{ZLd}
 \Z_L^d:=\left( \Z\cap ( -L/2 , L/2]\right) ^d. 
 \ee
 We will view $\Z_L^d$ as a torus and denote  by  $[x-y]_L$ the representative of $x-y$ in $\Z_L^d$, i.e.,  
 \be\label{representativeL}[x-y]_L:= \left[(x-y)+L\Z^d\right]\cap \Z_L^d.\ee
 Clearly, $\|x-y\|_L:=\| [x-y]_L \|$ is  the {periodic} distance on $\Z_L^d$ for any norm $\|\cdot\|$ on $\Z^d$.  For definiteness, we use $\ell^\infty$-norm in this paper, i.e. $\|x-y\|_L:=\|[x-y]_L\|_\infty$. 
 In this paper, we consider the following class of $d$-dimensional random band matrices.

\begin{assumption}[Random band matrix $H\equiv H_{d,f,W,L}$] \label{assmH}
	Fix any $d\in \N$. For $L\gg W\gg 1$ and $N:=L^d$, we assume that $ H\equiv H_{d,f,W,L}$ is an $N\times N$ complex Hermitian random matrix whose entries $(\re h_{xy}, \im   h_{xy}: x,y \in \Z_L^d)$ are independent Gaussian random variables (up to the Hermitian condition $h_{xy}=\overline h_{yx}$) such that  
	\be\label{bandcw0}
	\mathbb E h_{xy} = 0, \quad \E (\re h_{xy})^2 =  \E (\im h_{xy})^2 = s_{xy}/2, \quad x , y \in \bZ_L^d,
	\ee
	where the variances $s_{xy}$ satisfy that
	\be\label{sxyf}s_{xy}= f_{W,L}\left( [x-y]_L \right)\ee
	for some positive symmetric function $f_{W,L}$ satisfying Assumption \ref{var profile} below. Then we say that $H$ is a $d$-dimensional  random band matrix with the linear size $L$,  band width $W$ and variance profile $f_{W,L}$. Denote the variance matrix by $S : = (s_{xy})_{x,y\in \Z_L^d}$, which is a doubly stochastic symmetric $N \times N$ matrix. 
	%
\end{assumption}   

\begin{assumption}[Variance profile]\label{var profile}
	We assume that $f_{W,L}:\Z_L^d\to \mathbb R_+$ is a positive symmetric function on $\Z_L^d$ that can be expressed by the Fourier transform 
	\be\label{choicef}
	f_{W,L}(x):= \frac{1}{(2\pi)^d Z_{W,L}}\int \psi(Wp)e^{\ii p\cdot x} \dd p.  \ee
	Here $\Z_{W,L}$ is the  normalization constant so that $\sum_{x\in \Z_L^d} f_{W,L}(x)=1$, and $\psi\in C^\infty(\R^d)$ is a symmetric smooth function independent of $W$ and $L$ and satisfies the following properties:
	\begin{itemize}
		\item[(i)] $\psi(0)=1$ and $\|\psi\|_\infty \le 1$;  
		\item[(ii)] $\psi(p)\le \max\{1 - c_\psi |p|^2 , 1-c_\psi  \}$ for a constant $c_\psi>0$;
		\item[(iii)]  $\psi$ is in the Schwartz space, i.e.,
		\be\label{schwarzpsi} \lim_{|p|\to \infty}(1+|p|)^{k}|\psi^{(l)}(p)| =0, \quad \text{for any }k,l\in \N.\ee
	\end{itemize}
\end{assumption}

Clearly, $f_{W,L}$ is of order $\OO(W^{-d})$
and decays faster than any polynomial, that is, for any fixed $k\in \N$, there exists a constant $C_k>0$ so that
\be\label{subpoly}
|f_{W,L}(x)|\le C_k W^{-d}\left( {\|x\|_L}/{W}\right)^{-k}.
\ee
Hence the variance profile $S$ defined in \eqref{sxyf}  has a banded structure, namely,  for any constants $\tau,D>0$,
\be\label{app compact f}
\mathbf 1_{|x-y|\ge W^{1+\tau}}|s_{xy}|\le W^{-D}.
\ee
Combining \eqref{schwarzpsi} and \eqref{subpoly} with the Poisson summation formula, we obtain that  
\be\label{bandcw1} 
Z_{W,L} =   \psi(0) + \OO(W^{-D})=1+ \OO(W^{-D}),
\ee
for any large constant $D>0$ as long as $L\ge W^{1+\e}$ for a constant $\e>0$. 
Note that Assumption \ref{var profile} does not cover non-smooth profile functions. For example, it does not  include the indicator function $f_{W,L}(x)=W^{-d}\mathbf 1_{x\in (-W/2,W/2]^d}$. While we believe that Assumption \ref{var profile} is not essential, we will not get into this technical issue in this paper.

Denote the eigenvalues and normalized eigenvectors of $H$ by $\{\lambda_\al\}$ and $\{\mathbf u_\alpha\}$. According to \cite{Spencer2}, an eigenvector $\bu_\al$ is localized with a localization length $\ell$ if for some $x_0\in \Z_L^d$ its entries satisfy that 
\be\label{localu}  
|  u_\alpha(x) | \le Ce^{-c\|x-x_0\|_L/\ell} ,
\ee
for some constants $c,C>0$. Inspired by this definition, 
for any fixed constants $K>1$ and $0<\gamma \le 1 $, we define a random subset of indices as
$$ {\mathcal B}_{\gamma, K,\ell} :=\left\{\alpha: \lambda_\alpha \in (-2+\kappa, 2-\kappa) \ \text{ so that  } \min_{x_0\in \Z_L^d} \sum_x |  u_\alpha(x)|^2 \exp\left[\left(\frac{\|x-x_0\|_L}{\ell}\right)^\gamma\right] \le K \right\}, $$
which contains all indices associated with bulk eigenvectors that have localization lengths bounded by $\OO(\ell)$. Here we have relaxed the exponential function in \eqref{localu} to a more general family of sub-exponential functions. Then we have the following theorem for random band matrices in dimensions $d\ge 8$. 



\begin{theorem}[Weak delocalization of bulk eigenvectors in high dimesnions]\label{comp_delocal}
	Fix $d\ge 8$, small constants $c_0, c_1, \gamma, \kappa>0$ and a large constant $K > 1$.  Suppose that $W\le  \ell \le L^{1 - c_0}$,  $L^{c_1} \le W \le L$ and $H$ is a $d$-dimensional random band matrix satisfying Assumptions \ref{assmH} and \ref{var profile}. 
	Then we have that for any constants $\tau,D>0$,
	\be\label{uinf_locallength}  \mathbb P \left[ \frac{|\mathcal B_{\gamma,K,\ell}|}{N} \le W^\tau \left(\frac{\ell^2}{L^2} + W^{-d/2}\right)\right] \ge 1-L^{-D},
	\ee
	provided that $L$ is sufficiently  large  depending on these constants. Moreover, 
	for any eigenvalue $\lambda_\al$ of $H$ satisfying $\lambda_\al\in (-2+\kappa,2-\kappa)$, its eigenvector $\bu_\al$ satisfies that
	\be\label{uinf} 
	\mathbb P \left(\|\bu_\al\|_\infty \le W^{1+\tau}/L\right)\ge  1-L^{-D} ,
	\ee
	for any constants $ \tau, D>0$ and sufficiently large $L$.
\end{theorem}

The estimate \eqref{uinf_locallength} asserts that, for random band matrices with band width essentially of order one ($L^{c_1}$ 
for any small constant $c_1> 0$), the majority of bulk eigenvectors have localization lengths essentially of the size of the system (in the sense that they are larger than $L^{1 - c_0}$ for any small constant $c_0 > 0$). The bound \eqref{uinf} implies that $\|\bu_\al\|_4^4\le W^{2+2\tau}/L^2$ with high probability, which converges to 0 as $L\to \infty$, another commonly used weak notion of delocalization in physics (see, e.g., \cite{Spencer2}).  We also remark that all the results in this paper hold only for large enough $W$ and $L$, and, for simplicity, we do not repeat it again in all our statements.

To prove Theorem \ref{comp_delocal}, we study the resolvent (or Green's function) of $H$ defined by 
\be\nonumber
G(z)=(H-z)^{-1},\quad z\in \C_+:=\{x\in \C: \im z>0\}.
\ee
In \cite{Semicircle, ErdYauYin2012Univ}, it has been shown that for any small constant $\e>0$,
\be  \label{intro_semi}
\max_{x,y}|G_{xy}(z)-m(z)\delta_{xy}| \le \frac{W^\e}{\sqrt{W^d \eta}} ,\quad z= E+ \ii \eta,
\ee
with high probability for all $E\in (-2+\kappa,2-\kappa)$ and $\eta\ge W^{-d + \e}$, where $m(z)$ is the Stieltjes transform of Wigner's semicircle law,
\be\label{msc}
m(z):=\frac{-z+\sqrt{z^2-4}}{2} = \frac{1}{2\pi}\int_{-2}^2 \frac{\sqrt{4-\xi^2}}{\xi-z}\dd\xi,\quad z\in \C_+.
\ee
The bound \eqref{intro_semi} implies a lower bound on the localization length of order $W$, which is far shorter than $L^{1 - c_0}$ 
stated in Theorem \ref{comp_delocal} when $W = L^{c_1}$. For our purpose, we need to decrease $\eta$ from $\eta\ge W^{-d + \e}$
to a much smaller scale $\eta \ge W^{2}/L^{2-\e}$ and improve the error bound in \eqref{intro_semi} significantly. 
%
While the diagonal resolvent entry $G_{xx}$ is expected to be given by the semicircle law for a large range of $\eta$, we will show that the off-diagonal entries can be approximated by a \emph{diffusive kernel} $\Theta$ defined by  
\be\label{Thetapm}
\Theta(z):=\frac{|m(z)|^2 S}{1-|m(z)|^2 S} .
\ee
It is well-known that 
for $z=E+\ii \eta$ with $E\in (-2+\kappa,2-\kappa)$ and $\eta\ge W^{2}/L^{2-\e}$ for a constant $\e>0$,
\be\label{thetaxy}
\Theta_{xy}(z) \le   \frac{W^\tau \mathbf 1_{|x-y|\le   \eta^{-1/2}W^{1+\tau}}}{W^2\left(\|x-y\|_{ L}+W\right)^{d-2}}   + \frac{1}{ \left(\|x-y\|_{ L}+W\right)^{D}}  \le W^{\tau} B_{xy},
\ee
for any constants $\tau, D>0$, where we have abbreviated that 
\be\label{defnBxy}B_{xy}:=W^{-2}\left(\|x-y\|_{ L}+W\right)^{-d+2}.\ee 
The reader can refer to \cite{delocal,PartIII} for a proof of \eqref{thetaxy}. 
The following theorem provides an essentially sharp local law on the resolvent entries under the assumptions of Theorem \ref{comp_delocal}.

\begin{theorem}[Local law]\label{main thm0}
Under the assumptions of Theorem \ref{comp_delocal}, for any small constants $\e,\tau>0$ and large constant $D>0$, we have the following estimate on $G(z)$ for $z=E+\ii \eta$ and all $x,y \in \Z_{L}^d$: 
\be\label{locallaw}
\P\bigg(\sup_{E\in (-2+\kappa, 2- \kappa)}\sup_{W^{2}/L^{2-\e}\le \eta\le 1} |G_{xy} (z) -m(z)\delta_{xy}|^2 \le  W^\tau B_{xy}\bigg) \ge 1- L^{-D} . 
\ee
\end{theorem}

The proof of Theorem \ref{main thm0} is based on an expansion method and can be readily adapted to non-Gaussian random band matrices after some technical modifications (cf. Remark \ref{rem general distr} below for more details).  We choose to present it  for Gaussian cases to avoid technical complexities associated with non-Gaussian distributions (which will increase the number of terms in  expansions). The proof for the real symmetric Gaussian case is very similar to the complex case except that, as usual, the number of terms will double in every expansion step.  (This is due to the fact that $\E h_{xy}^2 = 0$ in the complex case but not in the real case.) The condition $d\ge 8$ can also be improved; it remains to be seen whether this method can  reach the physical dimension $d=3$. 
We will deal with  these  improvements in forthcoming papers.

A very strong notion of delocalization is to require that
$$\mathbb P \left(\|\bu_\al\|_\infty \le L^{-d/2 +\tau}\right)\ge  1-L^{-D} 
$$
for any constants $\tau, D>0$. This was first proved for Wigner matrices in \cite{ESY_local,ESY1,ErdYauYin2012Univ,ErdYauYin2012Rig} and later extended to many other classes of mean-field type random matrices (see e.g. \cite{isotropic,PartI,EKYY_ER1,HY_Regulard,BHY2019,HKM2019,BKH2017}). This estimate  was proved in \cite{ErdYauYin2012Univ} as a consequence of the following bound on the diagonal resolvent entries, i.e.,  for some constant $C>0$, 
$$
\max_{x\in \Z_L^d}   |G_{xx} (z) | \le  C  \quad \text{for all $\eta\gg L^{-d}$}.
$$
We believe  that the resolvents of random band matrices satisfy the following stronger estimate with high probability:
\be
|G_{xy} (z) -m(z)\delta_{xy}|^2 \le W^\tau B_{xy} + W^\tau (L^d\eta)^{-1} \quad \text{for all $\eta\gg L^{-d}$}.
\ee
The restriction $\eta\gg W^2/L^2$ in this paper is not intrinsic and can be substantially improved to, say, $\eta\gg L^{-d/4}$. However, it seems to be a difficult problem to reach the optimal threshold $\eta\gg L^{-d}$.

\subsection{Quantum diffusion} \label{sec_diffu}


A key quantity in the analysis of random band matrices is the $T$-matrix introduced in \cite{delocal}: 
\be\label{defGT}
T_{xy}(z):=|m|^2\sum_\al s_{x\al}|G_{\al y}(z)|^2,\quad x,y \in \Z_L^d.
\ee
Note that $T_{xy}$ is very similar to $|G_{xy}|^2$, and it is known that the $T$-matrix controls the asymptotic behaviors of the resolvent  (see, e.g., Lemma \ref{lem G<T} below). Moreover, the $T$-variables are slightly easier to use in our proof, because the diagonal $T$-variables $T_{xx}$ can be dealt with in the same way as the off-diagonal $T$-variables $T_{xy}$ with $x\ne y$, while this is not the case for the $|G_{xy}|^2$ variables.
In the following theorem, we show that  $\E T_{xy}$ is governed by a diffusion profile. 

\begin{theorem}[Quantum diffusion of the $T$-matrix]\label{main thm2}
	Suppose the assumptions of Theorem \ref{comp_delocal} hold. 
	Fix any small constant $\e>0$ and large constant $M\in \N$. Then  for all $x,y \in \Z_{L}^d$ and $z=E+\ii \eta$ with $E\in (-2+\kappa,2-\kappa)$ and $W^2/L^{2-\e}\le \eta\le 1$, we have that 
	\be\label{E_locallaw}
	\begin{split}
		\E T_{xy} &=  \left[  \Theta^{(M)} \left(|m|^2 + \cal G^{(M)} \right)\right]_{xy}  + \OO( W^{-Md/2} ) . 
	\end{split}
	\ee
	Here $\Theta^{(M)}$ is the $M$-th order \emph{renormalized diffusive matrix}
	\be\label{theta_renormal}\Theta^{(M)}:= \frac{1}{1-|m|^2  S\left(1+\wtSdelta^{(M)}\right)}|m|^2S,\ee
	and it  satisfies the bound 
	\be\label{Self_theta}\left|\Theta^{(M)}_{xy}\right| \le L^{\tau}B_{xy}, \ee 
	for any constant $\tau>0$. Furthermore, the \emph{self-energy correction} $\wtSdelta^{(M)}$ is given by $\wtSdelta^{(M)}(z):=\sum_{l=4}^M \Sele_l(z)$ where $\{\Sele_l\}_{l=4}^M$ is a sequence of deterministic matrices satisfying the following properties:  
    \be\label{two_properties0}
	\Selek_{l} (x, x+a) =  \Selek_{l} (0,a), \quad   \Selek_{l} (0, a) = \Selek_{l}(0,-a), \quad \forall \ x,a\in \Z_L^d,
	\ee 
	and for any constant $\tau>0$,
	\be\label{4th_property0}
	\left|  (\Selek_{l})_{0x}(z) \right| \le L^\tau {W^{-(l-4)d/2}}B_{0x}^2  , \quad \forall \ x\in \Z_L^d,\ \eta\in [W^2/L^{2-\e}, 1],
	\ee
	\be\label{3rd_property0}
	\Big|\sum_{x\in \Z_L^d} (\Selek_{l})_{0x}(z)\Big|   \le L^\tau \eta W^{-(l-2)d/2 } , \quad \forall  \ \eta\in [W^2/L^{2-\e},1].
	\ee
 Here in \eqref{two_properties0} and throughout the rest of this paper, we use $\cal A_{xy}$ and $ \cal A(x,y)$ interchangeably for any matrix $\cal A$.	
	The	 $M$-th order \emph{local correction} $ \cal G^{(M)}$ satisfies that
	\be\label{bound_calG}
	\left|\cal G_{xy}^{(M)}\right|\le L^\tau B_{xy}^{3/2} , 
	\ee
	for any constant $\tau>0$. 
\end{theorem}
We believe that the bound \eqref{bound_calG} can be improved to $|\cal G_{xy}^{(M)}|\le L^\tau B_{xy}^{2}$ with some extra work, but we do not pursue this improvement in this paper for simplicity. From Theorem \ref{main thm2}, we can readily obtain the following quantum diffusion of the resolvent entries. 
\begin{corollary}[Quantum diffusion]\label{main_cor}
	Under the assumptions of Theorem \ref{main thm2}, we have that 
	\be\label{EGxy} \E |G_{xy}|^2 =  \bigg[ \frac{1}{1-\left(1+\wtSdelta^{(M)}\right)|m|^2  S} \left( |m|^2 + \cal G^{(M)}\right) \bigg]_{xy} 
	+ \OO( W^{-Md/2} ) ,
	\ee
	for all $x,y \in \Z_{L}^d$ and $z=E+\ii \eta$ with $E\in (-2+\kappa,2-\kappa)$ and $W^2/L^{2-\e}\le \eta\le 1$.
\end{corollary}

Taking $M=0$, we get the (0-th order) diffusive matrix $\Theta^{(0)} \equiv \Theta$ in \eqref{Thetapm}.
We can expand $\Theta$ into a geometric series 
\be\label{RW-rep}
\Theta_{xy}= \sum_{k=1}^\infty |m|^{2k}(S^{k})_{xy}.
\ee   
By \eqref{fxy}, $S$ is the transition matrix of a random walk on $\Z_L^d$ with step size $\OO(W)$. With direct calculations, we can check that $|m|^2 =  1- \eta/r(E) +\OO(\eta^2)$ where $r(E):=\sqrt{4-E^2}/2$ is proportional to the semicircle density. Hence \eqref{RW-rep} shows that $\Theta_{xy}$ is a superposition of random walks up to the time $\eta^{-1}$, which is the main reason why we call $\Theta$ the diffusive matrix. Due to the form of $s_{xy}$ in \eqref{sxyf}, $\Theta_{xy}$ is translationally invariant on $\Z_L^d$. Moreover, the Fourier transform of $\Theta_{xy}$ with respect to $x-y$ 
is given by  
\be\nonumber
\wh\Theta(p):=  \sum_x \Theta_{0x} e^{\ii p \cdot x} = \frac{|m|^2 \wh S_{W,L}(p)}{1- |m|^2 \wh S_{W,L}(p)},\quad \text{with} \quad p\in  \mathbb T_L^d:=\left(\frac{2\pi }{L}\Z_L\right)^d, \ \ \wh S_{W,L}(p):= \sum_x s_{0x} e^{\ii p \cdot x}.
\ee
Note that by \eqref{choicef}, $\wh S_{W,L}(p)$ is equal to $\psi(Wp)$ up to a small error when $L$ is large. In the regime $|p|\ll W^{-1}$ and $\eta\ll 1$, this equation gives the following diffusion approximation: 
\be\label{diffu_Theta}
\wh\Theta(p) = \frac{|m|^2 \wh S_{W,L}(p)}{(1- |m|^2) + |m|^2[1-\wh S_{W,L}(p)]}= \frac{r(E) \left[1+\OO(\eta+W^2 |p|^2)\right]}{\eta+ W^2 p\cdot \cal D(E) p + \OO(\eta^2 + W^3|p|^3)},
\ee
with an effective diffusion coefficient (matrix) $ \cal D(E)$ defined by
$$ \cal D_{ij} :=\frac{r(E) }{2}\sum_{i,j}\frac{x_i x_j}{W^2}s_{0x},\quad 1\le i, j \le d. 
$$

The matrix $\Theta^{(M)}$ can be viewed as a diffusion propagator with an $M$-th order \emph{self-energy renormalization} to the diffusion constant.  
For any $4\le l \le M$, the property \eqref{two_properties0} shows that $(\Sele_l)_{xy}$ is translationally invariant, and $(\Sele_l)_{0x}$ is symmetric in $x$. Thus its Fourier transform in $x$, 
\smash{$\wh\Sele_l(p)$}, is a symmetric function in $p$. 
Using the properties \eqref{4th_property0} and \eqref{3rd_property0}, it is easy to check that for $|p|\ll W^{-1}$ and any constant $\tau>0$, 
\be\label{FT_of_self}
\wh\Sele_l(p)=W^{-(l-2)d/2}   W^2 p\cdot \cal D_l(z) p + \OO\left[W^{-(l-2)d/2+\tau} \left(\eta + W^3|p|^3\right)\right], 
\ee
where $\cal D_l(z)$ is defined by 
$$ (\cal D_l)_{ij} (z):=W^{(l-2)d/2}\cdot \frac{1}{2}\sum_{i,j}\frac{x_i x_j}{W^2}(\Sele_l)_{0x}(z),\quad 1\le i, j \le d. $$
Note that 
the main error in \eqref{FT_of_self} comes from the $l=4$ case. Using \eqref{diffu_Theta} and \eqref{FT_of_self}, we can write the Fourier transform of $\Theta^{(M)}$ as 
\be\label{TM}
\wh \Theta^{(M)}(p) = \frac{r(E)\left[1+\OO(\eta+W^2 |p|^2)\right]}{\eta+ W^2 p\cdot \cal D_{eff}^{(M)}(z) p+ \OO\left( \eta^2+ W^{-d+\tau} \eta + W^3|p|^3\right)},
\ee
for $|p|\ll W^{-1}$, where the renormalized effective diffusion coefficient is defined as
\be  \cal D_{eff}^{(M)}(z):= \cal D(E) + r(E) \sum_{l=4}^M W^{-(l-2)d/2}\cal D_l(z).\ee
Therefore, $\Theta^{(M)}$ is a diffusion propagator with $\Sele_l$ being the $l$-th order $\self$. 

The matrix $\cal G^{(M)}$ in \eqref{bound_calG} represents the collective effects of local recollisions. Notice that each row of $\cal G^{(M)}$ has a summable decay and its $\ell^1$ norm is small in the sense that $\sum_y |\cal G_{x y}^{(M)}| =\OO(W^{-d/2+\tau})$. In particular, this shows that $|\wh {\cal G}^{(M)}(p)|=\OO(W^{-d/2+\tau})$. 
Thus the Fourier transform of \eqref{E_locallaw} is given by 
\be\label{T law}
\begin{split}
 \sum_x \E T_{0x} (z) e^{\ii p \cdot x}  = 
\wh \Theta^{(M)}(p)\left[|m|^2+\OO(W^{-d/2+\tau})\right] + \OO(W^{-M}) , 
\end{split} 
\ee
as long as $M$ is sufficiently large. 
By \eqref{EGxy}, the Fourier transform of $\E |G_{xy}|^2$ has a similar behavior for $|p|\ll W^{-1}$. 
To summarize, we have obtained the following corollary from Theorem \ref{main thm2} and Corollary \ref{main_cor}.
\begin{corollary}\label{main_cor2}
Under the assumptions of Theorem \ref{main thm2}, let $M$ be a large constant satisfying $M\ge 4\log_W L$. For $p\in \mathbb T_L^d$ with $|p|\ll W^{-1}$ and $W^2/L^{2-\e}\le \eta \ll 1$, \eqref{T law} and the following estimate hold for any small constant $\tau>0$:
\be\label{T law2}
\begin{split}
 \sum_x \E |G_{0x}(z)|^2 e^{\ii p \cdot x}  = \frac{r(E)\left[|m|^2+\OO(W^{-d/2+\tau})\right]}{\eta+ W^2 p\cdot \cal D_{eff}^{(M)}(z) p+ \OO\left( \eta^2 + W^{-d+\tau} \eta + W^3|p|^3 \right)} + \OO(W^{-M}) . 
\end{split} 
\ee
\end{corollary}

It is commonly believed in physics literature (see, e.g., \cite{Spencer1,Spencer2}) that
\be\label{dif}
 \sum_x \E|G_{0x}(z)|^2 e^{\ii p \cdot x}  \sim \frac 1 { \eta + a(p)} 
\ee
with $a(p)$ being a quadratic form of $p$ for small $|p|$ is a signature of quantum diffusion. Hence Corollary \ref{main_cor2} shows that the resolvent is diffusive for  $\eta \ge W^2/L^{2-\e} $.
The quantum diffusion for the Anderson's model was proved in \cite{ESY2008} for time scale 
$t \sim \lambda^{-2 - c }$ for some small constant $c > 0$. If we take the correspondence $t \sim \eta^{-1}$ and $\lambda \sim W^{-d/2}$, the result in \cite{ESY2008} amounts to establishing the quantum diffusion for  $\eta \sim W^{-d - c}$ in the current language. The quantum diffusion in \cite{ESY2008} was established for the unitary evolution $e^{\ii t H}$ instead of  \eqref{dif} in terms of the resolvent. {While the two formulations of the quantum diffusion are generally believed to be roughly equivalent, lots of works are still required to prove the quantum diffusion for the unitary evolution $e^{\ii t H}$ of random band matrices. However, we believe that there are no intrinsic difficulties for such results.}

The \emph{Thouless time} \cite{Edwards_1972,Thouless_1977, Spencer2} for random band matrices is defined to be 
the time for a particle to reach the boundary of the system, which is roughly $t_{Th}=L^2/W^2$ if we assume that the particle evolves as a diffusion. 
It is generally believed, at least heuristically, that the localization/delocalization and quantum diffusion properties of a disordered system can be determined by the behavior of the resolvent up to the Thouless time. Since $\eta$ and the time $t$ are dual variables, 
the assumption $\eta \gg W^/L^2$ in Corollary \ref{main_cor2} exactly corresponds to that the evolution time is less than the Thouless time. In other words, Corollary \ref{main_cor2}  establishes the quantum diffusion in resolvent sense up to the Thouless time. 

\subsection{$T$-expansion}\label{sec mainref}

The main tool to prove Theorems \ref{comp_delocal}, \ref{main thm0} and \ref{main thm2} is an expansion of the $T$-matrix 
up to arbitrarily high order. In \cite{delocal}, the $T$-matrix was shown to satisfy a $T$-equation to the leading order, which gives a $T$-expansion up to second order in $W^{-d/2}$ (i.e., up to order $W^{-d}$) as follows. 
From \eqref{defGT}, it is trivial to derive the following equation 
\be\label{T-equation}
T _{xy}=\Theta_{xy} ( |G_{y y}|^2- T_{yy})+\sum_{\al\ne y}\Theta_{x\al} (|G_{\al y}|^2- T_{\al y}) .
\ee
Since we have  
\[
T_{y y} =  |m|^2 s_{y y}|G_{y y}(z)|^2+ |m|^2 \sum_{\al \ne y} s_{y \al}|G_{\al y}(z)|^2 \le CW^{-d}|G_{y y}(z)|^2 + \sup_{\al\ne y} |G_{\al y}(z)|^2, 
\]
we expect that $T_{y y}=\OO(W^{-d})$ with high probability and thus is an error term. We now show that $|G_{\al y}|^2- T_{\al y}$ also gives a higher order term. Using the equation $(z+ m) m = -1$ for $m(z)$, we get that
\be\label{GmH}
G = - \frac 1 { z+m} + \frac 1 { z+m} (H+m) G \quad \Rightarrow \quad G-   m = - m  (H+m) G .
\ee
Here the expansion in terms of $H+m$, instead of $H$, can be viewed as a naive renormalized expansion.  
 Define $\E_x$ as the partial expectation with respect to the $x$-th row and column of $H$, i.e., $\E_x(\cdot) := \E(\cdot|H^{(x)}),$	where $H^{(x)}$ denotes the $(N-1)\times(N-1)$ minor of $H$ obtained by removing the $x$-th row and column. 
For simplicity, in this paper we will use the notations 
$$P_x :=\E_x , \quad Q_x := 1-\E_x.$$
Using \eqref{GmH}, we get that for $x\ne y$,
\begin{align*}
	&  |G_{xy}|^2 =P_x \left( G_{xy}  \overline G_{xy}  \right) + Q_x  |G_{xy}|^2  
	=  - P_x \Big[ \Big(m^2 G_{xy} + m \sum_\al h_{x\al }G_{\al y}\Big)  \overline G_{xy}\Big] + Q_x |G_{xy}|^2   .
\end{align*}
Using Gaussian integration by parts with respect to $h_{x\al }$, we obtain that
\begin{align}
	|G_{xy}|^2 
	&= Q_x    |G_{xy}|^2 - m^2 P_x | G_{xy} |^2  - m P_x \Big[  \sum_\al s_{x\al } \partial_{ h_{ \al x }} \big ( 
	G_{\al y}  \overline G_{xy}  \big )\Big] \nonumber\\
	& =  Q_x   |G_{xy}|^2+ m P_x   \Big[\sum_\al s_{x\al }\big ( G_{\al \al}- m \big ) |G_{xy}|^2 \Big] +  m P_x\Big(  \overline G_{xx}   \sum_\al s_{x\al }|G_{\al y}|^2  \Big)  \nonumber \\
	& =   |m|^2   \sum_\al s_{x\al } |G_{\al y}|^2 + \Omega_{xy} = T_{xy} + \Omega_{xy},\label{G-T}
\end{align}
where $\Omega_{xy}$ consists of diagonal error terms (i.e., terms depending on $ G_{\al \al}- m$ and $\overline G_{xx}-\overline m$) and fluctuations (i.e., terms of the form  $Q_x[\cdot]$). Inserting \eqref{G-T} into \eqref{T-equation}, we obtain that 
\be\label{T-equation2}
T _{xy} = \left[|m|^2+ \OO( W^{-d/2})\right] \Theta_{xy}  +   \sum_{\al \ne y}\Theta_{x \al}  \Omega_{\al y}  ,
\ee
if we have a diagonal estimate $G_{yy}=m+\OO(W^{-d/2})$ with high probability.



We expect the second term in \eqref{T-equation2} to be an error term. But we have that 
\be\label{sumTheta}\sum_{y}\Theta_{xy}(z)=\frac{|m(z)|^2 }{1-|m(z)|^2 }\sim \eta^{-1},\ee
which makes the error $\sum_{\al \ne y}\Theta_{x\al}\Omega_{\al y}$ bigger than the order of $\max_{\al, y}|\Omega_{\al y}|$ by a huge factor $\eta^{-1}$ if we bound the sum naively. Thus this error is very difficult to bound when $\eta \ll 1$ (in particular, when $\eta=W^2/L^{2-\e}$). The estimate of $\sum_{\al \ne y}\Theta_{x\al}\Omega_{\al y}$ can be improved by a fluctuation averaging lemma, which was first discovered in \cite{EYY bernoulli} and later extended to random band matrices in \cite{EKY_Average}.  This leads to, roughly speaking, the following bound in \cite{delocal}: for any small constant $\tau>0$,
\be
\big|\sum_{\al \ne y}\Theta_{x \al}  \Omega_{\al y}\big| \le \eta^{-1}  W^{-3d/2+\tau} \quad \text{with high probability for}\quad \eta\gg W^{-d}.
\ee
 It was noticed later \cite{PartIII} that one can take advantage of the decay of $\Theta_{x\al}$ and $\Omega_{\al y}$ with respect to $\al$ to improve the error estimate. In order to achieve the regime $L \ge W^C$ for an arbitrarily large constant $C>0$, the previous methods will require that $|\Omega_{\alpha y}| \ll \eta$, which is almost impossible to establish and very  likely to be incorrect.  

While the $T$-equation has drawbacks, it is already a big step towards the understanding of the $T$-matrix. Recall that $\Theta$ is a random walk expansion up to the time $\eta^{-1}$. Hence to prove that $T_{xy}\sim \Theta_{xy}$ for $\eta \sim W^{-d}$, it amounts to expanding the resolvent $(H-z)^{-1}$ at least $W^{d}$ times. This will generate a huge combinatorial factor $(W^d \times W^d ) !$ in calculating $\E |G_{xy}|^2$ using Gaussian contractions. This combinatorial factor makes it infeasible to use the naive expansion method even taking into account various renormalization simplifications in the calculations. The $T$-equation method bypasses the problem of analyzing the $(W^d \times W^d)!$ many error terms at the expense of showing that the error is bounded up to the accuracy $|\Omega_{\alpha y}| \ll \eta$. Returning to the current case with $W = L^\e$ and $\eta \sim W^2/L^2$, the naive expansion will generate $(L^2/W^2\times L^2/W^2) !$ many terms and it is again hopeless to analyze them. Thus we have to study the $T$-equation more deeply and seek for a crucial replacement of the bound $|\Omega_{\alpha y}| \ll  \eta$.

One key observation of this work is that main contributions to the term  $\sum_{\al \ne y}\Theta_{x \al}\Omega_{\al y}$ come from  self-energy related terms such as  $ (\Theta \wtSdelta^{(M)})^k \Theta$ in the Taylor expansion of $\Theta^{(M)}$. Suppose for now we replace the property \eqref{3rd_property0} by a stronger sum zero property
\be\label{pseduo_sum_zero} \sum_{x} (\Sele_l)_{xy} =0 .\ee
Together with the fact that $(\Sele_l)_{xy}$ is symmetric in $x$ and $y$, we can sum by parts twice in the expression 
$\sum_\al \Theta_{x\al}\, (\Sele_l)_{\al y}$  to get 
\be\label{label_Theta}
\left|(\Theta \Sele_l)_{x y}\right| \le  W^{-(l-4)d/2+\tau}  \sum_{\alpha} \left|\partial_{\alpha }^2 \Theta_{x \alpha }\right|B_{\al y}^2 \le \frac{W^{-(l-2)d/2+2\tau}}{\left(\|x-y\|_{ L}+W\right)^{d}},
\ee 
where we also used the bound \eqref{4th_property0} for $\Sele_l$ 
and $ | \partial_\al^2 \Theta_{x \al} | \lesssim \left(\|x-\al\|_{ L}+W\right)^{-d+\tau}$ for any constant $\tau>0$. 
(Strictly speaking, $\partial_{\alpha }^2 \Theta_{x \alpha }$ should be replaced by the second order difference of  $ \Theta_{x \alpha }$ in $\alpha$.)  Using this estimate, it is easy to get that for any small constant $\tau>0$,  
\be \Big|\big[(\Theta \wtSdelta^{(M)})^k \Theta\big]_{xy}\Big| \le W^{\tau} B_{xy}. \ee 
Although the row sums of $\Sele_l$ are not exactly equal to zero by \eqref{3rd_property0}, the $\eta$ factor in the error term will be small enough to cancel the factor from $\sum_\al \Theta_{x \al}\sim \eta^{-1}$. To summarize, the $\selfs$ in the $T$-expansion need to either satisfy a sum zero property or contain effectively an $\eta$ factor. 
If we take $\eta\to 0$ as $L\to\infty$, then an exact sum zero property will hold for the \emph{infinite space limit} of $\Sele_l$ (see equation \eqref{weaz} for a more precise statement). 
Hence we will call \eqref{3rd_property0} a \emph{sum zero property}. 


Our main task is thus to design an expansion method to derive a $T$-equation with the leading term $\Theta^{(n)}$ and an error of order   $\OO(W^{-(n+1)d/2})$ for any fixed $n\in \N$. But there will also be many other types of terms. Roughly speaking, we will derive an expression of the form 
\be\label{T-exp-intro}
T = \Theta^{(n)} + (\text{recollision term}) + (\text{higher order term}) + (\text{fluctuation term}) + (\text{error term}),
\ee
where the recollision term consists of expressions with coincidences in summation indices, the higher order term consists of expressions that are of order smaller than $W^{- nd/2}$, 
the fluctuation term consists of expressions that can be written into the form $\sum_x Q_x(\cdot)$ (which can be analyzed 
via the fluctuation averaging mechanism), and the error term can be neglected for all of our proofs. In the expansion process, we will need to give a precise construction of $\Theta^{(n)}$. Furthermore, the recollision, higher order and fluctuation terms will also need to be tracked relatively explicitly and some key structures (which we call the \emph{doubly connected structures}) 
need to be maintained in order to derive the final estimates on these terms. The expansion \eqref{T-exp-intro} is constructed inductively in $n$.  Roughly speaking, with the $T$-expansion \eqref{T-exp-intro} for a given $n$, 
we insert itself into a suitable subset of expressions in \eqref{T-exp-intro} to derive the $(n+1)$-th order $T$-expansion. The main technical difficulties are to verify the sum zero properties for the $\selfs$ $\Sele_l$ order by order, and to maintain the doubly connected structures for all the other expressions so that we can estimate them. 
We want to point out that the typical sizes of $\Theta_{xy}$ and $G_{xy}$ are of order $B_{xy}$ and \smash{$B_{xy}^{1/2}$}, respectively. Moreover, the row sums of $B_{xy}$ are bounded by $L^2/W^2$, while the row sums of $B_{xy}^{3/2}$ are bounded by $\OO(W^{-d/2}$). The 
doubly connected structures defined in Definition \ref{def 2net} below ensure that in each sum, we have at least a product of a $\Theta$ factor and a $G$ factor, so that the sum can be bounded independently of $L$.

 
The proof of the main results in this paper and \cite{PartII_high} can be roughly divided into the following three parts:
(i) construction of the $T$-expansion, (ii)  proof of the sum zero properties for the $\selfs$,  
(iii) proof of Theorems \ref{comp_delocal}, \ref{main thm0} and \ref{main thm2} using the $T$-expansion. 
In this paper, we will complete  (ii) and (iii), while  (i) and some estimates used in (iii) will be proved in the second paper of this series \cite{PartII_high}. 
We stress that our strategy to construct the $T$-expansion is not  a straightforward extension of the one used  in \cite{EKY_Average, PartIII}. In terms of the terminology to be  introduced in this paper, the expansions in \cite{EKY_Average, PartIII} are  \emph{local expansions}. The  construction of the full $T$-expansion will require the much more sophisticated \emph{global expansions}, which will be explained in  
 Section \ref{subsec global} and Section \ref{sec strategy}.

 
 \medskip

The rest of this paper is organized as follows. 
In Section \ref{sec outline}, we introduce the graphical tools and use them to define the core concepts of this paper---the $T$-expansion and $\selfs$. In Section \ref{sec_basiclocal}, we introduce the basic graph operations that are used to construct the $T$-expansion. In Section \ref{sec lower_order}, we give some examples of how to use the basic graph operations to obtain some lower order $T$-expansions. In Section \ref{sec pfmain}, we give the proof of Theorem \ref{main thm}, a slightly weaker version of Theorem \ref{main thm0},  based on some lemmas that will be proved in Sections \ref{sec double}--\ref{sec infspace} and the second paper of this series \cite{PartII_high}. We will also discuss the restriction $d \ge 8$ after the continuity estimate, Lemma \ref{lem: ini bound}. 
In Section \ref{sec double}, we introduce the doubly connected structures of the graphs. In Section \ref{sec infspace}, we study the infinite space limits of $\selfs$. Finally, the proofs of Theorem \ref{comp_delocal}, Theorem \ref{main thm0}, Theorem \ref{main thm2} and Corollary \ref{main_cor} will be presented in Section \ref{sec pf0}. In Section \ref{sec strategy}, we discuss some key new ideas in \cite{PartII_high} that are used to prove the relevant lemmas in Section \ref{sec pfmain}.   

\medskip
\noindent{\bf Acknowledgements.} J.Y. would like to thank Xinyi Cui and Xuyang Tian for helpful discussions.

\section{$T$-expansion and $\selfs$}\label{sec outline}



The major part of this paper is devoted to proving Theorem \ref{main thm0}. We will mainly focus on proving the following slightly weaker form of Theorem \ref{main thm0}, which assumes a stronger compactly supported condition on $\psi$. The reason is that in the setting of Theorem \ref{main thm}, the sum zero property \eqref{weaz} can be stated in a cleaner form. In Section \ref{sec pf0}, we will show how to adapt the proof for Theorem \ref{main thm} to the proof of Theorem \ref{main thm0}. 

\begin{theorem}\label{main thm}
Under the assumptions of Theorem \ref{main thm0}, we assume in addition that $\psi$ in Definition \ref{var profile} is a compactly supported smooth function. Then for any small constants $\e,\tau>0$ and large constant $D>0$, the estimate \eqref{locallaw} holds. 
\end{theorem}

In this section, we introduce the following three key tools for the  proof of Theorem \ref{main thm}: $\selfs$ in Definition \ref{collection elements}, the $T$-expansion in Definition \ref{defn genuni}, and the $\incomp$ 
in Definition \ref{defn incompgenuni}. With these tools, we will give an outline of the proof of Theorem \ref{main thm} in Section \ref{sec pfmain}. 
Before stating the $T$-expansion, we introduce two deterministic matrices 
\be\label{Thetapm2}
S^+(z):=\frac{m^2(z) S}{1-m^2(z) S}, \quad S^-(z):=\overline S^+(z),  
\ee
which satisfy the following estimate \eqref{S+xy}. For simplicity, throughout the rest of this paper, we abbreviate
\be\label{Japanesebracket} |x-y|\equiv \|x-y\|_L,\quad \langle x-y \rangle \equiv \|x-y\|_L + W.\ee

\begin{lemma} \label{lem deter}
Suppose Assumptions \ref{assmH} and \ref{var profile} hold, and $z=E+\ii \eta$ with $E\in (-2+\kappa,2-\kappa)$ for a constant $\kappa>0$. 
Then for any constants $\tau, D>0$, we have that
\be\label{S+xy}|S^\pm_{xy}(z)| \lesssim    W^{-d}\mathbf 1_{|x-y|\le W^{1+\tau}} +  \langle x-y\rangle^{-D} . \ee 
\end{lemma}
\begin{proof}
The estimate \eqref{S+xy} is a folklore result. A formal proof for the $d=1$ case is given in equation (4.21) of \cite{PartII}. This proof can be extended directly to the general $d$ case. 
\end{proof}

In this paper, we adopt the following convention of stochastic domination \cite{EKY_Average}. 

\begin{definition}[Stochastic domination and high probability event]\label{stoch_domination}
	(i) Let
	\[\xi=\left(\xi^{(W)}(u):W\in\mathbb N, u\in U^{(W)}\right),\hskip 10pt \zeta=\left(\zeta^{(W)}(u):W\in\mathbb N, u\in U^{(W)}\right),\]
	be two families of non-negative random variables, where $U^{(W)}$ is a possibly $W$-dependent parameter set. We say $\xi$ is stochastically dominated by $\zeta$, uniformly in $u$, if for any fixed (small) $\tau>0$ and (large) $D>0$, 
	\[\mathbb P\bigg[\bigcup_{u\in U^{(W)}}\left\{\xi^{(W)}(u)>W^\tau\zeta^{(W)}(u)\right\}\bigg]\le W^{-D}\]
	for large enough $W\ge W_0(\tau, D)$, and we will use the notation $\xi\prec\zeta$. 
	If for some complex family $\xi$ we have $|\xi|\prec\zeta$, then we will also write $\xi \prec \zeta$ or $\xi=\OO_\prec(\zeta)$. 
	
	\vspace{5pt}
	\noindent (ii) As a convention, for two deterministic non-negative quantities $\xi$ and $\zeta$, we will write $\xi\prec\zeta$ if and only if $\xi\le W^\tau \zeta$ for any constant $\tau>0$. 
	
	
	\vspace{5pt}
	\noindent (iii) We say that an event $\Xi$ holds with high probability (w.h.p.) if for any constant $D>0$, $\mathbb P(\Xi)\ge 1- W^{-D}$ for large enough $W$. More generally, we say that an event $\Omega$ holds $w.h.p.$ in $\Xi$ if for any constant $D>0$,
	$\P( \Xi\setminus \Omega)\le W^{-D}$ for large enough $W$.
\end{definition}

\subsection{Second order $T$-expansion} \label{sec Thetaexp} 
 We generalize  the $T$-variable in \eqref{defGT} to the following $T$-variables with three subscripts: 
 \be\label{general_T}
T_{x,yy'}:=|m|^2\sum_\al s_{x\al}G_{\al y}\overline G_{\al y'},\quad \text{and}\quad T_{yy',x}:=|m|^2\sum_\al G_{ y\al }\overline G_{y' \al}s_{\al x}. 
\ee
By definition,  the $T$-variable in \eqref{defGT} can be written as $T_{xy}\equiv T_{x,yy}$.
Our $T$-expansion will be formulated in terms of these generalized $T$-variables. In this subsection, we define the second order $T$-expansion of $T_{x,yy'}$ using the following $\Theta$ expansion, which is derived from Gaussian integration by parts. The expansion of $T_{yy',x}$ can be obtained by considering the transposition of $T_{x,yy}$.  

  \begin{lemma}[$\Theta$-expansion]   \label{expandtheta}
In the setting of Theorem \ref{comp_delocal}, consider the expression $|m|^2\sum_\al s_{x\al}G_{\al y}  \overline G_{\al y' }  f(G)$, where  
$f$ is a differentiable function of $G$. 
Then we have the identity 
\begin{align}
  |m|^2 \sum_{\al}s_{x\al}G_{\al y}  \overline G_{\al y'}  f(G) = m  \Theta_{xy}   \overline G_{yy'} f(G)  +  m \sum_{\al,\beta} \Theta_{x\al} s_{\al\beta}   (G_{\beta \beta}-m) G_{\al y} \overline G_{\al y'} f(G)\nonumber \\
 +m \sum_{\al,\beta} \Theta_{x\al}  s_{\al\beta} ( \overline G_{\al\al} -\overline m)G_{\beta y}\overline G_{\beta y'}  f(G)- m\sum_{\al,\beta}  \Theta_{x\al} s_{\al\beta}  G_{\beta y} \overline G_{\al y'} \partial_{ h_{\beta \al}}f(G) +\cal Q_\Theta ,  \label{Otheta}
 \end{align} 
 where
 \begin{align*}
	\cal Q_\Theta&:=  \sum_{\al}\Theta_{x\al} Q_\al \left[G_{\al y} \overline G_{\al y'} f(G)\right]- m  \Theta_{xy}  Q_y\left[  \overline G_{yy'} f(G)\right] - \sum_{\al,\beta} m \Theta_{x\al} s_{\al\beta}  Q_\al\left[   (G_{\beta \beta}-m) G_{\al y} \overline G_{\al y'} f(G)\right]  \\
	& - \sum_{\al,\beta} m\Theta_{x\al}  s_{\al\beta}  Q_\al\left[    \overline G_{\al\al} G_{\beta y}\overline G_{\beta y'}  f(G)\right] + \sum_{\al,\beta} m \Theta_{x\al} s_{\al\beta} Q_\al\left[  G_{\beta y} \overline G_{\al y'} \partial_{ h_{\beta\al}}f(G)\right].
\end{align*}
\end{lemma}

\begin{proof}
With \eqref{GmH} and the identity $|m|^2 S= \Theta - |m|^2  \Theta S$, we can write that
\begin{align}
 & |m|^2 \sum_{\al}s_{x\al} P_\al\left[G_{\al y}  \overline G_{\al y'}  f(G)\right]  =   \sum_\al \left[ \Theta_{x\al} - |m|^2 (\Theta S)_{x\al}\right] P_\al\left[G_{\al y} \overline G_{\al y'} f(G)\right] \label{GGfH}\\
& = - \sum_\al  |m|^2 (\Theta S)_{x\al} P_\al\left[G_{\al y} \overline G_{\al y'} f(G)\right] +\sum_\al  \Theta_{x\al} P_\al\left[ (m\delta_{\al y} -m^2G_{\al y} - m (HG)_{\al y}) \overline G_{\al y'} f(G)\right].\nonumber
\end{align}
For the $ HG $ term, using Gaussian integration by parts we obtain that
\begin{align*}
   P_\al\bigg[ -m \sum_\beta h_{\al \beta}G_{\beta y} \overline G_{\al y'} f(G)\bigg] &= P_\al\bigg[ m \sum_\beta  s_{\al \beta}G_{\beta \beta} G_{\al y} \overline G_{\al y'} f(G) + m \sum_\beta s_{\al\beta} G_{\beta y} \overline G_{\al\al } \overline G_{\beta y'}  f(G)\bigg] \\
  &+   P_\al\bigg[ - m \sum_\beta s_{\al\beta}G_{\beta y} \overline G_{\al y'} \partial_{ h_{ \beta \al}}f(G)\bigg].
\end{align*}
Using this equation, we get that
\begin{align*}
& \sum_\al  \Theta_{x\al} P_\al\left[ (m\delta_{\al y} -m^2G_{\al y} - m (HG)_{\al y}) \overline G_{\al y'} f(G)\right]\\
& = m  \Theta_{xy}  P_y\bigg[  \overline G_{yy'} f(G)\bigg] + \sum_\al \Theta_{x\al}  P_\al\bigg[ m \sum_\beta  s_{\al\beta} (G_{\beta \beta}-m) G_{\al y} \overline G_{\al y'} f(G)\bigg] \\
&+ \sum_\al \Theta_{x\al}  P_\al\bigg[ m ( \overline G_{\al\al} -\overline m)\sum_\beta s_{\al\beta} G_{\beta y}\overline G_{\beta y'}  f(G)\bigg] +  \sum_\al \Theta_{x\al}  P_\al\bigg[ |m|^2\sum_\beta s_{\al\beta} G_{\beta y}\overline G_{\beta y'}  f(G)\bigg] \\
&- \sum_\al \Theta_{x\al}  P_\al\bigg[  m \sum_\beta s_{\al\beta}G_{\beta y} \overline G_{\al y'} \partial_{ h_{\beta\al}}f(G)\bigg].
\end{align*}
Plugging it into \eqref{GGfH}, writing $P_\al = 1-Q_\al$, and using the identity
\begin{align*}  
&\sum_{\al,\beta} |m|^2 \Theta_{x\al} s_{\al\beta}   P_\al\left[ G_{\beta y}\overline G_{\beta y'}  f(G)\right]  -  \sum_\al  |m|^2 (\Theta S)_{x\al} P_\al\left[G_{\al y} \overline G_{\al y'} f(G)\right] \\
=& \sum_\al  |m|^2 (\Theta S)_{x\al} Q_\al\left[G_{\al y} \overline G_{\al y'} f(G)\right] - \sum_{\al,\beta} |m|^2 \Theta_{x\al} s_{\al\beta}   Q_\al\left[ G_{\beta y}\overline G_{\beta y'}  f(G)\right]  ,
\end{align*} 
we can obtain \eqref{Otheta} after some simple calculations.
\end{proof}

Using the $\Theta$-expansion \eqref{Otheta}, we obtain the following second order $T$-expansion. 
  
\begin{lemma}\label{2nd3rd T}
Under the assumptions of Theorem \ref{comp_delocal}, we have that for any $\fa, \fb_1,\fb_2 \in \Z_L^d$, 
\begin{align}\label{seconduniversal}
& T_{\fa,\fb_1\fb_2}= 
m \Theta_{\fa\fb_1} \overline G_{\fb_1\fb_2}    + (\AT^{(>2)})_{\fa,\fb_1\fb_2} + (\QT^{(2)})_{\fa,\fb_1\fb_2}  ,
 \end{align}
 where 
 \begin{align}
  (\AT^{(>2)})_{\fa,\fb_1\fb_2}& :  =  m\sum_{x,y} \Theta_{\fa x} s_{xy} (G_{yy}-m) G_{x \fb_1}\overline G_{x \fb_2}  + m \sum_{x,y} \Theta_{\fa x} s_{xy}  ( \overline G_{xx} - \overline m) G_{y \fb_1}\overline G_{y \fb_2}, \label{Aho>2}  \\
    (\QT^{(2)})_{\fa,\fb_1\fb_2}& := \sum_x Q_x\left(\Theta_{\fa x} G_{x \fb_1}   \overline G_{x \fb_2} \right) - m Q_{\fb_1} \left( \Theta_{\fa\fb_1} \overline G_{\fb_1 \fb_2}\right)  \nonumber\\
    & -  m\sum_{x,y} Q_x\left[\Theta_{\fa x} s_{xy} (G_{yy}-m) G_{x \fb_1} \overline G_{x \fb_2} \right] - m \sum_{x,y} Q_x\left[ \Theta_{\fa x} s_{xy}  \overline G_{xx} G_{y\fb_1}\overline G_{y\fb_2} \right] .\label{QT>2}
 \end{align}
\end{lemma}

 \begin{proof}
 Taking $f(G)\equiv 1$ in Lemma \ref{expandtheta} and replacing $x,y,y'$ by $\fa, \fb_1,\fb_2$, we immediately conclude \eqref{seconduniversal}.
\end{proof}

We  will see in Section \ref{sec lower_order} that the third and fourth order $T$-expansions are already rather lengthy. For even higher order $T$-expansions, the number of terms will grow exponentially (actually there are about $n^{Cn}$ many terms in the $n$-th order $T$-expansion). These terms have complicated structures and we will use graphical notations to represent them.

\subsection{Graphical notations} \label{unisec}

Our goal is to expand the generalized $T$-variable $T_{\fa,\fb_1 \fb_2}$ for $\fa, \fb_1,\fb_2 \in \Z_L^d$. 
We represent these three indices by special vertices
\be\label{rep_abc}\fa\equiv \otimes, \quad \fb_1\equiv \oplus, \quad \fb_2\equiv \ominus,\ee
in the graphs. In other words, we use $\fa,\fb_1,\fb_2$ in expressions, and draw them as $\otimes,\oplus,\ominus$ in the graphs. 
Now we first introduce the atomic graphs, and the concept of subgraphs. 

\begin{definition}[Atomic graphs] \label{def_graph1} 
Given a standard oriented graph with vertices and edges, we assign the following structures and call the resulting graph an atomic graph.  

\begin{itemize}
%
%

\item {\bf Atoms:} We will call the vertices atoms (vs. molecules in Definition \ref{def_poly} below). Each graph has some external atoms and internal atoms. The external atoms represent external indices whose values are fixed, while internal atoms represent summation indices that will be summed over. In particular, each graph in the expansions of $T_{\fa,\fb_1\fb_2}$ has the following 
external atoms: one $\otimes$ atom representing  the $\fa$ index, one $\oplus$ atom representing  the $\fb_1$ index, and one $\ominus$ atom representing the $\fb_2$ index (where some of them can be the same atom). 
By fixing the value of an internal atom, it will become an external atom; by summing over an external atom, it will become an internal atom. 
 
\item {\bf Regular weights:}  A regular weight on the atom $x $ represents a $G_{xx}$ or $\overline G_{xx}$ factor. 
 Each regular weight has a charge, where ``$+$" charge indicates that the weight is a $G$ factor, represented by a blue solid $\Delta$, and ``$-$" charge indicates that the weight is a $\overline G$ factor, represented by a red solid $\Delta$. 
%
%

 \item \noindent{\bf Light weights}:  
Corresponding to the regular weights defined above, we define the light weights representing $G_{xx}-m$ and $\overline G_{xx}-\overline m$. 
They are drawn as blue or red hollow $\Delta$ in graphs depending on their charges. 

%

\item {\bf Edges:} The edges are divided into the following types. 

\begin{enumerate}

\item{\bf  Solid edges:} A solid edge represents a $G$ factor. More precisely, 
\begin{itemize}
\item each oriented edge from atom $\alpha$ to atom $\beta$ with $+$ charge represents a $G_{\al\beta}$ factor; 
\item each oriented edge from atom $\alpha$ to atom $\beta$ with $-$ charge represents a $\overline G_{\al\beta}$ factor.
\end{itemize}
The plus $G$ edges will be drawn as blue solid edges, while minus $G$ edges will be drawn as red solid edges. In this paper, whenever we say ``$G$ edges", we mean both the plus and minus $G$ edges.

\item {\bf Waved edges:} We have neutral black, positive blue and negative red waved edges:  
\begin{itemize}
\item a neutral waved edge between atoms $x$ and $y$ represents an $s_{xy}$ factor; 

\item a blue waved edge of positive charge between atoms $x$ and $y$ represents a $S^+_{xy}$ factor;

\item a red waved edge of negative charge between atoms $x$ and $y$ represents a $S^-_{xy}$ factor.
\end{itemize}



\item {\bf $\Dashed$ edges:} 
A $\dashed$ edge connecting atoms $x$ and $y$ represents a ${\Theta}_{xy}$ factor;
 we draw it as a double-line edge between atoms $x$ and $y$. 

\item {\bf Dotted edges:} A dotted line connecting atoms $\al$ and $\beta$  represents the factor $\mathbf 1_{\al=\beta}\equiv \delta_{\al\beta}$; a dotted line with a cross ($\times$) represents the factor $\mathbf 1_{\al\ne \beta} \equiv  1-\delta_{\al\beta} $. There is at most one dotted or $\times$-dotted edge between each pair of atoms. 
By definition, a $\times$-dotted edge between the two ending atoms of a $G$ edge  indicates that this $G$ edge is off-diagonal. We also allow for dotted edges between external atoms.   
\end{enumerate}
The orientations of non-solid edges do not matter. The edges between internal atoms are called \emph{internal edges};  the edges with at least one end at an external atom are called \emph{external edges}.

\item{\bf $P$ and $Q$ labels:} Some solid edges and weights may have a label $P_x$ or $Q_x$, where $x$ is an atom in the graph. Moreover, each edge or weight can have at most one $P$ or $Q$ label. 

\item{\bf Coefficients:} There is a coefficient (which is a polynomial of $m$, $m^{-1}$, $(1-m^2)^{-1}$ and their complex conjugates) associated with each graph.  


\end{itemize}

\end{definition}

\begin{definition}[Sugraphs]\label{def_sub}
A graph $\cal G_1$ is said to be a subgraph of $\cal G_2$, denoted by $\cal G_1\subset \cal G_2$, if every graphical component of $\cal G_1$ is also in $\cal G_2$. Moreover, $\cal G_1 $ is a proper subgraph of $\cal G_2$ if  $\cal G_1\subsetneq \cal G_2$. 
Given a subset $\cal S$ of atoms in a graph $\cal G$, the subgraph induced on $\cal S$ 
refers to the subgraph of $\cal G$ with atoms in $\cal S$ as vertices, the edges between these atoms, and the weights on these atoms.  
\end{definition}

  \begin{example}
   As an example, we draw the graphs for $\AT^{(>2)}$ in \eqref{Aho>2}: 
\be  \label{Aho2} 
\parbox[c]{0.3\linewidth}{\includegraphics[width=4cm]{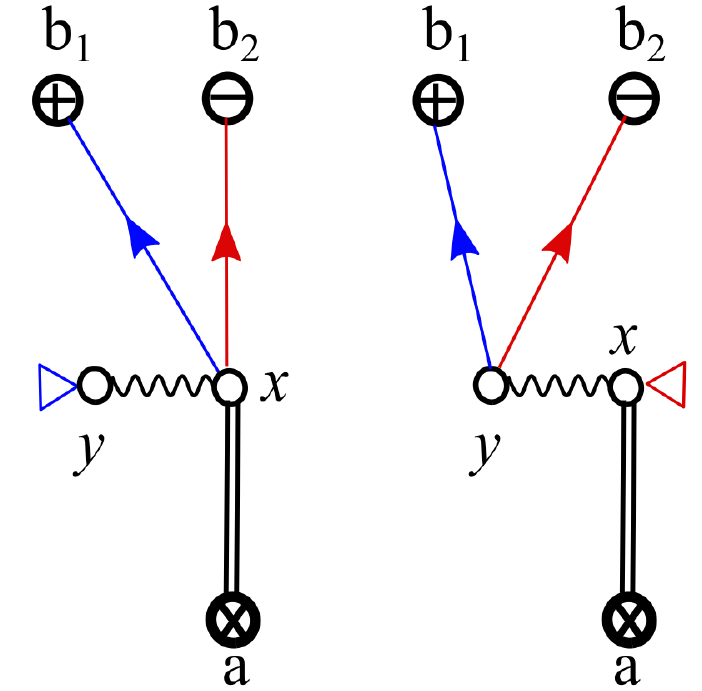}} 
\ee
For conciseness, we do not draw the coefficients of these graphs.
\end{example}

To each graph, we assign a value as follows.
  
 \begin{definition}[Values of graphs]\label{ValG} 
For an atomic graph $\mathcal G$, we define its value, denoted by $\llbracket \cal G\rrbracket$, as an expression obtained as follows. We first take the product of 
all the edges, all the weights and the coefficient of $\cal G$. 
Then for the edges and weights with the same $P_x$ or $Q_x$ label, we group them together and apply $P_x$ or $Q_x$ to them. Finally, we sum over all the internal indices represented by the internal atoms. The values of the external indices are fixed by their given values. 
 For a linear combination of graphs  $\sum_i c_i \cal G_i$, where $\{c_i\}$ is a sequence of coefficients and $\{\cal G_i\}$ is a sequence of graphs, we define  its value by 
 $$ \Big\llbracket\sum_i c_i  \cal G_i\Big\rrbracket=\sum_i c_i \left\llbracket  \cal G_i\right\rrbracket.$$
For simplicity, we will abuse the notation by identifying  a  graph (which is a geometric object) with its value (which is an analytic expression). 
 \end{definition}
\begin{example}  As an example of Definition \ref{ValG}, we write down the value for the following graph:
  \be\nonumber
 \parbox[c]{0.25\linewidth}{\includegraphics[width=\linewidth]{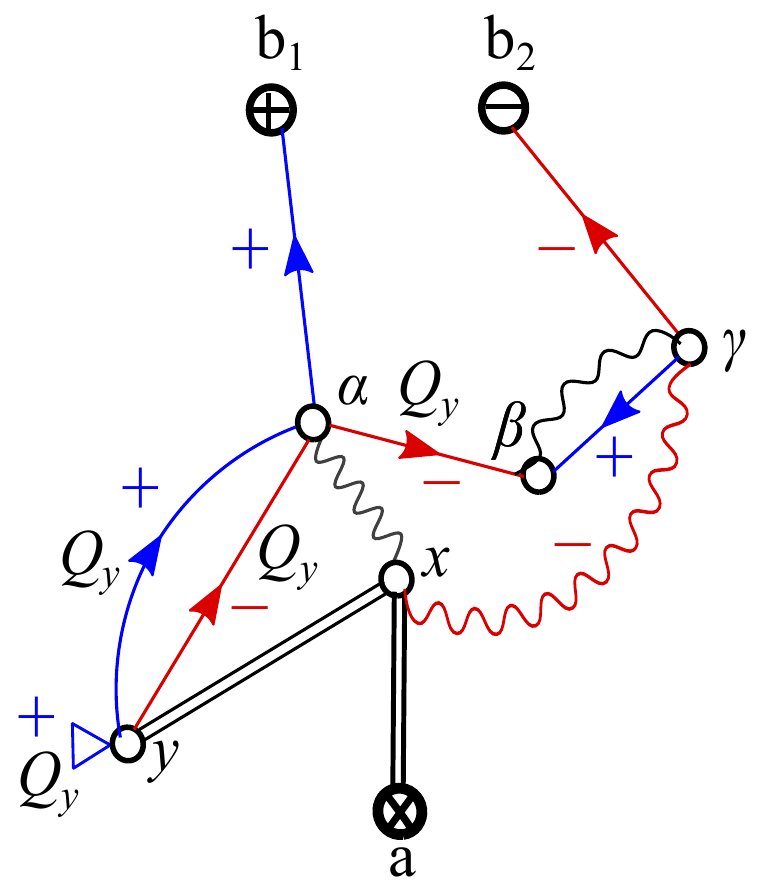}} \quad 
 = \sum_{x,y,\al,\beta,\gamma} \Theta_{\fa x}s_{x\al}S^-_{x\gamma}s_{\beta\gamma}G_{\al\fb_1} G_{\gamma\beta} \overline G_{\gamma\fb_2} \Theta_{xy}Q_y\left[\overline G_{\al\beta}  | G_{y\al }|^2 (G_{yy}-m)\right].
 \ee
\end{example}

Next, we introduce the concept of \emph{regular graphs}, which include (almost) all the graphs appearing in this paper, and a stronger concept of \emph{normal regular graphs}.

 \begin{definition}[Normal regular graphs]  \label{defnlvl0} 
 We say an atomic graph $\cal G$ is {\bf regular} if it satisfies the following properties:
\begin{itemize}
\item[(i)] it is a connected graph that contains at most $\OO(1)$ many atoms and edges;
\item[(ii)] all the internal atoms are connected together through paths of waved and $\dashed$ edges;
\item[(iii)] there are no dotted edges between internal atoms.
\end{itemize}
 Moreover, we say a regular graph is {\bf normal} if it satisfies the following additional property:
\begin{itemize}
\item[(iv)] any pair of atoms $\al$ and $\beta$ in the graph are connected by a $\times$-dotted edge \emph{if and only if} they are connected by a $G$ edge.
\end{itemize}

 \end{definition}

By this definition,  every $G$ edge in a normal regular graph is off-diagonal, while all the diagonal $G$ factors will be represented by weights. There are two reasons for introducing the property (iv): {(1)} in Definition \ref{def scaling}, we need to distinguish between the diagonal and off-diagonal $G$ entries; {(2)} the weight expansion (cf. Definition \ref{Ow-def} below) of diagonal $G$ entries and the edge expansions (cf. Definitions \ref{multi-def}, \ref{GG-def}, \ref{GGbar-def} below) of off-diagonal $G$ entries are very different in nature. 

   \begin{definition} [Scaling order] \label{def scaling}
   Given a normal regular  graph $\cal G$, we define its scaling order as 
\begin{align}
\text{ord}(\cal G):&= \#\{\text{off-diagonal }  G  \text{ edges}\} + \#\{\text{light weights}\} + 2\#\{ \text{waved edges}\}  + 2 \#\{\text{$\dashed$ edges}\} \nonumber\\
&  - 2\left[ \#\{\text{internal atoms}\}- \#\{\text{dotted edges}\} \right]. \label{eq_deforderrandom2}
\end{align}
Here each dotted edge in a normal regular graph means that an internal atom is equal to an external atom, so we lose one free summation index. 
The concept of scaling order can  be also defined for subgraphs. 
  \end{definition} 

The motivation behind this definition is as follows. Consider the Wigner ensemble with $W=L$. 
By \eqref{subpoly}, \eqref{thetaxy} and \eqref{S+xy}, each waved edge is of order $\OO(W^{-d})$ and each $\dashed$ edge is of order $\OO_\prec(W^{-d})$. 
Moreover, if we know that $|G_{xy}-m\delta_{xy}|\prec W^{-d/2}$, then each off-diagonal $G$ edge or light weight is bounded by $\OO_\prec(W^{-d/2})$.  Finally, each summation leads to a factor $W^d$. Hence it is easy to obtain the bound 
$$\llbracket \cal G\rrbracket\prec W^{-\text{ord}(\cal G)\cdot d/2}.$$ 
Later in Lemma \ref{no dot}, 
we will show that this bound holds even if $L\gg W$ as long as the graph satisfies the doubly connected property to be introduced in Section \ref{sec doublenet}. 

In the following proof whenever we say the order of a graph, we are referring to its scaling order. We emphasize that in general the scaling order does not imply the ``order of the graph value" directly. 


\subsection{$\Selfs$}\label{sec self}

The $T$-expansion is defined using a collection of special sums of deterministic graphs, which satisfy some important properties given by Definition \ref{collection elements} below. Following the notations in Feynman diagrams, we call them ``$\selfs$". 
As we explained before, 
the sum zero property \eqref{3rd_property0} of the $\selfs$ is one of the key reasons why we can define the $T$-expansion up to any order. In previous works \cite{delocal,PartIII}, the $T$-expansion can only be performed to third order without using the concept of $\selfs$.

\begin{definition}[$\Selfs$]\label{collection elements}
Under the assumptions of Theorem \ref{main thm}, for a fixed $l\in \N$, let $\Sele_l(z)\equiv  \Selek_{l}^L(z)$ be a deterministic matrix depending on $m(z)$, $S$, $S^\pm(z)$ and $\Theta(z)$ only, and satisfying the following properties. (In this paper, we will often omit the dependence on $L$ in $\Selek_{l}^L$.)


 \begin{itemize}
\item[(i)] For any $x,y \in \Z_L^d$, $  (\Selek_{l})_{xy}$ is a sum of at most $C_l$ many deterministic graphs of scaling order $l$ and with external atoms $x$ and $y$. Here $C_l$ is a large constant depending on $l$. Some graphs, say $\cal G$, in $\Sele_l$ can be diagonal matrices satisfying $  \cal G_{xy}= \cal G_{xx} \delta_{xy}$, i.e. there is a dotted edge between the atoms $x$ and $y$. 

\item[(ii)] $ \Selek_{l} (z)$ satisfies the properties \eqref{two_properties0}--\eqref{3rd_property0}.

\item[(iii)]  For any $x,y \in \Z^d$ and $z_L:= E+\ii\eta_L$ with $\eta_L \in [W^{2}/L^{2-\tau} , L^{-\tau}]$ for a small constant $\tau>0$, we denote the infinite space limit (with $W$ being fixed) of $ (\Selek^L_{l})_{xy}(z_L)$ by $ (\Seleinf_{l})_{xy}(E):=\lim_{L\to \infty} (\Selek^L_{l})_{xy}(z_L)$,  which is independent of $L$ and with  $\eta_\infty = 0$.  
\end{itemize} 
We call $\Selek_l$ the $l$-th order $\self$ ($\Selek_l$ will be unique from our construction) and graphically 
we will use a square, $\square$, between atoms $x$ and $y$ with a label $l$ to represent $(\Sele_l)_{xy}$. 
\end{definition}


We will show that the infinite space limits of the $\selfs$ satisfy the following properties: 
%
\be\label{two_properties0V}
\Seleinf_{l}(x, x+a) = \Seleinf_{l}(0,a), \quad  \Seleinf_{l}(0, a) =\Seleinf_{l} (0,-a), \quad \forall \ x,a\in \Z^d,
\ee 
\be\label{4th_property0V}
\left| (\Seleinf_{l})_{0x}(E) \right| \le W^{-ld/2} \frac{W^{2d-4} }{\langle x\rangle^{2d-4-\tau}}, \quad \forall \ x\in \Z^d, 
\ee
\be\label{3rd_property0 diff}
\left| (\Selek^L_{l})_{0x}(z)-(\Seleinf_{l})_{0x}(E) \right| \le W^{-ld/2} \frac{\eta W^{2d-6} }{\langle x\rangle^{2d-6-\tau}}, \quad \forall \  x\in \Z_L^d \subset \Z^d,\ \eta\in [W^{2}/L^{2-\tau}, L^{-\tau}],
\ee
\begin{equation}\label{weaz}
\sum_{x\in \Z^d} (\Seleinf_{l})_{0x}(E)=0.
\end{equation} 
In \eqref{3rd_property0 diff}, with slight abuse of notation, we identify the torus $\Z_L^d$ in \eqref{ZLd} as a subset of $\Z^d$. The properties \eqref{two_properties0V} and \eqref{4th_property0V} take the same forms as the properties \eqref{two_properties0} and \eqref{4th_property0}. The property \eqref{weaz} is an exact \emph{sum zero property} and is thus stronger than \eqref{3rd_property0}. 
These properties will be proved in Lemma \ref{cancellation property}.

By Definition \ref{def scaling}, the scaling order of a deterministic graph can only be even. Moreover, every nontrivial $\self$ $\Selek_{l}$ used in this paper has scaling order $\ge 4$. Hence we always have  
\be\label{trivial conv}
\Sele_1=\Sele_2=\Sele_3=0, \quad \text{and}\quad \Sele_{2l+1}:=0,\quad l\in \N.
\ee

By property \eqref{two_properties0V}, $\Seleinf_l$ is translationally invariant and symmetric (so are all the deterministic graphs in this paper by Lemma \ref{lem Rsymm}). 
The properties \eqref{4th_property0} and \eqref{4th_property0V} show that the rows of $\Selek_{l}$ or $\Seleinf_{l}$ are absolutely summable, i.e.,  
 for any constant $\tau>0$,
 \be\label{sum S2l}\sum_{x}\left| (\Selek_{l})_{0x}  \right| \le W^{-(l-2)d/2 + \tau},\quad \sum_{x}\left| (\Selek_{l}^\infty)_{0x} \right| \le W^{-(l-2)d/2 + \tau} .
 \ee
The bound \eqref{3rd_property0} is stronger than the first estimate in \eqref{sum S2l} by an extra $\eta$ factor, which, as discussed in Section \ref{sec mainref}, is crucial for our proof.  

The property \eqref{3rd_property0 diff} controls the difference between $(\Selek_{l})_{0x}(z)$ and $(\Seleinf_{l})_{0x}(E)$. 
The property \eqref{3rd_property0} actually 
can be derived from \eqref{4th_property0V}, \eqref{3rd_property0 diff}, and the sum zero property \eqref{weaz} for $\Seleinf_l$.  More precisely, 
\begin{align*}
 \Big|\sum_{x\in \Z_L^d} (\Selek_{l})_{0x}(z)\Big| &=  \Big|\sum_{x\in \Z_L^d} (\Selek_{l})_{0x}(z) - \sum_{x\in \Z^d} (\Seleinf_{l})_{0x}(E)\Big|  \le \sum_{|x|\le L/2} \left| (\Selek_{l})_{0x}(z) -(\Seleinf_{l})_{0x}(E) \right| + \sum_{|x|>L/2}\left|  (\Selek_{l})_{0x}(z) \right| \\
& \le \sum_{|x|\le L/2} W^{-ld/2} \frac{\eta W^{2d-6} }{\langle x\rangle^{2d-6-\tau}} + \sum_{|x|>L/2}W^{-ld/2} \frac{W^{2d-4} }{\langle x\rangle^{2d-4-\tau}} \\
&\lesssim L^\tau \left(\eta+\frac{W^2}{L^2}\right)W^{-(l-2)d/2} \le 2L^\tau \eta W^{-(l-2)d/2},
\end{align*}
where in the first step we used  \eqref{weaz}, in the third step we used \eqref{4th_property0V} and \eqref{3rd_property0 diff}, and in the last step we used $\eta\gg W^2/L^2$. 
 
The $l$-th order $\self$ $\Sele_l$ in this paper is constructed through a specific expansion procedure of the $T$-variables. In general, if a different expansion procedure is used, a different $l$-th order $\self$ may be obtained. Although we expect the $\selfs$ constructed in different procedures to be the same up to negligible errors, this property is not needed in this paper and we will not pursue it. 

\subsection{Definition of the $T$-expansion}\label{sec Texp2}
 
Given $n\in \N$, we will define the $n$-th order $T$-expansion in Definition \ref{defn genuni}, which is an extension of the second order $T$-expansion in \eqref{seconduniversal}. To this end, we first introduce the following two types of graphs. 

\begin{definition}[Recollision graphs and $Q$-graphs]\label{Def_recoll}

(i) We say a graph is a $\oplus$/$\ominus$-recollision graph, if there is at least one dotted edge connecting $\oplus$ or $\ominus$ to an internal atom. In other words, a recollision graph represents an expression where we set at least one summation index to be equal to $\fb_1$ or $\fb_2$.

\medskip
\noindent(ii) We say a graph is a $Q$-graph if all  $G$ edges and $G$ weights  in the graph have the same $Q$ label with a specific atom $x$, i.e., all $Q$ operators are given by the same $Q_x$. 


\end{definition}

We now  define a general $n$-th order $T$-expansion for any fixed $n\in \N$. Besides the properties in Definition \ref{defn genuni}, the graphs in the definition satisfy several additional properties to be stated in Definition \ref{def genuni}. 

\begin{definition} [$n$-th order $T$-expansion]\label{defn genuni}
Fix any $n\in \N$ and let $D>n$ be an arbitrary large constant. For $\fa, \fb_1,\fb_2 \in \Z_L^d$, an $n$-th order $T$-expansion of $T_{\fa,\fb_1\fb_2}$ with $D$-th order error is an expression of the following form: 
\be\label{mlevelTgdef}
\begin{split}
T_{\fa,\fb_1 \fb_2}&= m  \Theta_{\fa \fb_1}\overline G_{\fb_1\fb_2} +m   (\Theta \Sdelta^{(n)}\Theta)_{\fa\fb_1} \overline G_{\fb_1\fb_2} \\
&+ (\PTn)_{\fa,\fb_1 \fb_2} +  (\ATn)_{\fa,\fb_1\fb_2}  + (\QTn)_{\fa,\fb_1\fb_2}  +  (\Err_{n,D})_{\fa,\fb_1\fb_2}.
\end{split}
\ee
The graphs on the right side depend only on $n$, $D$, $m(z)$, $S$, $S^{\pm}(z)$, $\Theta(z)$ and $G(z)$, but do not depend on $W$, $L$ and $d$ explicitly. Moreover, they satisfy the following properties with $C_n$ and $C_D$ denoting large constants depending on $n$ and $D$, respectively.
\begin{enumerate}
\item  The graphs on the right side are normal regular graphs (recall Definition \ref{defnlvl0}) with external atoms $\otimes\equiv \fa$, $\oplus\equiv \fb_1$ and $\ominus\equiv \fb_2$, and with at most $C_D$ many atoms.

\item $\Sdelta^{(n)}$ is a sum of at most $C_n$ many deterministic normal regular graphs. 
We decompose it according to the scaling order as 
\be\label{decompose Sdelta0} \Sdelta^{(n)} = \sum_{k\le n} \Sdeltak.\ee
Moreover, we have a sequence of $\selfs$ $\Sele_{k}$ satisfying Definition \ref{collection elements} and properties \eqref{two_properties0V}--\eqref{weaz} for $4\le k \le n$ such that $\Sdeltak$ can be written into the following form 
\be\label{chain S2k}
\Sdeltak=\Sele_{k} + \sum_{ l =2}^k \sum\limits_{ \mathbf k=(k_1,\cdots, k_l) \in \Omega_{k}^{(l)} } \Sele_{k_1}\Theta  \Sele_{k_2}\Theta \cdots \Theta  \Sele_{k_l} . 
\ee
Here all the deterministic graphs with $l=1$ are included into $\Sele_k$ so that the summation starts with $l=2$. Moreover,  $\Omega_{k}^{(l)} \subset \N^l$ is the subset of vectors $\mathbf k$ satisfying that 
\be\label{omegakl}
4\le k_i \le k-1,\quad \text{ and }\quad  \sum_{i=1}^l  k_i  -2(l-1)=k. 
\ee
The second condition in \eqref{omegakl} guarantees that the subgraph $(\Sele_{k_1}\Theta  \Sele_{k_2}\Theta \cdots \Theta  \Sele_{k_l})_{xy}$ has scaling order $k$.

\item $(\PTn)_{\fa,\fb_1\fb_2}$ is a sum of at most $C_n$ many $\oplus/\ominus$-recollision graphs of scaling order $\le n$ and without any $P/Q$ labels.  Moreover, it can be decomposed as 
 \be\label{decomposeP}(\PTn)_{\fa,\fb_1\fb_2} =\sum_{ k =3}^{ n}(\PTk)_{\fa,\fb_1\fb_2},\ee
 where each $(\PTk)_{\fa,\fb_1\fb_2}$ is a sum of the $\oplus/\ominus$-recollision graphs of scaling order $k$ in $(\PTn)_{\fa,\fb_1\fb_2}$.

\item  $(\ATn)_{\fa,\fb_1\fb_2}$ is a sum of at most $C_D$ many graphs of scaling order $> n$ and without any  $P/Q$ labels.

\item $(\QTn)_{\fa,\fb_1\fb_2} $ is a sum of at most $C_D$ many $Q$-graphs.  Moreover, it can be decomposed as 
\be\label{decomposeQ} (\QTn)_{\fa,\fb_1\fb_2}= \sum_{k=2}^{n} (\QTk)_{\fa,\fb_1\fb_2} + (\cal Q^{(>n)}_T)_{\fa,\fb_1\fb_2},\ee
where  $(\QTk)_{\fa,\fb_1\fb_2}$ is a sum of the scaling order $k$ $Q$-graphs in $ (\QTn)_{\fa,\fb_1\fb_2}$ and $(\cal Q^{(>n)}_T)_{\fa,\fb_1\fb_2}$ is a sum of all the scaling order $>n$ $Q$-graphs in $ (\QTn)_{\fa,\fb_1\fb_2}$.

\item $\Sdeltak$, $\PTk$ and $\QTk$ are independent  of $n$.

\item $(\Err_{n,D})_{\fa,\fb_1\fb_2}$ is a sum of at most $C_D$ many graphs, each of which has scaling order $> D$ and may contain some $P/Q$ labels in it. 

\item In each graph of $(\PTn)_{\fa,\fb_1\fb_2}$, $(\ATn)_{\fa,\fb_1\fb_2}$, $(\QTn)_{\fa,\fb_1\fb_2} $ and $(\Err_{n,D})_{\fa,\fb_1\fb_2}$,  there is a unique  $\dashed$ edge connected to $\otimes$. Furthermore,  there is at least an edge, which is either  plus solid $G$ or $\dashed$ or dotted,  connected to  $\oplus$, 
 and  there is at least an edge, which is either  minus solid $G$ or $\dashed$ or dotted,  connected to $\ominus$.
 

%
 
\end{enumerate}
The graphs on the right-hand side of \eqref{mlevelTgdef} satisfy some additional properties, which will be given in Definition \ref{def genuni} below. 
\end{definition}



In accordance with \eqref{trivial conv}, we have that
$$\Sigma_{T,1}=\Sigma_{T,2}=\Sigma_{T,3}=\Sigma_{T,2l+1}=0,\quad l\in \N.$$
%
With \eqref{label_Theta}, we can bound $\left(\Theta \Sdeltak \Theta\right)_{\fa\fb_1}$ by 
\be\label{intro_redagain}\left(\Theta \Sdeltak \Theta\right)_{\fa\fb_1} \prec W^{-(k-2)d/2}B_{\fa\fb_1}. \ee
This bound shows that when $\fb_1=\fb_2=\fb$, the second term on the right-hand side of \eqref{mlevelTgdef} can be bounded by  \smash{$m\overline G_{\fb\fb}(\Theta \Sdelta^{(n)} \Theta)_{\fa\fb}\prec B_{\fa\fb}$}, which is necessary for \eqref{locallaw} to hold.
The rigorous proof of \eqref{intro_redagain} will be given in Lemma \ref{lem redundantagain}. 
When $\fb_1=\fb_2= \fb$, by \eqref{Aho>2} the two graphs in $(\cal R_{T,3})_{\fa,\fb\fb}$ are 
$$ m\sum_{x,y}\delta_{x \fb} \Theta_{\fa x} s_{xy} (G_{yy}-m) |G_{x \fb}|^2  + m \sum_{x,y}\delta_{y \fb} \Theta_{\fa x} s_{xy}  ( \overline G_{xx} - \overline m) |G_{y \fb}|^2,$$
which can be easily bounded by $\Theta_{\fa\fb}\prec B_{\fa\fb}$. 
In general, there are many more complicated graphs in \smash{$(\PTn)_{\fa,\fb_1 \fb_2}$}, but they all satisfy good enough bounds for our prupose. 
If $D>0$ is sufficiently large, the term $(\Err_{n,D})_{\fa,\fb_1\fb_2}$ will be negligible for all proofs. If a graph in $(\Err_{n,D})_{\fa,\fb_1\fb_2}$ does not contain any  $P/Q$ label, then it can  be also included into {$(\ATn)_{\fa,\fb_1\fb_2}$}.

In Section \ref{sec_basiclocal}, we will describe the basic graph operations that are used to obtain the $T$-expansion, and more details will  be given in \cite{PartII_high}. Assuming the $n$-th order $T$-expansion, we can prove Theorem \ref{main thm}. 
 
\begin{theorem}\label{thm ptree}
Fix any $n\in \N$. Suppose the assumptions of Theorem \ref{main thm} hold, and we have an $n$-th order $T$-expansion given in Definition \ref{defn genuni} (together with the additional properties in Definition \ref{def genuni}). 
Assume that $L$ satisfies 
\be\label{Lcondition1}  {L^2}/{W^2}  \le W^{(n-1)d/2-c_0} \ee
for some constant $c_0>0$. Then for any constant $\e>0$, the local law 
\be\label{locallaw1}
|G_{xy} (z) -m(z)\delta_{xy}|^2 \prec B_{xy}
\ee
holds uniformly in all $z= E+\ii\eta$ with $E\in (-2+\kappa,2-\kappa)$ and $\eta \in [W^{2}/L^{2-\e},1]$.
\end{theorem}


If we have obtained the $n$-th order $T$-expansion for $n= n_{W,L}:= \left\lceil\frac{4}{d}\left(\log_W L  - 1 + \frac{c_0}{2}\right)\right\rceil+1$, then we can conclude Theorem \ref{main thm} by using Theorem \ref{thm ptree}. The proof of Theorem \ref{thm ptree} will be given Section \ref{sec pfmain}.


\subsection{Definition of the $\Incomp$}

In this subsection, we define the concept of \emph{$\incomp$}. 

\begin{definition} [$n$-th order $\incomp$]\label{defn incompgenuni}
Fix any $n\in \N$ and let $D>0$ be an arbitrary large constant. For $\fa, \fb_1,\fb_2 \in \Z_L^d$, an $n$-th order $\incomp$ of $T_{\fa,\fb_1\fb_2}$ with $D$-th order error is an expression of the following form: 
	\be \label{mlevelT incomplete}
	\begin{split}
		T_{\fa,\fb_1 \fb_2} &= m  \Theta_{\fa \fb_1}\overline G_{\fb_1\fb_2}   +\sum_x (\Theta \wtSdeltan)_{\fa x} T_{x,\fb_1\fb_2} \\
		&+ (\PITn)_{\fa,\fb_1 \fb_2} +  (\AITn)_{\fa,\fb_1 \fb_2} + (\QITn)_{\fa,\fb_1 \fb_2} + (\Err'_{n,D})_{\fa,\fb_1\fb_2}  ,
	\end{split}
	\ee
	where the graphs on the right-hand side depend only on $n$, $D$, $m(z)$, $S$, $S^{\pm}(z)$, $\Theta(z)$ and $G(z)$, but do not depend on $W$, $L$ and $d$ explicitly. 
	Moreover, they satisfy the following properties.
	\begin{enumerate}
		\item  $(\PITn)_{\fa,\fb_1 \fb_2}$, $(\AITn)_{\fa,\fb_1 \fb_2}$, $(\QITn)_{\fa,\fb_1 \fb_2}$ and $(\Err'_{n,D})_{\fa,\fb_1\fb_2}$ respectively satisfy the same properties as $(\PTn)_{\fa,\fb_1 \fb_2}$, $(\ATn)_{\fa,\fb_1 \fb_2}$, $(\QTn)_{\fa,\fb_1 \fb_2}$ and $(\Err_{n,D})_{\fa,\fb_1\fb_2}$ in Definition \ref{defn genuni}. Furthermore,  $(\PITn)_{\fa,\fb_1 \fb_2}$ and $(\QITn)_{\fa,\fb_1 \fb_2}$ can be decomposed as
		\be\label{decomposePIT}(\PITn)_{\fa,\fb_1\fb_2} =\sum_{ k =3}^{ n}(\PITk)_{\fa,\fb_1\fb_2},\ee
		and
		\be\label{decomposeQIT} (\QITn)_{\fa,\fb_1\fb_2}= \sum_{k=2}^{n} (\QITk)_{\fa,\fb_1\fb_2} + (\cal Q^{(>n)}_{IT})_{\fa,\fb_1\fb_2},\ee
		where  $(\PITk)_{\fa,\fb_1\fb_2}$ is a sum of the scaling order $k$ $\oplus/\ominus$-recollision graphs in $(\PITn)_{\fa,\fb_1\fb_2}$,  $(\QITk)_{\fa,\fb_1\fb_2}$ is a sum of the scaling order $k$ $Q$-graphs in $(\QITn)_{\fa,\fb_1 \fb_2}$, and $(\cal Q^{(>n)}_{IT})_{\fa,\fb_1\fb_2}$ is a sum of the scaling order $>n$ $Q$-graphs in  $(\cal Q^{(>n)}_{IT})_{\fa,\fb_1\fb_2}$.  Moreover, $\PITk$ and $\QITk$ are independent of $n$.
		
		\item $\wtSdeltan$ 
		can be decomposed according to the scaling order as 
		\be\label{decompose Sdelta} 
		\wtSdeltan = \Sele_n + \sum_{l=4}^{n-1} \Sele_{l},
		\ee
		where $\Sele_{l}$, $1\le l \le n-1$, is a sequence of $\selfs$ satisfying Definition \ref{collection elements} and properties \eqref{two_properties0V}--\eqref{weaz}. For any $x,y \in \Z_L^d$, $(\Sele_{n})_{xy}$ is a sum of at most $C_n$ many deterministic graphs of scaling order $n$ and with external atoms $x$ and $y$. 
		
		\item Each graph of $(\PITn)_{\fa,\fb_1\fb_2}$, $(\AITn)_{\fa,\fb_1\fb_2}$, $(\QITn)_{\fa,\fb_1\fb_2} $ and $(\Err_{n,D}')_{\fa,\fb_1\fb_2}$ can be written into 
		$$ \sum_{x}\Theta_{\fa x} {\cal G}_{x,\fb_1\fb_2},$$ 
		where $ {\cal G}_{x,\fb_1\fb_2}$ is a normal regular graph with external atoms $x$, $\fb_1$ and $\fb_2$. Moreover, $ {\cal G}_{x,\fb_1\fb_2}$ has at least an edge, which is either plus solid $G$ or $\dashed$ or dotted,  connected to  $\oplus$, and at least an edge, which is either minus solid $G$ or $\dashed$ or dotted, connected to $\ominus$.
		
	
	\end{enumerate}
	The graphs on the right side of \eqref{mlevelT incomplete} satisfy some additional properties, which will be given in Definition \ref{def incompgenuni} below.
\end{definition} 
The form of \eqref{mlevelT incomplete} is different from \eqref{mlevelTgdef} only in the second term on the right-hand side. We can regard \eqref{mlevelT incomplete} as a linear equation of the $T$-variable $T_{x,\fb_1\fb_2}$. 
In fact, taking $\fb_1=\fb_2$ and $n=M$ in \eqref{mlevelT incomplete}, if we move the second term on the right-hand side to the left-hand side, multiply both sides by $(1-\Theta\wtSdeltan)^{-1}$ and take expectation, then we will get \eqref{E_locallaw}. More details of the proof 
will be given in Section \ref{sec pf0}.

The $\Sele_l$'s in \eqref{decompose Sdelta} are the same $\selfs$ as in Definition \ref{defn genuni}. 
We remark that the sequence of $\incomp$s is constructed inductively. In particular, before constructing the $n$-th order $\incomp$, we have obtained the $k$-th order $\incomp$ and proved the properties \eqref{two_properties0}--\eqref{3rd_property0} and \eqref{two_properties0V}--\eqref{weaz} for $\Sele_k$ for all $4\le k\le n-1$. On the other hand, $ \Sele_n$ is a new sum of deterministic graphs obtained in the $n$-th order $\incomp$, whose properties \eqref{two_properties0}--\eqref{3rd_property0} and \eqref{two_properties0V}--\eqref{weaz} are yet to be shown.  


\section{Basic graph operations}\label{sec_basiclocal}

A graph operation $\cal O[\cal G]$ on a graph $\cal G$ is a linear combination of new graphs such that the graph value of  $\cal G$ 
is unchanged, i.e. $\llbracket \cal O[\cal G]\rrbracket=\llbracket \cal G\rrbracket$. 
All graph operations are linear, that is, 
\be\label{linear graphO} \cal O\Big[\sum_i c_i  \cal G_i\Big] =\sum_i c_i \cal O\left[ \cal G_i\right].
\ee

\subsection{Dotted edge operations} 

Recall that a dotted edge between atoms $\al$ and $\beta$ represents a $\delta_{\al\beta}$ factor. We will identify  internal atoms  connected by dotted edges, but we will {\it not} identify an external and an internal atom due to their different roles in graphs. Dotted edges between internal atoms may appear in intermediate steps, so we define the following merging operation. 
\begin{definition}[Merging operation]\label{def merge}
Given a graph $\cal G$ that contains dotted edges between different atoms, we define an operator $\cal O_{merge}$ in the following way: $\cal O_{merge}[\cal G]$ is a graph obtained by merging every pair of internal atoms, say $\al$ and $\beta$, that are connected by a path of dotted edges into a single internal atom, say $\gamma$. Moreover, the weights and edges attached to $\al$ and $\beta$ 
are now attached to the atom $\gamma$ 
 in $\cal O_{merge}[\cal G]$. In particular, the $G$ edges between $\al$ and $\beta$ become weights on $\gamma$, and the waved and $\dashed$ edges between $\al$ and $\beta$ become self-loops on $\gamma$. 
\end{definition}
It is easy to see that the graph operator $\cal O_{merge}$ is an identity in the sense of graph values:
 $ \llbracket \cal O_{merge}[\cal G]\rrbracket=\llbracket \cal G\rrbracket$.
Given any regular graph, we can rewrite it as a linear combination of normal regular graphs 
using the following \emph{dotted edge partition} operation. 


\begin{definition}[Dotted edge partition] \label{dot-def}
 Given a regular graph $\cal G$, for any pair of atoms $\al$ and $\beta$, if there is at least one $G$ edge but no $\times$-dotted edge between them, then we write 
 $$1=\mathbf 1_{\al=\beta} + \mathbf 1_{\al \ne \beta} ;$$
  if there is a $\times$-dotted line $ \mathbf 1_{\al\ne \beta}$ but no $G$ edge between them, then we write 
 $$\mathbf 1_{\al\ne \beta} =1 - \mathbf 1_{\al = \beta}.$$
 Expanding the product of all these sums on the right-hand sides, we can expand $\cal G$ as
 \be\label{odot}\cal O_{dot}[\cal G] := \sum \cal O_{merge}[{\Dot} \cdot \cal G],\ee
 where each ${\Dot}$ is a product of dotted and $\times$-dotted edges together with a $+$ or $-$ sign. 
  In ${\Dot} \cdot \cal G$, if there is a $\times$-dotted edge between $\al$ and $\beta$, then the $G$ edges between them are off-diagonal; otherwise, the $G$ edges between them become weights after the merging operation. If $\Dot$ is ``inconsistent" (i.e., two atoms are connected by a  $\times$-dotted edge and a path of dotted edges), then we trivially have $\llbracket \Dot \cdot \cal G\rrbracket=0$. Thus we will drop all inconsistent graphs.  Finally, if the graph $\cal G$ is already normal, then $\cal O_{dot}$ acting on $\cal G$ is a null operation and we let $\cal O_{dot}[\cal G]:=\cal G$.

  \end{definition}
\begin{lemma}\label{lem Odot}
Given any regular graph $\cal G$, $\cal O_{dot}[\cal G]$ is a sum of normal regular graphs and $\llbracket \cal O_{dot}[\cal G]\rrbracket  =\llbracket \cal G\rrbracket.$
\end{lemma}

Lemma \ref{lem Odot} trivially follows from Definition \ref{dot-def}. 
We now  introduce the  concept of molecules and  local expansions. 



\begin{definition}[Molecules]\label{def_poly}
	We partition the set of all atoms into a union of disjoint sets $\{\text{all atoms}\}=\cup_j \cal M_j$, where each $\cal M_j$ is called a molecule. Two internal atoms belong to the same molecule if and only if they are connected by a path of neutral/plus/minus {waved edges} and {dotted edges} (note there may be dotted edges between internal atoms if the graph is not regular). 
	Each external atom will be called an external molecule (such as $\otimes$, $\oplus$ and $\ominus$ molecules) by definition.  An edge is said to be inside a molecule if its ending atoms belong to this molecule.
\end{definition}
By \eqref{subpoly} and \eqref{S+xy}, if two atoms $x$ and $y$ are in the same molecule, then we essentially have $|x-y|\le W^{1+\tau}$ up to a negligible error $\OO(W^{-D})$. 
Given an atomic graph, we will call the subgraph inside a molecule (i.e. the subgraph induced on the atoms inside this molecule) the  {\it local structure} of the molecule. The \emph{molecular graph} (cf. Definition \ref{def moleg} below) is the quotient graph with each molecule regarded as a vertex. Then the {\it global structure} of a graph refers to its molecular graph.
Note that the local structures can only vary on scales of order $\OO (W^{1+\tau})$, while the global structure varies on scales up to  $L$. The \emph{two-level structure} of an atomic graph---a global structure plus several local structures---has been explored in \cite{PartIII} already.

We will call an expansion \emph{local} if it does not create new molecules, that is, every molecule in the new graphs is obtained by adding new atoms to existing molecules or merging some molecules in the original graph. It is easy to see that $\cal O_{dot}$ is a local expansion. In Sections \ref{sec localexp} and \ref{sec localexp edge}, we will introduce more local expansions. 
We point out that local expansions  can change the global structure. However, as we will explain in \cite{PartII_high}, they will maintain the doubly connected properties of the graphs (cf. Definition \ref{def 2net}). 




\subsection{Weight expansion} \label{sec localexp}

\begin{lemma} \label{ssl} 
In the setting of Theorem \ref{comp_delocal}, suppose that $f$ is a differentiable function of $G$. Then we have the following identity:  
\be\label{weightG}
\begin{split} 
     (G_{xx}  -  m) f(G)  & = m \sum_{\al,\beta}b_{x\al} s_{\al\beta} (G_{\al\al}-m) (G_{\beta\beta}-m) f(G) -  m \sum_{\al,\beta} b_{x\al} s_{\al\beta} P_\al  \left[ G_{\beta\al }\partial_{  h_{ \beta \al}} f(G)\right]  \\
  &+ \sum_{\al}b_{x\al}Q_\al \left[(G_{\al\al}  -  m) f(G)\right]  - m \sum_{\al,\beta}b_{x\al}s_{\al\beta}Q_\al\left[  (G_{\beta\beta}-m)G_{\al\al} f(G) \right] ,
\end{split} 
\ee
where for simplicity we introduced the matrix
\be\label{defn b}  b:= (1- m^2 S)^{-1} = 1 + S^+  .\ee
\end{lemma} 
\begin{proof}
Using \eqref{GmH} and Gaussian integration by parts, we obtain that
\begin{align*}
& (G_{xx}  -  m) f(G) =Q_x \left[(G_{xx}  -  m) f(G)\right] + P_x \left[(G_{xx}  -  m) f(G)\right] \\
&=Q_x \left[(G_{xx}  -  m) f(G)\right] +  P_x \Big[ \Big( -m^2G_{xx} - m \sum_\al h_{x\al}G_{\al x}\Big) f(G)\Big]\\
  &=Q_x\left[ (G_{xx}  -  m) f(G) \right] + m \sum_{\al} s_{x\al}P_x\left[ (G_{\al\al}-m) G_{xx}f(G) \right]-  m P_x  \sum_\al s_{x\al} \left[ G_{\al x}\partial_{ h_{\al x}} f(G)\right]\\
  &=Q_x\left[ (G_{xx}  -  m) f(G) \right] - m \sum_{\al} s_{x\al}Q_x\left[ (G_{\al\al}-m) G_{xx}f(G) \right] + m^2 \sum_{\al} s_{x\al} (G_{\al\al}-m)  f(G) \\
  &+ m \sum_{\al} s_{x\al} (G_{\al\al}-m) (G_{xx}-m)f(G)-  m \sum_\al s_{x\al} P_x  \left[  G_{\al x}\partial_{  h_{\al x}} f(G)\right],
  \end{align*}
which gives the equation
\begin{align*}
 \sum_\al(1-m^2 S)_{x\al}(G_{\al\al}  -  m) f(G) =Q_x\left[ (G_{xx}  -  m) f(G) \right] - m \sum_{\al} s_{x\al}Q_x\left[ (G_{\al\al}-m) G_{xx}f(G) \right]  \\
 + m \sum_{\al} s_{x\al} (G_{\al\al}-m) (G_{xx}-m)f(G)-  m \sum_\al s_{x\al} P_x  \left[  G_{\al x}\partial_{  h_{\al x}} f(G)\right].
\end{align*}
Multiplying both sides with $(1-m^2 S)^{-1}$, we obtain \eqref{weightG}.  
\end{proof}


Expanding $b_{x\al}$ as $\delta_{x\al} + S^+_{x\al}$ and $P_\al$ as $1-Q_\al$ in \eqref{weightG}, we obtain the following weight expansion operator. 
\begin{definition} [Weight expansion operator]\label{Ow-def} \ 
Given a normal regular graph $\cal G$ 
which contains an atom $x$, if there is no weight on $x$, then 
we trivially define $\cal O_{weight}^{(x)}[\cal G]:=\cal G$. Otherwise,  we define ${\cal O}_{weight}^{(x)}[\cal G]$ in the following way. 
 
 \medskip

\noindent{\bf (i) Removing regular weights:} Suppose there are regular $G_{xx}$ or $\overline G_{xx}$ weights on $x$. Then we rewrite 
$$G_{xx}=m+ (G_{xx}-m),\quad \text{and}\quad \overline G_{xx}=\overline m+ (\overline G_{xx}-\overline m).$$
Expanding the product of all these sums, we can write $\cal G$ into a linear combination of normal regular graphs containing only light weights on $x$. We denote this graph operator as $ {\cal O}^{(x), 1}_{weight}$.  

 \medskip

\noindent{\bf (ii) Expanding light weight:} If $\cal G$ has a light weight $G_{xx}-m$ of positive charge on $x$ and is of  the form $\cal G= (G_{xx}-m)f (G)$,  then we define the light weight expansion on $x$ by  
\begin{align} 
 {\cal O}^{(x),2}_{weight} \left [  \cal G \right] &:=  m \sum_{  \al} s_{x\al}  (G_{xx}-m) (G_{\al\al}-m)f (G) +m \sum_{  \al,\beta}S^{+}_{x\al} s_{\al \beta}  (G_{\al\al}-m) (G_{\beta\beta}-m)f (G) \nonumber\\
 &-  m  \sum_{ \al} s_{x \al} G_{\al x}\partial_{ h_{ \al x}} f (G) -  m \sum_{ \al,\beta} S^{+}_{x\al}s_{\al \beta} G_{\beta \al}\partial_{ h_{ \beta\al}} f(G) + \cal Q_w ,\label{Owx}\end{align}
 where 
 $\cal Q_w$ is a sum of $Q$-graphs,
 \begin{align} 
  \cal Q_w &:=  Q_x \left[(G_{xx}  -  m) f(G)\right] + \sum_{\al} Q_\al \left[  S^{+}_{x\al}(G_{\al\al}  -  m) f(G)\right] \nonumber\\
 & - m  Q_x\Big[ \sum_{\al} s_{x\al}  (G_{\al\al}-m)G_{xx} f(G) \Big]  - m \sum_\al Q_\al\Big[ \sum_{ \beta}S_{x\al}^{+}s_{\al\beta}(G_{\beta\beta}-m)G_{\al\al} f(G) \Big] \nonumber\\
  &+  m  Q_x  \Big[  \sum_\al s_{x \al}  G_{\al x}\partial_{ h_{\al x}} f(G)\Big]+ m \sum_\al Q_\al \Big[  \sum_{ \beta} S^{+}_{x\al}s_{\al \beta} G_{\beta \al}\partial_{ h_{\beta \al}} f(G)\Big]. \nonumber
\end{align} 
If $\cal G= (\overline G_{xx}-\overline m)f (G)$, then we define 
\be\label{Owx neg}
 {\cal O}^{(x),2}_{weight}\left [  \cal G \right] := \overline{ {\cal O}^{(x),2}_{weight} \left [  (G_{xx}  -  m) \overline{f(G)} \right ] } ,
\ee
where the right-hand side can be defined using \eqref{Owx}.
When there are more than one light weights on $x$, we pick any positive light weight and apply \eqref{Owx}; if there is no positive light weight, then we pick any negative light weight and apply \eqref{Owx neg}. 
 
Combining the above two graph operators, given any normal regular graph $\cal G$, we define
\be\label{eq defOdot}
{\cal O}_{weight}^{(x)}[ \cal G]:={\cal O}^{(x),1}_{weight} \circ\cal O_{dot} \circ  {\cal O}^{(x),2}_{weight} \circ {\cal O}^{(x),1}_{weight}  [\cal G], \ee
where the operator $\cal O_{dot}$ is applied to make sure that the resulting graphs after applying $\cal O_{dot}$ are normal regular.  The reason for the last  ${\cal O}^{(1)}_{weight}(x)$ operator will be explained in Remark \ref{remark remove_reg} below.

\end{definition}

\begin{remark}\label{remark remove_reg}
Consider the third term on the right-hand side of \eqref{Owx} as an example. First, when applying $\cal O_{dot}$, we will have a graph with $\al=x$, in which case $G_{\al x}$ becomes a weight $G_{xx}$. Second, we consider the partial derivative $\partial_{ h_{  \al x}} f(G) $.  Suppose $f(G)$ is of the form
$$f(G)=\sum_{\{y_i\},\{y_i'\},\{w_i\},\{w'_i\}} \prod_{i=1}^{k_1}G_{x y_i}  \cdot  \prod_{i=1}^{k_2}\overline G_{x y'_i} \cdot \prod_{i=1}^{k_3} G_{ w_i x} \cdot \prod_{i=1}^{k_4}\overline G_{ w'_i x} \cdot G_{xx}^{l_1} \overline G_{xx}^{l_2}(G_{xx}-m)^{l_3} (\overline G_{xx}-\overline m)^{l_4}  g(G),$$
where $ g(G)\equiv g (G,\{y_i\},\{y_i'\},\{w_i\},\{w'_i\})$ does not contain any weight or solid edge attached to atom $x$. Then we take the partial derivative of the weights and solid edges in $f(G)$ using the identities
\be\label{partialhG} \partial_{ h_{\al x}} G_{ab} = -G_{a\al}G_{xb}, \quad \partial_{ h_{\al x}} \overline G_{ba} =- \overline G_{bx}\overline G_{\al a},\quad a,b\in \Z_L^d.\ee
Note that it is possible to have $b=x$ (e.g.\;when we take the partial derivative $\partial_{ h_{\al x}}$ of $\overline G_{x y'_i} $, $G_{ w_i x}$ or a weight on $x$), which will lead to a weight $G_{xx}$ or $\overline G_{xx}$ on atom $x$.  Hence we can have regular weights in the graphs in $\cal O_{dot} \circ  {\cal O}^{(x),2}_{weight} \circ {\cal O}^{(x),1}_{weight}  [\cal G]$. These regular weights are  removed  by applying another ${\cal O}^{(x),1}_{weight} $.
\end{remark}

\begin{definition}[Canonical local expansions]\label{def_canonical_local}
A local expansion $\cal O^{(x)}$ of a normal regular graph $\cal G$ at an atom $x$ is said to be \emph{canonical}
if it satisfies the following properties.
\begin{itemize}
	\item[(i)] The graph value is unchanged after the expansion, i.e.,  $ \llbracket {\cal O} ^{(x)}[\cal G]\rrbracket  =\llbracket \cal G \rrbracket.$ 
		
	\item[(ii)] $  {\cal O} ^{(x)}  [\cal G  ] $  is a linear combination of normal regular graphs. 
	
	\item[(iii)] Every graph in ${\cal O}^{(x)}[\cal G]$ has scaling order $\ge \ord(\cal G)$.
	
	\item[(iv)] 
	If there is a new atom in a graph after the expansion, then it is connected to $x$ through a path of waved edges.
\end{itemize}
\end{definition}

The property (iv) shows that all the new atoms created in the expansions are included in the existing molecule containing atom $x$
and hence is consistent with the local property of $\cal O^{(x)}$.



  \begin{lemma}\label{expandlabel}
Given a normal regular graph $ \cal G$ with an atom $x$, $ {\cal O}_{weight}^{(x)}[\cal G]$ is a canonical local expansion. If $\cal G$ contains at least one weight at $x$, then every graph without $Q$-labels, say $\cal G_1$, in \smash{$ {\cal O}_{weight}^{(x)}[\cal G]$} satisfies one of the following two properties:


\begin{itemize}
	\item[(a)] its scaling order is strictly higher than $\ord(\cal G)$, i.e., $\ord(\cal G_1)\ge \ord(\cal G) +1$;
	\item[(b)] 
	$\ord(\cal G_1)= \ord(\cal G)$, and $\cal G_1$ has strictly fewer weights than $\cal G$ (more precisely, it contains at least one fewer weight on $x$, no weights on the new atoms, and the same number of weights on any other atom).
\end{itemize}

  \end{lemma}
 
 The proof of Lemma \ref{expandlabel} follows straightforwardly by using Definition \ref{Ow-def} and we postpone it to Appendix \ref{appd localpf}. The properties (a) and (b) in Lemma \ref{expandlabel}  show that, by applying the weight expansion repeatedly, we can get either new graphs without weights, or $Q$-graphs and graphs of sufficiently high scaling orders.

\subsection{Edge expansions} \label{sec localexp edge}

In this subsection, we introduce three basic edge expansion operators. First, we define a multi-edge expansion, which aims to remove atoms that have degrees larger than 2. Here we use the following notion of degrees of solid edges (i.e. plus and minus $G$ edges):
\be\label{degree}
{\deg}(x) := \# \{\text{solid edges connected with $x$}\} . 
\ee

\begin{lemma}
\label{Oe14}
In the setting of Theorem \ref{comp_delocal}, suppose that $f$ is a differentiable function of $G$. Consider a graph
\be\label{multi setting}
\cal G := \prod_{i=1}^{k_1}G_{x y_i}  \cdot  \prod_{i=1}^{k_2}\overline G_{x y'_i} \cdot \prod_{i=1}^{k_3} G_{ w_i x} \cdot \prod_{i=1}^{k_4}\overline G_{ w'_i x} \cdot f(G),
\ee
where the atoms $y_i,$ $y'_i$, $w_i$ and $w'_i$ are all not equal to $x$. If $k_1\ge 1$, then we have the following identity: 
\begin{align} 
& \cal G  = \sum_{i=1}^{k_2}m P_x \left[ \overline G_{xx}  \left( \sum_\al s_{x\al }G_{\al y_1} \overline G_{\al y'_i}\right)\frac{\cal G}{G_{x y_1} \overline G_{xy_i'}} \right] + \sum_{i=1}^{k_3} m P_x \left[ G_{xx} \left(\sum_\al s_{x\al }G_{\al y_1} G_{w_i \al} \right)\frac{\cal G}{G_{xy_1}G_{w_i x}}\right] \nonumber \\
&+  m P_x \left[  \sum_\al s_{x\al }\left(G_{\al \al}-m\right) \cal G\right] +(k_1-1) m P_x \left[  \sum_\al s_{x\al }G_{\al y_1} G_{x \al} \frac{ \cal G}{G_{xy_1}}\right] + k_4 m P_x \left[  \sum_\al s_{x\al }G_{\al y_1} \overline G_{\al x} \frac{\cal G}{G_{xy_1}}\right] \nonumber\\
 & - m P_x \left[  \sum_\al s_{x\al }\frac{\cal G}{G_{x y_1}f(G)}G_{\al y_1}\partial_{ h_{\al x}}f (G)\right]+ Q_x \left(  \cal G\right)  .\label{gH0} 
 \end{align} 
Here the fractions are used to simplify the expression. For example, the fraction ${\cal G}/({G_{x y_1} \overline G_{xy_i'}})$ is the  graph obtained by removing the factor ${G_{x y_1} \overline G_{xy_i'}}$ from the product in \eqref{multi setting}.
 \end{lemma}
 \begin{proof}
Using \eqref{GmH} and $x\ne y_1$, we can write that
 \be
  \begin{split}\label{gH}
 & P_x \left(  \cal G \right) =P_x \left[ \left( - m^2 G_{xy_1} - m \sum_\al h_{x\al }G_{\al y_1}\right) \prod_{i=2}^{k_1}G_{x y_i}  \cdot  \prod_{i=1}^{k_2}\overline G_{x y'_i} \cdot \prod_{i=1}^{k_3} G_{ w_i x} \cdot \prod_{i=1}^{k_4}\overline G_{ w'_i x}\cdot f(G)\right] .
 \end{split} 
 \ee
We apply Gaussian integration by parts to the $HG$ term to get that
  \begin{align*}
 & - m P_x \left[  \sum_\al h_{x\al }G_{\al y_1}\cdot \prod_{i=2}^{k_1}G_{x y_i}  \cdot  \prod_{i=1}^{k_2}\overline G_{x y'_i} \cdot \prod_{i=1}^{k_3} G_{ w_i x} \cdot \prod_{i=1}^{k_4}\overline G_{ w'_i x}\cdot f(G)\right] \\
 &=  m P_x \left[  \left(\sum_\al s_{x\al }G_{\al \al}\right) \cal G\right] +(k_1-1) m P_x \left[  \sum_\al s_{x\al }G_{\al y_1} G_{x \al} \frac{\cal G}{G_{xy_1}}\right] + k_4 m P_x \left[  \sum_\al s_{x\al }G_{\al y_1} \overline G_{\al x} \frac{\cal G}{G_{xy_1}}\right] \\
 &+  \sum_{i=1}^{k_2}m P_x \left[ \overline G_{xx}  \left( \sum_\al s_{x\al }G_{\al y_1} \overline G_{\al y'_i}\right)\frac{\cal G}{G_{x y_1} \overline G_{xy_i'}} \right] + \sum_{i=1}^{k_3} m P_x \left[ G_{xx} \left(\sum_\al s_{x\al }G_{\al y_1} G_{w_i \al} \right)\frac{\cal G}{G_{xy_1}G_{w_i x}}\right] \\
 & - m P_x \left[  \sum_\al s_{x\al }G_{\al y_1}\cdot \prod_{i=2}^{k_1}G_{x y_i}  \cdot  \prod_{i=1}^{k_2}\overline G_{x y'_i} \cdot \prod_{i=1}^{k_3} G_{ w_i x} \cdot \prod_{i=1}^{k_4}\overline G_{ w'_i x} \cdot\partial_{  h_{\al x}}f(G)\right].
 \end{align*}
Plugging it into \eqref{gH} and using $\cal G=P_x(\cal G)+Q_x(\cal G)$, we conclude \eqref{gH0}.
\end{proof}

Applying $P_x=1-Q_x$ to \eqref{gH0}, we can define the following multi-edge expansion operator.

\begin{definition} [Multi-edge expansion operator]\label{multi-def} 
Given a normal regular graph $\cal G$, if there are no solid edges connected with an atom $x$, then we trivially define \smash{$\cal O_{multi-e}^{(x)}[\cal G]:=\cal G$}. Otherwise, we define \smash{${\cal O}_{multi-e}^{(x)}$} in the following way. 
Suppose $\cal G$ takes the form \eqref{multi setting},
where 
the atoms $y_i,$ $y'_i$, $w_i$ and $w'_i$ are all not equal to $x$. 

\medskip
\noindent{(i)} If $k_1\ge 1$, then we define the multi-edge expansion on $x$ as 
\begin{align} 
& \wh{\cal O}_{multi-e}^{(x)} \left[\cal G\right] : = \sum_{i=1}^{k_2}|m|^2  \left( \sum_\al s_{x\al }G_{\al y_1} \overline G_{\al y'_i}\right)\frac{\cal G}{G_{x y_1} \overline G_{xy_i'}}   + \sum_{i=1}^{k_3} m^2 \left(\sum_\al s_{x\al }G_{\al y_1} G_{w_i \al} \right)\frac{\cal G}{G_{xy_1}G_{w_i x}} \nonumber \\
& + \sum_{i=1}^{k_2}m  (\overline G_{xx} -\overline m) \left( \sum_\al s_{x\al }G_{\al y_1} \overline G_{\al y'_i}\right)\frac{\cal G}{G_{x y_1} \overline G_{xy_i'}}   + \sum_{i=1}^{k_3} m  (G_{xx}-m) \left(\sum_\al s_{x\al }G_{\al y_1} G_{w_i \al} \right)\frac{\cal G}{G_{xy_1}G_{w_i x}} \nonumber \\
& +  m   \sum_\al s_{x\al }\left(G_{\al \al}-m\right) \cal G  +(k_1-1) m  \sum_\al s_{x\al } G_{x \al} G_{\al y_1}\frac{ \cal G}{G_{xy_1}}  + k_4 m   \sum_\al s_{x\al }\overline G_{\al x} G_{\al y_1}  \frac{\cal G}{G_{xy_1}}  \nonumber\\
 & - m    \sum_\al s_{x\al } \frac{\cal G}{G_{x y_1}f(G)}G_{\al y_1}\partial_{ h_{\al x}}f (G)  +\cal Q_{multi-e} .\label{Oe1x}
\end{align}
On the right-hand side of \eqref{Oe1x}, the first two terms are main terms with the same scaling order as $\cal G$, but the degree of atom $x$ is reduced by 2 and a new atom $\al$ with degree 2 is created; the third to fifth terms contain one more light weight and hence are of  strictly higher scaling orders than $\cal G$; the sixth to eighth terms contain at least one more off-diagonal $G$ edge and hence are of strictly higher scaling orders than $\cal G$.  The last term $\cal Q_{multi-e} $ is a sum of $Q$-graphs defined by 
\begin{align*}
  \cal Q_{multi-e} &:= Q_x \left( \cal G\right)  -  \sum_{i=1}^{k_2}m Q_x \left[ \overline G_{xx}  \left( \sum_\al s_{x\al }G_{\al y_1} \overline G_{\al y'_i}\right)\frac{\cal G}{G_{x y_1} \overline G_{xy_i'}} \right] \\
& - \sum_{i=1}^{k_3} m Q_x \left[ G_{xx} \left(\sum_\al s_{x\al }G_{\al y_1} G_{w_i \al} \right)\frac{\cal G}{G_{xy_1}G_{w_i x}}\right] -  m Q_x \left[  \sum_\al s_{x\al }\left(G_{\al \al}-m\right) \cal G\right] \\
&-(k_1-1) m Q_x \left[  \sum_\al s_{x\al }G_{x \al} G_{\al y_1}  \frac{ \cal G}{G_{xy_1}}\right] - k_4 m Q_x \left[  \sum_\al s_{x\al }\overline G_{\al x} G_{\al y_1}  \frac{\cal G}{G_{xy_1}}\right] \\
& +m Q_x \left[  \sum_\al s_{x\al } \frac{\cal G}{G_{x y_1}f(G)}G_{\al y_1}\partial_{h_{\al x}}f(G)  \right].
 \end{align*} 

\noindent{(ii)} If $k_1=0$ and $k_2\ge 1$, then we define 
$$\wh{\cal O}_{multi-e}^{(x)} \left[\cal G\right]:= \overline{  \wh{\cal O}_{multi-e}^{(x)} \left[\prod_{i=1}^{k_2}G_{x y'_i} \cdot \prod_{i=1}^{k_3} \overline  G_{ w_i x} \cdot \prod_{i=1}^{k_4} G_{ w'_i x} \cdot \overline{f (G)}\right] },$$
where the right-hand side can be defined using (i).

\medskip

\noindent{(iii)} If $k_1=k_2=0$ and $k_3 \ge 1$, then we define $ \wh{\cal O}_{multi-e}^{(x)} \left[\cal G\right]$ by exchanging the order of matrix indices in (i). More precisely, we define
\begin{align} 
 \wh{\cal O}_{multi-e}^{(x)} \left[\cal G\right] & := \sum_{i=1}^{k_4}|m|^2  \left( \sum_\al s_{x\al }G_{w_1\al} \overline G_{w'_i\al}\right)\frac{\cal G}{G_{w_1x} \overline G_{w_i'x}}   + \sum_{i=1}^{k_4}m  (\overline G_{xx} -\overline m) \left( \sum_\al s_{x\al }G_{w_1 \al } \overline G_{ w'_i \al}\right)\frac{\cal G}{G_{w_1x} \overline G_{w_i' x}}    \nonumber \\
& +  m   \sum_\al s_{x\al }\left(G_{\al \al}-m\right) \cal G  +(k_3-1) m  \sum_\al s_{x\al }G_{w_1\al} G_{\al x} \frac{ \cal G}{G_{w_1x}} \nonumber\\
&  - m    \sum_\al s_{x\al } \frac{\cal G}{G_{w_1x}f(G)}G_{  w_1 \al}\partial_{ h_{x\al}}f(G)  +\cal Q_{multi-e} ,\label{Oe1x3}
\end{align}
where 
\begin{align*}
  \cal Q_{multi-e} &:= Q_x \left( \cal G\right)  -  \sum_{i=1}^{k_4}m Q_x \left[ \overline G_{xx}  \left( \sum_\al s_{x\al }G_{w_1 \al } \overline G_{ w'_i \al}\right)\frac{\cal G}{G_{w_1 x} \overline G_{w_i' x}} \right]  -  m Q_x \left[  \sum_\al s_{x\al }\left(G_{\al \al}-m\right) \cal G\right] \\
&  -(k_3-1) m Q_x \left[  \sum_\al s_{x\al }G_{w_1 \al } G_{\al x} \frac{ \cal G}{G_{w_1 x}}\right]  +m Q_x \left[  \sum_\al s_{x\al } \frac{\cal G}{G_{w_1x}f(G)}G_{ w_1\al}\partial_{ h_{x\al }}f(G) \right] .
 \end{align*} 

\noindent{(iv)} If $k_1=k_2=k_3=0$ and $k_4\ge 1$, then we define 
$$\wh{\cal O}_{multi-e}^{(x)} \left[\cal G\right]:= \overline{  \wh{\cal O}_{multi-e}^{(x)} \left[  \prod_{i=1}^{k_4} G_{ w'_i x} \cdot \overline{f (G)}\right] },$$
where the right-hand side can be defined using (iii).
 
Finally, applying the $\cal O_{dot}$ in Definition \ref{dot-def}, we define
$${\cal O}_{multi-e}^{(x)}[ \cal G]:= \cal O_{dot} \circ  \wh{\cal O}_{multi-e}^{(x)}   [\cal G]. $$ 
 \end{definition}

 The multi-edge expansion motivates the following definition of \emph{matched} and \emph{mismatched} solid edges.

\begin{definition}[Matched and mismatched edges]
Consider an internal atom $x$ of degree 2 in a graph. We say the two edges connected with $x$ are {\bf mismatched} if they are of the following forms: 
\begin{center}
 \parbox[c]{0.8\linewidth}{\includegraphics[width=13cm]{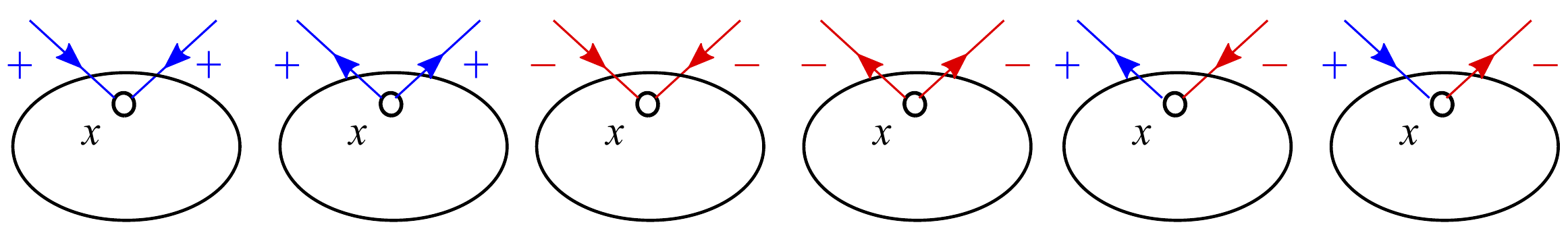}} 
 \end{center}
Otherwise the two edges are {\bf matched} and of the following forms:
  \begin{center}
 \parbox[c]{0.56\linewidth}{\includegraphics[width=9cm]{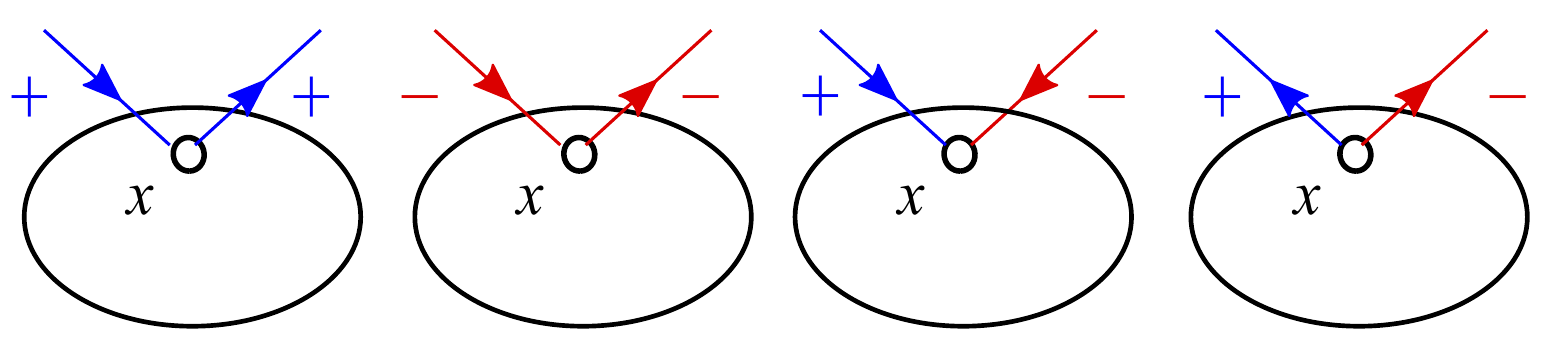}} 
\end{center}
Alternatively, a degree 2 atom $x$ is said to be connected with two matched edges if and only if its charge is 0, where  the charge of an atom is defined by  
$$\#\{\text{incoming $+$ and outgoing $-$ solid edges}\}- \#\{\text{outgoing $+$ and incoming $-$ solid edges}\} .$$
\end{definition}

By Definition \ref{multi-def}, we can see that if $x$ is connected with two mismatched edges, then $\cal O_{multi-e}(x) \left[\cal G\right] $ is a sum of graphs that are all of strictly higher scaling orders than $\cal G$. For example, we take  $\cal G=G_{xy}   G_{x y'}  f(G)$, i.e., $k_1=2$ and $k_2=k_3=k_4=0$ in \eqref{multi setting}. Then the first two main terms on the right-hand side of \eqref{Oe1x} are both zero.

The following lemma describes the basic properties of multi-edge expansions. 
Its proof is a straightforward application of Definition \ref{multi-def}, and we postpone it to Appendix \ref{appd localpf}.
   
\begin{lemma}\label{expandmulti}
Consider a normal regular graph $\cal G$ taking the form \eqref{multi setting}, where $f (G)$ does not contain any $G$ edges or weights attached to $x$, and the atoms $y_i,$ $y'_i$, $w_i$ and $w'_i$ are all not equal to $x$. Then \smash{${\cal O}_{multi-e}^{(x)}[\cal G]$} is a canonical local expansion satisfying the following properties. 
\begin{itemize}
\item[(a)] 
Suppose that $\deg(x)\ge 4$ in $\cal G$. Then every graph without $Q$-labels, say $\cal G_1$, in \smash{${\cal O}_{multi-e}^{(x)}[\cal G]$} either has a strictly higher scaling order than $\cal G$ or satisfies one of the following two properties:
\begin{itemize}
\item[(a.1)] 
$\ord(\cal G_1)= \ord(\cal G)$; $\cal G_1$ has one new atom with degree 2;  $\deg(x)$ in $\cal G_1$ is smaller than $\deg(x)$ in $\cal G$ by 2, and the degree of any other atom stays the same as in $\cal G$;
\item[(a.2)] 
$\ord(\cal G_1)= \ord(\cal G)$; $\cal G_1$ has no new atom; $\deg(x)$ in $\cal G_1$ is smaller than $\deg(x)$ in $\cal G$ by 2, and the degree of any other atom either stays the same or decreases by 2. 
\end{itemize}
\item[(b)] 
Suppose that $\deg(x)=1$, or $x$ is connected with exactly two mismatched solid edges in $\cal G$. Then 
every graph without $Q$-labels has a strictly higher scaling order than $\cal G$.
\end{itemize}
  \end{lemma}

 Lemma \ref{expandmulti} shows that, by applying the multi-edge expansion repeatedly, we can either make all atoms in the resulting graphs to be connected with \emph{exactly two matched solid edges}, or get $Q$-graphs and graphs of sufficiently high scaling orders. 

If an atom is connected with exactly two matched solid edges, then applying the multi-edge expansion cannot improve the graph anymore. Instead, we will apply the $G G$ expansion given by the following lemma if these two edges are of the same charge.

  \begin{lemma} \label{T eq0}
In the setting of Theorem \ref{comp_delocal}, consider a graph $\cal G= G_{xy}   G_{y' x }  f (G)$ where $f$ is a differentiable function of $G$ and $y,y'\ne x$. Then we have that 
\be\label{G2fH}
\begin{split} 
 \cal G   &  =   m b_{xy} P_y\left[ G_{y' y} f(G)\right]  +m \sum_{\al,\beta}  b_{x\al} s_{\al\beta}P_\al\left[  (G_{\beta \beta}-m)  G_{\al y}   G_{y' \al }  f (G)\right]  \\
 & + m \sum_{\al,\beta}  b_{x\al} s_{\al\beta} P_\al\left[ (G_{\al\al}-m)   G_{\beta y}   G_{y' \beta }  f (G)\right]- m \sum_{\al,\beta} b_{x\al}    s_{\al\beta}P_\al\left[ G_{\beta y} G_{y' \al} \partial_{ h_{\beta \al}}f(G)\right]\\
  & + \sum_\al b_{x\al}  Q_\al\left[G_{\al y}   G_{y' \al}  f (G)\right]   - m^2 \sum_{\al, \beta}b_{x\al} s_{\al\beta}  Q_\al\left[ G_{\beta y}   G_{y' \beta }  f (G)\right] ,
 \end{split} 
 \ee 
where $b$ is  defined in \eqref{defn b}.
\end{lemma} 
\begin{proof}
Using \eqref{GmH} and $I_N=b - m^2 bS$, we get that  
\begin{align}
   P_x\left(  \cal G\right) & = \sum_\al \left[ b_{x\al} -m^2 (b S)_{x\al}\right] P_\al\left[G_{\al y}   G_{y' \al}f(G)\right] \nonumber\\
   &= -\sum_\al m^2 (b S)_{x\al}  P_\al\left[G_{\al y}   G_{y' \al}f(G)\right] + \sum_\al  b_{x\al} P_\al\left[ \left(m\delta_{\al y}-m^2G_{\al y} - m (HG)_{\al y}\right) G_{y' \al} f(G)\right] .\label{dervG2}
  \end{align}
 Applying Gaussian integration by parts to the $ HG $ term, we get that
\begin{align*}
   &P_\al\Big[  -m \sum_\beta h_{\al\beta}G_{\beta y} G_{y' \al} f(G)\Big] \\
   =& P_\al\Big[ m \sum_\beta  s_{\al\beta}G_{\beta\beta} G_{\al y}   G_{y' \al}f(G) + m \sum_\beta s_{\al\beta} G_{\al\al}  G_{\beta y}   G_{y' \beta}f(G) - m \sum_\beta s_{\al\beta}G_{\beta y} G_{y' \al} \partial_{ h_{ \beta \al}}f(G)\Big].
\end{align*}
Plugging it into \eqref{dervG2},  using $1=P_x+Q_x$, $\delta_{x\al}+m^2 (b S)_{x\al}=b_{x\al}$ and 
\begin{align*}
& \, m \sum_{\al, \beta}b_{x\al} s_{\al\beta}  P_\al\left[G_{\al\al}  G_{\beta y}   G_{y' \beta}f(G)\right] -\sum_\al m^2 (b S)_{x\al}  P_\al\left[G_{\al y}   G_{y' \al}f(G) \right] \\
=&\,  m \sum_{\al, \beta}b_{x\al} s_{\al\beta}  P_\al\left[  (G_{\al\al}-m)  G_{\beta y}   G_{y' \beta}f(G)\right] + \sum_\al m^2 (b S)_{x\al}  Q_\al\left[G_{\al y}   G_{y' \al}f(G)\right] \\
&\,   - m^2 \sum_{\al, \beta}b_{x\al} s_{\al\beta}  Q_\al\left[G_{\beta y}   G_{y' \beta}f(G)\right] ,
\end{align*}
 we  obtain equation \eqref{G2fH}. 
  \end{proof}


Using \eqref{G2fH}, $b =1 + S^+ $ and $P_\al=1-Q_\al$, we can define the following $G G$ expansion operator. 
   \begin{definition}[$G G$ expansion operator] \label{GG-def}
  Given a normal regular graph $\cal G$, suppose an atom $x$ is connected with exactly two matched $G$ edges of the same charge. Suppose $\cal G$ takes the form $\cal G=G_{xy}   G_{y' x } f(G)$ with $y,y' \ne x$. Then we define
\begin{align}
  \wh{\cal O}_{GG}^{(x)} [\cal G] & := 
m S^{+}_{xy} G_{y' y} f(G)  + m \sum_\al  s_{x\al} (G_{\al \al}-m) \cal G + m \sum_{\al,\beta}  S^{+}_{x\al}  s_{\al\beta} (G_{\beta \beta}-m) G_{\al y}   G_{y'\al} f(G) \nonumber \\
    &  + m(G_{xx }-m)   \sum_\al s_{x\al}G_{\al y}   G_{y'\al} f(G) + m \sum_{\al,\beta}  S^+_{x\al} s_{\al\beta}  (G_{\al\al }-m) G_{\beta y}   G_{y'\beta} f(G)  \nonumber\\
    & - m \sum_{ \al}  s_{x\al}G_{\al y} G_{y' x} \partial_{ h_{\al x}}f(G) - m \sum_{\al,\beta} S^{+}_{x\al} s_{\al\beta}G_{\beta y} G_{y' \al} \partial_{ h_{\beta \al}}f(G) + \cal Q_{GG} .\label{Oe2x}
    \end{align}
On the right-hand side of \eqref{Oe2x}, the first term is the main term which is either of the same scaling order as $\cal G$ if $y=y'$ or has a strictly higher scaling order if $y\ne y'$; the second to fifth terms contain one more light weight and hence are of strictly higher scaling orders than $\cal G$; the sixth and seventh terms contain at least one more off-diagonal $G$ edge and hence are of strictly higher scaling orders than $\cal G$. The last term $\cal Q_{GG} $ is a sum of $Q$-graphs defined by
    \begin{align}
  \cal Q_{GG}&:=  Q_x \left(\cal G\right)+  \sum_\al Q_\al\Big[ S^+_{x\al}  G_{\al y}   G_{y'\al} f(G) \Big]  - m Q_y\Big[S^{+}_{xy}  G_{y' y} f(G)\Big] - m Q_x\Big[\sum_\al  s_{x\al} (G_{\al\al}-m) \cal G\Big]  \nonumber\\
   &- m \sum_\al Q_\al\Big[\sum_{ \beta}  S^{+}_{x\al}  s_{\al\beta} (G_{\beta \beta}-m)G_{\al y}   G_{y'\al} f(G)\Big] - mQ_x\Big[ G_{xx }  \sum_\al s_{x\al} G_{\al y}   G_{y'\al} f(G)\Big] \nonumber\\
&- m\sum_\al Q_\al\Big[ \sum_{ \beta}  S^{+}_{x\al}  s_{\al\beta} G_{\al\al }  G_{\beta y}   G_{y'\beta} f(G)\Big]  + m Q_x\Big[\sum_{ \al}  s_{x\al} G_{\al y} G_{y' x} \partial_{ h_{ \al x}}f(G)\Big] \\
& + m\sum_\al Q_\al\Big[ \sum_{\beta} S^{+}_{x\al} s_{\al\beta}G_{\beta y} G_{y' \al} \partial_{h_{\beta \al}}f (G)\Big] .\nonumber 
 \end{align}
On the other hand, if $\cal G=\overline G_{xy}  \overline G_{y' x }  f(G)$, then we define 
$$\wh{\cal O}_{GG}^{(x)}\left[\cal G\right]:=\overline{ \wh{\cal O}_{GG}^{(x)}\left[G_{xy}   G_{y' x }  \overline{f(G)}\right] },$$
where the right-hand side can be defined using \eqref{Oe2x}. Finally, we define
$${\cal O}_{GG}^{(x)}[ \cal G]:= \cal O_{dot} \circ  \wh{\cal O}_{GG}^{(x)}   [\cal G]. $$  
\end{definition}

 We describe the basic properties of the $GG$ expansions in the following lemma. Its proof is straightforward by using Definition \ref{GG-def}, and we postpone it to Appendix \ref{appd localpf}.
  \begin{lemma}\label{expandG2}
 Given a normal regular graph $\cal G = G_{xy}   G_{y' x } f(G)$, where $f(G)$ contains no weights or solid edges attached to $x$  and $y,y'\ne x$. 
Then \smash{${\cal O}_{GG}^{(x)}[ \cal G]$} is a canonical local expansion. 
Moreover, every graph without $Q$-labels, say $\cal G_1$, in \smash{${\cal O}_{GG}^{(x)}[ \cal G]$} satisfies one of the following properties. 
\begin{itemize}
\item[(a)] If $y\not = y'$, then $\cal G_1$ has a strictly higher scaling order than $\cal G$.

\item[(b)] If $y = y'$, then either $\cal G_1$ has a strictly higher scaling order than $\cal G$, or $\cal G_1$ is obtained by replacing $G_{xy} G_{yx}$ in $\cal G$ with $m S^+_{xy}G_{yy}$.
\end{itemize}
Similar statements hold if $\cal G = \overline G_{xy}   \overline  G_{y' x } f(G)$. 
\end{lemma}


Lemma \ref{expandG2} shows that, by applying the $GG$ expansion repeatedly, we can either get rid of atoms 
that are connected with a pair of edges of the same charge, or obtain $Q$-graphs and graphs of sufficiently high scaling orders. 

%
%

Now we define the following concept of \emph{standard neutral atoms}. Roughly speaking, the edges connected with a standard neutral atom almost form a $T$-variable (but not an exact $T$-variable because of the $\times$-dotted edges; see Section \ref{subsec global} for more details).

 \begin{definition}[{Standard neutral atoms}]\label{def SNA}
An atom is said to be \emph{standard neutral} if it is only connected with three edges besides the $\times$-dotted edges: two {matched} $G$ edges of {opposite charges} and one waved $S$ edge. 
 \end{definition}
 
Given a graph with a non-standard neutral atom $x$ (for example, the atom $x$ in graph (f) of \eqref{Aho3} below) that is connected with two matched $G$ edges of opposite charges, we can apply the following $G\overline G$ expansion. The $G\overline G$ expansion \eqref{Oe3x} is a special case of the multi-edge expansion in Definition \ref{multi-def} with $k_1=k_2=1$, $k_3=k_4=0$ or $k_1=k_2=0$, $k_3=k_4=1$.

\begin{definition} [$G\overline G$ expansion operator]\label{GGbar-def} 
Given a normal regular graph $\cal G$, suppose the atom $x$ is connected with exactly two matched $G$ edges of opposite charges,  and $\cal G$ takes the form $\cal G= G_{xy}  \overline G_{xy'} f(G)$ with $y,y' \ne x$.  
Then we define  
\begin{align}
 \wh{\cal O}_{G\overline G}^{(x)}[\cal G] &:=   |m|^2  \sum_\al s_{x\al }G_{\al y} \overline G_{\al y'} f(G) + m \sum_\al s_{x\al }\left(G_{\al \al} -m \right) \cal G \nonumber  \\
&  +  m (\overline G_{xx} - \overline m)   \sum_\al s_{x\al }G_{\al y} \overline G_{\al y'} f(G)  - m  \sum_\al s_{x\al }G_{\al y}  \overline G_{xy'} \partial_{ h_{\al x}} f(G) + \cal Q_{G\overline G}, \label{Oe3x}
 \end{align}
  where on the right-hand side, the first term is of the same scaling order as $\cal G$, and the new atom $\al$ is standard neutral; the second and third terms  contain one more light weight and hence are of strictly higher scaling orders than $\cal G$; the fourth term  contains at least one more off-diagonal $G$ edge and hence is of strictly higher scaling order than $\cal G$. The last term $\cal Q_{G\overline G} $ is a sum of $Q$-graphs defined by
 \begin{align*}
 \cal Q_{G\overline G}&:= Q_x \left(\cal G\right) -  m Q_x\Big[ \sum_\al s_{x\al }\overline G_{xx}G_{\al y} \overline G_{\al y'} f(G) \Big] - m Q_x\Big[\sum_\al s_{x\al }\left(G_{\al \al} -m \right) \cal G\Big] \\
 & + mQ_x\Big[  \sum_\al s_{x\al }G_{\al y}  \overline G_{xy'} \partial_{ h_{\al x}} f(G) \Big].
\end{align*}
On the other hand, if $\cal G= G_{yx}  \overline G_{y'x} f(G)$, then we define $\cal O_{G\overline G}(x)[ \cal G] $ by taking $k_1=k_2=0$ and $k_3=k_4=1$ in Definition \ref{multi-def}, and we omit the explicit expression for simplicity.  Finally, we define
$${\cal O}_{G\overline G}^{(x)}[ \cal G]:= \cal O_{dot} \circ  \wh{\cal O}_{G\overline G}^{(x)}   [\cal G]. $$  
\end{definition}

The purpose of the $G\overline G$ expansion is to turn the non-standard neutral atom $x$ into a new standard neutral atom $\al$ in the first term. The following lemma describes the basic properties of the $G\overline G$ expansion. Its proof is straightforward by using Definition \ref{GGbar-def}, and we postpone it to Appendix \ref{appd localpf}.  
  \begin{lemma}\label{expandGGbar}
  Given a normal regular graph $\cal G= G_{xy}  \overline G_{xy'} f(G)$, where $f(G)$ contains no weights or solid edges attached to $x$  and $y,y'\ne x$. Then \smash{${\cal O}_{G\overline G}^{(x)}[ \cal G]$} is a canonical local expansion. 
  Moreover, every graph without $Q$-labels, say $\cal G_1$, in \smash{${\cal O}_{G\overline G}^{(x)}[ \cal G]$} either has a strictly higher scaling order than $\cal G$, or satisfies one of the following properties:
%
\begin{enumerate}
\item[(a)] $\ord(\cal G_1)= \ord(\cal G)$, $\deg(x)=0$ in $\cal G_1$, and $\cal G_1$ contains one more standard neutral atom;
\item[(b)] $\ord(\cal G_1)= \ord(\cal G)$, and $\cal G_1$ is obtained by replacing $G_{xy}  \overline G_{xy}$ with $|m|^2 s_{xy} |G_{y y}|^2$ in the $y=y'$ case. 
\end{enumerate}
Similar statements hold if $\cal G = G_{y x}   \overline  G_{y' x } f(G)$.

  \end{lemma}

Lemma \ref{expandGGbar} shows that by applying the $G\overline G$ expansions repeatedly, we can get either new graphs containing only standard neutral atoms and degree 0 atoms, or $Q$-graphs and graphs of sufficiently high scaling orders.


\subsection{Local expansion strategy}

We define the concept of \emph{locally standard graphs}.

 \begin{definition} [Locally standard graphs] \label{deflvl1}
A graph $\cal G$ is  \emph{locally standard}  if 
\begin{itemize}
\item[(i)] it is a normal regular graph without $P/Q$ labels; 



\item[(ii)] it has no weights or light weights;

\item[(iii)] the degree of any internal atom is $0$ or $2$; 

\item[(iv)] all degree 2 internal atoms are \emph{standard neutral atoms}. 



\end{itemize}
 \end{definition}

Applying local expansions in Definitions \ref{dot-def}, \ref{Ow-def}, \ref{multi-def}, \ref{GG-def} and \ref{GGbar-def} repeatedly, we can  expand any regular graph into a linear combination of locally standard, recollision, higher order and $Q$ graphs. The expansions will be performed according to the flow chart in Figure \ref{Fig chart1}. More precisely, given a regular graph $\cal G$, we first apply $\cal O_{dot}$ to expand it into a sum of 
normal regular graphs, then apply the weight expansion to remove the weights, and then apply the multi-edge, $GG$  and $G\overline G$ expansions one by one to remove all  atoms that are not standard neutral.  After an expansion, we may need to perform  earlier expansions to the resulting graphs. For example, after a multi-edge expansion, we may get graphs that contain weights. Then before performing another multi-edge expansion, we first need to perform weight expansions to these graphs. This explains why we have loops in Figure \ref{Fig chart1}.


To describe precisely the local expansion process in  Figure \ref{Fig chart1}, we  define the following  stopping rules.  
Given a cut-off order $n$, we stop the expansion of a graph if it is a normal regular graph 
and satisfies at least one of the following properties:
\begin{itemize}
\item[(S1)] it is locally standard;

\item[(S2)] it is a $\oplus$/$\ominus$-recollision graph;

\item[(S3)] its scaling order is at least $n+1$; 

\item[(S4)] it is a $Q$-graph.
\end{itemize}

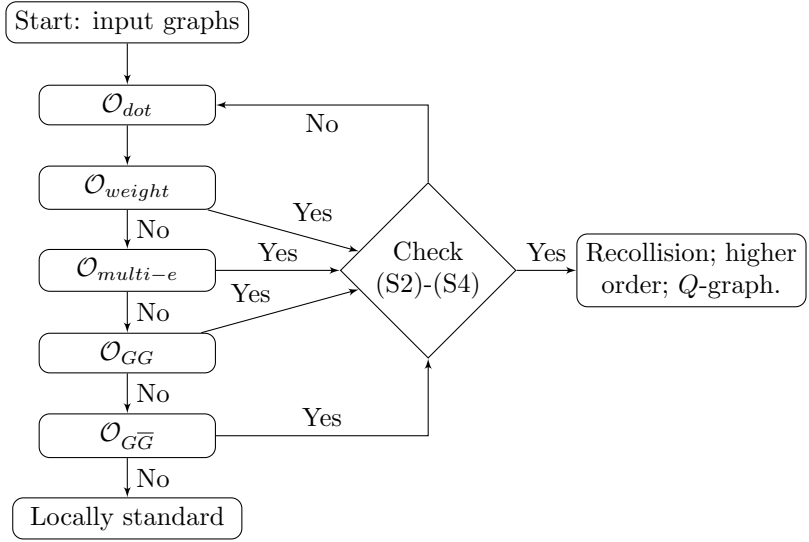
\begin{figure}
\color{black}
\tikzstyle{startstop} = [rectangle,rounded corners, minimum width=1cm,minimum height=0.5cm,text centered, draw=black]
\tikzstyle{startstop3} = [rectangle,rounded corners, minimum width=3cm, text width=8em, minimum height=0.5cm,text badly centered, draw=black]
\tikzstyle{block} = [rectangle, draw=black, 
    text width=6em, text centered, rounded corners, minimum height=1.5em]
\tikzstyle{line} = [draw=black, -latex']
\tikzstyle{line0} = [draw=black]
\tikzstyle{decision} = [diamond, draw=black, 
    text width=4.5em, text badly centered, node distance=2cm, 
    inner sep=0pt]

\tikzstyle{cloud} = [draw=black, ellipse, 
node distance=2.5cm,
    minimum height=2em]
\tikzstyle{null} = [draw=none,fill=none,right]

\begin{center}  
\begin{tikzpicture}[node distance = 1.1cm, auto]

    \node [startstop] (start) {Start: input graphs};
    \node [block, below of= start] (op1) {$\cal O_{dot}$};
    \node [block, below of=op1] (op2) {${\cal O}_{weight}$};
    \node [block, below of=op2] (op3) {$\cal O_{multi-e}$};
    \node [block, below of=op3] (op4) {$\cal O_{GG} $};
    \node [block, below of=op4] (op5) {$\cal O_{G\overline G}$};
    \node [startstop3, below of=op5] (stop_LS) {Locally standard};
    \node [decision, right of=op3, node distance=4cm] (decide_stop) {Check (S2)-(S4)};    
    
    \node [startstop3, right of=decide_stop, node distance=3.5cm] (stop2) {Recollision; higher order; $Q$-graph.};
    
    \path [line] (start) -- (op1);
    \path [line] (op1) -- (op2);
    \path [line] (op2) --node [midway] {No} (op3);
    \path [line] (op3) --node [midway] {No} (op4);
    \path [line] (op4) --node [midway] {No} (op5);
    
    \path [line] (op2) -- node [midway] {Yes} (decide_stop);
    \path [line] (op3) -- node [midway] {Yes} (decide_stop);
    \path [line] (op4) -- node [midway] {Yes} (decide_stop);
    \path [line] (op5) -| node [near start] {Yes} (decide_stop);
    \path [line] (decide_stop) -- node [midway] {Yes} (stop2);
    \path [line] (decide_stop) |- node [near end] {No} (op1);
    
    \path [line] (op5) --node [midway] {No}  (stop_LS);
\end{tikzpicture}
\end{center}
\caption{The flow chart for local expansions. If the weight, multi-edge, $GG$, or $G\overline G$ expansion does not do anything to an input graph (in which case we call it a \emph{null operation}), then we have ``No" and send it to the next operation. In particular, if all  graph operations are null for a graph, then it is locally standard and will be sent to the output.
On the other hand, if a non-trivial graph operation is acted on an input graph, then we have ``Yes"  and we will check whether the resulting graphs satisfy the stopping rules (S2)--(S4). If a graph indeed satisfies the stopping rules, then we send it to the output. Otherwise, we send it back to the first step $\cal O_{dot}$.}\label{Fig chart1}
\end{figure}

\begin{strategy}[Local expansion strategy]\label{strat_local}
We apply the following local expansion strategy.
\begin{itemize}
\item[(1)] We first assign dotted edge partitions of the input graph using $\cal O_{dot}$ such that all resulting graphs are normal regular. 


\item[(2)] For any input graph, pick an atom $x$ and apply ${\cal O}_{weight}^{(x)}$ to expand  the weights on $x$. For the resulting graphs from this expansion, we send the ones satisfying the stopping rules (S2)--(S4) to the outputs, and the remaining graphs back to the first operation $\cal O_{dot}$. If the input graph has no weight, then  ${\cal O}_{weight}$ is a null operation and we send the graph to the next operation.   

\item[(3)]  For any input graph, if it contains atoms of degrees $\notin \{0,2\}$ or atoms connected with two mismatched edges, then we pick one of them, say $x$, and apply \smash{$\cal O_{multi-e}^{(x)}$} to expand the graph. For the resulting graphs, we send the ones satisfying the stopping rules (S2)--(S4) to the outputs, and the remaining graphs back to the first operation $\cal O_{dot}$. If every internal atom in the input graph either has degree 0 or is connected with exactly two matched solid edges, then $\cal O_{multi-e}$ is a null operation and we send the graph to the next operation.  

\item[(4)] For any input graph, if it contains atoms connected with exactly two matched solid edges of the same charge, then we pick one of them, say $x$, and apply \smash{$\cal O_{GG}^{(x)}$} to expand the graph. For the resulting graphs, we send the ones satisfying the stopping rules (S2)--(S4) to the outputs, and the remaining graphs back to the first operation $\cal O_{dot}$. If every internal atom in an input graph is connected with exactly two matched edges of opposite charges, then $\cal O_{GG}$ is a null operation and we send the graph to the next operation.  

\item[(5)] For any input graph, if it contains non-standard neutral atoms, then we pick one of them, say $x$, and apply {$\cal O_{G\overline G}^{(x)}$} to expand the graph. For the resulting graphs, we send the ones satisfying the stopping rules (S2)--(S4) to the outputs, and the remaining graphs back to the first operation $\cal O_{dot}$. 

\item[(6)] Finally, if all the above operations are null, then the input graph is locally standard, and we send it to the output.  
\end{itemize}
\end{strategy}

Finally, we collect all the output graphs of Strategy \ref{strat_local} and obtain the following lemma. The proof of Lemma \ref{lvl1 lemma} is based on Lemmas \ref{expandlabel}, \ref{expandmulti}, \ref{expandG2} and \ref{expandGGbar}, and is postponed to Appendix \ref{appd localpf}.

\begin{lemma} \label{lvl1 lemma}
Let $\cal G_{\fa, \fb_1\fb_2}$ be a normal regular graph without solid edges connected with $\fa$. Then for any fixed $n\in \N$, we can expand it into a sum of $\OO(1)$ many graphs:
\begin{align}\label{expand lvl1}
\cal G_{\fa, \fb_1\fb_2} =  (\cal G_{local})_{\fa, \fb_1\fb_2}  + \PGn_{\fa,\fb_1\fb_2} +  (\AGn)_{\fa,\fb_1\fb_2} + \QGn_{\fa,\fb_1\fb_2} ,
\end{align}
where $(\cal G_{local})_{\fa, \fb_1\fb_2} $ is a sum of locally standard graphs, \smash{$ \PGn_{\fa,\fb_1\fb_2} $} is a sum of  $\oplus$/$\ominus$-recollision graphs, \smash{$(\AGn)_{\fa,\fb_1\fb_2}$} is a sum of graphs of scaling order $> n$, and \smash{$(\QGn)_{\fa,\fb_1\fb_2} $} is a sum of $Q$-graphs.  Every molecule in the graphs on the right side is obtained by merging some molecules in the original graph $\cal G_{\fa, \fb_1\fb_2}$. 
\end{lemma}


We have  noted that local expansions will not create new molecules. Hence if there are no dotted or waved edges added between different molecules, then the molecules in the new graphs are the same as those in $\cal G_{\fa, \fb_1\fb_2}$. In general, there may be newly added dotted edges (due to the dotted edge partition $\cal O_{dot}$) or waved edges (due to the first term on the right-hand side of \eqref{Oe2x}) to the graphs, 
so the molecules in the new graphs are obtained from merging the molecules connected by dotted or waved edges. 

\subsection{Global expansions}\label{subsec global}


In this section, we introduce the global expansions. Suppose that we have the $(n-1)$-th order $T$-expansion by induction. Given a locally standard graph, say $\cal G$, a global expansion consists of the following three steps:
\begin{itemize}
	\item[(i)] choosing a standard neutral atom in $\cal G$;
	\item[(ii)] replacing the $T$-variable containing the atom in (i) by the $(n-1)$-th order $T$-expansion;
	\item[(iii)] applying $Q$-expansions to the resulting graphs with $Q$-labels  from (ii).
\end{itemize}
This procedure 
is called ``global" because it may 
create new molecules in the resulting graphs. For example, if we replace $T_{x,y_1y_2}$ with the right-hand side of \eqref{Otheta} (with $f(G)\equiv 1$), then the new atoms $\al$ and $\beta$ are in a different molecule from $x$. Unlike the local expansions, a global expansion 
may break the doubly connected properties of our graphs (cf. Definition \ref{def 2net}). To avoid this issue,
we need to follow a delicate procedure to choose the standard neutral atom in (i). This will be done fully in \cite{PartII_high} and a brief discussion will be given in Section \ref{sec strategy}. 

 
We now explain briefly the items (ii) and (iii) in the above procedure. Picking a standard neutral atom, say $\al$, in a locally standard graph, the edges connected to it take one of the following forms:
\begin{equation}\label{12in4T}
	t_{x,y_1 y_2} :=|m|^2\sum_\al s_{x\al}G_{\al y_1}\overline G_{\al y_2}\mathbf 1_{\al\ne y_1}\mathbf 1_{\al\ne y_2},\quad \text{or}\quad t_{y_1y_2, x}:=|m|^2\sum_\al G_{ y_1\al }\overline G_{y_2 \al}s_{\al x}\mathbf 1_{\al\ne y_1}\mathbf 1_{\al\ne y_2} .
\end{equation}
Then we apply the $(n-1)$-th order $T$-expansion in \eqref{mlevelTgdef} to these variables in the following way: 
\begin{align*}
	  t_{x,y_1 y_2} & = m  \Theta_{xy_1}\overline G_{y_1y_2} +m (\Theta \Sdelta^{(n-1)}\Theta)_{xy_1} \overline G_{y_1y_2} + (\PT^{(n-1)})_{x,y_1 y_2} +  (\AT^{(>n-1)})_{x,y_1y_2}  + (\QT^{(n-1)})_{x,y_1y_2} \nonumber\\
	& +  (\Err_{n-1,D})_{x,y_1y_2}  - |m|^2\sum_\al s_{x\al}G_{\al y_1}\overline G_{\al y_2}\left(\mathbf 1_{\al\ne y_1}\mathbf 1_{\al = y_2}+\mathbf 1_{\al=y_1}\mathbf 1_{\al \ne y_2} +\mathbf 1_{\al = y_1}\mathbf 1_{\al = y_2}\right) .
\end{align*}
The last term on the right-hand side gives one (if $y_1=y_2$) or two (if $y_1\ne y_2$) recollision graphs, so we combine it with $ (\PT^{(n-1)})_{x,y_1 y_2} $ and denote the resulting expression by $(\wtPT^{(n-1)})_{x,y_1 y_2}$. Hence we have the final expansion formula 
\be\label{replaceT}
 \begin{split}
	t_{x,y_1 y_2} &=   m  \Theta_{xy_1}\overline G_{y_1y_2} +m   (\Theta \Sdelta^{(n-1)}\Theta)_{xy_1} \overline G_{y_1y_2} \\
	& +  (\wtPT^{(n-1)})_{x,y_1 y_2} +  (\AT^{(>n-1)})_{x,y_1y_2}  + (\QT^{(n-1)})_{x,y_1y_2} +  (\Err_{n-1,D})_{x,y_1y_2} .
\end{split}
\ee
The expansion of $t_{y_1y_2, x}$ can be obtained by exchanging the order of matrix indices in the above equation.

In a global expansion, if we replace $t_{x,y_1y_2}$ in a graph, say $\cal G_0$, with a graph in $(\QT^{(n-1)})_{x,y_1y_2}$, we will get a graph of the form  
\be\label{QG}\cal G=\sum_y \Gamma Q_y (\cal G_1) ,\ee
where both $\Gamma$ and $\cal G_1$ are graphs without $P/Q$ labels (more precisely, $\Gamma$ is the subgraph obtained by removing $t_{x,y_1 y_2}$ from $\cal G_0$, and $Q_y (\cal G_1)$ is a $Q$-graph in \smash{$(\QT^{(n-1)})_{x,y_1y_2}$}). Applying the so-called $Q$-expansions, we can expand the above graph into a sum of $Q$-graphs and some graphs without $P/Q$ labels. 
We will give the precise definition of $Q$-expansions in \cite{PartII_high}. Here we only describe briefly the basic ideas. For any $y\in \Z_L^d$, let $H^{(y)}$ be the $(N-1)\times(N-1)$ minor of $H$ obtained by removing the $y$-th row and column of $H$, and define the resolvent minor $G^{(y)}(z):=(H^{(y)}-z)^{-1}$. Using Schur complement formula, we can obtain the following resolvent identity:
$$G_{x_1x_2}=G_{x_1x_2}^{(y)}+\frac{G_{x_1 y}G_{yx_2}}{G_{yy}},\quad x_1,x_2\in \Z_L^d.$$
Applying this identity to expand the resolvent entries in $\Gamma$ one by one, we can write it as 
\be\label{decompose_gamma} \Gamma=\Gamma^{(y)}+\sum_{\omega} \Gamma_\omega.\ee 
Here $\Gamma^{(y)}$ is a graph whose weights and solid edges are $G^{(y)}$ entries, so it is independent of the $y$-th row and column of $H$. The other term is a sum of $\OO(1)$ many graphs, where each $\Gamma_\omega$ has a strictly higher scaling order than $\Gamma$, at least two new solid edges connected with atom $y$, and a factor of the form $(G_{yy})^{-k}(\overline G_{yy})^{-l}$ for some $k,l\in \N$. 
The entry $1/G_{yy}$ can be expanded using Taylor expansion
$$\frac{1}{G_{yy}}=\frac{1}{m} + \sum_{k=1}^{D}\frac1m\left(-\frac{G_{yy}-m}{m}\right)^k + \cal W_{err},\quad \cal W_{err}:=\sum_{k>D}\left(-\frac{G_{yy}-m}{m}\right)^k . $$
We will regard $\cal W_{err}$ as a weight of scaling order $>D$ and collect all  graphs containing it into $\Err_{n,D}$ in \eqref{mlevelTgdef}. Using \eqref{decompose_gamma}, we can expand \eqref{QG} as
\be\label{QG2}
\cal G=\sum_{\omega}\sum_y \Gamma_\omega Q_y(\cal G_1) + \sum_y Q_y\left( \Gamma \cal G_1\right)- \sum_{\omega}\sum_y Q_y\left(\Gamma_\omega \cal G_1\right),
\ee
where the second and third terms are sums of $Q$-graphs. For the first term, we will remove $Q_y$ using some operations that will be introduced in \cite{PartII_high}. The above $Q$-expansion is an expansion of the commutator $[\Gamma, Q_y]$. It has the following important properties: (i) the scaling order of any graph $\sum_y \Gamma_\omega Q_y(\cal G_0)$ is strictly higher than $\cal G$; (ii) for any $\omega$,  at least one weight or solid edge in $\Gamma$ is replaced by two solid edges connected with $y$ in  $\Gamma_\omega$. 

%
%
%
%

 \begin{remark}\label{rem general distr}
The local and global expansions can be readily extended to non-Gaussian band matrices. 
The Gaussian integration by parts will be replaced by the following cumulant expansion in \cite[Proposition 3.1]{Cumulant1} and \cite[Section II]{Cumulant2}. Fix an integer  $l\in \N$ and let $h$ be a real-valued random variable with finite moments up to order $l+2$. Then for  any  $f\in \cal C^{l+1}(\R)$,  we have that
$$\mathbb E [f(h) h]=\sum_{k=0}^{l}\frac1{k!}\kappa_{k+1}(h)\mathbb Ef^{(k)}(h)+ R_{l+1},$$
where $\kappa_{k}(h)$ is the $k$-th cumulant of $h$ and $R_{l+1}$ satisfies that for any $K>0$,
$$R_{l+1}\lesssim \mathbb E\left| h^{l+2} \mathbf 1_{|h|>K}\right|\cdot \|f^{(l+1)}\|_{\infty} + \mathbb E\left| h\right|^{l+2}\cdot \sup_{|x|\le K}|f^{(l+1)}(x)|.$$
Using the cumulant expansions, we can extend the expansions in Lemma \ref{expandtheta} and Definitions \ref{Ow-def}, \ref{multi-def}, \ref{GG-def} and \ref{GGbar-def} to general cases. These general expansions will make the $\Theta$-expansion and the local expansions more complicated, but there are no new ``essential" difficulties. Moreover, the global expansions defined in this subsection can be 
used without any change regardless of the distributions of the matrix entries. With these remarks, we can prove our main results for random band matrices with entries satisfying only certain moment assumptions. Due to the length constraint of the current paper, we will postpone the details of this generalization  to a future work.
\end{remark}

\section{Examples of low order $T$-expansions}\label{sec lower_order}

To help the reader to understand how operations in Section \ref{sec_basiclocal} are applied, in this section we give some examples of low order $T$-expansions. We remark that these examples will not be used in the proof of Theorem \ref{main thm}, so the reader can skip this section and go to Section \ref{sec pfmain} directly for the main proof. 

\subsection{Third order $T$-expansion}
We can derive the third order $T$-expansion by further expanding \eqref{seconduniversal}. Applying the weight expansion \eqref{Owx} to the two terms in $ (\AT^{(>2)})_{\fa,\fb_1\fb_2} $, we can obtain that  
\begin{align*}
  (\AT^{(>2)})_{\fa,\fb_1\fb_2} &=   m^2\sum_{x,y,\al,\beta}\Theta_{\fa x} s_{xy} \left( \delta_{y\al}+  S^+_{y\al}\right)s_{\al\beta}(G_{\al\al}-m)(G_{\beta\beta}-m) G_{x \fb_1}\overline G_{x\fb_2} \\
&-m^2 \sum_{x,y,\al,\beta} \Theta_{\fa x} s_{xy}\left(\delta_{y\al}+S^+_{y\al}\right)s_{\al\beta}G_{\beta \al}\partial_{ h_{\beta\al}}(G_{x \fb_1}\overline G_{x \fb_2}) \\
&+ |m|^2 \sum_{x,y,\al,\beta} \Theta_{\fa x} s_{xy}\left( \delta_{x\al} + S^-_{x\al} \right)s_{\al\beta}( \overline G_{\al\al} - \overline m)( \overline G_{\beta\beta} - \overline m) G_{y \fb}\overline G_{y\fb_2}\\
& - |m|^2 \sum_{x,y,\al,\beta} \Theta_{\fa x} s_{xy} \left(\delta_{x\al}+S^-_{x\al}\right)s_{\al\beta}\overline G_{\beta\al} \partial_{h_{\al\beta}}( G_{y \fb_1}\overline G_{y\fb_2} )+\cal Q_{T,3} .
\end{align*}
Here $\cal Q_{T,3}$ is a sum of $Q$-graphs that can be derived from \eqref{Owx}, but we do not write down its expression for simplicity. If we expand the partial derivatives using \eqref{partialhG}, and use the identity 
\be\label{iden_S+}\sum_y m^2 s_{xy}\left(\delta_{y\al}+S^+_{y\al}\right) = S^+_{x\al},\ee  
we can reduce the above expansion to $(\AT^{(>2)})_{\fa,\fb_1\fb_2} =(\AT^{(>3)})_{\fa,\fb_1\fb_2} + (\cal Q_{T,3})_{\fa,\fb_1\fb_2},$ where
\begin{align}
 (\AT^{(>3)})_{\fa,\fb_1\fb_2} &:=   \sum_{x,\al,\beta} \Theta_{\fa x}  S^+_{x\al}s_{\al\beta} (G_{\al\al}-m)(G_{\beta\beta}-m) G_{x\fb_1}\overline G_{x\fb_2} \label{Aho>3}\\
&+ |m|^2 \sum_{x,y,\beta} \Theta_{\fa x} s_{xy} s_{x\beta}( \overline G_{xx} - \overline m)( \overline G_{\beta\beta} - \overline m) G_{y \fb_1}\overline G_{y\fb_2} \nonumber  \\
&+ |m|^2 \sum_{x,y,\al,\beta} \Theta_{\fa x} s_{xy} S^-_{x\al} s_{\al\beta}( \overline G_{\al\al} - \overline m)( \overline G_{\beta\beta} - \overline m) G_{y \fb_1}\overline G_{y\fb_2} \nonumber  \\
&+ \sum_{x,\al,\beta} \Theta_{\fa x} S^+_{x\al}s_{\al\beta}G_{\beta\al} G_{x\beta}G_{\al \fb_1}\overline G_{x\fb_2}  +  \sum_{x,\al,\beta} \Theta_{\fa x} S^+_{x\al}s_{\al\beta}G_{\beta\al} G_{x \fb_1}\overline G_{x\al}\overline G_{\beta\fb_2} \nonumber \\
&+ |m|^2 \sum_{x,y,\beta} \Theta_{\fa x} s_{xy} s_{x\beta}\overline G_{\beta x}  G_{y x}G_{\beta\fb_1}\overline G_{y\fb_2}  + |m|^2 \sum_{x,y,\beta} \Theta_{\fa x} s_{xy} s_{x\beta}\overline G_{\beta x}   G_{y\fb_1}\overline G_{y\beta}\overline G_{x\fb_2}  \nonumber\\
& + |m|^2 \sum_{x,y,\al,\beta} \Theta_{\fa x} s_{xy}  S^-_{x\al} s_{\al\beta}\overline G_{\beta\al}  G_{y\al}G_{\beta\fb_1}\overline G_{y\fb_2} + |m|^2 \sum_{x,y,\al,\beta} \Theta_{\fa x} s_{xy}  S^-_{x\al} s_{\al\beta}\overline G_{\beta\al}  G_{y\fb_1}\overline G_{y\beta}\overline G_{\al\fb_2}  .\nonumber 
\end{align}
Now we draw the 9 graphs of \eqref{Aho>3} in the following figure: 
\be  \label{Aho3} 
\parbox[c]{0.8\linewidth}{\includegraphics[width=12.5cm]{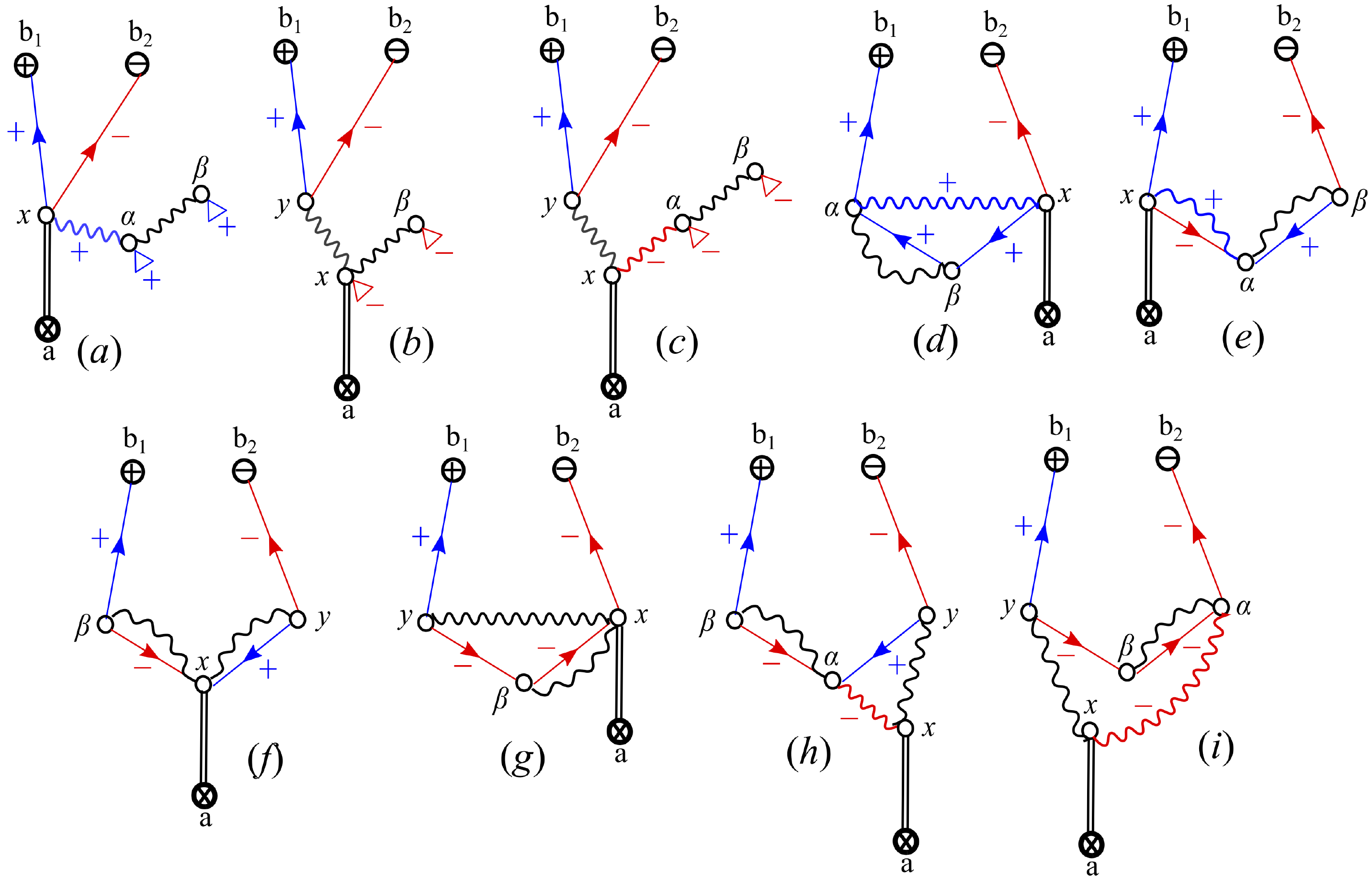}} 
\ee
For conciseness, we do not draw the coefficients of these graphs. The graphs in \eqref{Aho3} are not yet normal regular, but it is easy to see that after applying $\cal O_{dot}$ to them, all the resulting graphs are of scaling order $\ge 4$. Thus we have obtained the following third order $T$-expansion  
\begin{align} 
 T_{\fa,\fb_1\fb_2} &=  m  \Theta_{\fa\fb_1}\overline G_{\fb_1\fb_2} 
  +  (\AT^{(>3)})_{\fa,\fb_1\fb_2} + (\QT^{(3)})_{\fa,\fb_1\fb_2},\label{thirduniversal}
 \end{align}
 where $ \QT^{(3)}:= \cal Q_{T,3} + \QT^{(2)}$.
 
 \subsection{Fourth order $T$-expansion}
Next we can perform local and global expansions to the graphs in $\AT^{(>3)}$ to construct the fourth order $T$-expansion. Since the expression of the fourth order $T$-expansion is rather lengthy and does not help our proof, we will not give its explicit form in this paper. Instead, we will describe the expansions of several typical graphs to show that we actually have $\Sdelta^{(4)}=0$ in the current setting where $H$ has complex Gaussian entries.  

First, the graphs (a), (b), (c) in \eqref{Aho3} all have two light weights in them. Taking graph (a) as an example, we apply the weight expansion in Definition \ref{Owx} to the weight $G_{\beta\beta}-m$ and get that
\begin{align*}
& \sum_{x,\al,\beta} \Theta_{\fa x}  S^+_{x\al}s_{\al\beta} (G_{\al\al}-m)(G_{\beta\beta}-m) G_{x\fb_1}\overline G_{x\fb_2}=  \cal Q_a\\
& + m^{-1} \sum_{x,\al,\gamma_1,\gamma_2} \Theta_{\fa x} S^+_{x\al}S^+_{\al\gamma_1}s_{\gamma_1\gamma_2}(G_{\al\al}-m)(G_{\gamma_1\gamma_1}-m)(G_{\gamma_2\gamma_2}-m) G_{x\fb_1}\overline G_{x\fb_2} \tag{a1}\\
& +m^{-1} \sum_{x,\al, \gamma_1,\gamma_2} \Theta_{\fa x}  S^+_{x\al}S^+_{\al \gamma_1}s_{\gamma_1\gamma_2} G_{\gamma_2\gamma_1} G_{\al\gamma_2}G_{\gamma_1\al} G_{x\fb_1}\overline G_{x\fb_2} \tag{a2} \\
& + m^{-1} \sum_{x,\al,\gamma_1,\gamma_2} \Theta_{\fa x}  S^+_{x\al}S^+_{\al \gamma_1}s_{\gamma_1\gamma_2} (G_{\al\al}-m) G_{\gamma_2\gamma_1}  G_{x \gamma_2}G_{\gamma_1 \fb_1}\overline G_{x \fb_2} \tag{a3}  \\
& +m^{-1} \sum_{x,\al, \gamma_1,\gamma_2} \Theta_{\fa x}  S^+_{x\al} S^+_{\al\gamma_1}s_{\gamma_1\gamma_2} (G_{\al\al}-m) G_{\gamma_2\gamma_1}  G_{x\fb_1}\overline G_{x \gamma_1} \overline G_{\gamma_2 \fb_2} , \tag{a4}
 \end{align*}
where we used \eqref{iden_S+} in the derivation, and $\cal Q_a$ is a sum of $Q$-graphs. In \eqref{Aho41}, we draw the four graphs (a1)--(a4), where for conciseness we do not draw the coefficients of them. 
\be  \label{Aho41} 
\parbox[c]{0.8\linewidth}{\includegraphics[width=12.5cm]{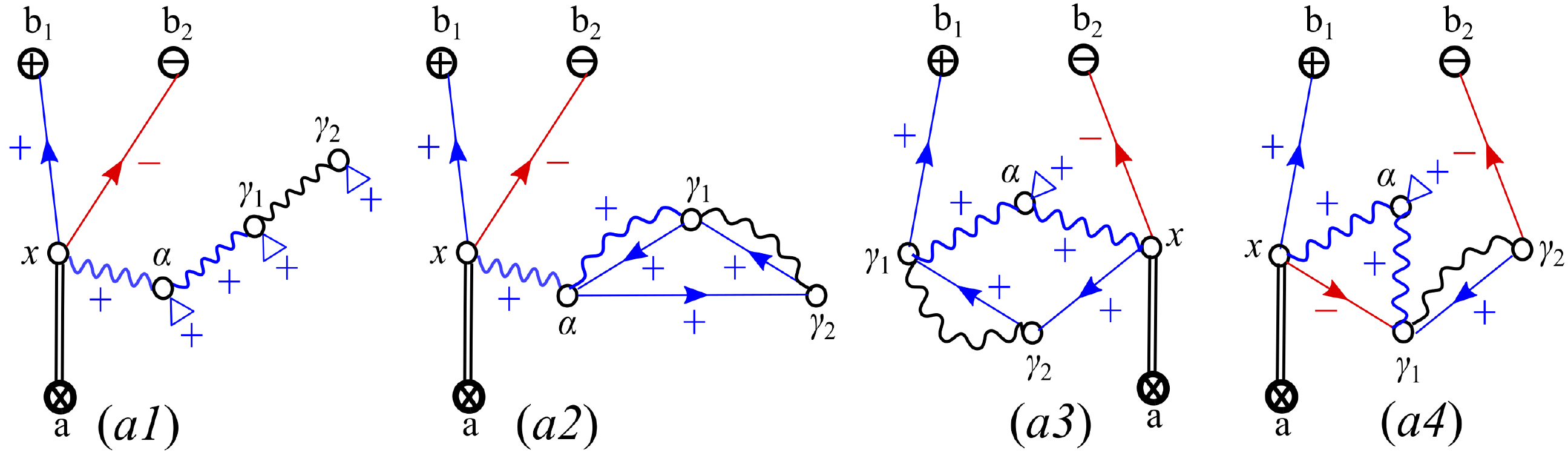}} 
\ee
The graphs in \eqref{Aho41} are not yet normal regular, but it is easy to see that after applying $\cal O_{dot}$ to them, all the resulting graphs are of scaling order $\ge 5$.   
Similarly, we can check that applying the weight expansion to the light weights in graphs (b) and (c) of \eqref{Aho3} will give graphs of scaling order $\ge 5$. 

Second, the graphs (d), (g) and (i) in \eqref{Aho3} all have an atom $\beta$ connected with two matched edges of the same charge. Taking graph (i) as an example, we apply the $GG$ expansion in Definition \ref{GG-def} to the two edges connected with $\beta$, and get that 
\begin{align*}
& 	|m|^2 \sum_{x,y,\al,\beta} \Theta_{\fa x} s_{xy}  S^-_{x\al} s_{\al\beta}\overline G_{\beta\al} \overline G_{y\beta} G_{y\fb_1}\overline G_{\al\fb_2}=\cal Q_i\\
& + |m|^2 \overline m\sum_{x,y,\al,\beta} \Theta_{\fa x} s_{xy}  S^-_{x\al} s_{\al\beta}S^-_{\al\beta}\overline G_{y\al} G_{y\fb_1}\overline G_{\al\fb_2} \tag{i1}\\
&  + m\sum_{x,y,\al, \gamma_1,\gamma_2} \Theta_{\fa x} s_{xy}  S^-_{x\al} S^-_{\al \gamma_1}s_{\gamma_1\gamma_2} (\overline G_{\gamma_2\gamma_2}-\overline m ) \overline G_{\gamma_1\al} \overline G_{y \gamma_1} G_{y\fb_1}\overline G_{\al \fb_2} \tag{i2}\\
&  +  m\sum_{x,y,\al, \gamma_1,\gamma_2} \Theta_{\fa x} s_{xy}  S^-_{x\al}S^-_{\al\gamma_1}s_{\gamma_1\gamma_2} (\overline G_{\gamma_1\gamma_1}-\overline m )\overline G_{\gamma_2\al} \overline G_{y \gamma_2}  G_{y\fb_1}\overline G_{\al\fb_2} \tag{i3}\\
& +  m\sum_{x,y,\al, \gamma_1,\gamma_2} \Theta_{\fa x} s_{xy}  S^-_{x\al} S^-_{\al\gamma_1}s_{\gamma_1\gamma_2} |G_{y \gamma_1}|^2 \overline G_{\gamma_2\al}   G_{\gamma_2 \fb_1}\overline G_{\al \fb_2} \tag{i4}\\
& +  m\sum_{x,y,\al,\gamma_1,\gamma_2} \Theta_{\fa x} s_{xy}  S^-_{x\al}  S^-_{\al\gamma_1}s_{\gamma_1\gamma_2} \overline G_{y \gamma_1}\overline G_{\al \gamma_2}  \overline G_{\gamma_2\al}  G_{y\fb_1}\overline G_{\gamma_1 \fb_2} , \tag{i5}
\end{align*}
where we used the complex conjugate of \eqref{iden_S+} in the derivation, and $\cal Q_i$ is a sum of $Q$-graphs. In \eqref{Aho42}, we draw the five graphs (i1)--(i5), where for conciseness we do not draw the coefficients of the graphs.  
\be  \label{Aho42} 
\parbox[c]{0.85\linewidth}{\includegraphics[width=13.5cm]{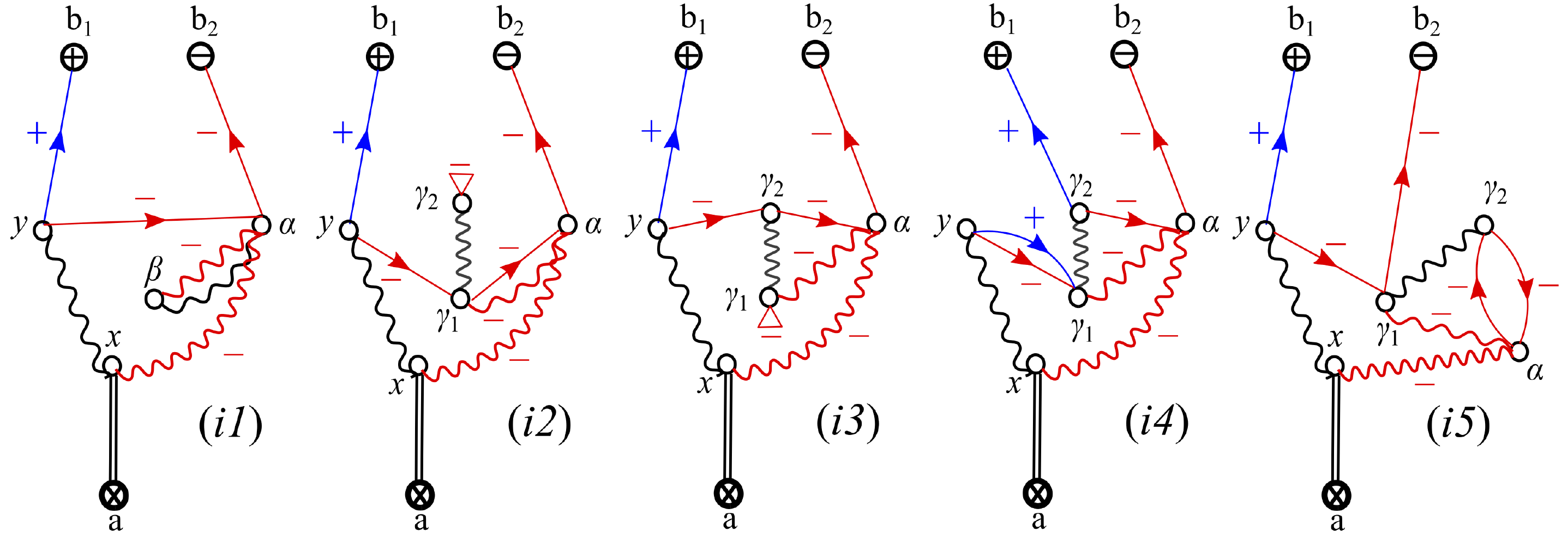}} 
\ee
The graphs in \eqref{Aho42} are not yet normal regular, but it is easy to see that after applying $\cal O_{dot}$ to them, all the resulting graphs are of scaling order $\ge 5$. Similarly, we can check that applying the $GG$ expansion to the two $G$ edges connected with atom $\beta$ in the graphs (d) and (g) of \eqref{Aho3} will give graphs of scaling order $\ge 5$. 

Finally, the graphs (e), (f) and (h) of \eqref{Aho3} only contain degree 2 atoms connected with two matched edges of opposite charges, but not all atoms in them are standard neutral, such as the atom $x$ in (f) and the atom $\al$ in (e) and (h). Taking graph (f) as an example, we apply the $G\overline G$ expansion in Definition \ref{GGbar-def} to the two edges connected with atom $x$, and get that 
\begin{align*}
	&|m|^2 \sum_{x,y,\beta} \Theta_{\fa x} s_{xy} s_{x\beta}\overline G_{\beta x}  G_{y x}G_{\beta\fb_1}\overline G_{y\fb_2}= \cal Q_f\\
&+ |m|^4 \sum_{x,y,\al,\beta } \Theta_{\fa x} s_{xy} s_{x\beta}s_{x\al}\overline G_{\beta\al}  G_{y \al}G_{\beta\fb_1}\overline G_{y\fb_2} \tag{f1}\\
&+ |m|^2 m \sum_{x,y,\al,\beta} \Theta_{\fa x} s_{xy} s_{x\beta}s_{x\al} (G_{\al\al}-m)\overline G_{\beta x}  G_{y x}G_{\beta \fb_1}\overline G_{y\fb_2} \tag{f2}\\
&+ |m|^2 m \sum_{x,y,\al,\beta} \Theta_{\fa x} s_{xy} s_{x\beta}s_{x\al} (\overline G_{xx}-\overline m)\overline G_{\beta\al}  G_{y\al}G_{\beta\fb_1}\overline G_{y\fb_2} \tag{f3}\\
&+ |m|^2 m\sum_{x,y,\al,\beta } \Theta_{\fa x} s_{xy} s_{x\beta}s_{x\al} G_{y\al}  |G_{\beta x}|^2 G_{\al\fb_1}\overline G_{y\fb_2}   \tag{f4}\\
&+ |m|^2 m\sum_{x,y,\al,\beta} \Theta_{\fa x} s_{xy} s_{x\beta}s_{x\al}\overline G_{\beta x}  |G_{y \al}|^2  G_{\beta \fb_1} \overline G_{x\fb_2} , \tag{f5}
\end{align*}
where $\cal Q_f$ is a sum of $Q$-graphs. In \eqref{Aho43}, we draw the five graphs (f1)--(f5), where for conciseness we do not draw the coefficients of them. 
 \be  \label{Aho43} 
\parbox[c]{0.85\linewidth}{\includegraphics[width=13.5cm]{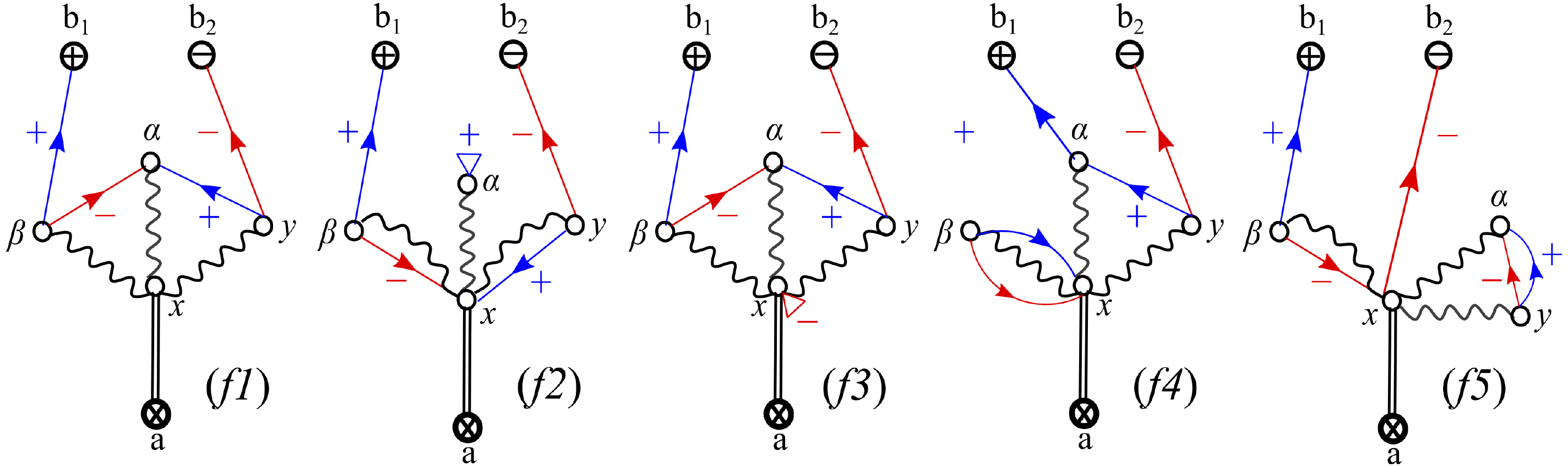}} 
\ee
Here the graph (f1) is the main term, while the graphs (f2)--(f5) all give graphs of scaling order $\ge 5$ after a dotted edge partition $\cal O_{dot}$.  Next we apply a global expansion to (f1), that is, we replace $|m|^2\sum_{\al}s_{x\al}G_{y\al}\overline G_{\beta\al }$ with the second order $T$-expansion in \eqref{seconduniversal}:
\begin{align*}
&	|m|^4 \sum_{x,y,\al,\beta } \Theta_{\fa x} s_{xy} s_{x\beta}s_{x\al}\overline G_{\beta\al}  G_{y \al}G_{\beta\fb_1}\overline G_{y\fb_2}  =  |m|^2 \sum_{x,y,\beta} \Theta_{\fa x} s_{xy} s_{x\beta}(\QT^{(2)})_{y\beta, x}  G_{\beta \fb_1}\overline G_{y\fb_2} \\
&+ |m|^2 m \sum_{x,y,\beta} \Theta_{\fa x} s_{xy} s_{x\beta}\Theta_{xy} \overline G_{\beta y}  G_{\beta\fb_1}\overline G_{y\fb_2} \tag{f1.1}\\
&+|m|^2 m\sum_{x,y,\beta ,\gamma_1,\gamma_2} \Theta_{\fa x} s_{xy} s_{x\beta}\Theta_{x\gamma_1} s_{\gamma_1\gamma_2} (G_{\gamma_2\gamma_2}-m) \overline G_{\beta \gamma_1}  G_{y \gamma_1}G_{\beta \fb_1}\overline G_{y\fb_2} \tag{f1.2}\\
&+|m|^2 m\sum_{x,y,\beta ,\gamma_1,\gamma_2} \Theta_{\fa x} s_{xy} s_{x\beta}\Theta_{x \gamma_1} s_{\gamma_1\gamma_2} (\overline G_{\gamma_1\gamma_1}-\overline m) \overline G_{\beta \gamma_2}  G_{y \gamma_2}G_{\beta \fb_1}\overline G_{y \fb_2}. \tag{f1.3}
\end{align*}
We need to further apply a $Q$-expansion to the first term on the right-hand side. In \eqref{Aho44}, we draw the three graphs (f1.1)--(f1.3), where for conciseness we do not draw the coefficients of them. It is easy to see that after applying $\cal O_{dot}$ to these graphs, all the resulting graphs are of scaling order $\ge 5$.  
 \be  \label{Aho44} 
\parbox[c]{0.7\linewidth}{\includegraphics[width=11cm]{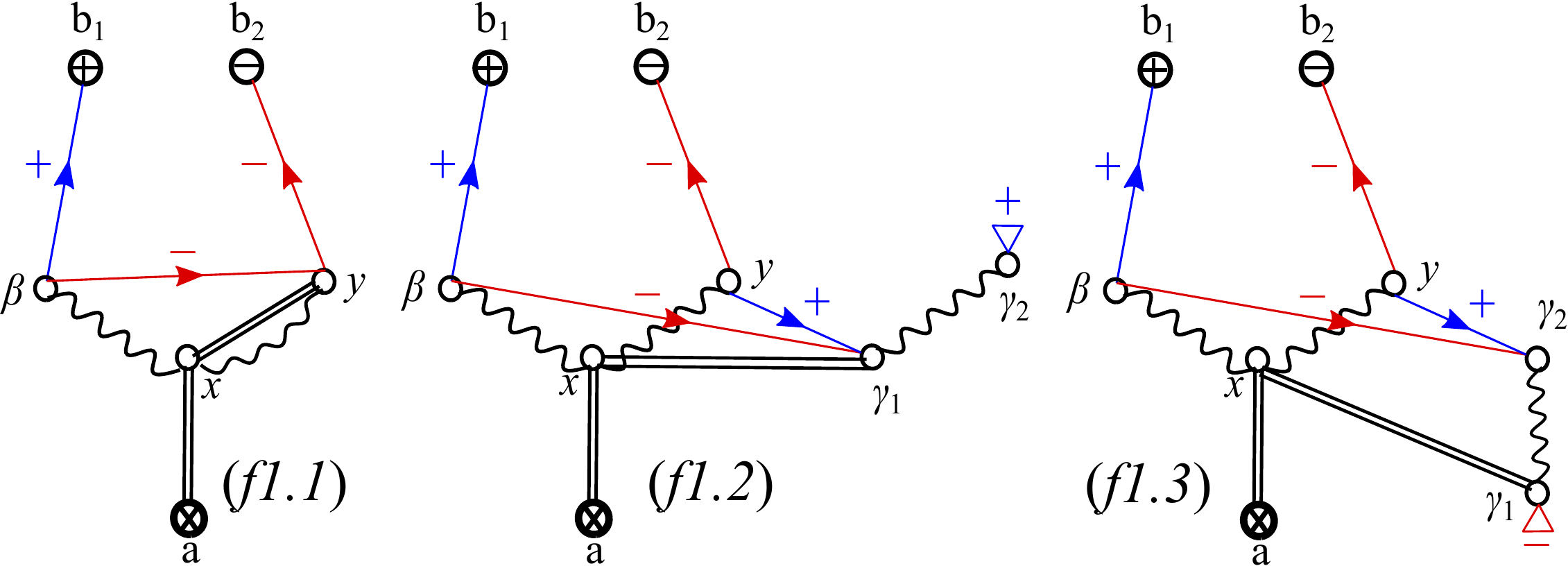}} 
\ee
Similarly, we can check that after applying the $G\overline G$ expansion to the two $G$ edges connected with atom $\al$ in graphs (e) and (h) of \eqref{Aho3} and then applying a global expansion, we will get graphs of scaling order $\ge 5$. 
    
To summarize, we have found that in the fourth order $T$-expansion, the fourth order $\self$ $\Sele_4$ vanishes.

\subsection{Examples of graphs in $\self$ $\Sele_6$}
In this subsection, we use some examples to show that the 6th order $\self$ actually contains non-trivial graphs. Hence, unlike $\Sele_4$, its sum zero property \eqref{weaz} is not trivial anymore. We remark there are hundreds of ways to get graphs in $\Sele_6$, and we are not trying to exhaust all of them. 
   
First, if we assign the dotted edge partition such that the two internal $G$ edges in (d)--(i) of \eqref{Aho3} are diagonal, then we will get sixth order graphs. For example, for the graph (d) of \eqref{Aho3}, we assign dotted edges $\delta_{x\beta}$ and $\delta_{\al \beta}$, and then replace the two weights $G_{xx}^2$ with $m^2$; for the graph (h), we assign dotted edges $\delta_{y\al}$ and $\delta_{\al\beta}$, and then replace the two weights $|G_{yy}|^2$ with $|m|^2$. Then we get the following two graphs:
\be  \label{Aho61} 
\parbox[c]{0.3\linewidth}{\includegraphics[width=5cm]{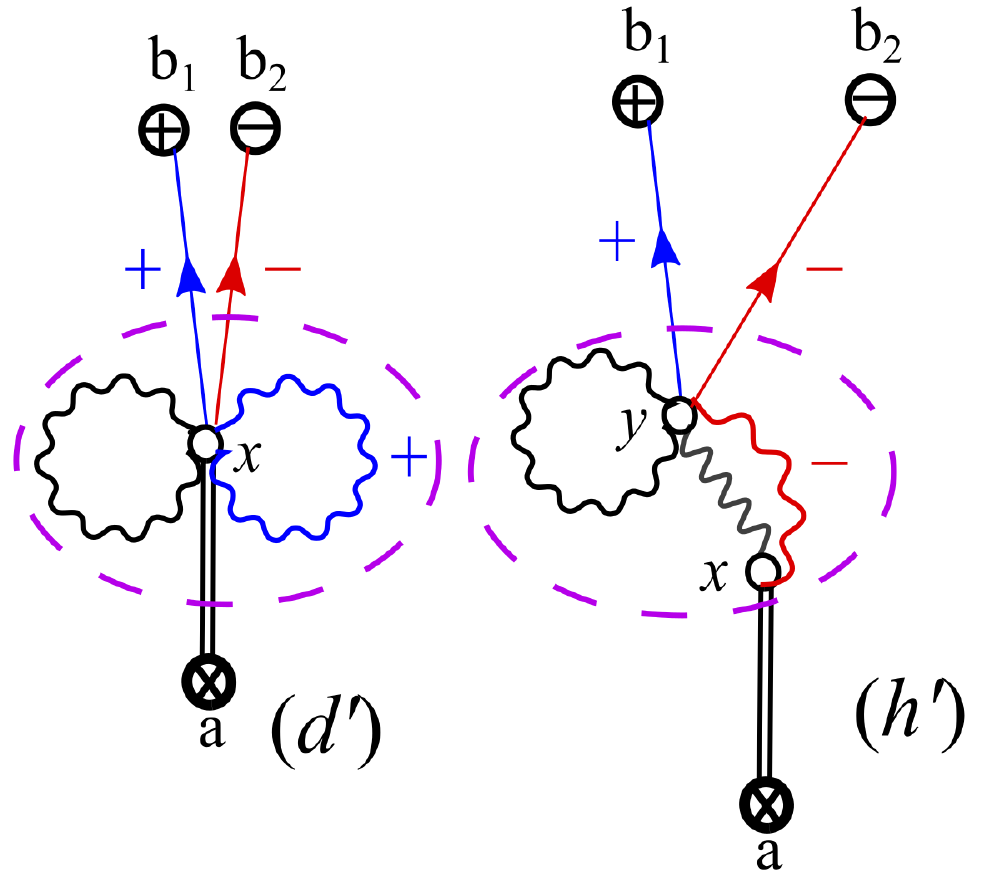}} 
\ee
Inside the purple dashed circles are two deterministic graphs in $(\Sele_6)_{xy}$ (except for the coefficients):
$$(d'): \ m^2 \delta_{xy} s_{xx}S^+_{xx},\quad (h'):\ |m|^4 s_{yy}s_{xy} S^-_{xy}. $$ 
 
As the second example, in the graph (i4) of \eqref{Aho42}, if we assign a dotted edge $\delta_{\al\gamma_2}$, replace the weight $\overline G_{\al\al}$ with $\overline m$, replace the $T$-variable $|m|^2\sum_{y}s_{xy}|G_{y\gamma_1}|^2$ with $|m|^2\Theta_{x\gamma_1}$ in a global expansion, and rename $\al$ as $y$, we then get the graph (i4$'$) in \eqref{Aho62}. In the graph (i5) of \eqref{Aho42}, if we assign a dotted edge $\delta_{y \gamma_1}$, replace the weight $\overline G_{yy}$ with $\overline m$, and replace the two edges  $\overline G_{\al\gamma_2}\overline G_{\gamma_2\al}$ with $\overline m^2 S^-_{\al\gamma_2}$ in the $GG$ expansion, then we get the graph (i5$'$) in \eqref{Aho62}.
    \be  \label{Aho62} 
\parbox[c]{0.4\linewidth}{\includegraphics[width=6.5cm]{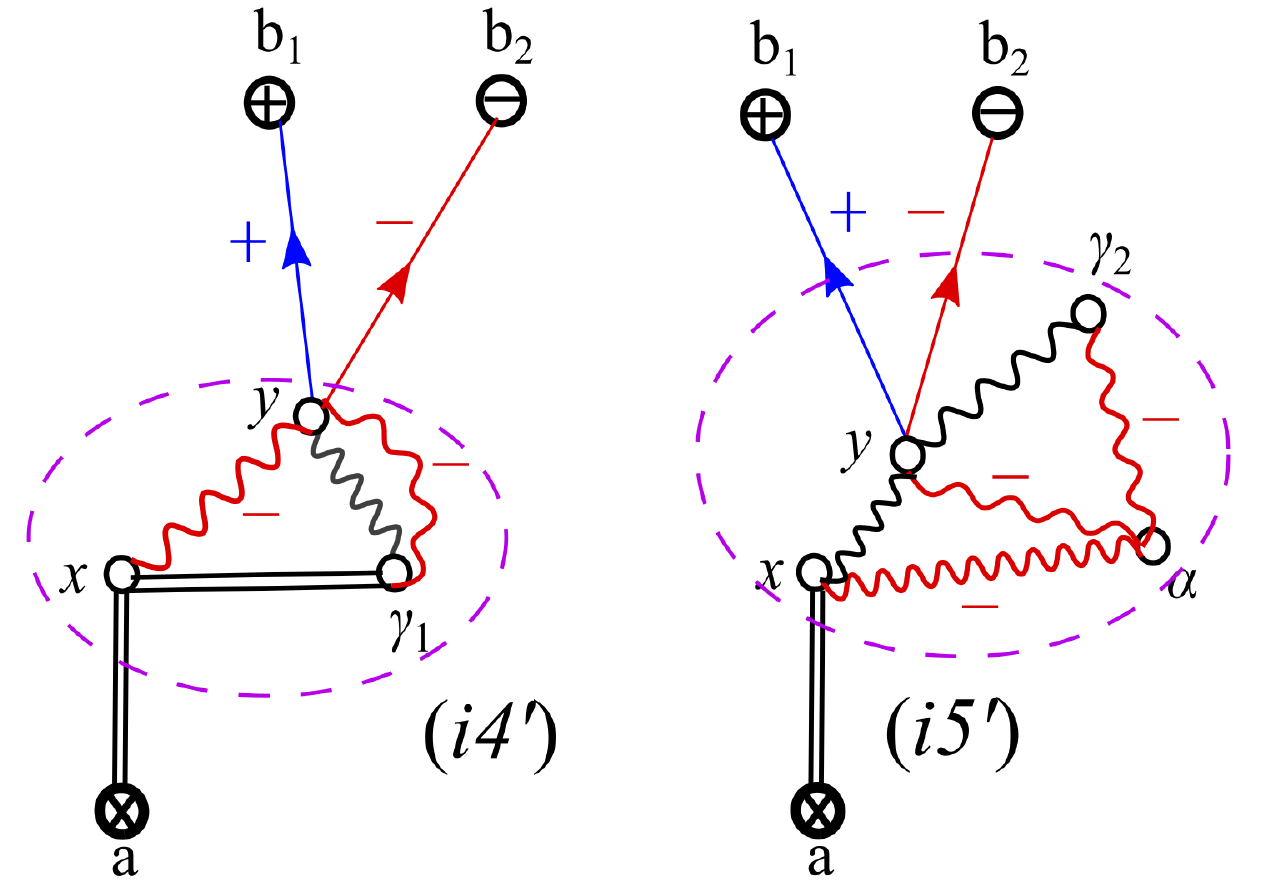}} 
\ee
Inside the purple dashed circles are two deterministic graphs in $(\Sele_6)_{xy}$ (except for the coefficients):
$$  (i4'):\ |m|^2 S^-_{xy}\sum_{\gamma_1}s_{y\gamma_1}S^-_{y\gamma_1} \Theta_{x\gamma_1}, \quad (i5'): \ |m|^2\overline m^2 s_{xy} \sum_{\al,\gamma_2}S^-_{x\al} S^-_{y\al} S^-_{\al \gamma_2}s_{y\gamma_2}.$$

As the last example, in the graph (f4) of \eqref{Aho43}, if we assign a dotted edge $\delta_{y\al}$, replace the weight $  G_{yy}$ with $ m$, and replace the $T$-variable $|m|^2\sum_{\beta}s_{x\beta}|G_{\beta x}|^2$ with a $|m|^2\Theta_{xx}$ edge in a global expansion, then we get the graph (f4$'$) in \eqref{Aho63}. In the graph (f1.1) of \eqref{Aho44}, if we assign a dotted edge $\delta_{y\beta}$ and replace the weight $\overline G_{yy}$ with $\overline m$, then we get the graph (f1.1$'$) in \eqref{Aho63}.
    \be  \label{Aho63} 
\parbox[c]{0.3\linewidth}{\includegraphics[width=5cm]{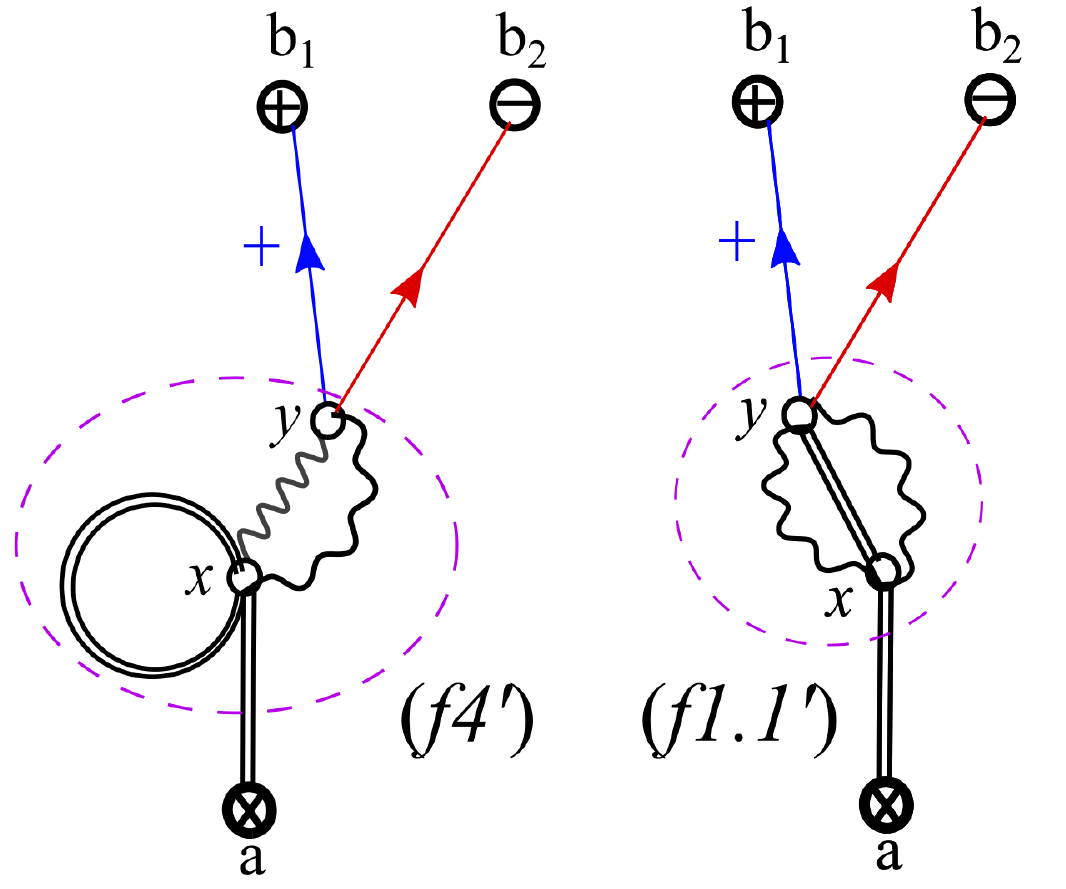}} 
\ee
Inside the purple dashed circles are two deterministic graphs in $(\Sele_6)_{xy}$ (except for the coefficients):
$$ \quad (f4'):\   |m|^2 m^2 \Theta_{xx}(s_{xy})^2 , \quad (f1.1'): \  |m|^4(s_{xy})^2\Theta_{xy}.$$

\section{Proof of Theorem \ref{main thm}}\label{sec pfmain}

In this section, we give an outline of the proof of Theorem \ref{main thm}. Some lemmas used in the proof will be proved in subsequent sections and \cite{PartII_high}. We first recall the following large deviation estimates in Lemma \ref{lem G<T}, which show that the resolvent entries can be bounded using the $T$-variables in \eqref{defGT}. The bound \eqref{offG largedev} was proved in equation (3.20) of \cite{PartIII}, while \eqref{diagG largedev} was proved in Lemma 5.3 of \cite{delocal}. Given a matrix $M$, we will use $\|M\|_{\max}=\max_{i,j}|M_{ij}|$ to denote its maximum norm. 

\begin{lemma}\label{lem G<T}
Suppose for a constant $\delta_0>0$ and deterministic parameter $W^{-d/2}\le \Phi\le W^{-\delta_0}$ we have that
\be\label{initialGT} 
\|G(z)-m(z)\|_{\max}\prec W^{-\delta_0},\quad \|T\|_{\max} \prec \Phi^2 ,
\ee
uniformly in $z\in \mathbf D$ for a subset $\mathbf D\subset \C_+$. Then 
\be\label{offG largedev} \mathbf 1_{x\ne y} |G_{xy}(z)|^2  \prec T_{xy}(z)\ee
uniformly in $x\ne y \in \Z_L^d$ and $z\in \mathbf D$, and
\be\label{diagG largedev} |G_{xx}(z)-m(z)| \prec \Phi
\ee
	uniformly in $x\in \Z_L^d$ and $z\in \mathbf D$.
\end{lemma}

\subsection{Main structure of the proof}  \label{sec_main_struct}
The proof of Theorem \ref{main thm} will proceed by induction on $n$, the scaling order of the $T$-expansion. 

\vspace{5pt}

\noindent{\bf Step 1: Second order $T$-expansion.} The second order $T$-expansion has been given by Lemma \ref{2nd3rd T}.

\vspace{5pt}

\noindent{\bf Step 2: Local law.}  Assume by induction that we have obtained the $k$-th order $T$-expansion for $2\le k \le n-1$. Then 
we will prove in Theorem \ref{thm ptree} that  the local law \eqref{locallaw1} holds when $L$ satisfies the condition (with $n$ in \eqref{Lcondition1} replaced by $n-1$) 
\be\label{Lcondition0}  {L^2}/{W^2}  \le W^{(n-2)d/2-c_0} .\ee

%

 \vspace{5pt}

\noindent{\bf Step 3: $n$-th order $\incomp$.}  Given the $k$-th order $T$-expansion for $2\le k\le n-1$, we will construct an $n$-th order $\incomp$ in Lemma \ref{incomplete Texp}.  
 
\vspace{5pt}

\noindent{\bf Step 4: Sum zero property.} With the $n$-th order $\incomp$ in Step 3, using the local law proved in Step 2 we will show in Lemma \ref{cancellation property} that the $n$-th order $\self$ $\Sele_n$ satisfies the sum zero property. 

\vspace{5pt}

\noindent{\bf Step 5: $n$-th order $T$-expansion.} With the $n$-th order $\incomp$ in Step 3 and the sum zero property for $\Sele_n$, we will construct an $n$-th order $T$-expansion in Lemma \ref{lem completeTexp}.  

 
\vspace{5pt}

\noindent Combining these steps, by induction on $n$ we obtain an $n$-th order $T$-expansion for any fixed $n\in \N$. Theorem \ref{thm ptree} then implies that the local law \eqref{locallaw1} holds for $L$ satisfying \eqref{Lcondition1}. This concludes Theorem \ref{main thm} since $n$ is arbitrary. In Figure \ref{Fig pfchart1}, we illustrate the structure of the whole proof of Theorem \ref{main thm} with a flow chart.

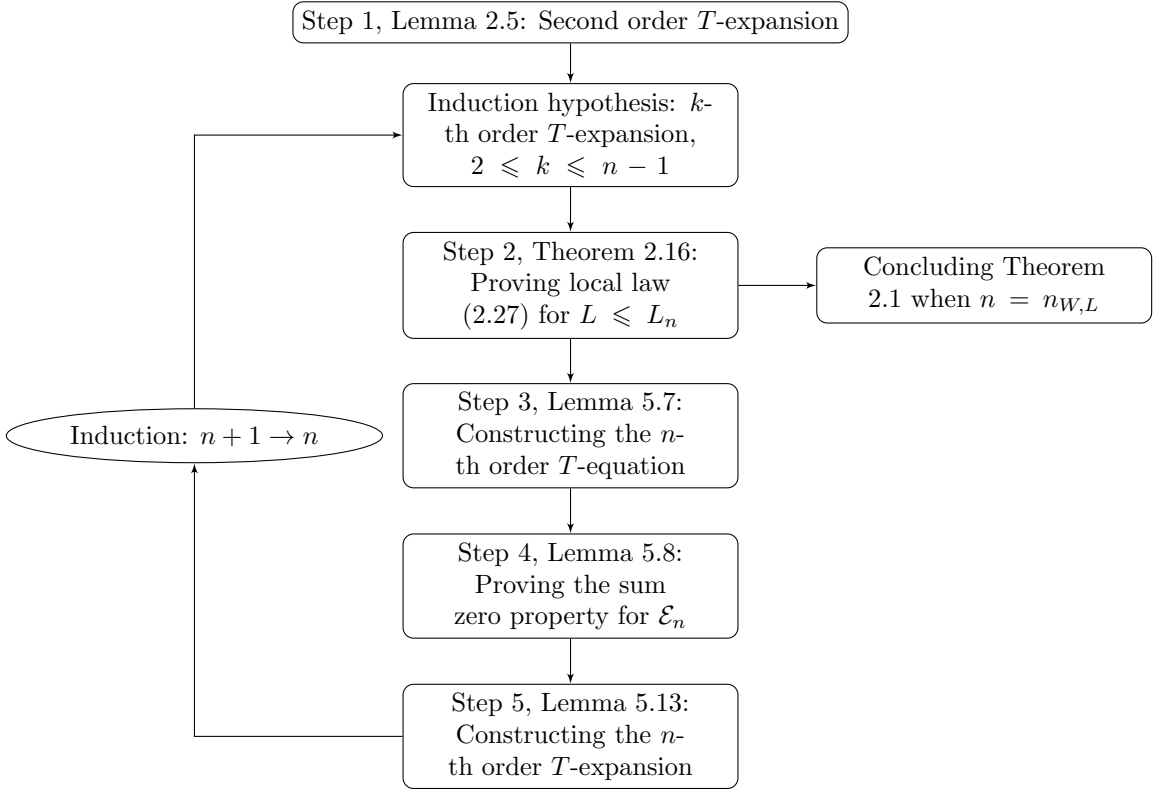
\begin{figure}[htb]
\color{black}
\tikzstyle{startstop} = [rectangle,rounded corners, minimum width=1cm,minimum height=0.5cm,text centered, draw=black]
\tikzstyle{startstop3} = [rectangle,rounded corners, minimum width=3cm, text width=6em, minimum height=0.5cm,text badly centered, draw=black]
\tikzstyle{block} = [rectangle, draw=black, 
    text width=12em, text centered, rounded corners, minimum height=1.5em]
\tikzstyle{line} = [draw=black, -latex']
\tikzstyle{line0} = [draw=black]
\tikzstyle{decision} = [diamond, draw=black, 
    text width=3em, text badly centered, node distance=2cm, 
    inner sep=0pt]

\tikzstyle{cloud} = [draw=black, ellipse, 
node distance=2.5cm,
    minimum height=2em]
\tikzstyle{null} = [draw=none,fill=none,right]

\begin{center}  
\begin{tikzpicture}[node distance = 1.5cm, auto]

    \node [startstop] (start) {Step 1, Lemma \ref{2nd3rd T}: Second order $T$-expansion};
    \node [block, below of= start, node distance=1.5cm] (op1) {Induction hypothesis: $k$-th order $T$-expansion, $2\le k\le n-1$};
    \node [block, below of=op1, node distance=2cm] (op2) {Step 2, Theorem \ref{thm ptree}: Proving local law \eqref{locallaw1} for $L\le L_n$};
    \node [block, right of=op2, node distance=5.5cm] (output) {Concluding Theorem \ref{main thm} when $n=n_{W,L}$}; 
    \node [block, below of=op2, node distance=2cm] (op3) {Step 3, Lemma \ref{incomplete Texp}: Constructing the $n$-th order $\incomp$ };
    \node [block, below of=op3, node distance=2cm] (op4) {Step 4, Lemma \ref{cancellation property}: Proving the sum zero property for $\Sele_n$};
    \node [block, below of=op4, node distance=2cm] (op5) {Step 5, Lemma \ref{lem completeTexp}: Constructing the $n$-th order $T$-expansion};
    \node [cloud, left of=op3, node distance=5cm] (update) {Induction: $n+1 \to n$};

    \path [line] (start) -- (op1);
    \path [line] (op1) -- (op2);
    \path [line] (op2) --node [midway] {} (op3);
    \path [line] (op2) -- (output);
    \path [line] (op3) --node [midway] {} (op4);
    \path [line] (op4) --node [midway] {} (op5);
    \node [null, left of=op2, node distance=4.12cm] (null2) {};
    \node [null, right of=op2, node distance=4.12cm] (null22) {};
    \node [null, right of=op3, node distance=4.12cm] (null33) {};
    \node [null, left of=op4, node distance=4.12cm] (null4) {};
    \node [null, right of=op4, node distance=4.12cm] (null44) {};
     
    \path [line] (update) |- (op1);
    \path [line] (op5) -| node [near start] {} (update);

\end{tikzpicture}
\end{center}
\caption{The main structure of the proof of Theorem \ref{main thm}. Corresponding to \eqref{Lcondition0}, we denote $L_n:=W^{1+(n-2)d/{4} - {c_0}/{2}}$ for a  constant $c_0>0$. Moreover, given the $L$ in Theorem \ref{main thm}, it is enough to perform the induction up to $n_{W,L}:= \left\lceil\frac{4}{d}\left(\log_W L  - 1 + \frac{c_0}{2}\right)\right\rceil+2$.}\label{Fig pfchart1}
\end{figure}

\subsection{Step 2: Proof of Theorem \ref{thm ptree}}\label{sec_strat_step1}
%

In this subsection, we prove Theorem \ref{thm ptree}, which is based on three main ingredients, Lemmas \ref{eta1case0}, \ref{lem: ini bound} and \ref{lemma ptree}. The first step is an initial estimate when $\eta=1$. The following lemma is a folklore result, and has been proved in e.g. \cite{delocal} in a different setting. We will give a formal proof in \cite{PartII_high}. 


\begin{lemma}[Initial estimate, Lemma 7.2 of \cite{PartII_high}] \label{eta1case0} 
Under the assumptions of Theorem \ref{thm ptree}, for any $z=E+\ii\eta$ with $E\in (-2+\kappa,2-\kappa)$ and $\eta=1$, we have that 
\be\label{locallaw eta1}
|G_{xy} (z) -m(z)\delta_{xy}|^2\prec  B_{xy}  ,\quad \forall \ x,y \in \Z_L^d.  
\ee
\end{lemma}

The second step is the following continuity estimate, Lemma \ref{lem: ini bound}, whose proof will be given at the end of this subsection. It allows us to get some a priori estimates on $G(z)$ from the local law \eqref{locallaw1} on $G(\wt z)$ for $\wt z$ with a larger imaginary part $\im \wt z = W^{\e_0} \im z$ for a small constant $\e_0>0$. 

 \begin{lemma}[Continuity estimate]\label{lem: ini bound}
 Under the assumptions of Theorem \ref{thm ptree}, suppose that 
\be\label{ulevel}
 |G_{xy} (\wt z) -m(\wt z)\delta_{xy}|^2\prec B_{xy}(\wt z),\quad \forall\ x,y \in \Z_L^d,
\ee
with $\wt z =E+ \ii \wt\eta$ for some $E\in (-2+\kappa,2-\kappa)$ and $\wt\eta\in [W^{2}/L^{2-\e},1]$. Then we have that
\be \label{initial Txy2}
\max_{x,x_0}\frac{1}{K^{d}}\sum_{y:|y-x_0|\le K} \left(|G_{xy}( z)|^2 + |G_{yx}( z)|^2\right) \prec \left(\frac{\wt \eta} {\eta}\right)^2 \frac{1}{W^4 K^{d-4}},
\ee
uniformly in $K \in[W,L/2]$ and $z=E+ \ii\eta$ with $W^{2}/L^{2-\e}\le \eta\le \wt \eta$. Moreover, for any constant $\e_0\in (0,d/20)$, we have that
\be\label{Gmax}
\|G(z)-m(z)\|_{\max}\prec W^{-d/2+\e_0},
\ee
uniformly in $z=E+\ii \eta$ with $\max\{W^{-\e_0}\wt\eta, W^{2}/L^{2-\e}\} \le \eta \le \wt\eta$.
 \end{lemma}
Compared with \eqref{locallaw1}, the $\ell^\infty$ bound \eqref{Gmax} is sharp up to a factor $W^{\e_0}$. The estimate \eqref{initial Txy2} is an averaged bound instead of an entrywise bound and the right-hand side of \eqref{initial Txy2} loses an $W^2/K^2$ factor when compared with the sharp averaged bound $W^{-2}K^{-(d-2)}$. In our proof, we will need to bound terms of the form $\sum_x \Theta_{xy_1}|G_{xy_2}|$. Using \eqref{thetaxy} and \eqref{initial Txy2}, it is not hard to get the bound $\sum_x \Theta_{xy_1}|G_{xy_2}| \prec W^{-d/2}{\wt \eta} /{\eta}$ when $d\ge 8$ (cf. Claim \ref{cor: ini bound}). This is one key reason why we require $d\ge 8$ in Theorem \ref{main thm0} and Theorem \ref{main thm}. 

In order to improve the weaker estimates \eqref{initial Txy2} and \eqref{Gmax} to the stronger local law \eqref{locallaw1}, we use the following lemma, whose proof will be given in \cite{PartII_high}. Note that \eqref{initial Txy2} verifies the assumption \eqref{initial Txy222} as long as we have $W^{-\e_0}\wt\eta \le \eta\le \wt\eta$. 
 

\begin{lemma}[Entrywise bound on $T$-variables, Lemma 7.4 of \cite{PartII_high}]\label{lemma ptree}  
Suppose the assumptions of Theorem \ref{thm ptree} hold. Fix any $z=E+ \ii\eta$ with $E\in (-2+\kappa,2-\kappa)$ and $\wt\eta\in [W^{2}/L^{2-\e},1]$. Suppose \eqref{Gmax} and the following estimate hold:
 \be \label{initial Txy222}
 \max_{x,x_0} \frac1{K^{d}}\sum_{y:|y-x_0|\le K} \left(|G_{xy}( z)|^2 + |G_{yx}( z)|^2\right) \prec  \frac{W^{2\e_0}}{W^4 K^{d-4}} ,
\ee
for all $K \in[W,L/2]$. As long as $\e_0$ is a sufficiently small constant (depending on $n$ and $c_0$ in \eqref{Lcondition0}), we have that 
\be\label{pth T}
 T_{xy}(z) \prec B_{xy} ,\quad \forall \ x,y\in \Z_L^d.
\ee
\end{lemma}

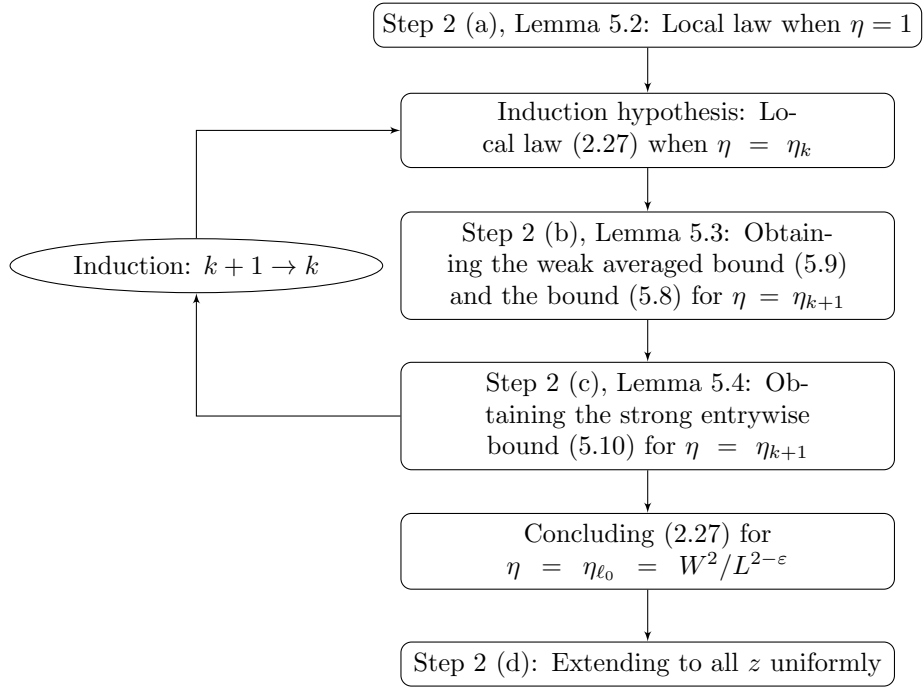
\begin{figure}[htb]
\color{black}
\tikzstyle{startstop} = [rectangle,rounded corners, minimum width=1cm,minimum height=0.5cm,text centered, draw=black]
\tikzstyle{startstop3} = [rectangle,rounded corners, minimum width=3cm, text width=6em, minimum height=0.5cm,text badly centered, draw=black]
\tikzstyle{block} = [rectangle, draw=black, 
    text width=18em, text centered, rounded corners, minimum height=1.5em]
\tikzstyle{line} = [draw=black, -latex']
\tikzstyle{line0} = [draw=black]
\tikzstyle{decision} = [diamond, draw=black, 
    text width=3em, text badly centered, node distance=2cm, 
    inner sep=0pt]

\tikzstyle{cloud} = [draw=black, ellipse, 
node distance=2.5cm,
    minimum height=2em]
\tikzstyle{null} = [draw=none,fill=none,right]

\begin{center}  
\begin{tikzpicture}[node distance = 1.5cm, auto]

    \node [startstop] (start) {Step 2 (a), Lemma \ref{eta1case0}: Local law when $\eta=1$};
    \node [block, below of= start, node distance=1.4cm] (op1) {Induction hypothesis: Local law \eqref{locallaw1} when $\eta=\eta_{k}$};
    \node [block, below of=op1, node distance=1.8cm] (op2) {Step 2 (b), Lemma \ref{lem: ini bound}: Obtaining the weak averaged bound \eqref{initial Txy222} and the bound \eqref{Gmax} for $\eta=\eta_{k+1}$};
 
    \node [block, below of=op2, node distance=2.0cm] (op3) {Step 2 (c), Lemma \ref{lemma ptree}: Obtaining the strong entrywise bound \eqref{pth T} for $\eta=\eta_{k+1}$};
    \node [block, below of=op3, node distance=1.8cm] (op4) {Concluding \eqref{locallaw1} for $\eta=\eta_{\ell_0}=W^{2}/L^{2-\e}$};
     \node [block, below of=op4, node distance=1.5cm] (op5) {Step 2 (d): Extending to all $z$ uniformly};
     \node [cloud, left of=op2, node distance=6cm] (update) {Induction: $k+1 \to k$};

    \path [line] (start) -- (op1);
    \path [line] (op1) -- (op2);
    \path [line] (op2) --node [midway] {} (op3);
    \path [line] (op3) --node [midway] {} (op4);
     \path [line] (op4) --node [midway] {} (op5);
    \node [null, left of=op2, node distance=4.12cm] (null2) {};
    \node [null, right of=op2, node distance=4.12cm] (null22) {};
    \node [null, right of=op3, node distance=4.12cm] (null33) {};
    \node [null, left of=op4, node distance=4.12cm] (null4) {};
    \node [null, right of=op4, node distance=4.12cm] (null44) {};
     
    \path [line] (update) |- (op1);
    \path [line] (op3) -| node [near start] {} (update);

\end{tikzpicture}
\end{center}
\caption{The structure of the proof of Theorem \ref{thm ptree}, where $\eta_k$ is defined in \eqref{def etak}. }\label{Fig pfchart2}
\end{figure}

Combining Lemmas \ref{eta1case0}--\ref{lemma ptree}, we can complete the proof of Theorem \ref{thm ptree} using a bootstrapping argument  on a sequence of multiplicatively decreasing $\eta$ given below.  

\begin{proof}[Proof of Theorem \ref{thm ptree}]
Given a small constant $\e_0>0$ and a fixed $E\in (-2+\kappa, 2-\kappa)$, we define the following sequence of decreasing imaginary parts: 
\be\label{def etak}
 \eta_{k}:=\max\left\{W^{- k\e_0}, \; W^{2}/L^{2-\e}\right\} ,\quad 0\le k \le \ell_{0},
\ee
where $\ell_0$ is the smallest integer such that $W^{-\ell_0\e_0}\le W^{2}/L^{2-\e}$. Note that by definition $\eta_{k+1}=W^{-\e_0}\eta_k$ for $k\le \ell_0-2$ and we always have $\eta_{\ell_0}=W^{2}/L^{2-\e}\ge W^{-\e_0}\eta_{\ell_0-1}$. Then we prove Theorem \ref{thm ptree} through an induction on $k$ as illustrated in Figure \ref{Fig pfchart2}. More precisely, we have the following procedure.  

\medskip
\noindent{\bf Step 2 (a):} By Lemma \ref{eta1case0}, \eqref{locallaw1} holds for $z_0=E+\ii\eta_0$. 

\medskip
\noindent{\bf Step 2 (b):} For any $0\le k \le \ell_0-1$, suppose \eqref{locallaw1} holds for $z_{k}=E+ \ii \eta_k$. 
Then by Lemma \ref{lem: ini bound},  \eqref{Gmax} and \eqref{initial Txy222} hold for all $z=E+\ii \eta$ with $\eta_{k+1}\le \eta \le \eta_{k}$. 

\medskip
\noindent{\bf Step 2 (c):} Applying Lemma \ref{lemma ptree}, we obtain that \eqref{pth T} holds for $z=z_{k+1}$. Using Lemma \ref{lem G<T}, we conclude \eqref{locallaw1} for $z=z_{k+1}$.

\medskip

Repeating the above Steps 2 (b) and 2 (c) for $\ell_0$ steps, we obtain that 
\begin{itemize}
	\item[(i)]\eqref{locallaw1} holds for all $z_k$ with $0\le k \le \ell_0$;
	\item[(ii)]\eqref{Gmax} and \eqref{initial Txy222} hold for all $z=E+\ii \eta$ with $\eta_{\ell_0}\le \eta\le 1$.
\end{itemize}  
To conclude Theorem \ref{thm ptree}, we still need to extend \eqref{locallaw1} uniformly to all $z=E+\ii \eta$ with $E\in (-2+\kappa,2-\kappa)$  and $\eta\in [\eta_{\ell_0}, 1]$. 

\medskip

\noindent{\bf Step 2 (d):} For a fixed $E\in (-2+\kappa, 2-\kappa)$ and $0\le k \le \ell_0-1$, we consider the following interpolations between $z_{k}$ and $z_{k+1}$: 
\be\label{fine interp}
z_{k,j}=E + \ii \eta_k-  \ii (\eta_k -\eta_{k+1})\cdot jL^{-10d}, \quad   j\in \llbracket 0, L^{10d}\rrbracket. 
\ee
By the above item (ii), \eqref{Gmax} and \eqref{initial Txy222} hold for all $z=z_{k,j}$, $ j\in \llbracket 0, L^{10d}\rrbracket$. Now applying Lemma \ref{lemma ptree}, we obtain that \eqref{pth T} holds for all $z_{k,j}$: 
$$T_{xy}(z_{k,j})\prec B_{xy}(z_{k,j}),\quad \forall \ x,  y\in \Z_L^d, \ j\in \llbracket 0, L^{10d}\rrbracket.$$
Using Lemma \ref{lem G<T} and taking a union bound, we conclude that \eqref{locallaw1} holds uniformly for all $z_{k,j}$. Then using the simple resolvent identity 
\be\label{resol exp0}G_{xy}(E+\ii \eta) = G_{xy}(E+\ii \eta') + \ii (\eta-\eta') \sum_{\al}G_{x\al}(E+\ii \eta)G_{\al y}(E+\ii \eta'), \ee
and the trivial bound
$$\|G(E+\ii \eta)\|_{\max}\le  \eta^{-1} \le L^2/W^2,\quad \forall \ \eta\ge W^2/L^2 ,$$
we can easily obtain the perturbation estimate
\be\label{perturb G}\|G(z)-G(z_{k,j})\|\le L^{-d},\quad \forall\  z=E+\ii \eta, \ \eta \in [\im z_{k,j},\im z_{k,j-1}].\ee
Together with the local law \eqref{locallaw1} at $z_{k,j}$, \eqref{perturb G} implies that \eqref{locallaw1} holds uniformly for all $z=E+\ii \eta$ with $\eta\in[\eta_k, \eta_{k+1}]$. This concludes \eqref{locallaw1} for a fixed $E\in (-2+\kappa, 2-\kappa)$ and uniformly for all $\eta\in [\eta_{\ell_0}, 1]$. Finally, to extend \eqref{locallaw1} uniformly to all $E$, we choose an $L^{-10d}$-net of $(-2+\kappa,2-\kappa)$ and use a similar perturbation argument as above. 
We omit the details.   
\end{proof}

Now we give the proof of Lemma \ref{lem: ini bound}.  We first recall the following classical Ward's identity. Its proof is a simple application of the spectral decomposition of $G(z)$.
\begin{lemma}[Ward's identity]\label{lem ward}
	For any $y,y'\in \Z_L^d$ and $z=E+\ii \eta$, we have 
	\be\label{eq_Ward0}
	\sum_x \overline G_{xy'}( z) G_{xy}( z) = \frac{G_{y'y}(z)-\overline{G_{yy'}(z)}}{2\ii \eta},\quad \sum_x \overline G_{y' x}( z) G_{yx}( z) = \frac{G_{yy'}(z)-\overline{G_{y'y}(z)}}{2\ii \eta}.
	\ee
	As a special case, if $y=y'$, we have
	\be\label{eq_Ward}
	\sum_x |G_{xy}( z)|^2 =\sum_x |G_{yx}( z)|^2 = \frac{\im G_{yy}(z) }{ \eta}.
	\ee
\end{lemma}

 \begin{proof}[Proof of Lemma \ref{lem: ini bound}]
 We first prove \eqref{initial Txy2}. The proof of \eqref{Gmax} will be based on \eqref{initial Txy2}.

\medskip

\noindent{\bf Proof of \eqref{initial Txy2}:} By \eqref{ulevel}, the following event is a high probability event: 
$$\Xi:=\big\{ \max_{x}|G_{xx}(\wt z) -m(\wt z)| \le W^{-d/4}\big\}.$$
With Ward's identity \eqref{eq_Ward}, we obtain that on $\Xi$,
\be\label{l2vw0}
\sum_{x}|G_{xy}(\wt z)|^2 =\sum_{x}|G_{yx}(\wt z)|^2 = \frac{\im G_{yy}(\wt z)}{\wt\eta} \sim \wt \eta^{-1}  .
\ee
Moreover, using the inequality $\eta \im G_{yy}(z) \le\wt\eta \im G_{yy}(\wt z)$, we obtain that on $\Xi$,
\be\label{l2vw}
\sum_{x}|G_{xy}(z)|^2=\sum_{x}|G_{yx}(z)|^2 = \frac{\im G_{yy}(z)}{\eta} \le \frac{\wt\eta \im G_{yy}(\wt z)}{\eta^2}\lesssim \frac{ {\wt \eta} }{\eta^2}.
\ee
Now we define the family of vectors $\bv_x$, $\bw_x$, $\wh  \bv_x$ and $\wh  \bw_x$ as
$$ v_x(y):=G_{xy}(z), \quad \wh  \bv_x:= \frac{\bv_x}{\|\bv_x\|_2}, \quad  w_x(y):=\overline G_{yx}(\wt z), \quad \wh  \bw_x:= \frac{\bw_x}{\|\bw_x\|_2}.$$
By \eqref{l2vw0} and \eqref{l2vw}, we have that on $\Xi$,
\be\label{l2vw2} \|\bv_x\|_2^2 \lesssim { {\wt \eta} }/{\eta^2},\quad \|\bw_x\|_2^2 \sim \wt\eta^{-1},\quad \forall \ x\in \Z_L^d.\ee


Now let $\mathcal I$ be any subset of indices. Suppose the orthogonal projection of $\wh  \bv_x$ onto the subspace spanned by $\{\wh  \bw_y: y\in \mathcal I\}$ can be written as $\bu_x = \sum_{y\in\mathcal  I} a_x(y) \wh  \bw_y$. Then we have 
\be\label{imply ba}
b_x(y) : = (\wh  \bv_x , \wh  \bw_y)= (\bu_x , \wh  \bw_y)= \sum_{y'\in \mathcal I} a_x({y'})A_{y'y}.
\ee
Here the inner product is defined as $(\bv,\bw):=\sum_{x}\bv(x)\overline \bw(x)$,
and the matrix $A$ is defined by 
\begin{align} \label{Ayy}
A_{y'y}:= (\wh  \bw_{y'},\wh  \bw_y) = \frac{\sum_x \overline G_{xy'}(\wt z) G_{xy}(\wt z)}{\|\bw_y\|_2\|\bw_{y'}\|_2} = \frac{G_{y'y}(\wt z)-\overline{G_{yy'}(\wt  z)}}{2\ii \wt \eta\|\bw_y\|_2\|\bw_{y'}\|_2}, 
\end{align}
where we used \eqref{eq_Ward0} in the third step. Notice that by definition, $A \equiv A(\cal I)$ is a positive definite Hermitian matrix with indices in $\cal I$. 
We define two \emph{row vectors} $\mathbf a_x:=(a_x(y))_{y\in \mathcal I}$ and $\mathbf b_x:=(b_x(y))_{y\in \mathcal I}$. Then \eqref{imply ba} gives that $\mathbf a_x = \mathbf b_x A^{-1}$, with which we can get that 
$$1=\|\wh  \bv_x\|_2^2 \ge \|\bu_x \|_2^2 = \mathbf a_x A \mathbf a_x^* =\mathbf b_x A^{-1}\mathbf b_x^* \ge \|\mathbf b_x\|^2 \|A\|_{\ell^2 \to \ell^2}^{-1} .$$
This inequality implies that on $\Xi$,
\be\label{b<A} \|\mathbf b_x\|^2 \le \|A\|_{\ell^2 \to \ell^2}\lesssim \|\cal A\|_{\ell^2 \to \ell^2},\ee
where the matrix $\cal A$ is defined by 
\begin{align*}
	\cal A_{y'y}= \frac1{2\ii }\left[G_{y'y}(\wt z)-\overline{G_{yy'}(\wt  z)}\right].
\end{align*}
On the other hand, we have the resolvent identity 
\be \label{vwrel}
\begin{split}
G_{xy}(z) = G_{xy}(\wt z) - \ii (\wt\eta-\eta) \sum_{\al}G_{x\al}(z)G_{\al y}(\wt z) &= \overline w_{y}(x) - \ii (\wt\eta-\eta)(\wh  \bv_x, \wh  \bw_y) \|\bv_x\|_2 \|\bw_y\|_2 \\
 &= \overline w_{y}(x) - \ii (\wt\eta-\eta) b_x(y) \|\bv_x\|_2 \|\bw_y\|_2 .
\end{split}
\ee
With this identity, we obtain that on $\Xi$, 
\be\label{vwrel3}
\sum_{y\in \mathcal I} |G_{xy}(z)|^2 \lesssim \sum_{y\in \mathcal I} |G_{xy}(\wt z)|^2 + \left( \frac{\wt\eta}{\eta}\right)^2 \sum_{y\in \mathcal I} |b_{x}(y)|^2 \lesssim \sum_{y\in \mathcal I} |G_{xy}(\wt z)|^2 +  \left( \frac{\wt\eta}{\eta}\right)^2 \|\cal A\|_{\ell^2\to \ell^2},
\ee
where we used \eqref{l2vw2} in the first step and \eqref{b<A} in the second step, 
Similarly, we can get that on $\Xi$, 
\be\label{vwrel4}
\sum_{y\in \mathcal I} |G_{yx}(z)|^2 \lesssim \sum_{y\in \mathcal I} |G_{yx}(\wt z)|^2 + \left( \frac{\wt\eta}{\eta}\right)^2  \|\cal A\|_{\ell^2\to \ell^2} .
\ee

For the specific index set $\cal I=\{y:|y-x_0|\le K\}$, using \eqref{ulevel} we can bound that
\be\label{vwrel5}\sum_{y\in \mathcal I}\left( |G_{xy}(\wt z)|^2+ |G_{yx}(\wt z)|^2\right) \prec 1+ \sum_{y\in \mathcal I}B_{xy}\lesssim \frac{K^2}{W^2}.\ee
It remains to bound $ \|\cal A\|_{\ell^2\to \ell^2}$ in \eqref{vwrel3} and \eqref{vwrel4}. A simple bound can be obtained by using the Hilbert-Schmidt norm:
$$ \|\cal A\|_{\ell^2\to \ell^2}^2 \le \|\cal A\|_{HS}^2\lesssim \sum_{y,y'\in \cal I}\left(|G_{y'y}(\wt z)|^2+|G_{yy'}(\wt z)|^2\right) \prec |\cal I|+ \sum_{y,y'\in \mathcal I}B_{yy'} \lesssim \frac{K^{d+2}}{W^2}. $$
This estimate is not strong enough to give the bound \eqref{initial Txy2}. To obtain a better bound on $ \|\cal A\|_{\ell^2\to \ell^2}$, we  use the following lemma, which is based on the classical method of moments. The proof of this lemma will be given in \cite{PartII_high}.

\begin{lemma}[Lemma 9.1 of \cite{PartII_high}]\label{lem normA} 
Suppose the assumptions of Lemma \ref{lem: ini bound} hold. We choose the index set $\cal I=\{y:|y-x_0|\le K\}$ for  $ K\in [W, L/2]$. Then for any fixed $p\in \N$ and small constant $\e>0$, 
we have the estimate  
\be\label{Moments method}
\E \tr\left(\cal A^{2p}\right)  \le K^d\left( W^\e \frac{K^4}{W^4}\right)^{2p-1}.
\ee
\end{lemma}
With Lemma \ref{lem normA}, we obtain that 
$$\E\|\cal A\|_{\ell^2\to \ell^2}^{2p} \le  \E \tr\left(\cal A^{2p}\right)\le K^d\left( W^\e\frac{K^4}{W^4}\right)^{2p-1}.$$
Since $p$ can be arbitrarily large, using Markov's inequality we get that 
\be\label{strong normA}
\|\cal A\|_{\ell^2\to \ell^2} \prec \frac{K^4}{W^4}.
\ee
Inserting \eqref{vwrel5} and \eqref{strong normA} into \eqref{vwrel3} and \eqref{vwrel4}, we obtain \eqref{initial Txy2}.

\vspace{10pt}

\noindent{\bf Proof of \eqref{Gmax}:} To prove \eqref{Gmax}, as in \eqref{fine interp}, we define the following interpolations between $z$ and $\wt z$:  
\be\label{fine interp2}
z_{j}=E + \ii \wt \eta-  \ii (\wt\eta -\eta)\cdot jL^{-10d}, \quad   j\in \llbracket 0, L^{10d}\rrbracket. 
\ee
By \eqref{perturb G}, we have the perturbation estimate 
\be\label{perturb G2}\|G(z_j)-G(z_{j-1})\|\le L^{-d} . 
\ee
Moreover, by \eqref{initial Txy2} and the fact $\wt\eta/\eta\le W^{\e_0}$, we know that \eqref{initial Txy222} holds for $z=z_{j}$ for all $j\in \llbracket 0, L^{10d}\rrbracket$. Taking a union bound, we get that for any small constant $\tau>0$ and large constant $D>0$, the event 
\begin{align}
\Xi_0 :=&\left\{|G_{xy}(\wt z)-m(\wt z)\delta_{xy}|\le W^{\tau}B_{xy}, \ \forall \ x,y\in \Z_L^d\right\}\nonumber \\
\cap &\bigg\{ \max_{0\le j \le L^{10d}}  \sum_{K \in[W,L/2]} \max_{x,x_0\in \Z_L^d}\frac{W^4}{K^4}\sum_{|y-x_0|\le K} \left(|G_{xy}( z_j)|^2 + |G_{yx}( z_j)|^2\right) \le  W^{2\e_0+\tau} \bigg\}\label{defn Xi0}
\end{align}
holds with probability $\PP(\Xi_0)\ge 1-L^{-D}$. Now fix a small constant $\delta_0\in(0,d/20)$, we define the events
$$A_j:=\{\|G(z_{j})-m(z_{j})\|_{\max}\le W^{-\delta_0}\},\quad B_j :=\{ \|T(z_{j})\|_{\max}\le W^{-d+2\e_0+3\tau} \},\quad   j\in \llbracket 0, L^{10d}\rrbracket.  $$
By Lemma \ref{lem G<T}, we have that 
\begin{align}
&\PP\left( \|G(z_{j})-m(z_{j})\|_{\max}\ge W^{-d/2+\e_0+2\tau};A_j\cap B_j\right) \le L^{-D} .\label{largedev G2}
\end{align}
Now we use the above facts to prove that 
\be\label{unionG}
\PP\Big( \max_{0\le j \le L^{10d}} \|G(z_{j})-m(z_{j})\|\le W^{-d/2+ \e_0 + 2\tau}\Big)\ge 1 - 2L^{-D+10d},
\ee
which concludes \eqref{Gmax} since $\tau$ and $D$ are arbitrary. 

 Using \eqref{perturb G2}, we can obtain that $\|G (z_{1})-m(z_{1})\|_{\max}\le W^{-d/4}$ on $\Xi_0$, which gives $\Xi_0\subset A_1$. 
Moreover, on $\Xi_0$, we have that for any $x,y\in \Z_L^d$,
\begin{align*}
T_{xy}(z_{1})&=|m(z_{1})|^2\sum_\al s_{x\al}|G_{\al y}(z_{1})|^2   
 \lesssim  s_{xy} +  W^{-d} \sum_{|\al-y|\le W^{1+\frac{\tau}{4}}}|G_{\al y}(z_{1})|^2 + \OO(W^{-100d}) \lesssim   \frac{W^{2\e_0+2\tau }}{W^d} ,
\end{align*}
where in the second step we used \eqref{subpoly} and in the third step we used the averaged bound in the definition of $\Xi_0$. This estimate gives that $\Xi_0\subset A_1\cap B_1$. Then \eqref{largedev G2} implies that $\PP(\Xi_1)\ge \PP(\Xi_0) - L^{-D}  \ge 1-2L^{-D}$, where the event $\Xi_1$ is defined by
$$\Xi_1:=\Xi_0\cap \left\{\|G_{xy}(z_{1})-m(z_{1})\|_{\max}\le W^{-d/2+\e_0+2\tau}\right\}.$$
Repeating the above argument, for any $j\in \llbracket 0, L^{10d}\rrbracket$, we can obtain that $\PP(\Xi_{j})\ge 1-(j+1)\cdot L^{-D}$ for
$$\Xi_{j}:=\Xi_0 \cap \left\{\max_{0\le k \le j}\|G_{xy}(z_{k})-m(z_{k})\|_{\max}\le W^{-d/2+\e_0+2\tau}\right\} .  $$
Taking $j=L^{10d}$, we conclude \eqref{unionG}. 
\end{proof}

\subsection{Step 3: $n$-th order $\incomp$}
In this step, we construct the $n$-th order {$\incomp$} in Lemma \ref{incomplete Texp}, whose proof will be postponed to \cite{PartII_high}. In general, it is difficult to define the $\incomp$ explicitly (there are already hundreds of terms when $n=6$). Instead, we will give a prescription to generate the $\incomp$ by applying local and global expansions.  
Section \ref{sec strategy} contains some more explanations.  

\begin{lemma} [$n$-th order $\incomp$, Theorem 3.7 of \cite{PartII_high}]   \label{incomplete Texp} 
Fix any $n\in \N$. Suppose we have defined the $(n-1)$-th order $T$-expansion. Then we can construct an $n$-th order $\incomp$ satisfying Definition \ref{defn incompgenuni}.  
 \end{lemma}



In Step 5, we will solve the $n$-th order $\incomp$ \eqref{mlevelT incomplete} to get the $n$-th order $T$-expansion \eqref{mlevelTgdef}. Before doing that, we need to show that $\Sele_n$ satisfies the properties \eqref{two_properties0}--\eqref{3rd_property0} and \eqref{two_properties0V}--\eqref{weaz},  i.e. $\Sele_n$ is indeed an $n$-th order $\self$. This is the purpose of Step 4, where the proof of the sum zero properties \eqref{3rd_property0} and \eqref{weaz} will be the core argument. 

\subsection{Step 4: Proving the sum zero properties}\label{sec sumzero1} 

In this subsection, we prove that $\Sele_n$ constructed in Lemma \ref{incomplete Texp} is indeed a $\self$. 

\begin{lemma}[Properties of $\Sele_n$]\label{cancellation property}
 Fix any $n\in \N$. Suppose we have defined the $(n-1)$-th order $T$-expansion.  
 The deterministic matrix $\Sele_n$ constructed in the $n$-th order $\incomp$ in Lemma \ref{incomplete Texp} satisfies the properties \eqref{two_properties0}--\eqref{3rd_property0} and \eqref{two_properties0V}--\eqref{weaz} with $l=n$.


 \end{lemma}
The proof of Lemma \ref{cancellation property} is based on three main ingredients, Lemmas \ref{lem V-R wt}, \ref{lem FT0} and \ref{lem maincancel}. In Lemma \ref{lem V-R wt}, we show that the estimates \eqref{4th_property0}, \eqref{4th_property0V} and \eqref{3rd_property0 diff} hold. 
Its proof depends on the doubly connected property of $\Sele_{n}$ (cf. Definition \ref{def 2net}) and we postpone it to Section \ref{sec infspace}. 

\begin{lemma} \label{lem V-R wt}
Under the assumptions of Theorem \ref{main thm} and Lemma \ref{cancellation property}, $\Sele_n^{\infty}$ exists. Moreover, $\Sele_n$ and $\Sele_n^{\infty}$ satisfy \eqref{4th_property0}, \eqref{4th_property0V}, \eqref{3rd_property0 diff} and
 \be\label{property V-R wt}
 \bigg|\sum_{\fa\in \Z_L^d} ( \Sele_{n})_{0\fa} (m(z),\psi, W, L) - \sum_{\fa\in \Z^d} ( \Sele^{\infty}_{n})_{0\fa}(m(E), \psi, W)  \bigg|   \le \eta W^{-(n-2)d/2+\e},\quad \forall\ \eta\in[  W^{2}/L^{2-\e} ,L^{-\e}],
\ee
for any small constant $\e>0$. Here we have abbreviated $m(E)\equiv m(E+\ii 0_+)$. 
\end{lemma}

By taking $L\to \infty$, the infinite space limit $\sum_\fa  (\Sele^{\infty}_{n})_{0\fa}$ depends only on $n$, $m(E)$, the function $\psi$ in Assumption \ref{var profile}, and the band width $W$. Now using a standard calculation with Fourier transforms, we show that the $W$ dependence can be pulled out as a scaling factor if $\psi$ is compactly supported.

\begin{lemma}\label{lem FT0}
Fix any $n\in \N$. Under the assumptions of Theorem \ref{main thm} and Lemma \ref{cancellation property},
we have that 
\be\label{eq FT0} \sum_{\fa\in \Z^d}( \Sele^{\infty}_{n})_{0\fa }(m(E), \psi, W) =W^{-(n-2)d/2} {\Sint}_{n}(m(E), \psi)  ,\ee
where ${\Sint}_{n}$ is a constant independent of   $W$. 
\end{lemma}
The full proof of Lemma \ref{lem FT0} will be given in Section \ref{sec pf FT0}. Here we give a sketch of the proof. 

\begin{proof}[Sketch of the proof of Lemma \ref{lem FT0}]
The proof is straightforward if we replace the variance profile $f_{W,L}$ in \eqref{choicef} with an exact Fourier series on $\Z_L^d$:
\be\label{choicef0}
\wt f_{W,L}(x):= \frac1{L^d}\sum_{p\in \mathbb T_L^d} \psi(Wp)e^{\ii p\cdot x}, \quad \text{with}\quad \mathbb T_L^d =\left(\frac{2\pi }{L}\Z_L\right)^d.  \ee
We define the matrix $\wt S=(\wt s_{xy})$ with entries $\wt s_{xy}= \wt f_{W,L}([x-y]_L)$ and 
\be\label{choicefS}\wt S^+(z):=\frac{m^2(z)\wt S}{1-m^2(z)\wt S}, \quad \wt S^-(z):=\overline{\wt S^+(z)},\quad \wt \Theta(z):=\frac{|m(z)|^2 \wt S}{1-|m(z)|^2\wt S}.\ee
Then their entries can be expressed as 
\be\label{FTtheta}
\wt S^{+}_{xy}(z) = \frac1{L^d}\sum_{p\in \mathbb T_L^d} \frac{ m^2(z)\psi(Wp)}{1-m^2(z)\psi(Wp)}e^{\ii p\cdot (x-y)} ,\quad \wt \Theta_{xy}  =  \frac1{L^d}\sum_{p\in \mathbb T_L^d}\frac{|m(z)|^2  \psi(Wp)}{1- |m(z)|^2 \psi(Wp)}e^{\ii p\cdot (x-y)} .
\ee
By replacing $ S$, $ S^\pm$ and $ \Theta$ with $\wt S$, $\wt S^\pm$ and $\wt\Theta$ in $\Sele_n$, we get a new matrix $\wt\Sele_n$.  
In Section \ref{sec pf FT0}, we will show that $\sum_{\fa}(\Sele_n)_{ 0\fa}$ has the same infinite space limit as $\sum_{\fa}(\wt \Sele_n)_{ 0\fa}$. Let $\cal G$ denote the graphs in $\wt\Sele_{n}$, $p_e$ denote the momentum associated with each edge $e$ in $\cal G$, $\Xi_L$ be a subset of $(\mathbb T_L^{d})^{n_e}$ given by the constraint that the total momentum at each vertex is equal to 0, where $n_e\equiv n_e(\cal G)$ is the total number of edges in $\cal G$. Then using the Fourier series \eqref{choicef0} and \eqref{FTtheta}, we can write that
$$\sum_{\fa} (\wt\Sele_{n})_{0\fa}(m(z),\psi,W,L)=\frac{1}{L^{(n-2) d/2 }}\sum_{\cal G}\sum_{ \{p_e\} \in \Xi_{L}}\cal F_{\cal G}\left( \{Wp_e\},z\right), $$
where $\cal F_{\cal G}$ is a function expressed in terms of $\psi(Wp_e)$ (cf. Section \ref{sec pf FT0} for more details). 
Taking $L\to \infty$ and $\eta \to 0$, we get that
\begin{align*} 
\sum_{\fa} (\Seleinf_{n})_{0\fa } (m(E), \psi, W) &=\frac{1}{(2\pi)^{(n-2) d/2}}\sum_{\cal G}\int_{ \{p_e\} \in \Xi} \cal F_{\cal G}\left( \{Wp_e\},E\right) \prod_e \dd p_e ,
\end{align*}
where $\Xi$ is a union of hyperplanes in the torus $(-\pi,\pi]^{d n_e}$ with the  constraint that the total momentum at each vertex is equal to 0. Then applying a change of variables $q_e=Wp_e$ and using that $\psi$ is compactly supported, we obtain that 
$$\sum_{\fa} (\Seleinf_{n})_{ 0\fa }(m(E), \psi, W) =\frac{1}{(2\pi)^{(n-2) d/2}W^{(n-2) d/2}}\sum_{\cal G}\int_{\wt\Xi} \cal F_{\cal G}\left( \{q_e\},E\right) \prod_e \dd q_e,$$
where $\wt\Xi$ is a union of hyperplanes in $(\R^d)^{n_e}$ given by the constraints of $\Xi$. Renaming the right-hand side, we conclude \eqref{eq FT0}. 
\end{proof}

In Lemma \ref{lem maincancel}, we show that the row sum $\sum_{\fa} ( \Sele_n)_{0\fa}$ is much smaller than $W^{-(n-2)d/2}$ under some particular choices of $L$ and $z$. 


\begin{lemma}\label{lem maincancel}
Fix any $n\in \N$. Under the assumptions of Theorem \ref{main thm} and Lemma \ref{cancellation property}, suppose $L\equiv L_n$ satisfies that
\be\label{Lcondition} W^{(n-3)d/2+c_0} \le {L_n^2}/{W^2}  \le W^{(n-2)d/2-c_0} \ee
for a constant $c_0>0$, and $\eta\equiv \eta_n=W^{2+\e}/L_n^2$ for a small enough constant $\e>0$. Then  for $z_n= E+\ii \eta_n$,  
\be\label{small SnE0}
\begin{split}
\Big| \sum_{\fa} (\Sele_n)_{0\fa } \left(m(z_n),\psi, W, L_n\right) \Big|\le W^{-(n -2)d/2}\cdot W^{-c} 
\end{split}
\ee
for a constant $c>0$ depending only on $c_0$ and $d$.
\end{lemma}

Comparing \eqref{small SnE0} with \eqref{eq FT0}, we see that $\sum_{\fa} (\Sele_n)_{0\fa}$ is much smaller than its scaling size $W^{-(n-2)d/2}$ if $\Sint_n\ne 0$. We will use this contradiction to show that $\Sint_n=0$, and hence conclude the sum zero property \eqref{weaz} for $\Seleinf_n$. To prove Lemma \ref{lem maincancel}, we need to use the following lemma, whose proof will be postponed to Section \ref{secpfinformal}. The proof is based on some additional properties of \smash{$\PITn$, $\AITn$ and $\Err_{n,D}'$} that will be introduced in Definition \ref{def incompgenuni} below (more precisely, their doubly connected properties that will be defined in Definition \ref{def 2net}).


\begin{lemma} \label{lem informal}
Fix any $n\in \N$. Under the assumptions of Lemma \ref{lem maincancel}, 
 we have the following estimates:  
\be\label{intro_recol}
\sum_{\fa,\fb}\left| (\PITk)_{\fa,\fb\fb}\left(z_n,\psi, W, L_n\right)\right| \prec L_n^d \cdot \eta_n^{-1} W^{-(k-2)d/2} ,\quad 3\le k \le n;
\ee
\be\label{intro_error}
\sum_{\fa,\fb}\left| (\AITn)_{\fa,\fb\fb}\left(z_n,\psi, W, L_n\right)\right| \prec L_n^d \cdot \eta_n^{-2} W^{-(n -1)d/2} ;
\ee
\be\label{intro_higherror}
\sum_{\fa,\fb}\left|(\Err_{n,D}')_{\fa,\fb\fb}\left(z_n,\psi, W, L_n\right)\right| \prec L_n^d \cdot \eta_n^{-2} W^{-(D-1)d/2}.
\ee
\end{lemma}
In the proof of Lemma \ref{lem maincancel}, we will use these estimates to control
$$ 
\Big| \sum_{\fa,\fb}\sum_x \left(\Theta \Sele_n\right)_{\fa x} T_{x,\fb\fb}\Big| \sim \frac{L_n^{d}}{\eta_n^2} \Big| \sum_{\fa} (\Sele_n)_{0\fa } \Big| .$$
Compared to the scaling size $L_n^d \eta_n^{-2} W^{-(n -2)d/2}$ of the right-hand side  given by \eqref{eq FT0}, \eqref{intro_recol} gains a factor $W^{(n-3)d/2}\eta_n$ when $k=3$, \eqref{intro_error} gains a factor $W^{-d/2}$, and \eqref{intro_higherror} is negligible because $D$ is arbitrarily large.

%
%


\begin{proof}[Proof of Lemma \ref{lem maincancel}]
In the setting of Lemma \ref{lem maincancel}, we have  the $(n-1)$-th order $T$-expansion. Hence by Theorem \ref{thm ptree}, the local law \eqref{locallaw1} holds for $G(z_n)\equiv G(z_n, \psi, W, L_n)$ if $L_n$ satisfies ${L_n^2}/{W^2}  \le W^{(n-2)d/2-c_0} $. This explains the upper bound in \eqref{Lcondition}. 

Now given an $n$-th order $\incomp$ \eqref{mlevelT incomplete} with $\fb_1=\fb_2=\fb$, taking the expectation of both sides and summing over $\fa,\fb \in \Z_L^d$, we obtain that for large enough $D>0$, 
\begin{align}
  \sum_{\fa,\fb}\E T_{\fa\fb}(z_n,  L_n) & =  m(z_n)  \sum_{\fa,\fb}\E\overline G_{\fb\fb}(z_n,  L_n)\Theta_{\fa\fb}(z_n,  L_n) + \Big[\sum_{\fa,x}  (\Theta \wtSdelta^{(n)} )_{\fa x} \sum_{\fb}\E T_{x \fb}\Big](z_n, L_n)\nonumber \\
&   + \sum_{\fa,\fb}\E(\PITn)_{\fa,\fb\fb}(z_n , L_n)   +  \sum_{\fa,\fb}\E(\AITn)_{\fa,\fb\fb}(z_n,  L_n) + \OO (W^{-D}) ,\label{mlevelT sum}
\end{align}
where we used $\E\QITn  =0$ and \eqref{intro_higherror}. For simplicity, we have omitted the arguments $\psi$ and $W$ from the above equation. To further simplify the notation,  we will also omit the arguments $z_n$ and $ L_n$ in the following proof.

For the left-hand side of \eqref{mlevelT sum}, using Ward's identity \eqref{eq_Ward} we get that 
\begin{align}\label{Rcancel10}
&\sum_{\fa,\fb} \mathbb E T_{\fa \fb} = |m|^2 \sum_{\fa, x}s_{\fa x}\cdot \mathbb E \sum_\fb |G_{x\fb}|^2 = |m|^2 \frac{ \sum_{x}\im  \left(\mathbb E G_{xx}\right)}{\eta_n} =L_n^d |m|^2 \frac{\im  \left(\mathbb E G_{00}\right)}{\eta_n},   
\end{align}
where in the last step we used $\mathbb E G_{xx}= \mathbb E G_{00}$ for all $x\in \Z_{L_n}^d$ by translational invariance.
For the first term on the right-hand side of \eqref{mlevelT sum}, using the identity in \eqref{sumTheta} we obtain that 
\begin{align}\label{Rcancel20}
 m  \sum_{\fa,\fb}\E\overline G_{\fb\fb} \Theta_{\fa\fb} = m \mathbb E \overline G_{00} \cdot \sum_{\fa,\fb}\Theta_{\fa\fb}  = L_n^d\frac{|m|^2 m \cdot \mathbb E \overline G_{00}}{1-|m|^2} .
\end{align}
 For the third term on the right-hand side of \eqref{mlevelT sum}, using \eqref{intro_recol} we obtain that 
\begin{align}\label{Rcancel30} 
\sum_{\fa,\fb}\big|\E(\PITn)_{\fa,\fb\fb}\big|\le \sum_{k=3}^n \sum_{\fa,\fb}\left|\E(\PITk)_{\fa,\fb\fb}\right| \le L_n^d \frac{W^{-d/2+\e}}{\eta_n}.
\end{align}
For the fourth term on the right-hand side of \eqref{mlevelT sum}, using \eqref{intro_error} we obtain that 
\begin{align}\label{Rcancel40} 
\sum_{\fa,\fb}\big|  \E(\AITn)_{\fa,\fb\fb}\big| \le L^d_n\frac{W^{-(n-1)d/2+\e}}{\eta_n^2}.
\end{align}
Finally, for the second term on the right-hand side of \eqref{mlevelT sum}, we decompose $\wtSdeltan$ as \eqref{decompose Sdelta}. For $4\le l \le n$, we can calculate that  
\begin{align}\label{Rcancel21}
\sum_{\fa,x} \left(\Theta \Sele_{l} \right)_{\fa x} \sum_{\fb}\E T_{x \fb} &=\frac{|m|^2}{1-|m|^2}   \sum_{\al,x} \left( \Sele_{l}\right)_{\al x}  \sum_{\fb}\E T_{x \fb}   = L^d_n \frac{|m|^4\cdot \im\left(\mathbb E G_{00}\right)}{(1-|m|^2)\eta_n}\sum_{\al} \left(\Sele_{l}\right)_{ 0\al} ,
\end{align}
where in the first step we used \eqref{sumTheta} and in the second step we used the translational invariance of $\Sele_{\al x}$ and \eqref{Rcancel10}. Applying \eqref{3rd_property0} to $\Sele_l$,  $4\le l\le n-1$, we get that
$$\Big|\sum_{\al}  (\Sele_{l})_{0\al }\Big| \le \eta_n W^{-(l-2)d/2+\e},\quad 4\le l \le n-1.$$
Inserting it into \eqref{Rcancel21} and using $1-|m|^2\sim \eta_n$, we obtain that 
\begin{align}\label{Rcancel22}
\Big|\sum_{\fa,x} \left(\Theta \Sele_{l} \right)_{\fa x} \sum_{\fb}\E T_{x \fb}\Big| \lesssim L^d_n \frac{W^{-(l-2)d/2+\e} \im \left(\mathbb EG_{00}\right)}{\eta_n} ,\quad 4\le l \le n-1.
\end{align}

Now plugging \eqref{Rcancel10}--\eqref{Rcancel22} into \eqref{mlevelT sum} and cancelling the $L^d_n$ factor on both sides, we obtain that
\begin{align}
  |m|^2 \frac{ \im \left(\mathbb EG_{00}\right)}{\eta_n} &=  \frac{|m|^2 m \cdot \mathbb E \overline G_{00}}{1-|m|^2}  +  \frac{|m|^4\cdot\im\left( \mathbb E G_{00}\right)}{(1-|m|^2)\eta_n}\sum_{\al} \left(\Sele_{n} \right)_{0\al } \label{Rcancelsum0}\\
&+   \OO \left(  \frac{W^{-d/2+\e}}{\eta_n}+\frac{W^{-(n -1 )d/2+\e}}{\eta_n^2} +\sum_{ l=4}^{n-1} \frac{W^{-(l -2)d/2+\e}\im \left( \mathbb E G_{00}\right)}{\eta_n}\right) .\nonumber
\end{align}
Since \eqref{locallaw1} holds for $G(z_n,  L_n)$, we have that
\be\label{insert Gxx}\E G_{00}(z_n,  L_n)=m(z_n)+\OO(W^{-d/2+\e}).\ee 
Moreover, taking the imaginary part of the equation $z_n=-m(z_n)-m^{-1}(z_n)$, we obtain that 
\be\label{insert m^2}\frac{|m(z_n)|^2}{1-|m(z_n)|^2}  = \frac{ \im m(z_n)}{\eta_n} .\ee
Inserting \eqref{insert Gxx} and \eqref{insert m^2} into \eqref{Rcancelsum0} and using $1-|m|^2\sim \eta_n$, we get that
$$\Big|\sum_{\al} (\Sele_{n} )_{0\al } (m(z_n),\psi, W, L_n) \Big|\lesssim   W^{-d/2+\e}\eta_n + W^{-(n-1)d/2+\e} = W^{-d/2+\e}\frac{W^{2+\e}}{L_n^2} + W^{-(n-1)d/2+\e}.$$
Together with the lower bound in condition \eqref{Lcondition}, we conclude \eqref{small SnE0} for $c= \min(d/2-\e, c_0-2\e)$. 
 \end{proof}
 
Finally, combining Lemmas \ref{lem V-R wt}--\ref{lem maincancel}, we can conclude Lemma \ref{cancellation property}.

\begin{proof}[Proof of Lemma \ref{cancellation property}] 
	The estimates \eqref{4th_property0}, \eqref{4th_property0V} and \eqref{3rd_property0 diff} for $\Sele_n$ follow from  Lemma \ref{lem V-R wt}, and the equations \eqref{two_properties0} and \eqref{two_properties0V} follow from Lemma \ref{lem Rsymm}. Now we pick 
	$ L_{n}= W^{1+\left(n/4-5/8\right)d}$,
	which satisfies the condition \eqref{Lcondition} with $c_0=d/4$. Applying Lemma \ref{lem V-R wt} and Lemma \ref{lem FT0} with $L=L_n$ and $z=z_n$ defined in Lemma \ref{lem maincancel}, we obtain that 
	\be\label{lower Sdelta} 
	\Big|\sum_{\fa} (\Sele_n)_{0\fa} (m(z_n),\psi, W, L_n) \Big| \ge W^{-(n-2)d/2}\left|{\Sint}_{n}(m(E), \psi)\right| - \eta_n W^{-(n-2)d/2+\e}.
	\ee
	Combining \eqref{small SnE0} and \eqref{lower Sdelta}, we obtain that 
	$$ \left|{\Sint}_{n}(m(E), \psi)\right| \le \eta_n W^{\e} + W^{-c} =\oo(1).$$
	Since ${\Sint}_{n}$ is a constant, we must have ${\Sint}_{n}(m(E), \psi)=0$, which by Lemma \ref{lem FT0} implies \eqref{weaz} for $\Seleinf_n $.  Combining \eqref{weaz} for $\Seleinf_n $ with Lemma \ref{lem V-R wt}, we obtain \eqref{3rd_property0} for $\Sele_n $.
\end{proof}

%
%
%
%

\subsection{Step 5: The $n$-th order $T$-expansion}\label{sec expT}

After showing that $ \Sele_{n} $ is a $\self$ satisfying Definition \ref{collection elements}, we can now 
solve the $n$-th order $\incomp$ to obtain the $n$-th order $T$-expansion.

 \begin{lemma}[$n$-th order $T$-expansion]   \label{lem completeTexp} 
 Given the $n$-th order $\incomp$ constructed in Lemma \ref{incomplete Texp}, if $\Sele_n $ satisfies \eqref{two_properties0}--\eqref{3rd_property0} and \eqref{two_properties0V}--\eqref{weaz}, then we can construct an $n$-th order $T$-expansion satisfying Definition \ref{defn genuni}.
\end{lemma}

\begin{proof}
By property (iii) of Definition \ref{defn incompgenuni}, we can write that 
\be\label{form_IT1}
\begin{split}
(\PIT^{(n)})_{\fa,\fb_1 \fb_2} = \sum_x \Theta_{\fa x} (\Gamma_R^{(n)})_{x,\fb_1 \fb_2},\quad & (\AITn)_{\fa,\fb_1 \fb_2} = \sum_x \Theta_{\fa x}(\Gamma_A^{(>n)})_{x,\fb_1 \fb_2} ,\\
 (\QITn)_{\fa,\fb_1 \fb_2}= \sum_x \Theta_{\fa x}(\Gamma_Q^{(n)})_{x,\fb_1 \fb_2},\quad & (\Err'_{n,D})_{\fa,\fb_1\fb_2}=\sum_x \Theta_{\fa x}(\Gamma_{err}^{(n,D)})_{x,\fb_1 \fb_2} ,
 \end{split}
\ee
for some sums of graphs $\Gamma_R^{(n)}$, $\Gamma_A^{(>n)}$, $\Gamma_Q^{(n)}$ and $\Gamma_{err}^{(n,D)}$. Then moving the second term on the right-hand side of \eqref{mlevelT incomplete} to the left-hand side and multiplying both sides by $(1-\Theta\wtSdeltan)^{-1}$, we get that 
\begin{align}
	T_{\fa,\fb_1\fb_2} &=  m \Theta^{(n)}_{\fa\fb_1} \overline G_{\fb_1\fb_2}   + \sum_x \Theta^{(n)}_{\fa x} \left[(\Gamma_R^{(n)})_{x,\fb_1 \fb_2} + (\Gamma_A^{(>n)})_{x,\fb_1 \fb_2}+  (\Gamma_Q^{(n)})_{x,\fb_1 \fb_2} +   (\Gamma_{err}^{(n,D)})_{x,\fb_1 \fb_2}\right] . \label{solving_T2}
\end{align}	
where we used $\Theta^{(n)}=(1-\Theta \wtSdelta^{(n)})^{-1}\Theta$ by \eqref{theta_renormal}. We can expand $\Theta^{(n)}$ as 
\be\label{expand_thetan}
\Theta^{(n)}= \sum_{k=0}^D ( \Theta \wtSdelta^{(n)})^{k}\Theta + \Theta^{(n)}_{err},\quad \Theta^{(n)}_{err}:= \sum_{k>D} ( \Theta \wtSdelta^{(n)})^{k}\Theta.
\ee
This expansion is well-defined because $\|\Theta\wtSdelta^{(n)}\|_{\ell^\infty \to \ell^\infty} \le W^{-d+\tau}$ for any constant $\tau>0$ by estimate \eqref{redundant again} below. Every $( \Theta \wtSdelta^{(n)})^{k}\Theta$ can be expanded into a sum of \emph{labelled $\dashed$ edges} (cf. Definition \ref{def_graph comp}), which are allowed in the $T$-expansion (cf. Definition \ref{def genuni}). Moreover, we regard \smash{$(\Theta^{(n)}_{err})_{xy}$} as a $\dashed$ edge of scaling order $> 2D$. Then we plug \eqref{expand_thetan} into \eqref{solving_T2} and rearrange the resulting graphs as follows: \smash{$m \Theta^{(n)}_{\fa\fb_1} \overline G_{\fb_1\fb_2}$} will give the first two terms in \eqref{mlevelTgdef} and some graphs in {$ (\ATn)_{\fa,\fb_1\fb_2}$}; \smash{$\sum_x \Theta^{(n)}_{\fa x} (\Gamma_R^{(n)})_{x,\fb_1 \fb_2}$} will give \smash{$(\PTn)_{\fa,\fb_1\fb_2}$} and some graphs in \smash{$ (\ATn)_{\fa,\fb_1\fb_2}$}; {$\sum_x \Theta^{(n)}_{\fa x} (\Gamma_A^{(>n)})_{x,\fb_1 \fb_2}$} will give some graphs in \smash{$ (\ATn)_{\fa,\fb_1\fb_2}$}; \smash{$\sum_x \Theta^{(n)}_{\fa x} (\Gamma_Q^{(n)})_{x,\fb_1 \fb_2}$} will give \smash{$(\QTn)_{\fa,\fb_1\fb_2}$}; {$\sum_x \Theta^{(n)}_{\fa x} (\Gamma_{err}^{(n,D)})_{x,\fb_1 \fb_2}$} will give $(\Err_{n,D})_{\fa,\fb_1\fb_2}$. This concludes Lemma \ref{lem completeTexp}. 
\end{proof}
 


Finally, we collect the results in Sections \ref{sec_strat_step1}--\ref{sec expT}  to complete the proof of Theorem \ref{main thm}. The following proof is simply a recap of the strategy described in Section \ref{sec_main_struct}.

\begin{proof}[Proof of Theorem \ref{main thm}]
We follow the  flow chart in Figure \ref{Fig pfchart1}. 

\vspace{5pt}
\noindent{\bf Step 1:} By Lemma \ref{2nd3rd T}, we have defined the second order $T$-expansion. 

\vspace{5pt}
\noindent{\bf Step 2:} Suppose that we have defined the $k$-th order $T$-expansion for all $2\le k \le n-1$. Then applying Theorem \ref{thm ptree}, we get that the local law \eqref{locallaw1} holds as long as $L\le L_n:=W^{1+(n-2)d/{4} - {c_0}/{2}}$. 

\vspace{5pt}
\noindent{\bf Step 3:} 
We can construct an $n$-th order $\incomp$ by Lemma \ref{incomplete Texp}. 

\vspace{5pt}
\noindent{\bf Step 4:} Using the local law in Step 2 and the $n$-th order $\incomp$ in Step 3, we show properties \eqref{two_properties0}--\eqref{3rd_property0} and \eqref{two_properties0V}--\eqref{weaz} for $\Sele_n$ in Lemma \ref{cancellation property}. 



\vspace{5pt}
\noindent{\bf Step 5:}  
Applying Lemma \ref{lem completeTexp} we obtain the $n$-th order $T$-expansion. 

\vspace{5pt}
By induction, we can construct the $n$-th order $T$-expansion for all $2\le n \le n_{W,L}$ with $$n_{W,L} = \left\lceil\frac{4}{d}\left(\log_W L  - 1 + \frac{c_0}{2}\right)\right\rceil+2.$$ 
Then we apply Theorem \ref{thm ptree} to  conclude Theorem \ref{main thm}. 
\end{proof}
For the reader's convenience, we  summarize the lemmas in this section which are still  to be proved. 
\begin{itemize}



\item Lemmas \ref{eta1case0}, \ref{lemma ptree}, \ref{lem normA} and \ref{incomplete Texp} will be proved in \cite{PartII_high}.


\item Lemma \ref{lem V-R wt} and Lemma \ref{lem FT0} will be proved in Section \ref{sec infspace}.

\item Lemma \ref{lem informal} will be proved in Section \ref{secpfinformal}.
 

 
\end{itemize}
In Section \ref{sec strategy}, we will describe some key ideas in   \cite{PartII_high} that are needed to  prove  Lemmas \ref{lemma ptree}, \ref{lem normA} and \ref{incomplete Texp}. 

\section{Doubly connected property}\label{sec double} 

In this section, we will introduce an important structural property---the doubly connected property---satisfied by the graphs in the $T$-expansion.

\subsection{Labelled $\Theta$ edges}
In this subsection, we show how \eqref{intro_redagain} follows from the properties \eqref{two_properties0}--\eqref{3rd_property0} of the $\selfs$. 
The following lemma is a simple consequence of the sum zero property and will be proved in Appendix \ref{appd}. 
 
\begin{lemma}\label{lem cancelTheta}
Fix any $z=E+\ii \eta$ with $E\in (-2+\kappa,2-\kappa)$ and $\eta\ge W^{2}/L^{2-\e}$ for a small constant $\e>0$. Let $g:\Z_L^d \to \R$ be a symmetric function (i.e. $g(x)=g(-x)$) supported on a box $\cal B_K:=\llbracket -K,K\rrbracket^d$ of scale $K\ge W$. Assume that $g$ satisfies the sum zero property $\sum_{x}g(x)=0.$ Then for any $x_0\in \Z_L^d$ such that $|x_0| \ge K^{1+c}$ for a constant $c>0$, we have that 
$$\Big|\sum_{x}\Theta_{0 x}(z) g(x-x_0)\Big| \le   \sum_{x\in \cal B_K}\frac{x^2}{|x_0|^2}|g(x)| \cdot \left(  |x_0|^\tau  B_{0 x_0} \mathbf 1_{|x_0|\le   \eta^{-1/2}W^{1+\tau}}   +  |x_0|^{-D} \right),$$
for any constants $\tau,D>0$.
\end{lemma}

With Lemma \ref{lem cancelTheta}, we can readily obtain the following lemma. The long proof is due to extra arguments needed to handle the facts that \eqref{3rd_property0} is only an approximate ``sum zero property" and $(\Sele_l)_{0x}$ satisfies the ``compactly supported property" of $g(x)$ only approximately. 

 \begin{lemma}\label{lem redundantagain}
Fix $d\ge 6$. Given a $\self$ $\Selek_{2l} $ satisfying Definition \ref{collection elements}, we have that 
 \begin{align}\label{redundant again}
\Big| \sum_{\al}\Theta_{x \al} (\Selek_{2l})_{\al y}\Big| \le \frac{W^\tau}{W^{(l-1)d}\langle x-y\rangle^d },\quad \forall \ x,y \in \Z_L^d ,
\end{align}
for any small constant $\tau>0$. Let $\Sele_{2k_1}, \, \Sele_{2k_2},\, \cdots , \, \Sele_{2k_l} $ be a sequence of $\selfs$ satisfying Definition \ref{collection elements}. We have that for any small constant $\tau>0$,
\be\label{BRB} 
\left|\left(\Theta \Sele_{2k_1}\Theta  \Sele_{2k_2}\Theta \cdots \Theta  \Sele_{2k_l}\Theta\right)_{xy}\right|\le W^{-(k-2)d/2+\tau}B_{xy}  , \quad \forall \ x,y \in \Z_L^d , \ee
where $k:=\sum_{i=1}^l  2k_i -2(l-1)$ is the scaling order of $\left(\Theta \Sele_{2k_1}\Theta  \Sele_{2k_2}\Theta \cdots \Theta  \Sele_{2k_l}\Theta\right)_{xy}$. The estimate \eqref{BRB} implies \eqref{intro_redagain} for the $\Sdeltak$ defined in \eqref{chain S2k}. 
 \end{lemma}
 \begin{proof}
We abbreviate $r:=\langle x-y\rangle$. To prove \eqref{redundant again}, we decompose the sum over $\al$ according to the dyadic scales $ \cal I_{n}:=\{\al \in \Z_L^d: K_{n-1} \le |\al-y| \le K_{n}\}$, where $K_n$ are defined by
\be\label{defn Kn}
K_n:= 2^n W \ \ \text{for}\ \ 1\le n \le \log_2 (L/W)-1, \ \ \text{and}\ \  K_{0}:=0. 
\ee
If $K_n\ge W^{-\e} r$ for a small constant $\e>0$, then we have that
\be\label{boundIfar} \Big| \sum_{\al\in \cal I_{n}}\Theta_{ x \al} (\Selek_{2l})_{\al y}\Big| \le   \sum_{\al\in \cal I_{n}}\left|\Theta_{ x \al}\right| \cdot \max_{\al\in \cal I_n} \left|(\Selek_{2l})_{\al y}\right| \le W^\e\frac{K_n^2}{W^2} \cdot W^\e\frac{W^2}{W^{(l-1)d} K_n^{d+2}}\le \frac{W^{(d+2)\e}}{W^{(l-1)d} r^{d}},\ee
where in the second step we used \eqref{4th_property0} (together with $2d-4\ge d+2$ when $d\ge 6$) and $\sum_{\al\in \cal I_{n}}\Theta_{ x \al}\le W^\e {K_n^2}/{W^2}$ by \eqref{thetaxy}. It remains to bound the sum
$$\sum_{\al\in \cal I_{near}}\Theta_{x \al} (\Selek_{2l})_{\al y},\quad \cal I_{near}:=\bigcup_{n: K_n \le W^{-\e}r} \cal I_n .$$
In order for $\cal I_{near}$ to be nonempty, it suffices to assume that $r\ge W^{1+\e}$. 

Using \eqref{4th_property0} and \eqref{3rd_property0}, we can obtain that 
\be\label{mean Snear}\sum_{\al\in \cal I_{near}} (\Selek_{2l})_{\al y}=\sum_\al (\Selek_{2l})_{\al y} - \sum_{x\notin \cal I_{near}} (\Selek_{2l})_{\al y} \le W^\e\left(\frac{\eta}{ W^{(l-1)d} } + \frac{W^2}{W^{(l-1)d} (W^{-\e}r)^{2}}\right).\ee
Then we write $ (\Selek_{2l})_{\al y} = \overline R + \mathring R_{\al y}$ for $\al \in \cal I_{near}$, where $\overline R:= \sum_{\al\in \cal I_{near}}(\Selek_{2l})_{\al y}/|\cal I_{near}|$ is the average of $(\Selek_{2l})_{\al y}$ over $\cal I_{near}$.   
By \eqref{4th_property0} and \eqref{mean Snear}, we have that
\be\label{bound barR} |\overline R| \le \frac{W^{(d+3)\e}W^2}{W^{(l-1)d}r^{d+2}}+\frac{ \eta W^{(d+1)\e}}{W^{(l-1)d}r^d},\quad | \mathring R_{\al y} | \le  \frac{W^{2+\e}}{W^{ (l-1)d}\langle \al-y\rangle^{d+2}} + |\overline R| .\ee
Thus we can bound that 
\begin{align}
\Big| \sum_{\al \in \cal I_{near}}\Theta_{  x \al} \overline R  \Big|&\le \left( \frac{W^{(d+3)\e}W^2}{W^{(l-1)d}r^{d+2}}+\frac{ \eta W^{(d+1)\e}}{W^{(l-1)d}r^d}\right)\sum_{\al\in \cal I_{near}}\Theta_{x\al} \nonumber\\
&\lesssim   \left( \frac{W^{(d+3)\e}W^2}{W^{(l-1)d}r^{d+2}}+\frac{ \eta W^{(d+1)\e}}{W^{(l-1)d}r^d}\right) \min\left\{W^\e\frac{W^{-2\e}r^2}{W^2},\eta^{-1}\right\} \lesssim  \frac{W^{(d+2)\e}}{W^{(l-1)d} r^{d}},\label{Rnear_sum1}
\end{align}
where in the second step we used \eqref{thetaxy} and \eqref{sumTheta} to bound $\sum_{\al\in \cal I_{near}}\Theta_{x\al} .$
Finally, we use Lemma \ref{lem cancelTheta} to bound the sum over $\mathring R$ as 
\begin{align}
\Big| \sum_{\al\in \cal I_{near}}\Theta_{ x \al} \mathring R_{\al y} \Big|& \le  \left(\sum_{\al\in \cal I_{near}}\frac{|\al-y|^2}{r^2}|\mathring R_{\al y}|\right) \left(  \frac{W^\e}{W^2r^{d-2}}\mathbf 1_{r\le   \eta^{-1/2}W^{1+\e}}   + W^{-D}\right) \nonumber\\
 & \le \left(\sum_{\al\in \cal I_{near}}\frac{W^{2+\e}}{r^2\langle\al-y\rangle^d W^{(l-1)d}} + W^{-(d+2)\e} r^d |\overline R|\right) \left(  \frac{W^\e}{W^2r^{d-2}}\mathbf 1_{r\le   \eta^{-1/2}W^{1+\e}}   + W^{-D}\right)\nonumber\\
& \le \left(\frac{W^{2+2\e}}{W^{(l-1)d}r^2 } + \frac{\eta W^{-\e}}{W^{(l-1)d}}\right) \left( \frac{W^\e \mathbf 1_{r\le   \eta^{-1/2}W^{1+\e}}}{W^2r^{d-2}}   + W^{-D}\right)\le \frac{W^{4\e}}{W^{(l-1)d} r^{d}},\label{Rnear_sum2}
\end{align}
where in the second and third steps we used \eqref{bound barR}, and in the last step we used 
$$\frac{\eta W^{-\e}}{W^{(l-1)d}}\frac{W^\e}{W^2r^{d-2}}\mathbf 1_{r\le   \eta^{-1/2}W^{1+\e}}  \le  \frac{W^{2\e}}{W^{(l-1)d} r^{d}}.$$
Combining \eqref{boundIfar}, \eqref{Rnear_sum1} and \eqref{Rnear_sum2}, we conclude \eqref{redundant again} since $\e$ is arbitrary. 

From \eqref{redundant again}, we can obtain \eqref{BRB} easily by using the following simple facts: if $f_1$, $f_2$ and $g$ are functions on $\Z_L^d\times \Z_L^d$ satisfying that
$$|f_1(x,y)| \le \langle x-y\rangle^{-d}, \quad |f_2(x,y)| \le \langle x-y\rangle^{-d},\quad |g(x,y)| \le W^{-2}\langle x-y\rangle^{-d+2},$$
then we have
$$\sum_\al \left|f_1(x,\al)f_2(\al,y)\right|\lesssim \langle x-y\rangle^{-d},\quad \sum_\al \left|f_1(x,\al)g(\al,y)\right|\lesssim W^{-2} \langle x-y\rangle^{-d+2}. $$
Finally, \eqref{intro_redagain} follows from \eqref{BRB} directly by definition \eqref{chain S2k}.
\end{proof}


%
%

The $\selfs$ in Definition \ref{collection elements} will appear in the following \emph{labelled $\dashed$ edges}, which are formed by joining the $\selfs$ with $\dashed$ edges. 

\begin{definition}[Labelled $\dashed$ edges] \label{def_graph comp} 
Given $l$ $\selfs$ $\Sele_{2k_i}$, $ i =1,2,\cdots,l$, we represent the entry
\be\label{eq label} \left(\Theta \Sele_{2k_1}\Theta  \Sele_{2k_2}\Theta \cdots \Theta  \Sele_{2k_l}\Theta\right)_{xy}\ee
 by a labelled $\dashed$ edge between atoms $x$ and $y$ with label $(k;  {2k_1}, \cdots , {2k_l})$, where $k:=\sum_{i=1}^l 2 k_i -2(l-1)$ is the scaling order of this edge.  
In graphs, each labelled $\dashed$ edge is drawn as one single double-line edge with a label but without any internal structure as in the following figure: 
\begin{center}
\includegraphics[width=10cm]{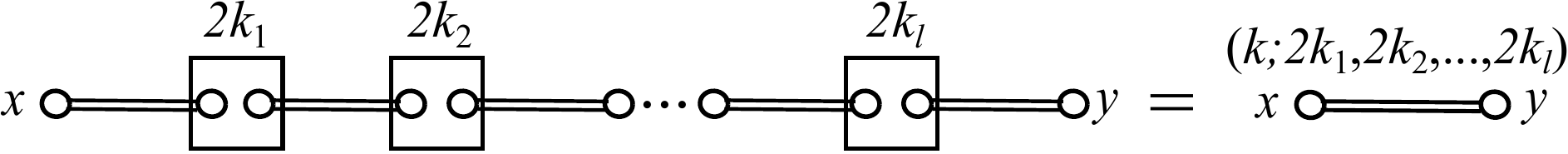}
\end{center}
\end{definition}

The scaling order of a labelled $\dashed$ edge is calculated as follows. Taking \eqref{eq label} as an example, there are $l+1$ $\dashed$ edges of total scaling order $2(l+1)$, $\selfs$ $\Sele _{2k_1},\cdots, \Sele_{2k_l}$ of total scaling order \smash{$\sum_{i=1}^l 2 k_i$}, and $2l$ internal atoms of total scaling order $-4l$. Hence the scaling order of \eqref{eq label} is $2(l+1)+\sum_{i=1}^l 2 k_i -4l =k. $ By \eqref{BRB}, \eqref{eq label} is bounded by $W^{-(k-2)d/2+\tau}B_{xy}$ for any constant $\tau>0$, i.e. it has the same decay with respect to $|x-y|$ as $\Theta_{xy}$ except for an extra $W^{-(k-2)d/2}$ factor.  As a convention, both $\dashed$ and labelled $\dashed$ edges will be called ``$\dashed$ edges".


The scaling order of a normal regular graph with labelled $\dashed$ edges can be equivalently counted as
\begin{align}
  \text{ord}(\cal G) &:= \#\{\text{off-diagonal }  G  \text{ edges}\} + \#\{\text{light weights}\} + 2\#\{ \text{waved edges}\}  + 2 \#\{\text{$\dashed$ edges}\} \nonumber\\
&+\sum_k k\cdot   \#\{k\text{-th order labelled $\dashed$ edges}\}  -  2\left[ \#\{\text{internal atoms}\}- \#\{\text{dotted edges}\} \right] . \label{eq_deforderrandom3}
\end{align}
In other words, a $k$-th order labelled $\dashed$ edges is simply counted as an edge of scaling order $k$, and there is no need to count its internal structures using Definition \ref{def scaling}.

\subsection{Doubly connected property}\label{sec doublenet}

Recall the definition of molecules in Definition \ref{def_poly}. We define the molecular graph as the quotient graph of the atomic graph  with the equivalence relation that atoms belonging to the same molecule are equivalent. 
  
\begin{definition}[Molecular graphs] \label{def moleg}
 Molecular graphs are graphs consisting of 
  \begin{itemize}
\item external molecules which represent the external atoms (such as the $\otimes$, $\oplus$ and $\ominus$ molecules);
\item internal molecules;
\item blue and red solid edges, which represent the plus and minus $G$ edges between molecules;
\item 
$\dashed$ edges between molecules;
\item dotted edges between 
external and internal molecules.
\end{itemize}
Given any atomic graph $\cal G$, we define its molecular quotient graph $ \cal G_{\cal M}$ in the following way:
\begin{itemize}
\item each molecule of $\cal G$ is represented by a vertex in $ \cal G_{\cal M}$;

\item each blue or red solid edge of $\cal G$ between atoms in different molecules is represented by a blue or red solid edge between these two molecules in $ \cal G_{\cal M}$;

\item each $\dashed$ edge of $\cal G$ between atoms in different molecules is represented by a $\dashed$ edge between these two molecules in $ \cal G_{\cal M}$;

\item each dotted edge of $\cal G$ between an external atom and an internal atom is represented by a dotted edge between the corresponding external and internal molecules;


\item we discard all the other components in $\cal G$ (including the weights, $\times$-dotted edges, 
and  
all edges inside any molecule). 
\end{itemize}
\end{definition}

We emphasize that molecular graphs are used solely to analyze the graph structures; the expansions in Section \ref{sec_basiclocal}  are only applied to atomic graphs. In the following proof, we assume that each atomic graph is automatically associated with a molecular graph. As discussed below Definition \ref{def_poly}, we call 
the structure of the molecular graph as the {\it global structure} of the atomic graph.  



The following \emph{doubly connected} property is a key global property for our proof. It allows us to establish a direct connection between the scaling order of a graph and a bound on its value (cf. Lemma \ref{no dot} below). In fact, all graphs in the $T$-expansion and $\incomp$ will satisfy this property (cf. Definitions \ref{def genuni} and \ref{def incompgenuni}). 

\begin{definition} [Doubly connected property]\label{def 2net}  
A subgraph $\cal G$ without external molecules is said to be \emph{doubly connected} if its molecular graph $\cal G_{\cal M}$ satisfies the following property. There exist a collection, say $\cal B_{black}$, of $\dashed$ edges and another collection, say $\cal B_{blue}$, of  either blue solid or  $\dashed$ edges  such that (a) $\cal B_{black}\cap \cal B_{blue}=\emptyset$, and (b) both  $\cal B_{black}$ and $\cal B_{blue}$ contain a spanning tree 
that connects all molecules in the graph. For simplicity of notations, we call the $\dashed$ edges in $\cal B_{black}$ as black edges, and the blue solid and $\dashed$ edges in $\cal B_{blue}$ as blue edges. Correspondingly, $\cal B_{black}$ and $\cal B_{blue}$ are referred to as \emph{black net} and \emph{blue net}, respectively,  where a ``net" refers to a subset of edges that contains a spanning tree.  


A graph  $\cal G$ with external molecules is said to be doubly connected if its subgraph with all external molecules removed is doubly connected, i.e. the spanning trees in the two nets are not required to contain the external molecules.
\end{definition}

The doubly connected property is defined on molecular graphs, and thus is a \emph{global property}. In the above definition, the $\dashed$ edges also include labelled $\dashed$ edges introduced in Definition \ref{def_graph comp}. The red solid edges are not tracked in the doubly connected property, and the path connectivity of red solid edges can be broken in our expansion procedure in \cite{PartII_high}. By symmetry, we can also define an expansion procedure so that graphs in the $T$-expansion satisfy the doubly connected property with a black net and a red net.

 \subsection{$T$-expansion with doubly connected structures}
 
For graphs in the $T$-expansion, they are all doubly connected in the sense of Definition \ref{def 2net}. By including this property and the labelled $\dashed$ edges in Definition \ref{def_graph comp}, we are now ready to state the rest of the details for the $T$-expansion in Definition \ref{defn genuni} and the $\incomp$ in Definition \ref{defn incompgenuni}. We will design an expansion strategy in \cite{PartII_high} so that all graphs in the $T$-expansion and $\incomp$ are doubly connected.
\begin{definition} [More properties of the $n$-th order $T$-expansion]
\label{def genuni}
An $n$-th order $T$-expansion of $T_{\fa,\fb_1 \fb_2}$ is an expression satisfying Definition \ref{defn genuni} and the following additional properties.


\begin{enumerate}

\item A diffusive edge in the graphs on the right-hand side of \eqref{mlevelTgdef} is either a $\Theta$ edge or a labelled $\dashed$ edge of the form \eqref{eq label} with $4\le 2k_i \le n$. 

\item Each graph in $(\PTn)_{\fa,\fb_1\fb_2}$, $(\ATn)_{\fa,\fb_1\fb_2}$, $(\QTn)_{\fa,\fb_1\fb_2}$ and $(\Err_{n,D})_{\fa,\fb_1\fb_2}$ is doubly connected in the sense of Definition \ref{def 2net}. 
\end{enumerate}

 \end{definition}
   

 \begin{definition} [More properties of the $n$-th order $\incomp$]\label{def incompgenuni}
 An $n$-th order $\incomp$ of $T_{\fa,\fb_1 \fb_2}$ is an expression satisfying Definition \ref{defn incompgenuni} and the following additional properties.
\begin{enumerate}
	
	
	\item A diffusive edge in $\Sele_n$ and the graphs on the right-hand side of \eqref{mlevelT incomplete} is either a $\Theta$ edge or a labelled $\dashed$ edge of the form \eqref{eq label} with $4\le 2k_i \le n-1$. 

	\item Each graph in $\Sele_n$, $(\PITn)_{\fa,\fb_1\fb_2}$, $(\AITn)_{\fa,\fb_1\fb_2}$, $(\QITn)_{\fa,\fb_1\fb_2}$ and $(\Err'_{n,D})_{\fa,\fb_1\fb_2}$ is doubly connected in the sense of Definition \ref{def 2net}. 
\end{enumerate}

\end{definition}

\subsection{Bounding doubly connected graphs}\label{sec bound double}

In this subsection, we give some important estimates on doubly connected graphs in Lemma \ref{no dot}. In particular, these estimates will be used crucially in the proofs of Lemma \ref{lem V-R wt} and Lemma \ref{lem informal}. Inspired by the maximum bound in \eqref{Gmax} and the weak averaged bound in \eqref{initial Txy222}, we introduce the following weak and strong norms, which will be a convenient tool for the proof of Lemma \ref{no dot}. 

\begin{definition}\label{Def PseudoG} 
Given a $\Z_L^d\times \Z_L^d$ matrix $\cal A$ and some fixed $a,b>0$, we define its weak-$(a,b)$ norm as
$$\|\cal A\|_{w;(a,b)}:= W^{a d/2}\max_{x,y\in \Z_L^d} \left|\cal A_{xy}\right| + \sup_{ K \in [W, L/2]} \left( \frac{W}{K}\right)^b K^{ad/2} \max_{x, x_0\in \Z_L^d} \frac1{K^d}\sum_{y:| y - x_0 |\le K} \left(\left|\cal A_{xy}\right|+\left| \cal A_{yx}\right| \right),$$
and its strong-$(a,b)$ norm as
$$\|\cal A\|_{s;(a,b)}:= \max_{x,y\in \Z_L^d} \left( \frac{W}{\langle x-y\rangle}\right)^b \langle x-y\rangle^{ad/2} \left|\cal A_{xy}\right|.$$
\end{definition}

In this paper, we only use weak or strong-$(a,b)$ norms with $a\le 2$. In this case, it is easy to check that the strong-$(a,b)$ norm is  strictly stronger than the weak-$(a,b)$ norm. By Definition \ref{Def PseudoG}, we immediately get the bounds
\be\label{eq PseudoG2} 
\max_{x,y\in \Z_L^d} \left|\cal A_{xy}\right| \le W^{-a d/2}\|\cal A\|_{w;(a,b)},
\ee
\be\label{eq PseudoG}
\max_{x, x_0\in \Z_L^d} \frac1{ K^{d}}\sum_{y:| y - x_0 |\le K} \left(\left|\cal A_{xy}\right|+\left| \cal A_{yx}\right| \right)\le   \frac{1}{W^b K^{ad/2-b}}  \|\cal A\|_{w;(a,b)},\quad \text{for all $  K \in [W, L/2]$,}
\ee
\be\label{eq unifG}
\left|\cal A_{xy}\right| \le    \frac{1}{W^b \langle x-y\rangle^{ad/2-b}} \|\cal A\|_{s;(a,b)}.
\ee
Here we list the weak or strong norms of some key deterministic or random variables. 
\begin{enumerate}
\item 
$\|B\|_{s;(2,2)}\le 1$ and $\|B^{(1/2)}\|_{s;(1,1)}\le 1$, where $B^{(1/2)}$ is the matrix with entries $(B_{xy})^{1/2}$; 

\item If \eqref{locallaw1} holds, then $\|G(z)-m(z)I_N\|_{s;(1,1)}\prec 1$. 

\item If \eqref{Gmax} and \eqref{initial Txy222} hold, then $\|G(z)-m(z)I_N\|_{w;(1,2)}\prec W^{\e_0}$ and $W^{-2\e_0}\|T(z)\|_{w;(2,4)}\prec W^{2\e_0}$. 



\item The following positive random variable $\Psi_{xy}$ was defined in \cite[Definition 3.4]{PartIII} for a small constant $\tau>0$ and a large constant $D>0$:
\be\label{eq defPsi}
\Psi^2_{xy}\equiv \Psi^2_{xy}(\tau,D) :=W^{-D}+ \max\limits_{ \substack{|x_1-x| \le   W^{1+\tau}  \\ |y_1-y|\le  W^{1+\tau}}}s_{x_1y_1} + W^{-(2+2\tau)d}\sum_{ |x_1-x| \le  W^{1+\tau}}\sum_{ |y_1-y|\le  W^{1+\tau}} |G_{x_1y_1}|^2  .\ee
Note that  $\|\Psi(z)\|_{w;(1,2)} \prec \|G(z)-m(z)I_N\|_{w;(1,2)} + 1$ and  $\|\Psi(z)\|_{s;(1,1)} \prec \|G(z)-m(z)I_N\|_{s;(1,1)} + 1$ as long as $D$ is large enough.

\end{enumerate}
The motivation for introducing the $\Psi$ matrix is as follows: given $x_1,x_2\in \Z_L^d$, suppose $y_1$ and $y_2$ satisfy that  
\be\label{y12x12}|y_1-x_1|\le W^{1+\tau/2} ,\quad |y_2-x_2|\le W^{1+\tau/2}.\ee 
If $y_1\ne y_2$ and we know that $\|G(z)\|_{\max}\prec 1$, then using Lemma \ref{lem G<T} we can obtain the bound 
\begin{align}
|G_{y_1y_2}(z)|^2&\prec T_{y_1y_2}(z) =|m|^2 s_{y_1y_2}|G_{y_2y_2}(z)|^2 +|m|^2 \sum_{\al\ne y_2}s_{y_1 \al}|G_{\al y_2}(z)|^2 \nonumber\\
&=|m|^2 s_{y_1y_2}|G_{y_2y_2}(z)|^2 +|m|^2 \sum_{\al\ne y_2}s_{y_1 \al}|G_{y_2\al }(\overline z)|^2 \nonumber\\
&\prec   s_{y_1y_2} +  \sum_{\al\ne y_2}s_{y_1\al} T_{ y_2 \al}(\overline z)  \le s_{y_1y_2} + \sum_{\al ,\beta} s_{y_1\al}s_{y_2\beta} |G_{\al\beta }(z)|^2 \nonumber\\
&\le W^{-D}+s_{y_1y_2} + W^{-2d}\sum_{|\al-y_1|\le W^{1+\tau/2}}\sum_{|\beta-y_2|\le W^{1+\tau/2}}  |G_{\al\beta}(z)|^2  \le W^{2d\tau} \Psi_{x_1x_2}^2(\tau,D),\label{Gpsi}
\end{align}
where in the third and fifth steps we used the simple identity $G_{xy}(z)=\overline{G_{yx}(\overline z)}$, and in the sixth step we used \eqref{subpoly}. In particular, if $y_1$ and $y_2$ are in the same molecules as $x_1$ and $x_2$, respectively, then we know that \eqref{y12x12} holds, since otherwise the graph value will be smaller than $W^{-D}$ for any fixed $D>0$ by \eqref{subpoly} and \eqref{S+xy}. Then \eqref{Gpsi} shows that all the $G$ edges between two molecules containing atoms $x_1$ and $x_2$ can be bounded with the same variable $\Psi_{x_1x_2}$. This fact will be convenient for our proof. 

%

By \eqref{sumTheta}, the row sums of $\Theta$ diverge when $L\to \infty$ (e.g. if $\eta=W^{2}/L^{2-\e}$). On the other hand, the following claim shows that the product of a $\Theta$ entry and a variable with bounded weak-$(1,2)$ norm is summable if $d\ge 8$. Although this claim will not be used in our proof directly, it explains why we require $d\ge 8$ in Theorem \ref{main thm0}. Our proof of Lemma \ref{no dot} is actually based on some more general versions of this claim in \eqref{keyobs3} and \eqref{keyobs2} below.

 
 \begin{claim}\label{cor: ini bound} 
Let $\cal A$ be a matrix satisfying $\|\cal A\|_{w;(a,b)}\prec 1$ 
for some fixed $a,b>0$. If 
 \be\label{a+b}
 a d/2 - b - 2\ge 0,
 \ee
then we have that
\be\label{BG14}
\max_{x,y\in \Z_L^d} \sum_{ \al }B_{x\al}\cal A_{y\al} \prec  W^{-a d/2}.
 \ee
  \end{claim}
  
  \begin{proof}
We decompose the sum over $\al$ according to the dyadic scales: 
$$\al\in \cal I_{n,m}:=\{\al\in \Z_L^d: K_{n-1} \le |x-\al| \le K_{n},K_{m-1} \le |y-\al| \le K_{m}\},$$ 
where $K_n$ are defined in \eqref{defn Kn}. Then using \eqref{eq PseudoG} and the fact that $\cal I_{n,m}$ is inside a box of scale $\OO(K_{n}\wedge K_{m})$, we can estimate that 
\begin{align*}
\sum_{\al\in \cal I_{n,m}}B_{x\al} \cal A_{y\al} &\prec \frac{1}{W^2K_n^{d-2}}\sum_{\al\in \cal I_{n,m}}\cal A_{y\al}  \prec   \frac1{W^{2} K_{n}^{d-2} }\cdot \frac{\left( K_n\wedge K_m\right)^{d}}{ W^b(K_n\wedge K_m)^{ad/2-b}}  \\
&\le  \frac{1}{ W^{b+2}(K_n\wedge K_m)^{ad/2-b-2}}  \le  W^{-a d/2},
\end{align*}
where in the last step we used \eqref{a+b}. Summing over $\OO( (\log L)^2)$ many such sets $\cal I_{n,m}$, we get that 
$$\sum_{ \al }B_{x\al} \cal A_{y\al} \prec (\log L)^2  W^{-a d/2}\prec W^{-a d/2}.$$
This concludes the proof.
\end{proof}

If $\|G(z)-m(z)I_N\|_{w;(1,2)}\prec 1$, 
then by Claim \ref{cor: ini bound} we have that
$$\max_{x,y\in \Z_L^d} \sum_{ \al }B_{x\al}|G_{y\al}| \prec  W^{- d/2} $$
if $d/2-4\ge 0$, which gives $d\ge 8$. 
We now prove the following key estimates on doubly connected graphs. 
 
\begin{lemma}\label{no dot}
 Suppose $d\ge 8$ and $\|G(z)-m(z)I_N\|_{w;(1,2)}\prec 1$. 
Let $\cal G$ be a doubly connected normal regular graph without external atoms. 
Pick any two atoms of $\cal G$ and fix their values $x , y\in \Z_L^d$. Then the resulting graph $\cal G_{xy}$ satisfies  that 
\be\label{bound 2net1}
\left|\cal G_{xy}\right| \prec W^{ - \left(n_{xy}-3\right)d/2 } B_{xy} \cal A_{xy},
\ee
where  $n_{xy}:=\ord(\cal G_{xy})$ is the scaling order of $\cal G_{xy}$ and  $\cal A_{xy}$ is some  positive variable  satisfying $\|\cal A\|_{w;(1,2)}\prec 1$. 
Furthermore, 
if $\|G(z)-m(z)I_N\|_{s;(1,1)}\prec 1$, then we have that
\be\label{bound 2net1 strong}
\left|\cal G_{xy}\right| \prec W^{ -  \left(n_{xy}-3\right)d/2} B_{xy}^{3/2}.
\ee
If we fix an atom $x \in \cal G$, then  the resulting graph  $\cal G_{x}$ satisfies  that  
\be\label{bound 2net1 singlex}
\left|\cal G_{x}\right| \prec W^{ - \ord(\cal G_{x}) \cdot d/2 } .
\ee
The above bounds  hold also  for the graph $ {\cal G}^{{\rm abs}}$, which is obtained by replacing each component (including edges, weights and coefficients) in $\cal G$ with its absolute value and ignoring all the $P$ or $Q$ labels (if any). We emphasize that in defining $ {\cal G}^{{\rm abs}}$, a labelled $\dashed$ edge \eqref{eq label} will be regarded as one single edge and replaced by $|\left(\Theta \Sele_{2k_1}\Theta  \Sele_{2k_2}\Theta \cdots \Theta  \Sele_{2k_l}\Theta\right)_{xy}|$.
 \end{lemma}

%

Note that a doubly connected graph $\cal G$ with at least two molecules must have $n_{xy}\ge 3$.  If $x$ and $y$ are in the same molecule, then \eqref{bound 2net1} gives the sharp bound $\left|\cal G_{xy}\right| \prec W^{ - n_{xy} d/2 }$.

 \begin{proof} [Proof of Lemma \ref{no dot}]
The estimate \eqref{bound 2net1 singlex} is a special case of \eqref{bound 2net1} with $x=y$. Hence we only need to prove \eqref{bound 2net1} and \eqref{bound 2net1 strong}. Moreover, due to the trivial bound $|\cal G_{xy}|\prec \cal G^{\abs}_{xy}$, it suffices to prove \eqref{bound 2net1} and \eqref{bound 2net1 strong} for the graph \smash{$\cal G_{xy}^{\abs}$}. 
As explained before, the $\eta^{-1}$ factor in \eqref{sumTheta} is the main trouble for our proof. 
We will show that if we choose the order of summation in a proper way, then the following key property holds: for every summation over the global scale $L$, it involves a product of at least one $\dashed$ edge and one variable whose weak-$(1,2)$ or strong-$(1,1)$ norm is bounded by $\OO_\prec(1)$. In particular, every such summation does not provide a large $\eta^{-1}$ factor as we have seen in \eqref{BG14}.


By \eqref{subpoly}, \eqref{thetaxy}, \eqref{S+xy} and \eqref{BRB}, we have the following maximum bounds on deterministic edges:
\be\label{Ssmall}\begin{split}
	\max_{x,y}s_{xy} = & \OO   (W^{-d}), \quad  \max_{x,y}|S^{\pm}_{xy}|  = \OO(W^{-d}),\quad \max_{x,y}\Theta_{xy} \prec W^{-d}, \\
	& \max_{x,y}\left|(\Theta \Sele_{2k_1}\Theta  \Sele_{2k_2}\Theta \cdots \Theta  \Sele_{2k_l}\Theta)_{xy}\right|  \prec W^{-kd /2 } ,
\end{split}
\ee
where $k:=\sum_{i=1}^l 2 k_i -2(l-1)$. 
For simplicity of notations, we will use $\al\sim_{\cal M} \beta$ to mean that ``atoms $\al$ and $\beta$ belong to the same molecule". Suppose there are $\ell$ internal molecules $\cal M_i$, $1\le i \le \ell$, in $\cal G_{xy}^{\abs}$.  We choose one atom in each $\cal M_i$, say $x_i$, as a representative. Moreover, let atoms $x$ and $y$ be the representatives of their respective molecules in $\cal G^{\abs}$. For definiteness, we assume that $x$ and $y$ belong to \emph{different molecules}. The case where $x$ and $y$ belong to the same molecule can be dealt with in a similar way, and we omit the details. In the following proof, we fix a small constant $\tau>0$ and a large constant $D>0$. For any $y_i\sim_{\cal M}x_i$, it suffices to assume that  
\be\label{yixi}
|y_i-x_i|\le W^{1+\tau/2},
\ee
because otherwise the graph is smaller than $W^{-D}$. Then under the assumption \eqref{yixi}, for $y_i\sim_{\cal M} x_i$ and $y_j\sim_{\cal M} x_j$, by \eqref{thetaxy}, \eqref{Gpsi} and \eqref{BRB} we have that
\be\label{intermole1}|G_{y_i y_j}|\prec W^{d\tau} \Psi_{x_i x_j}(\tau, D) ,\quad \Theta_{y_i y_j}\prec B_{y_i y_j} \lesssim W^{(d-2)\tau/2} B_{x_i x_j},\ee
\be\label{intermole2} \left|(\Theta \Sele_{2k_1}\Theta  \Sele_{2k_2}\Theta \cdots \Theta  \Sele_{2k_l}\Theta)_{y_i y_j}\right| \prec W^{-(k-2)d /2 + (d-2)\tau/2} B_{x_i x_j}. \ee 
These estimates show that we can bound the edges between different molecules with $\Psi$ or $B$ entries that only contain the representative atoms $x_i$ in their indices. 
 
First, we bound the edges between different molecules. Due to the doubly connected property of $\cal G_{xy}^{\abs}$, we can pick two spanning trees of the black net and blue net, which we refer to as the \emph{black tree} and \emph{blue tree}, respectively. We bound the edges that do  not belong to the two trees using the maximum bounds: 
\begin{itemize}
	\item[(i)] each solid edge that is not in the blue tree is bounded by $\OO_\prec(W^{-d/2})$ using \eqref{eq PseudoG2} with $a=1$; 
	
	\item[(ii)] each $\dashed$ edge that is not in the black and blue trees is bounded by $\OO_\prec(W^{-d})$; 
	
	\item[(iii)] each labelled $\dashed$ edge that is not in the black and blue trees is bounded by $\OO_\prec(W^{-kd/2})$, where $k$ is the scaling order of this edge. 
\end{itemize}
The edges in the two trees are bounded as follows:
\begin{itemize}
	\item[(iv)] the blue solid and $\dashed$ edges in the two trees are bounded using \eqref{intermole1} and \eqref{intermole2}.
\end{itemize}
%
%
In this way, we can bound that 
\be\label{reduce Gaux0}
\cal G_{xy}^{\abs}\prec W^{-n_1 d/2 + n_2 \tau}\sum_{x_1, \cdots, x_\ell} \Gamma_{global}(x_1, \cdots, x_\ell) \prod_{i=1}^\ell \cal G_{x_i}^{(i)} ,
\ee
where $W^{-n_1 d/2+ n_2\tau}$ is a factor coming from the above items (i)--(iv), $\Gamma_{global}$ is a product of blue solid edges that represent $\Psi$ entries and double-line edges that represent $B$ entries, and every \smash{$\cal G_{x_i}^{(i)}$} is the subgraph inside the molecule $\cal M_i$, which has $x_i$ as an external atom. 
We bound the local structure $\cal G_{x_i}^{(i)}$ inside $\cal M_i$ as follows: 
\begin{itemize}
\item each waved or $\dashed$ edge is bounded by $\OO_\prec(W^{-d })$ using \eqref{Ssmall}; 
\item each labelled $\dashed$ edge is bounded by $\OO_\prec (W^{-kd/2 })$, where $k$ is its scaling order; 
\item each off-diagonal $G$ edge and light weight is bounded by $\OO_\prec(W^{-d/2})$ using \eqref{eq PseudoG2} with $a=1$; 
\item each summation over an internal atom in $\cal M_i \setminus \{x_i\}$ provides a factor $\OO( W^{(1+\tau/2)d})$ due to \eqref{yixi}. 
\end{itemize}
Thus with the definition of the scaling order in \eqref{eq_deforderrandom3}, we get that
\begin{align}\label{internal struc1}
|\cal G_{x_i}^{(i)}|\prec W^{-\ord(\cal G^{(i)}_{x_i}) \cdot d/2 + k_{i}\cdot \tau d/2 }, 
\end{align} 
where $k_i$ is the number of internal atoms in $\cal G_{x_i}^{(i)}$. Finally, for convenience of proof, we bound each $\dashed$ edge in the \emph{blue} (but not black) tree of $\Gamma_{global}(x_1, \cdots, x_\ell)$ as
\be\label{bound thetaweak} B_{x_i x_j}\le W^{-d/2}B_{x_i x_j} ^{1/2}.\ee
Then every edge in the blue tree represents a $\Psi$ or $B^{(1/2)}$ entry, whose weak-$(1,2)$ or strong-$(1,1)$ norm is bounded by $\OO_\prec(1)$ (depending on whether we want to prove \eqref{bound 2net1} or \eqref{bound 2net1 strong}).
Plugging \eqref{internal struc1} and \eqref{bound thetaweak} into \eqref{reduce Gaux0}, we obtain that
\be\label{bound G aux}
 \cal G_{xy}^{\abs} \prec W^{- (n_{xy}-\ell -3) d/2 + n_3\tau} (\cal G_{xy})_{aux},
\ee
where $n_3:=n_2+\sum_{i=1}^\ell k_i d/2$ and the number $n_{xy}-\ell -3$ in the exponent can be obtained by counting carefully the number of $W^{-d/2}$ factors from the above arguments. 
Here $(\cal G_{xy})_{aux}$ is an \emph{auxiliary graph} defined as follows: 
 \begin{itemize}
\item it has two external atoms $x$ and $y$, and some internal atoms $x_i$, $1\le i \le \ell$, which are the representative atoms of the molecules in $\cal G_{xy}^{\abs}$; 

\item each $\dashed$ edge in the black tree of $\cal G_{xy}^{\abs}$ is replaced by a double-line edge representing a $B$ entry in $(\cal G_{xy})_{aux}$; 

\item each edge in the blue tree of $\cal G_{xy}^{\abs}$ is replaced by a blue solid edge representing 
a $\Psi$ or $B^{(1/2)}$ entry in $(\cal G_{xy})_{aux}$.
\end{itemize}

By the construction of $(\cal G_{xy})_{aux}$, it is doubly connected in the following sense: $(\cal G_{xy})_{aux}$ contains a black spanning tree consisting of black double-line edges and a blue spanning tree consisting of blue solid edges. Now with \eqref{bound G aux}, to conclude the proof it suffices to show that after summing over all the internal atoms in $ (\cal G_{xy})_{aux}$, the auxiliary graph can be bounded as
\be\label{boundaux0}
| (\cal G_{xy})_{aux}| \prec W^{-\ell d/2  }B_{xy}\cal A_{xy} ,
\ee
for a positive variable $\cal A_{xy}$ satisfying $\|\cal A\|_{w;(1,2)}\prec 1$ (resp. $\|\cal A\|_{s;(1,1)}\prec 1$) if $\|G(z)-m(z)I_N\|_{w;(1,2)}\prec 1$ (resp. $\|G(z)-m(z)I_N\|_{s;(1,1)}\prec 1$).
The estimate \eqref{boundaux0} is an easy consequence of the following Claim \ref{twonet claim}. Our auxiliary graph $(\cal G_{xy})_{aux}$ satisfies its assumptions. We postpone its proof until we complete the proof of Lemma \ref{no dot}.

\begin{claim}\label{twonet claim}
Let $\wt{\cal G}_{xy}$ be a graph with two external atoms $x$ and $y$, $\ell$ internal atoms $x_1, x_2,\cdots, x_\ell$, a black spanning tree consisting $\ell+1$ black double-line edges, and a blue spanning tree consisting $\ell+1$ blue solid edges. Suppose that each black edge between atoms, say $\al$ and $\beta$, represents a $B_{\al\beta}$ factor, and each blue edge represents a positive variable whose weak-$(a,b)$ norm is bounded by $\OO_\prec(1)$. If \eqref{a+b} holds, then 
\be\label{boundaux}
| \wt{\cal G}_{xy}| \prec W^{-a\ell  d /2  }B_{xy}\cal A_{xy} ,
\ee
for a positive variable $\cal A_{xy}$ satisfying $\|\cal A\|_{w;(a,b)}\prec 1$. Moreover, if the strong-$(a,b)$ norm of each blue edge is bounded by $\OO_\prec(1)$ and \eqref{a+b} holds, then \eqref{boundaux} holds for a positive variable $\cal A_{xy}$ satisfying $\|\cal A\|_{s;(a,b)}\prec 1$. 
\end{claim}


Note that both $(a,b)=(1,2)$ and $(a,b)=(1,1)$ satisfy \eqref{a+b} for $d\ge 8$. Hence taking $a=1$ in \eqref{boundaux}, we obtain \eqref{boundaux0}. Combining \eqref{bound G aux} and \eqref{boundaux0}, we conclude \eqref{bound 2net1} and \eqref{bound 2net1 strong} for $\cal G_{xy}^{\abs}$ since $\tau$ is arbitrary. 
\end{proof}

\begin{proof}[Proof of Claim \ref{twonet claim}]
Our proof is based on the following extensions of Claim \ref{cor: ini bound}. If $\cal A^{(1)}$ and $\cal A^{(2)}$ are two matrices whose weak-$(a,b)$ or strong-$(a,b)$ norms are bounded by $\OO_\prec(1)$, then we have that 
\be\label{keyobs3}
\sum_{x_i}\cal A^{(2)}_{x_i \beta}\cdot \prod_{j=1}^k B_{x_i y_j } \prec W^{-ad/2 } \Gamma(y_1,\cdots, y_k),
\ee
and 
\be\label{keyobs2}
\sum_{x_i}\cal A^{(1)}_{x_i \al}\cal A^{(2)} _{x_i \beta}\cdot \prod_{j=1}^k B_{x_i y_j } \le W^{-a d/2 }\Gamma(y_1,\cdots, y_k)  {\cal A}_{\al\beta},
\ee
where $ {\cal A}$ is a matrix with $\|\cal A\|_{w;(a,b)}\prec 1$ if  $\|\cal A^{(1)}\|_{w;(a,b)}+ \|\cal A^{(2)}\|_{w;(a,b)}\prec 1$ or $\|\cal A\|_{s;(a,b)}\prec 1$ if  $\|\cal A^{(1)}\|_{s;(a,b)}+ \|\cal A^{(2)}\|_{s;(a,b)}\prec 1$, and $\Gamma(y_1,\cdots, y_k)$ is defined as a sum of $k$ different products of $(k-1)$ double-line edges:  
\be\label{defn_Gamma}\Gamma(y_1,\cdots, y_k):= \sum_{i=1}^k \prod_{j\ne i}B_{y_{i}y_j}.\ee
Intuitively speaking, \eqref{keyobs2} means that after summing over a product of $k$ double-line edges and two solid edges, we lose one double-line edge and one solid edge, which leads to the $W^{-ad/2}$ factor as in Claim \ref{cor: ini bound}. In each new graph, we have $(k-1)$ double-line edges connected with one of the neighbors of $x_i$ on the black tree, and one solid edge between atoms $\al$ and $\beta$ representing $\cal A_{\al\beta}$. In the following figure, we draw an example with $k=3$, where there are three graphs corresponding to the three terms on the right-hand side of \eqref{defn_Gamma} and we have omitted the factor $W^{-ad/2}$ from them: 
\begin{center}
\includegraphics[width=12cm]{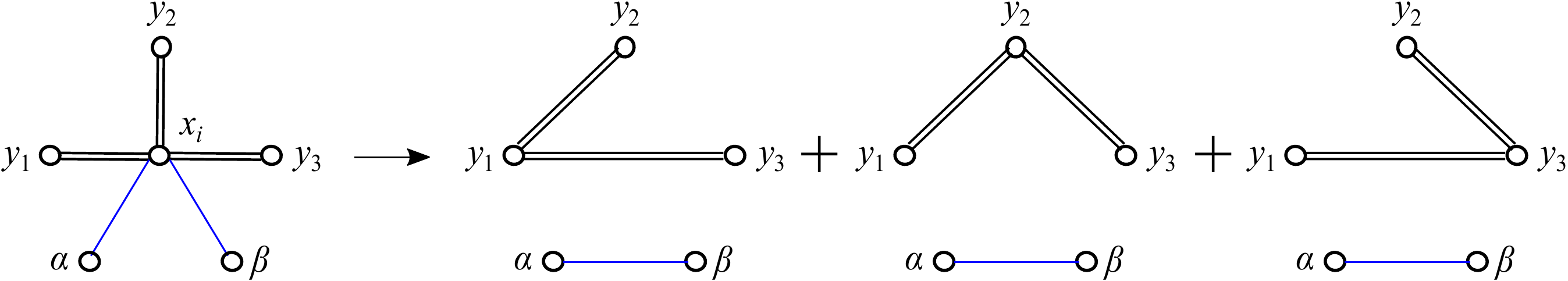}
\end{center}

To prove \eqref{keyobs3}, it suffices to assume the weaker condition $\|\cal A^{(2)}\|_{w;(a,b)}\prec 1$. We decompose the sum over $x_i$ according to dyadic scales $K_n$ defined in \eqref{defn Kn}. Consider the case $x_i \in \cal I_{\vec a}$ for some $\vec a:=(a_1, \cdots, a_k)\in (\N\setminus\{0\})^k$, where 
$$\cal I_{\vec a}:=\{ x_i: K_{a_j-1} \le |x_i - y_j| \le K_{a_j} ,  1\le j \le k\}.$$ For simplicity of notations, we abbreviate $L_j := K_{a_j}$ and $L_{\min}:=\min_{1\le j \le k}L_j $. Then we have that 
\begin{align}\label{sum_ba}
\sum_{x_i \in \cal I_{\vec a}} |\cal A^{(2)}_{x_i \beta}|\cdot \prod_{j=1}^k B_{x_i y_j }\le \prod_{j=1}^k\frac{1}{W^2L_j^{d-2}} \sum_{x\in \cal I_{\vec a}}  |\cal A ^{(2)}_{x_i \beta}|  \prec   \frac{L_{\min}^{d}}{W^bL_{\min}^{ad/2-b} }\frac{1}{\prod_{j=1}^k W^2L_j^{d-2}} ,  
\end{align}  
where in the second step we used \eqref{eq PseudoG} and the fact that $\cal I_{\vec a}$ is inside a box of scale $L_{\min}$.
Let $s \in \{1,2,\cdots, k\}$ be the value such that $L_s=L_{\min}$. Using \eqref{a+b}, we obtain that
$$\frac{L_{\min}^{d}}{ W^b L_{\min}^{ad/2-b} } \frac1{W^2L_s^{d-2}} = \frac{1}{W^{b+2}L_{\min}^{ad/2-b-2}}\le W^{-ad/2}.$$
Combining this bound with the fact that  
\be\label{yj-s}\langle y_j - y_s \rangle \le W+ L_j + L_s \le 3L_j,\quad j \ne s, \ee  
we can bound \eqref{sum_ba} as 
 \begin{align}
 \sum_{x_i \in \cal I_{\vec a}} |\cal A^{(2)}_{x_i \beta}|\cdot \prod_{j=1}^k B_{x_i y_j } \prec W^{-ad/2} \prod_{j\ne s} \frac1{W^2L_j^{d-2}} \lesssim W^{-ad/2} \prod_{j\ne s} B_{y_j y_s} \le  W^{-ad/2}{\Gamma(y_1,\cdots, y_k)} . \nonumber
\end{align}  
Summing over all possible scales $\cal I_{\vec a}$, we conclude \eqref{keyobs3}.

Next we prove \eqref{keyobs2} when $\|\cal A^{(1)}\|_{w;(a,b)}+ \|\cal A^{(2)}\|_{w;(a,b)}\prec 1$.
Applying \eqref{eq PseudoG2} to $\cal A^{(1)}_{x_i\al}$ and using \eqref{keyobs3}, we get that
\begin{align*}
& 
\sum_{x_i} |\cal A^{(1)}_{x_i \al}|  |\cal A^{(2)} _{x_i \beta}|\cdot \prod_{j=1}^k B_{x_i y_j } 
\prec W^{-ad/2}\sum_{x_i }  |\cal A^{(2)} _{x_i \beta}|\cdot \prod_{j=1}^k B_{x_i y_j }  \prec W^{-ad }\Gamma(y_1,\cdots, y_k).
\end{align*}
Applying \eqref{eq PseudoG} to $\cal A^{(1)}_{x_i\al}$ and using \eqref{keyobs3}, we get that for any $x_0\in \Z_L^d$ and $  K \in [W, L/2]$,
\begin{align*}
 \frac{1}{K^d}\sum_{\al: |\al - x_0|\le K}\sum_{x_i} |\cal A^{(1)}_{x_i \al}|  |\cal A^{(2)} _{x_i \beta}|\cdot \prod_{j=1}^k B_{x_i y_j }  & \prec \frac{\sum_{x_i } |\cal A^{(2)} _{x_i \beta}|\cdot  \prod_{j=1}^k B_{x_i y_j } }{W^{b}K^{ad/2-b}}  \prec W^{-ad/2} \frac{\Gamma(y_1,\cdots, y_k)}{W^{b}K^{ad/2-b}}. 
\end{align*} 
We can obtain a similar estimate for the average over $\{\beta:|\beta-x_0|\le K\}$. The above two estimates imply that $  \|\cal A\|_{w;(a,b)}\prec 1$, where $\cal A$ is defined by
\be\label{keyobs_add}
\begin{split}
	{\cal A}_{\al\beta}:=\frac{W^{ad/2 }}{\Gamma(y_1,\cdots, y_k)}\sum_{x_i} |\cal A^{(1)}_{x_i \al}|  |\cal A^{(2)} _{x_i \beta}|\cdot \prod_{j=1}^k B_{x_i y_j } . 
\end{split}
\ee
This concludes \eqref{keyobs2} in one case.
Then we prove \eqref{keyobs2} in the other case with $\|\cal A^{(1)}\|_{s;(a,b)}+ \|\cal A^{(2)}\|_{s;(a,b)}\prec 1$. We decompose the sum over $x_i$ according to dyadic scales as $x_i \in \cal I_{\vec a}$ for some $\vec a:=(a_1, \cdots, a_{k+2})\in (\N\setminus\{0\})^{k+2}$, where 
$$\cal I_{\vec a}:=\{ x_i: K_{a_j-1} \le |x_i - y_j| \le K_{a_j} ,  1\le j \le k; K_{a_{k+1}-1} \le |x_i-\al|\le K_{a_{k+1}}, K_{a_{k+2}-1} \le |x_i-\beta|\le K_{a_{k+2}}\}.$$ For simplicity of notations, we abbreviate $L_j := K_{a_j}$, $1\le j \le k+2$, and $L_{\min}:=\min_{1\le j \le k}L_j$. Let $s \in \{1,2,\cdots, k\}$ be the value such that $L_s=L_{\min}$. Then using \eqref{eq unifG} and the fact that $\cal I_{\vec a}$ is inside a box of scale $L_{\min}\wedge L_{k+1}\wedge L_{k+2}$, we obtain that 
\begin{align} \nonumber
& \sum_{x_i \in \cal I_{\vec a}} |\cal A^{(1)}_{x_i \al}|  |\cal A^{(2)} _{x_i \beta}|\cdot \prod_{j=1}^k B_{x_i y_j }    \prec \prod_{j=1}^k\frac{1}{W^2L_j^{d-2}} \cdot\frac{(L_{\min}\wedge L_{k+1}\wedge L_{k+2})^d}{W^{2b} L_{k+1}^{ad/2-b}L_{k+2}^{ad/2-b}} \\
&\le   \prod_{1\le j \le k, j\ne s}\frac{1}{W^2L_j^{d-2}} \cdot\frac{1}{W^{b} (L_{k+1}\vee L_{k+2})^{ad/2-b}}\cdot\frac{1}{W^{b+2}(L_{k+1}\wedge L_{k+2})^{ad/2-b-2}} \nonumber\\
& \lesssim W^{-ad/2} \frac{1}{W^{b} \langle \al-\beta \rangle^{ad/2-b}}\prod_{j\ne s} B_{y_j y_s}. \nonumber
\end{align}  
Here in the second step we used $L_{k+1} L_{k+2}=(L_{k+1}\vee L_{k+2})(L_{k+1}\wedge L_{k+2})$, and  in the third step we used \eqref{a+b}, \eqref{yj-s} and  $\langle \al - \beta \rangle \le W+ L_{k+1} + L_{k+2} \le 3L_{k+1}\vee L_{k+2} $. Summing the above estimate over all possible scales $\cal I_{\vec a}$, we get that $\|\cal A\|_{s;(a,b)}\prec 1$, which concludes \eqref{keyobs2}.



Now the proof of \eqref{boundaux} involves repeated applications of \eqref{keyobs3} and \eqref{keyobs2} with a carefully chosen order of summations. Without loss of generality, we regard $y$ as the root of the blue tree, and sum over the internal vertices from the leaves of the blue tree to the root. More precisely, we will sum over the vertices according to a partial order $x_{i_1}\preceq x_{i_2} \preceq \cdots \preceq x_{i_\ell}\preceq y$ that is compatible with the blue tree structure---if $x_i$ is a child of $x_j$, then we have $x_i \preceq x_j$. 
By renaming the labels of vertices if necessary, we can assume that the partial order is $x_1\preceq x_2  \preceq\cdots \preceq x_\ell$, so that we will perform the summations according to the order $\sum_{x_\ell}\cdots \sum_{x_2}\sum_{x_1}$. For simplicity of notations, we denote all the blue solid edges appearing in the proof by $\cal A$, including the old edges in  $\wt{\cal G}_{xy}$ and the new edges coming from applications of \eqref{keyobs2}. All these $\cal A$ variables have weak-$(a,b)$ or strong-$(a,b)$ norms bounded by $\OO_\prec(1)$, and their exact expressions may change from one line to another.

For the summation over $x_1$, using \eqref{keyobs3} (if  $x_1$ is not connected with $x$ in the blue tree) or \eqref{keyobs2} (if $x_1$ is connected with $x$ in the blue tree), we can bound $\wt{\cal G}_{xy}$ as
\be\label{reduceaux1}\wt{\cal G}_{xy} \prec W^{-ad/2 }\sum_{k=1}^{\ell_1} \cal G^{(1)}_{xy,k},\ee
where $\cal G^{(1)}_{xy, k}$ are new graphs obtained by replacing the edges connected to $x_1$ with the graphs on the right-hand side of \eqref{keyobs3} or \eqref{keyobs2}, and $\ell_1$ is the number of neighbors of $x_1$ on the black tree. More precisely, we perform the following operations to get these new graphs. 
\begin{itemize}
	\item We get rid of the blue solid and black double-line edges connected with $x_1$. 
	\item If $x$ and $x_1$ are connected through a blue solid edge in $\wt{\cal G}_{xy}$, then in each new graph $x$ is connected to the parent of $x_1$ on the blue tree through a blue solid edge.
	\item Suppose $w_{1},\cdots,w_{\ell_1}$ are the neighbors of $x_1$ on the black tree. Then corresponding to the $k$-th term in $\Gamma(w_{1},\cdots, w_{\ell_1})$, $k=1,\cdots, \ell_1$, the atoms $w_{1},\cdots,w_{k-1},w_{k+1},\cdots, w_{\ell_1}$ are connected to $w_{k}$ through double-line edges in the new graph. 
\end{itemize}
Now it is crucial to observe that each new graph \smash{$\cal G^{(1)}_{xy,k}$} is still doubly connected. 
In \eqref{Graphreduction}, we show the reduction from the first graph to the second one through a summation over $x_1$, where we have omitted the factor $W^{-ad/2}$ from the graphs.
\be  \label{Graphreduction}
\parbox[c]{0.9\linewidth}{\includegraphics[width=\linewidth]{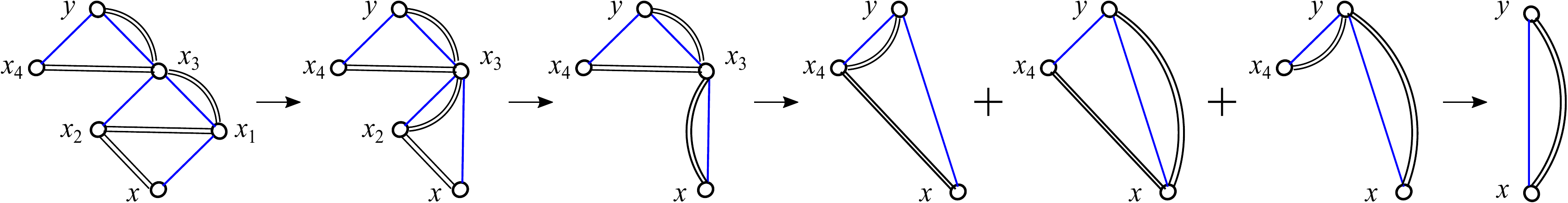}} 
\ee
Similarly, we can bound the summations over atoms $x_2, \cdots , x_{\ell}$ one by one using \eqref{keyobs3} and \eqref{keyobs2}. At each step we gain an extra factor $W^{-ad/2}$ and reduce the graphs into a sum of several new graphs, each of which has one fewer atom and a doubly connected structure. Finally, after summing over all internal atoms, we obtain a graph with atoms $x$ and $y$ only. In this case, the only doubly connected graph is the graph where $x$ and $y$ are connected by a double-line $B_{xy}$ edge and a blue solid edge whose weak-$(a,b)$ or strong-$(a,b)$ norm is bounded by $\OO_\prec(1)$. This concludes \eqref{boundaux}. In \eqref{Graphreduction}, we give an example of the above graph reduction process by summing over the four internal atoms. 
\end{proof}

If $\|G_{xy}(z)-m(z)I_n\|_{w;(a,b)} \prec W^{\e_0} $, then from \eqref{bound 2net1} we immediately get that 
\be\label{bound 2net1 cor}
\left|\cal G_{xy}\right| \prec W^{(n_{xy}-2)\e_0}\cdot  W^{ -  \left(n_{xy}-3\right)d/2} B_{xy} \cal A_{xy},
\ee
for a positive variable $\cal A_{xy}$ satisfying $\|\cal A\|_{w;(1,2)}\prec 1$. 
This follows from the fact that the number of light weights and off-diagonal $G$ edges in $\cal G_{xy}$ is at most $n_{xy} -2$, because by property (ii) of Definition \ref{defnlvl0}, the number of internal atoms in $\cal G_{xy}$ is smaller than the number of waved and $\dashed$ edges at least by 1. 


Deterministic doubly connected graphs satisfy better bounds than Lemma \ref{no dot}, because all edges in the blue net are now (labelled) $\dashed$ edges, whose strong-$(2,2)$ norms are bounded by $\OO_\prec(1)$.

\begin{corollary}\label{lem Rdouble}
Suppose $d\ge 6$. Let $\cal G$ be a deterministic doubly connected normal regular graph without external atoms. Pick any two atoms of $\cal G$ and fix their values as $x , y\in \Z_L^d$. Then the resulting graph $\cal G_{xy}$ satisfies that 
\be\label{bound 2net1B}
\left|\cal G_{xy}\right| \prec W^{ - \left(n_{xy} -4\right) d/2} B_{xy}^2 ,\quad \text{with}\quad n_{xy}:=\ord(\cal G_{xy}).
\ee
This bound also holds for the graph $ {\cal G}^{{\rm abs}}_{xy}$. 
\end{corollary}
\begin{proof} 
This corollary can be proved in the same way as Lemma \ref{no dot}, except that we need to apply Claim \ref{twonet claim} to an auxiliary graph whose blue edges have strong-$(2,2)$ norms bounded by $\OO_\prec(1)$. 
\end{proof}

 We also need another version of Corollary \ref{lem Rdouble}, which will be used in the proof of Lemma \ref{lem V-R wt} in Section \ref{subsec_inf_pf} below.
\begin{corollary}\label{lem Rdouble weak}
Under the assumptions of Corollary \ref{lem Rdouble}, suppose we replace every (labelled) $\dashed$ edge in $\cal G$ between atoms, say $\al$ and $\beta$, with an edge bounded by $\OO_\prec (B_{\al\beta})$. We treat these edges as double-line edges of scaling order $2$ and call the resulting graph $\cal G'$. We pick any two atoms of $\cal G'$ and fix their values as $x , y\in \Z_L^d$. Then the resulting graph  $\cal G'_{xy}$ satisfies the bound 
\be\label{bound 2net weak0}
\left|\cal G'_{xy}\right| \le \frac{W^{ - \left(n'_{xy}-4\right) d/2}}{W^4\langle x-y\rangle^{2d-4-\tau}} ,
\quad \text{with }\ \ n'_{xy}:=\ord(\cal G'_{xy}),
\ee
for any constant $\tau>0$. Furthermore, suppose we replace a double-line edge between atoms, say $ \al_0$ and $ \beta_0$, in $\cal G'_{xy}$ with an edge bounded by {$\OO_\prec (\wt B_{\al_0 \beta_0})$}, where \smash{$\wt B_{ \al_0 \beta_0}:= W^{-4}\langle \al_0 - \beta_0\rangle^{-(d-4)}.$}
We treat this edges as a double-line edge of scaling order $2$ and denote the resulting graph by $\cal G''_{xy}$. 
Then it satisfies the bound 
\be\label{bound 2net weak}
\left|\cal G''_{xy}\right| \le  \frac{W^{ - \left(n''_{xy}-4\right) d/2}}{W^6\langle x-y\rangle^{2d-6-\tau}},\quad \text{with }\ \ n''_{xy}:=\ord(\cal G''_{xy}),
\ee
for any constant $\tau>0$.
\end{corollary}
\begin{proof}
The estimate \eqref{bound 2net weak0} follows from \eqref{bound 2net1B}. The estimate \eqref{bound 2net weak} can be proved in the same way as Lemma \ref{no dot}, except that we need to apply Claim \ref{twonet claim} to an auxiliary graph whose blue edges have strong-$(2,4)$ norms bounded by $\OO_\prec(1)$. 
\end{proof}

\subsection{Proof of Lemma \ref{lem informal}}\label{secpfinformal}

In this subsection, we complete the proof of Lemma \ref{lem informal} using Lemma \ref{no dot}. Recall that by Theorem \ref{thm ptree}, the local law \eqref{locallaw1} holds for $G(z_n, \psi, W, L_n)$, so $\|G(z_n)-m(z_n)I_N\|_{w;(1,1)}\prec 1$. For simplicity of notations, in the following proof we abbreviate $G\equiv G(z_n, f, W, L_n)$. Moreover, in the setting of Lemma \ref{lem informal}, $\otimes$ represents the external atom $\fa$, while $\oplus$ represents the external atom $\fb$.

We first consider the $ \oplus$-recollision graphs in $(\PITk)_{\fa,\fb\fb}$. Take a graph from  $(\PITk)_{\fa,\fb\fb}$, say $\cal G_{\fa\fb}$.  By Definitions \ref{defn incompgenuni} and \ref{def incompgenuni}, it has at least one dotted edge connected with $\oplus$,  a $\dashed$ edge connected with $\otimes$, is of scaling order $\ge 3$, and is doubly connected in the sense of Definition \ref{def 2net}. 
%
%
Now we combine $\oplus$ with the internal atoms that connect to it through dotted edges. 
Then by property (iii) of Definition \ref{defn incompgenuni}, we can write that
\be\label{twoPks}
\cal G_{\fa\fb}=\sum_x \Theta_{\fa x}(\cal G_0)_{x\fb},\quad \text{or}\quad\cal G_{\fa\fb}=\Theta_{\fa \fb}(\cal G_0)_{\fb},\ee
for a graph $(\cal G_0)_{x\fb}$ or $(\cal G_0)_{\fb}$ satisfying the assumptions of Lemma \ref{no dot}. Using \eqref{sumTheta} and  \eqref{bound 2net1 strong}, we can bound  the first case of \eqref{twoPks} as
\be\label{recol case2}\sum_{\fa,\fb} |\cal G_{\fa\fb}|\le \sum_{\fa,\fb,x} \Theta_{\fa x}|(\cal G_0)_{x\fb}|\prec \eta_n^{-1}W^{ -\left(k-3\right)d/2 }\sum_{x,\fb} B_{x\fb}^{3/2}\lesssim L^d_n\frac{W^{-(k-2)d/2 }}{\eta_n}.\ee
The second case of \eqref{twoPks} is easier to bound and we omit the details. 

The proof of \eqref{intro_error} is similar. Recall that by Definitions \ref{defn incompgenuni} and \ref{def incompgenuni}, the graphs in $(\AITn)_{\fa,\fb\fb}$ are of scaling orders $\ge n+1$ and doubly connected in the sense of Definition \ref{def 2net} (i.e. the subgraphs induced on the internal atoms are doubly connected). 
Without loss of generality, we only consider the graphs in \smash{$(\AITn)_{\fa,\fb\fb} $} that are not $\oplus$-recollision graphs, because otherwise they can be bounded in the same way as the graphs in $(\PITn)_{\fa,\fb\fb}$. Pick one such graph $\cal G_{\fa\fb}$ in $(\AITn)_{\fa,\fb\fb}$. It can be written into 
\be\label{Aho 3forms}\cal G_{\fa\fb}= \sum_{x,y, y'}{\Theta}_{ \fa x} (\cal G_0)_{x,yy'} G_{y\fb}\overline G_{y'\fb}, \quad \text{or}\quad  \cal G_{\fa\fb}= \sum_{x,y}{\Theta}_{\fa x} (\cal G_0)_{xy}\Theta_{y\fb},\ee
or some forms obtained by setting some indices of $x,y,y'$ to be equal to each other. Without loss of generality, we only consider the two cases in \eqref{Aho 3forms}, while all the other cases can be dealt with in similar ways. 
By the doubly connected property of $\cal G_{\fa\fb}$, we know that $\cal G_0$ is doubly connected. 
Using \eqref{sumTheta} and \eqref{bound 2net1 strong}, we can bound the second term of \eqref{Aho 3forms} as
$$\sum_{\fa,\fb}|\cal G_{\fa\fb}| \prec \eta_n^{-2} \sum_{x,y} \left|(\cal G_0)_{xy}\right| \prec \eta_n^{-2}W^{-(n-2)d/2}\sum_{x,y}B_{xy}^{3/2} \prec L_n^d \frac{W^{-(n-1)d/2}}{\eta_n^2},$$
where in the second step we used that $\ord((\cal G_0)_{xy})\ge  n+1$. Using \eqref{sumTheta}, \eqref{bound 2net1 strong} and Ward's identity \eqref{eq_Ward0}, we can bound the first term in \eqref{Aho 3forms} as
\begin{align*}
	\sum_{\fa,\fb}|\cal G_{\fa\fb}|& \prec  \sum_{\fa,\fb}\sum_{x,y, y'}{\Theta}_{\fa x} \left|(\cal G_0)_{x,yy'} G_{y\fb}\overline G_{y'\fb}\right|\prec  \eta_n^{-2}W^{-(n-2)d/2}\sum_{x,y,y'} \left(\cal G_0^{\abs}\right)_{x,yy'}\\
	& \prec  \eta_n^{-2}W^{-(n-2)d/2}\sum_{y,y'}B_{ yy'}^{3/2} \prec L^d_n \frac{W^{-(n-1)d/2}}{\eta_n^2}.
\end{align*}
Here in the third step we used that $\sum_x ({\cal G}_0^{\rm{abs}})_{x, y y'}$ is a doubly connected graph satisfying the assumptions of Lemma \ref{no dot} with two fixed atoms $y$ and $ y'$, so that it satisfies \eqref{bound 2net1 strong}. 
Combining the above estimates, we conclude \eqref{intro_error}. 

Finally, \eqref{intro_higherror} can be proved in the same way as \eqref{intro_error} by using that the scaling orders of the graphs in $(\Err_{n,D}')_{\fa,\fb\fb}$ are at least $ D +1$.

\section{Infinite space limit}\label{sec infspace}


In this subsection, we study the infinite space limits of the $\selfs$ $\Seleinf_{2l}$. In particular, we will complete the proofs of Lemma \ref{lem V-R wt} and Lemma \ref{lem FT0}. We write the graphs $\Selek_{2l}$ as 
\be\label{defRnl} \Selek_{2l} \equiv  \Selek_{2l} \left(m(z),S, S^{\pm}(z),\Theta(z)\right), \ee
where the matrices $S$, $S^\pm$ and $\Theta$ depend on $W$, $L$ and $\eta$. We want to remove the $L$ and $\eta$ dependence by taking $L\to \infty$ and $\eta\to 0$. 
More precisely, we define the infinite space limit $\Seleinf_{2l}$ as follows. 

\begin{definition}[Infinite space limits]\label{def infspace}
Given a deterministic regular graph $\cal G\equiv \cal G\left(m(z),S, S^{\pm}(z),\Theta(z)\right)$ with $ z=E+\ii \eta$, we define 
 \be\label{defVnl} \cal G^\infty \equiv  \cal G^\infty \left( m(E), S_\infty ,  S_\infty^{\pm}(E), \Theta_\infty(E)\right), \quad  x\in \Z^d, \ee
in the following way. Recall that we denote $m(E):=m(E+\ii 0_+)$. 
\begin{enumerate}
\item We replace the $s_{\al\beta}$ edges in $\cal G $ with $(S_{\infty})_{\al\beta}$, where (recall \eqref{sxyf})
\be\label{eq_calS}
(S_\infty)_{\al\beta}:= \lim_{L\to \infty} f_{W,L}(\al-\beta).
\ee

\item We replace the $\cal S^{\pm}_{\al\beta}(z)$ edges in $\cal G $ with $(S_\infty^{\pm})_{\al\beta}(E)$, where
\be\label{eq_Sinftypm} S_\infty^{+}(E):= \frac{m^2(E)S_\infty}{1-m^2(E)S_\infty},\quad  S_\infty^{-}(E):=\overline S_\infty^{+}(E).
\ee

\item We replace the $\Theta_{\al\beta}$ edges in $\cal G $ with $(\Theta_\infty)_{\al\beta}$, where   
\be\label{defwhtheta}
(\Theta_\infty)_{\al\beta}:=\lim_{ L \to \infty }  \Theta_{\al\beta}\left(E+\ii \frac{W^2}{ L^2},  L\right). 
\ee


\item For all $m(z)$ in the coefficient (that is, $m(z)$'s that do not appear in $ S^{\pm}(z)$ and $\Theta(z)$ entries), we replace them with $m(E)$.


\item Finally, we let all the internal atoms take values over the whole $\Z^d$. 
\end{enumerate}
Note that $\cal G^\infty$ (if exists) only depends on $E$, $W$ and $\psi$ in Assumption \ref{var profile}, but does not depend on $L$ and $\eta$. 
\end{definition}

We first show that $S^\pm_\infty(E)$ and $\Theta_\infty(E)$ are well-defined, and give some basic estimates on them. The proof of Lemma \ref{lem esthatTheta} will be given in Appendix \ref{appd}. 
\begin{lemma}\label{lem esthatTheta}
For any $x\in \Z^d$ and $E\in (-2+\kappa,2-\kappa)$, $(S^\pm_\infty(E))_{0x}$ and $(\Theta_\infty(E))_{0x}$ exist and we have that
\be\label{S+xy1}
 |(S^+_\infty)_{0x}(E)| \lesssim   W^{-d} \mathbf 1_{|x|\le W^{1+\tau}} + \left(|x|+W\right)^{-D},
\ee
and 
\be\label{Theta-wh1}
|(\Theta_\infty)_{0x}(E)|\le  \frac{1}{W^2\left(|x|+W\right)^{ d-2 -\tau }} ,
\ee
for any constants $\tau, D>0$. Moreover, for any $L\ge W$ and $z=E+\ii\eta$ with $ {W^{2}}/{L^{2-\e}}\le \eta\le 1$ for a small constant $\e>0$, we have that
\be\label{S+xy2}
|(S^+_\infty)_{0x}(E)-  S^{+}_{0x}(z)| \lesssim \eta W^{-d}\mathbf 1_{|x|\le W^{1+\tau}} + \left(|x|+W\right)^{-D},\quad \forall \ x\in \Z_L^d,  
\ee
and 
\be\label{Theta-wh}
  |(\Theta_\infty)_{0x}(E)-\Theta_{0x}(z)|\le \frac{\eta  }{W^4(|x|+W)^{d-4-\tau}}+\left(|x|+W\right)^{-D},\quad \forall \ x\in \Z_L^d,  
\ee
for any constants $\tau,D>0$,
\end{lemma}


We have the following counterpart of Lemma \ref{lem cancelTheta} with $\Theta$ replaced by $\Theta_\infty$. The proof of Lemma \ref{lem cancelTheta2} will be given in Appendix \ref{appd}.
 

\begin{lemma}\label{lem cancelTheta2}
Fix any $E\in (-2+\kappa,2-\kappa)$. Let $g:\Z^d \to \R$ be a symmetric function supported on a box $\cal B_K:=\llbracket -K,K\rrbracket^d$ of scale $K\ge W$. Assume that $g$ satisfies the sum zero property $\sum_{x}g(x)=0.$ Then for any $x_0\in \Z^d$ such that $|x_0| \ge K^{1+c}$ for a constant $c>0$, we have that
$$\Big|\sum_{x} (\Theta_\infty)_{0  x }(E) g(x-x_0)\Big| \le   \left(\sum_{x\in \cal B_K}\frac{x^2}{|x_0|^2}|g(x)|\right)  \frac{1}{W^2 |x_0|^{d-2-\tau}},$$ 
for any constant $\tau >0$.
\end{lemma}

With this lemma, we can obtain the following counterpart of Lemma \ref{lem redundantagain}  for the infinite space limits of the labelled $\dashed$ edges. 

\begin{lemma}\label{lem redundantagain2}
Fix $d\ge 6$.  For any $\Seleinf_{2l}$ satisfying \eqref{two_properties0V}, \eqref{4th_property0V} and \eqref{weaz}, 
we have that
\begin{align}
&\Big| \sum_{\al}(\Theta_\infty)_{x \al} (\Seleinf_{2l} )_{\al y} \Big|\le \frac{W^{-(l-1)d}}{(|x-y|+W)^{d-\tau} },\quad \forall \ x,y \in \Z^d,  \label{redundant again2}
\end{align}
for any constant $\tau>0$. If $\Sele^{\infty}_{2k_1},  \cdots ,  \Sele^{\infty}_{2k_l} $ satisfy \eqref{two_properties0V}, \eqref{4th_property0V} and \eqref{weaz}, then we have that 
\be \label{BRB2} 
\left|\left(\Theta_\infty \Sele^{\infty}_{2k_1}\Theta_\infty  \Sele^{\infty}_{2k_2}\Theta_\infty \cdots \Theta_\infty  \Sele^{\infty}_{2k_l}\Theta_\infty\right)_{xy}\right|\le \frac{W^{-(k-2)d/2}}{W^2(|x-y|+W)^{d-2-\tau}} , \quad \forall \ x,y \in \Z^d,  \ee
 for any constant $\tau>0$, where $k:=\sum_{i=1}^l  2k_i -2(l-1)$. 
 \end{lemma}
 \begin{proof}
As in \eqref{Japanesebracket}, we abbreviate $\langle x-y\rangle:=|x-y|+W$ for our current setting with $L=\infty$. With Lemma \ref{lem cancelTheta2}, the proofs of \eqref{redundant again2} and \eqref{BRB2} are similar to the ones for \eqref{redundant again} and \eqref{BRB}. To prove \eqref{redundant again2}, we decompose the sum over $\al$ according to $\cal I_{far}:=\{\al: |\al-y|\ge \langle x-y\rangle^{1-\tau}\}$ and $\cal I_{near}:=\{\al: |\al-y|< \langle x-y\rangle^{1-\tau}\}$. Using \eqref{Theta-wh1} and \eqref{4th_property0V} (together with $2d-4\ge d+2$ when $d\ge 6$), we can bound that for any constant $\tau>0$, 
  \begin{align} 
 \Big|\sum_{\al\in \cal I_{far}}(\Theta_\infty)_{x \al} (\Seleinf_{2l})_{\al y} \Big| & \le \frac{1}{W^{(l-1)d}}\sum_{\al\in \cal I_{far}}\frac{1}{  \langle x-\al\rangle^{d-2-\tau} } \frac{1}{ \langle \al -y\rangle^{d+2-\tau} } \nonumber\\
 &\le \frac{1}{W^{(l-1)d}\langle x-y\rangle^{(1-\tau)(d-3\tau)}}\sum_{\al\in \cal I_{far}}\frac{1}{  \langle x-\al\rangle^{d-2-\tau} } \frac{1}{ \langle \al -y\rangle^{2+2 \tau} } \nonumber\\
 & \lesssim \frac{1}{W^{(l-1)d}\langle x-y\rangle^{(1-\tau)(d-3\tau)}}.\label{boundIfar2}
\end{align}
For the sum over $\al\in \cal I_{near}$, we decompose it as $(\Seleinf_{2l})_{\al y}= \overline R + \mathring R_{\al y}$ with
$$ \overline R:= \frac{\sum_{\al\in \cal I_{near}}(\Seleinf_{2l})_{\al y}}{|\cal I_{near}|} = -\frac{\sum_{\al\in \cal I_{far}}(\Seleinf_{2l})_{\al y}}{|\cal I_{near}|},$$
where we used \eqref{weaz} in the second step. Then using \eqref{4th_property0V}, we can obtain that
\be\label{bound barR2} |\overline R| \le  \frac{\langle x-y \rangle^{(d+3)\tau}W^2}{W^{(l-1)d}\langle x-y\rangle^{d+2}} ,\quad | \mathring R_{\al y} | \le  \frac{W^{2}}{W^{(l-1)d}\langle \al-y\rangle^{d+2-\tau}} + |\overline R| .\ee
We can bound the term with $\overline R$ as
\begin{align*}
\Big|\sum_{\al \in \cal I_{near}}(\Theta_\infty)_{x \al} \overline R\Big|  &\le \frac{\langle x-y \rangle^{(d+3)\tau}W^2}{W^{(l-1)d}\langle x-y\rangle^{d+2}} \sum_{\al\in \cal I_{near}}(\Theta_\infty)_{x\al} \le  \frac{\langle x-y \rangle^{(d+3)\tau}W^2}{W^{(l-1)d}\langle x-y\rangle^{d+2}} \frac{\langle x-y\rangle^{(1-\tau)\cdot (2+\tau)}}{W^2} \\ &\le  \frac{\langle x-y \rangle^{(d+3)\tau}}{W^{(l-1)d} \langle x-y\rangle^{d}},
\end{align*}
where in the second step we used \eqref{Theta-wh1} to bound $\sum_{\al\in \cal I_{near}}(\Theta_\infty)_{x\al} .$ On the other hand, we use Lemma \ref{lem cancelTheta2} and \eqref{bound barR2} to bound the term with $\mathring R$ as 
\begin{align*}
  \Big|\sum_{\al\in \cal I_{near}}\Theta_{ x \al} \mathring R_{\al y}\Big| & \le  \sum_{\al\in \cal I_{near}} |\al-y|^2  |\mathring R_{\al y}| \cdot \frac{\langle x-y\rangle^\tau}{W^2\langle x-y\rangle^{d}} \\
  & \le \left( \sum_{\al\in \cal I_{near}} \frac{W^{-(l-1)d}}{\langle \al-y\rangle^{d-\tau}} +\frac{ |\overline R| }{W^2} \langle x -y\rangle^{(1-\tau)(d+2)}\right) \frac{\langle x-y\rangle^\tau}{ \langle x-y\rangle^{d}} \lesssim \frac{\langle x-y\rangle^{2\tau}}{W^{(l-1)d}\langle x-y\rangle^{d} }.
\end{align*}
Combining the above two estimates with \eqref{boundIfar2}, we conclude \eqref{redundant again2}.  Finally, 
\eqref{BRB2} follows from \eqref{redundant again2}.
 \end{proof}
 
We will refer to the infinite space limits of the $\dashed$ and labelled $\dashed$ edges as $\Theta_\infty$ and labelled $\Theta_\infty$ edges. The estimates \eqref{Theta-wh1} and \eqref{BRB2} suggest that these two types of edges can be also used in the doubly connected property. 
 
 \begin{definition} [Doubly connected property with $\Theta_\infty$ egdes]\label{def 2net infty}  
We extend the doubly connected property in Definition \ref{def 2net} by including $\Theta_\infty$ and labelled $\Theta_\infty$ edges, 
which can be used either in the black net $\cal B_{black}$ or the blue net $\cal B_{blue}$.
\end{definition}
 
From Corollary \ref{lem Rdouble weak}, we immediately obtain the following result, which explains why \eqref{4th_property0V} holds.
 
\begin{corollary}\label{double inf}
Suppose $d\ge 6$. Let $\cal G$ be a deterministic doubly connected graph without external atoms. Denote its infinite space limit by $\cal G^\infty$. Pick any two atoms of $\cal G^\infty$ and fix their values as $x , y\in \Z^d$. Then the resulting graph $\cal G^\infty_{xy}$ satisfies that for any constant $\tau>0$,
\be \nonumber
\left|\cal G^\infty_{xy}\right| \le   W^{ -  n_{xy} d/2}\frac{W^{2d-4}}{(|x-y|+W)^{2d-4-\tau}}, \quad \text{with} \quad n_{xy}:=\ord(\cal G_{xy}).
\ee    
 \end{corollary}
 
 \begin{proof}
This result is a corollary of \eqref{bound 2net weak0} by taking $L\to \infty$. 
 \end{proof}
 

\subsection{Proof of Lemma \ref{lem V-R wt}}\label{subsec_inf_pf}
Now we prove the following lemma, which implies Lemma \ref{lem V-R wt} as a special case. 
\begin{lemma}\label{lemm V-R}
Fix $d\ge 6$. Suppose we have a sequence of $\selfs$ $\Sele_{2l}$, $4\le 2l \le n-1$, satisfying Definition \ref{collection elements} and properties \eqref{two_properties0V}--\eqref{weaz}. 
Let $\cal G$ be a deterministic graph satisfying the assumptions of Corollary \ref{lem Rdouble}, and let $\cal G^\infty $ be its infinite space limit. Moreover, suppose the labelled $\dashed$ edges in $\cal G_{xy}$ can only be of the form \eqref{eq label} with $4\le 2k_i \le n-1$.  Fix any $L\ge W$ and $z=E+\ii \eta$ with $E\in (-2+\kappa,2-\kappa)$ and $W^{2}/L^{2-\e}\le \eta\le L^{-\e}$ for a small constant $\e>0$. Then for any $x\in \Z_L^d$, we have that
\be\label{V-R1}  
\left|\cal G_{0x}(m(z),S, S^{\pm}(z),\Theta(z))-\cal G^\infty_{0x}( m(E), S_\infty,  S_\infty^{\pm}(E), \Theta_\infty(E))\right| \le W^{-n_0d/2} \frac{\eta W^{2d-6}}{\langle x\rangle^{2d-6-\tau}},
\ee
for any constant $\tau>0$, where $ n_0:=\ord(\cal G_{0x})$. Moreover, \eqref{V-R1} implies that for any constant $\tau>0$,
\be\label{V-R2}  \bigg|\sum_{x\in \Z_L^d} \cal G_{0x}(m(z),S, S^{\pm}(z),\Theta(z))-\sum_{x\in \Z^d} \cal G^\infty_{0x}( m(E), S_\infty,  S_\infty^{\pm}(E), \Theta_\infty(E)) \bigg| \le  \frac{L^\tau \eta }{ W^{(n_0-2)d/2}} .\ee
\end{lemma}
\begin{proof}
Using \eqref{V-R1} when $x\in \Z_L^d$ and applying Corollary \ref{double inf} to $\cal G^\infty_{0x}$ when $x\notin \Z_L^d$, we obtain that
\begin{align*}
 &\bigg|\sum_{x\in \Z_L^d}  \cal G_{0x}(m(z),S, S^{\pm}(z),\Theta(z))-\sum_{x\in \Z^d} \cal G^\infty_{0x}( m(E), S_\infty,  S_\infty^{\pm}(E), \Theta_\infty(E)) \bigg| \\
& \le \sum_{\|x\|_\infty \le L/2} \frac{\eta W^{2d-6}}{W^{n_0d/2}(|x|+W)^{2d-6-\tau}}  + \sum_{\|x\|_\infty > L/2}\frac{W^{2d-4 } }{W^{n_0d/2}|x|^{2d-4-\tau}}  \\
& \lesssim L^\tau \left(\eta + \frac{W^2}{L^2}\right)W^{-(n_0-2)d/2}  \lesssim L^\tau \eta W^{-(n_0-2)d/2}  ,
\end{align*}
which concludes \eqref{V-R2}. It remains to prove \eqref{V-R1}. 



First, using $|m(z)-m(E)|=\OO(\eta)$, we observe that replacing $m(z)$ in the coefficient with $m(E)$ leads to an extra factor $\eta$: 
$$\left|\cal G_{0x}(m(z),S , S^{\pm}(z),\Theta(z))-\cal G_{0x}( m(E), S , S^{\pm}(z),\Theta(z))\right| \le  \frac{ \eta W^{2d-4}}{W^{n_0d/2}\langle x \rangle^{2d-4-\tau}},$$
 for any small constant $\tau>0$. It remains to prove that for $x\in \Z_L^d$,
\be  \label{V-R3}
\left|\cal G_{0x}(m(E),S, S^{\pm}(z),\Theta(z))-\cal G^\infty_{0x}( m(E), S_\infty,  S_\infty^{\pm}(E), \Theta_\infty(E))\right| \le \frac{\eta W^{2d-6}}{W^{n_0d/2} \langle x\rangle^{2d-6-\tau}}.
\ee
For this purpose, we define a new graph $ {\cal G}^{[\eta]}_{0x}(m(E),S^{[\eta]}, S^{\pm,[\eta]}(z),\Theta^{[\eta]}(z))$ obtained by replacing the $S$, $S^\pm(z)$ and $\Theta(z)$ edges defined on $\Z_L^d$ with $S^{[\eta]}$, $S^{\pm,[\eta]}(z)$ and $\Theta^{[\eta]}(z)$ edges defined on $\Z^d$, where 
$$ S^{[\eta]}_{\al\beta}:=S_{ \al \beta } \mathbf 1_{|\al-\beta|\le L^\tau W \eta^{-1/2}},\quad S^{\pm, [\eta]}_{\al\beta}(z):=S^{\pm}_{ \al \beta } (z)\mathbf 1_{|\al-\beta|\le L^\tau W \eta^{-1/2}}, $$ 
and
$$ \Theta^{[\eta]}_{\al\beta}(z):=\Theta_{ \al \beta }(z)\mathbf 1_{|\al-\beta|\le L^\tau W \eta^{-1/2}},$$
for $\al,\beta\in \Z^d$. Note that for a sufficiently small $ \tau \in (0,\e/4)$, we have $L^\tau W \eta^{-1/2} \le L^{1-\e/4}$. Hence in order for {$\cal G_{0x}^{[\eta]}$} to be nonzero, any atom $\al$ in it must satisfy $|\al |\le C_0 L^\tau W\eta^{-1/2}\ll L$ for a constant $C_0>0$. By \eqref{subpoly}, \eqref{thetaxy} and \eqref{S+xy}, we have that for any constant $D>0$,
$$ \max_{x\in \Z_L^d}\left|\cal G_{0x}(m(E),S, S^{\pm}(z),\Theta(z))-{\cal G}^{[\eta]}_{0x}(m(E),S^{[\eta]}, S^{\pm,[\eta]}(z),\Theta^{[\eta]}(z))\right| \le L^{-D}.$$
Hence to prove \eqref{V-R3}, it remains to show that for $x\in \Z_L^d$,
\be\label{V-R4}
\left|{\cal G}^{[\eta]}_{0x}(m(E),S^{[\eta]}, S^{\pm,[\eta]}(z),\Theta^{[\eta]}(z))-\cal G^\infty_{0x}( m(E), S_\infty,  S_\infty^{\pm}(E), \Theta_\infty(E))\right|  \le \frac{\eta W^{2d-6}}{W^{n_0d/2} \langle x\rangle^{2d-6-\tau}}.
\ee
By Corollary \ref{double inf}, we have that for $|x| > C_0 L^\tau W \eta^{-1/2}$,
 \begin{align*}
& \left|{\cal G}^{[\eta]}_{0x}(m(E),S^{[\eta]}, S^{\pm,[\eta]}(z),\Theta^{[\eta]}(z))-\cal G^\infty_{0x}( m(E), S_\infty,  S_\infty^{\pm}(E), \Theta_\infty(E))\right| \\
 =&\left| \cal G^\infty_{0x}( m(E), S_\infty,  S_\infty^{\pm}(E), \Theta_\infty(E))\right| \le \frac{W^{2d-4}}{W^{ n_0d/2}\langle x\rangle^{2d-4-\tau}}  \le \frac{\eta W^{2d-6}}{W^{ n_0 d/2}\langle x\rangle^{2d-6 -\tau}} .
\end{align*}
  

It remains to prove \eqref{V-R4} for $|x| \le C_0 W^{1+\tau}\eta^{-1/2}$. We will replace the $S^{[\eta]}$, $S^{\pm,[\eta]}$, $\Theta^{[\eta]}$ and labelled $\Theta^{[\eta]}$ edges in ${\cal G}^{[\eta]}_{0x}$ with the $S_\infty$, $S_\infty^{\pm}$, $\Theta_\infty$ and labelled $\Theta_\infty$ edges one by one, and control the error of each replacement using the estimates \eqref{subpoly}, \eqref{S+xy2} and \eqref{Theta-wh}. We remark that when dealing with a labelled \smash{$\Theta^{[\eta]}$} edge, we will replace a $\self$ \smash{$\Sele_{2l}^{[\eta]}$}, $4\le 2l\le n-1$, with \smash{$\Sele_{2l}^{\infty}$} as a whole, and the estimate \eqref{3rd_property0 diff} will be used to bound the difference. For simplicity, in the following proof we use the notations
$$  {\cal G}^{[\eta]}_{0x}( S^{[\eta]}, S^{\pm,[\eta]},\Theta^{[\eta]}),\quad   \cal G^\infty_{0x}( S_\infty,  S_\infty^{\pm}(E), \Theta_\infty(E)),$$
with the understanding that the arguments $\Theta^{[\eta]}$ and $\Theta_\infty$ represent both $\dashed$ and labelled $\dashed$ edges. 
First, using \eqref{subpoly}, it is easy to see that replacing any $S^{[\eta]}$ edge with a $S_\infty$ edge gives an error of order $\OO(L^{-D})$. Second, using \eqref{S+xy2}, it is easy to show that replacing any $S^{\pm,[\eta]}(z)$ edge with a $S_\infty^{\pm}(E)$ edge leads to an extra factor $\eta$. Hence after replacing all $S^{[\eta]}$ and $S^{\pm,[\eta]}(z)$ edges with $S_\infty$ and $S_\infty^{\pm}$ edges, we get that 
$$\left|{\cal G}^{[\eta]}_{0x}(S^{[\eta]}, S^{\pm,[\eta]}(z),\Theta^{[\eta]}(z))-{\cal G}^{[\eta]}_{0x}( S_\infty,  S^{\pm}_\infty(E),\Theta^{[\eta]}(z))\right| \le \frac{ \eta W^{ 2d-4}}{ W^{ n_0d/2 }\langle x \rangle^{2d-4-\tau}}.$$
Here as a convention, we still add the superscript $[\eta]$ to the graph after the replacements, but its arguments are different from the original graph.
It remains to show that replacing the $\Theta^{[\eta]}$ and labelled $\Theta^{[\eta]}$ edges with  $\Theta_{\infty}$ and labelled $\Theta_{\infty}$ edges leads to a small enough error:
 \be\label{V-R3.5}
 \left|{\cal G}^{[\eta]}_{0x}(S_\infty, S^{\pm}_\infty(E),\Theta^{[\eta]}(z))-{\cal G}^\infty_{0x}( S_\infty,  S^{\pm}_\infty(E),\Theta_\infty (E))\right| \le  \frac{\eta W^{2d-6}}{W^{n_0d/2}\langle x\rangle^{2d-6-\tau}}.
 \ee
 Combining the above two estimates, we conclude \eqref{V-R4}. 
 

It remains to prove \eqref{V-R3.5}. 
Notice that ${\cal G}^{[\eta]}_{0x} (S_\infty, S^{\pm}_\infty(E),\Theta^{[\eta]}(z))-\cal G^{\infty}_{0x}
(S_\infty, S^{\pm}_\infty(E), \Theta_\infty (E))$ can be written into a sum of $\OO(1)$ many graphs, each of which is of scaling order $n_0$ and has a doubly connected structure consisting of $\Theta^{[\eta]}$ and $\Theta_\infty$ edges, labelled $\Theta^{[\eta]}$ and $\Theta_\infty$ edges, and one edge of the form $(\Theta^{[\eta]}-\Theta_\infty)_{\al\beta}$ or  
\be\label{diff label}
\left[\Theta^{[\eta]} \Sele^{[\eta]}_{2k_1}\Theta^{[\eta]}  \Sele^{[\eta]}_{2k_2}\Theta^{[\eta]} \cdots \Theta^{[\eta]}  \Sele^{[\eta]}_{2k_l}\Theta^{[\eta]}-\Theta_\infty \Sele^{\infty}_{2k_1}\Theta_\infty  \Sele^{\infty}_{2k_2}\Theta_\infty \cdots \Theta_\infty  \Sele^{\infty}_{2k_l}\Theta_\infty \right]_{\al\beta} ,\ee
with $4\le 2k_i \le n-1$, $1\le i \le l$, and scaling order $2s:= \sum_{i=1}^l 2k_i-2(l-1)$. 
Let $(\cal G_\omega)_{0x}$ be one of these graphs. We claim that 
\be\label{G-Ginfty2}   \left|(\cal G_\omega)_{0x}\right| \le \frac{\eta W^{2d-6}}{W^{n_0d/2}\langle x\rangle^{2d-6-\tau}} \quad \text{for}\quad |x| \le C_0 W^{1+\tau}\eta^{-1/2}.
\ee
With \eqref{G-Ginfty2}, we immediately conclude \eqref{V-R3.5}.

Finally we prove \eqref{G-Ginfty2}. If $(\cal G_\omega)_{0x}$ contains a  $(\Theta^{[\eta]}-\Theta_\infty)_{\al\beta}$ edge, then by \eqref{Theta-wh1} and \eqref{Theta-wh} we obtain that for any constant $\tau>0$,
\begin{align}\label{Deltatheta1}
\left|(\Theta^{[\eta]}-\Theta_\infty)_{\al\beta}\right| \le \frac{\eta \mathbf 1_{|\al-\beta|\le L^\tau W \eta^{-1/2}}}{W^4 (|\al-\beta|+W)^{d-4-\tau}}  + \frac{\mathbf 1_{|\al-\beta|> L^\tau W \eta^{-1/2}}}{W^2(|\al-\beta|+W)^{d-2-\tau}} \le  \frac{\eta }{W^4(|\al-\beta|+W)^{d-4-\tau}} ,
\end{align}
where in the second step we used that $W^2/| \al-\beta|^{2}\le \eta$ for $|\al-\beta|>L^\tau W \eta^{-1/2}$. Thus we can write that \smash{$\cal G_\omega= \eta  \wt{\cal G}_{\omega}$} for a graph \smash{$\wt{\cal G}_{\omega}$} which has a doubly connected structure consisting of $\Theta^{[\eta]}$ and $\Theta_\infty$ edges, labelled $\Theta^{[\eta]}$ and $\Theta_\infty$ edges, and one special edge between $\al$ and $\beta$ bounded by $\OO_\prec(\wt B_{\al\beta}).$ Then applying \eqref{bound 2net weak} (in the $L\to\infty$ case), we obtain that 
$$ \left|(\wt{\cal G}_\omega)_{0x}\right|\le  \frac{W^{-(n_0-4)d/2}}{W^6(|x|+W)^{2d-6-\tau}} 
,$$
which implies \eqref{G-Ginfty2}. 
On the other hand, suppose $\cal G_\omega$ contains an edge of the form \eqref{diff label}.
Following the same argument as above, in order to show \eqref{G-Ginfty2}, it suffices to prove that for any constant $\tau>0$,
\be\label{G-Ginfty3}
|(\ref{diff label})|\le \frac{\eta W^{-(s-1)d}}{W^4(|\al-\beta|+W)^{d-4-\tau}}.
\ee
We prove this estimate by replacing the $\Theta^{[\eta]}$ and $\Sele_{2k_i}^{[\eta]}$ entries one by one, and bounding the error of each replacement using \eqref{Theta-wh} and \eqref{3rd_property0 diff}. 
First, with \eqref{Theta-wh} and \eqref{redundant again}, we get that 
\begin{align*}
& \Big|\left[(\Theta_\infty-\Theta^{[\eta]})   \Sele^{[\eta]}_{2k_1}\Theta^{[\eta]}  \Sele^{[\eta]}_{2k_2}\Theta^{[\eta]} \cdots \Theta^{[\eta]}  \Sele^{[\eta]}_{2k_l}\Theta^{[\eta]}\right]_{\al\beta}\Big| \\
& \le \sum_{\al_1,\cdots, \al_{l}} \frac{\eta }{W^4(|\al-\al_1|+W)^{d-4-\tau}}\prod_{i=1}^{l-1} \frac{W^{-(k_i-1)d}}{(|\al_i-\al_{i+1}|+W)^{d-\tau}}\frac{W^{-(k_l-1)d}}{(|\al_l-\beta|+W)^{d-\tau} } \\
&\lesssim \frac{\eta W^{-(s-1)d}}{W^4 (| \al-\beta|+W)^{d-4-(l+1)\tau}}.
\end{align*}
Second, using \eqref{Theta-wh1}, \eqref{redundant again} and \eqref{3rd_property0 diff} for $\Sele_{2k_1}$, we get that 
\begin{align*}
&\Big|\left[\Theta_\infty \left( \Sele^{[\eta]}_{2k_1}- \Sele^{\infty}_{2k_1}\right)\Theta^{[\eta]}  \Sele^{[\eta]}_{2k_2}\Theta^{[\eta]} \cdots \Theta^{[\eta]}  \Sele^{[\eta]}_{2k_l}\Theta^{[\eta]}\right]_{\al\beta}\Big|\\
& \le \sum_{\al_1,\cdots,\al_l,\beta_1}\frac{1}{W^2(| \al-\al_1|+W)^{d-2-\tau}} \frac{\eta W^{2d-6}}{W^{k_1d}(| \al_1-\beta_1|+W)^{2d-6-\tau}} \frac{1}{W^2(|\beta_1-\al_2|+W)^{d-2-\tau}} \\
&\quad \times \prod_{i=2}^{l-1} \frac{W^{-(k_i-1)d}}{(| \al_i-\al_{i+1}|+W)^{d-\tau}}\frac{W^{-(k_l-1)d}}{(| \al_l-\beta|+W)^{d-\tau} }\\
&\lesssim \frac{\eta W^{-(s-1)d}}{W^4( | \al-\beta|+W)^{d-4-(l+2)\tau}}.
\end{align*}
Continuing the above process, we can replace $\Theta^{[\eta]}$ with $\Theta_\infty$ and $\Sele^{[\eta]}_{2k_i}$ with $\Sele^{\infty}_{2k_i}$ one by one. Moreover, using \eqref{3rd_property0 diff}, \eqref{redundant again}, \eqref{Theta-wh1}, \eqref{Theta-wh} and \eqref{redundant again2} at each step, we can show that each replacement gives an error at most 
$$\frac{\eta W^{-(s-1)d}}{W^4( | \al-\beta|+W)^{d-4-(l+2)\tau}}.$$ This implies \eqref{G-Ginfty3} since $\tau$ is arbitrarily small, and hence concludes \eqref{G-Ginfty2}. 
\end{proof}

Now we can complete the proof of Lemma \ref{lem V-R wt}. 
\begin{proof}[Proof of Lemma \ref{lem V-R wt}]
As given by Definition \ref{def incompgenuni}, $\Sele_{n}$ is a sum of $\OO(1)$ many deterministic doubly connected graphs satisfying the assumptions of Lemma \ref{lemm V-R}. Hence we immediately conclude Lemma \ref{lem V-R wt} using Corollary \ref{lem Rdouble}, Corollary \ref{double inf} and Lemma \ref{lemm V-R}.  
\end{proof}

%

\subsection{Proof of Lemma \ref{lem FT0}}\label{sec pf FT0}
 
Finally, in this subsection we give the full proof of Lemma \ref{lem FT0}. Recall the matrices $\wt S$, $\wt S^{\pm}$ and $\wt\Theta$ defined in \eqref{choicefS}. 
The following claim shows that $\wt S$, $\wt S^\pm$ and $\wt\Theta$ are close to $ S$, $ S^\pm$ and $ \Theta$. Its proof will be given in Section \ref{appd}.
\begin{claim}\label{claim wtS-S}
	Under the assumptions of Lemma \ref{lem FT0}, fix any $L\ge W$ and $z=E+\ii \eta$ with $E\in (-2+\kappa,2-\kappa)$ and $W^{2}/L^{2-\e}\le \eta\le L^{-\e}$ for a small constant $\e>0$. For any $x\in \Z_L^d$, we have that  
	\be\label{S0xy3}
	|\wt S_{0x} - S_{0x} | \lesssim \frac{W^2}{L^2}\frac{1}{W^d}\mathbf 1_{|x|\le W^{1+\tau}} +  \langle x\rangle^{-D} , 
	\ee
	\be\label{S+xy3}
	|\wt S^+_{0x}(z)- \cal S^{+}_{0x}(z)| \lesssim \frac{W^2}{L^2}\frac{1}{W^d }\mathbf 1_{|x|\le W^{1+\tau}} +  \langle x\rangle^{-D} , 
	\ee
	\be\label{Theta-wh3}
	|\wt \Theta_{0x}(z)-\Theta_{0x}(z)|\le \frac{W^2}{L^2}\frac{1}{W^4(|x|+W)^{d-4-\tau}}+ \langle x\rangle^{-D} ,
	\ee
 for any constants $\tau,D>0$.
\end{claim}

Corresponding to the $\selfs$ in Definition \ref{collection elements}, we define $\wt\Sele_{2l}$, $4\le 2l\le n$, as the sum of graphs obtained by replacing the $ S$, $ S^\pm$ and $ \Theta$ edges in $\Sele_{2l}$ with the $\wt S$, $\wt S^\pm$ and $\wt\Theta$ edges. With Claim \ref{claim wtS-S}, we can show that $\wt\Sele_{2l}$ is sufficiently close to $\Sele_{2l}$. 
\begin{claim}\label{claim_wtself}
Under the assumptions of Lemma \ref{lem FT0}, fix any $L\ge W$ and $z=E+\ii \eta$ with $E\in (-2+\kappa,2-\kappa)$ and $W^{2}/L^{2-\e}\le \eta\le L^{-\e}$ for a small constant $\e>0$. Then for any $4\le 2l \le n$, we have that 
\be\label{eq_wtS-S0}
\left| (\wt \Selek_{2l})_{0x}(z)-(\Sele_{2l})_{0x}(z) \right| \le W^{-ld} \frac{W^2}{L^2} \frac{W^{2d-6}}{\langle x\rangle^{2d-6-\tau}}, \quad \forall \  x\in \Z_L^d  , 
\ee
and
\be\label{eq_wtS-S1}
\Big|\sum_{x\in \Z_L^d} (\wt\Selek_{2l})_{0x}(z)\Big|   \le  L^\tau W^{-(l-1)d} \frac{W^2}{L^2} ,  
\ee
for any constant $\tau>0$.
\end{claim} 
\begin{proof}
	We prove \eqref{eq_wtS-S0} and \eqref{eq_wtS-S1} by induction on $l$. First, we trivially have $\wt\Sele_2=\Sele_2=0$. Now suppose we have shown that \eqref{eq_wtS-S0} and \eqref{eq_wtS-S1} hold for $\Sele_{2l}$ with $l\le k-1$. Then with this induction hypothesis and the estimates \eqref{S0xy3}--\eqref{Theta-wh3}, using the same argument as in the proof of Lemma \ref{lemm V-R},  we can prove that \eqref{eq_wtS-S0} and \eqref{eq_wtS-S1} hold for $\Sele_{2k}$.
\end{proof}

Claim \ref{claim_wtself} shows that $\sum_{\fa}(\Sele_{n})_{0\fa }$ has the same infinite space limit as $\sum_{\fa}(\wt\Sele_{n})_{0\fa}$. Hence to prove Lemma \ref{lem FT0}, we first calculate the sum \smash{$\sum_\fa (\wt \Sele_{n})_{0\fa }$} for a finite $L$, and then take $L\to \infty$. In the following proof, we choose $z \equiv z(L)=E+ \ii W^{2}/L^{2-\e}$ for a small enough constant $\e>0$. 
Now we express $ \sum_{\fa} (\wt\Sele_{n})_{0\fa }$ using the Fourier series \eqref{choicef0} and \eqref{FTtheta}. 
For simplicity of notations, 
we denote the $\wt S$, $\wt S^\pm$, $\wt\Theta$ and labelled $\wt\Theta$ edge in a unified way as
$$\wt S^{(a)}_{xy}(z) =  \frac1{L^d}\sum_{p\in \mathbb T_L^d} \psi_{a}(Wp, z)e^{\ii p\cdot (x-y)},$$
for
$$ a\in \{\emptyset, \pm, \Theta\} \cup \Big\{ (k; 2k_1, \cdots, 2k_l): l\ge 1, \max_{i}(2k_i)\le n-1, k=\sum_{i=1}^l 2 k_i -2(l-1)  \Big\},$$
where $ \wt S^{(\emptyset)} :=\wt S $, $ \wt S^{(\pm)} (z):=\wt S^\pm(z) $, $\wt S^{(\Theta)} (z):=\wt\Theta(z)$ and $\wt S^{(k; 2k_1, \cdots, 2k_l)}$ corresponds to a labelled $\wt\Theta$ edge as in \eqref{eq label} (with $\Theta$ and $\Sele_{2k_i}$ replaced by $\wt\Theta$ and $\wt\Sele_{2k_i}$). The functions $\psi_a(Wp,z)$ are given by \eqref{choicef0} and \eqref{FTtheta} for $a\in \{\emptyset, \pm, \Theta\}$, and 
we have 
\be\label{FT_finitespace}
\psi_{ (k; 2k_1, \cdots, 2k_l)}(Wp,z) =  \psi_{\Theta}(Wp, z)^{l+1}\prod_{i=1}^l \psi_{\Sele_{2k_i}}(Wp , z),  
\ee
where \smash{$\psi_{\wt\Sele_{2k_i}}(Wp , z)$} is the Fourier transform of $(\wt\Sele_{2k_i})_{0x}$ (which can be calculated inductively with respect to $2k_i$). 
For each edge $e$ in \smash{$ (\wt\Sele_{n})_{0\fa }$}, we assign a label $a_e$ and a momentum $p_e$ to it. 	

For a vertex $x$ in the graph, suppose that it is connected with $k$ edges with labels $a_i$, $1\le i \le k$. Then summing over $x\in \Z_L^d$, we get that
\begin{align*}
\sum_ {x\in \Z_L^d} \prod_{i=1}^k \wt S^{(a_i)} _{xy_i} & =\sum_{x\in \Z_L^d} \prod_{i=1}^k \frac{1}{L^d}\sum_{p_i\in \mathbb T_L^d}\psi_{a_i} (Wp_i,z)e^{\ii p_i \cdot (x-y_i)}  \\
& = \frac{1}{L^{(k-1)d}}\sum_{\substack{p_1,\cdots, p_k\in \mathbb T_L^d: \\ p_1+p_2+\cdots+p_k=0 \mod 2\pi}}\prod_{i=1}^k\psi_{a_i}(Wp_i,z)e^{-\ii p_i \cdot y_i}  ,
\end{align*}
where for a vector $v\in \R^d$, we use $v=0 \mod 2\pi$ to mean that $v_i=0 \mod 2\pi$ for all $1\le i \le d$. Note that $p_1+\cdots+p_k=0 \mod 2\pi$ is a momentum conservation condition. The momentum $p_i$ associated with $y_i$ will be used later in the summation over $y_i$, and so on. Let $\cal G$ denote the graphs in $\wt\Sele_{n}$, $c(\cal G, z)$ be the coefficient of $\cal G$, $p_e$ denote the momentum associated with each edge $e$ in $\cal G$, $\Xi_L$ be a subset of $(\mathbb T_L^{d})^{n_e}$ given by the constraint that the total momentum at each vertex is equal to 0 modulo $2\pi$, where $n_e\equiv n_e(\cal G)$ is the total number of edges in $\cal G$.  Then after summing over all indices in $(\wt\Sele_{n})_{0\fa}$, we obtain that 
\be\label{eq_discrete} 
\sum_{\fa} (\wt\Sele_{n})_{0\fa }(m(z),\psi,W,L)=\frac{1}{L^{(n-2) d/2 }}\sum_{\cal G}c(\cal G, z)\sum_{ \{p_e\} \in \Xi_{L}}\prod_{e}\psi_{a_e}(Wp_e,z) . 
\ee
Taking $L\to \infty$, \eqref{eq_discrete} gives that
\begin{align}\label{eq_discrete2} \sum_{\fa} (\Seleinf_{n})_{0\fa }&=\frac{1}{(2\pi)^{(n-2) d/2}}\sum_{\cal G}c(\cal G, E)\int_{ \{p_e\} \in \Xi}\prod_{e}\psi_{a_e}(Wp_e, E) \dd p_e ,
\end{align}
where $\Xi$ is a union of hyperplanes in the torus $(-\pi,\pi]^{d n_e}$ given by the constraint that the total momentum at each vertex is equal to zero modulo $2\pi$. To give a more rigorous proof of \eqref{eq_discrete2}, we need to deal with the singularities of $\psi_{a_e}(Wp_e, E)$ at $p_e=0$ for (labelled) $\dashed$ edges. We introduce an infrared cutoff on these edges, i.e. $\psi_{a_e,\e}(Wp_e,z):=\psi_{a_e}(Wp_e,z) \mathbf 1_{|W p_e|\le \e}$. Then we define \smash{$\sum_{\fa} (\wt\Sele_{n,\e})_{0\fa }$} by replacing  $\psi_{a_e}(Wp_e,z)$ with $\psi_{a_e,\e}(Wp_e,z)$ on the right-hand side of \eqref{eq_discrete}. Since $\psi_{a_e,\e}(Wp_e,z)$'s are nonsingular, taking $L\to \infty$ we readily get that 
$$\sum_{\fa} (\Seleinf_{n,\e})_{0\fa }=\frac{1}{(2\pi)^{(n-2) d/2}}\sum_{\cal G}c(\cal G, E)\int_{ \{p_e\} \in \Xi}\prod_{e}\psi_{a_e,\e}(Wp_e, E) \dd p_e . $$
Then taking $\e\to 0$, we can show that this equation converges to \eqref{eq_discrete2}, which again follows from the doubly connected property of the graphs $\cal G$ by using a similar argument as in the proof of Lemma \ref{lemm V-R}. We omit the details. 


Now applying a change of variables $q_e=Wp_e$ to \eqref{eq_discrete2}, we get that 
$$\sum_{\fa} (\Seleinf_{n})_{0\fa }=\frac{1}{(2\pi)^{(n-2) d/2}W^{(n-2) d/2}}\sum_{\cal G}c(\cal G, E)\int_{W\Xi}\prod_{e}\psi_{ a_e}(q_e, E) \dd q_e .$$
Since $\psi$ is compactly supported in the assumption of Theorem \ref{main thm}, we have that 
\be\label{psi_compact}\frac{1}{(2\pi)^{(n-2) d/2}W^{(n-2) d/2}}\int_{W\Xi}\prod_{e}\psi_{ a_e}(q_e, E) \dd q_e =\frac{1}{(2\pi)^{(n-2) d/2}W^{(n-2) d/2}}\int_{\wt\Xi }\prod_{e}\psi_{ a_e}(q_e, E) \dd q_e ,\ee
where $\wt\Xi$  is a union of hyperplanes in $(\R^d)^{n_e}$ given by the constraint that the total momentum at each vertex is equal to 0 (without modulo $2\pi$).
Combining the above two equations, 	we obtain  \eqref{eq FT0} by renaming
$$\Sint_{n}(m(E), \psi) :=\frac{1}{(2\pi)^{(n-2) d/2} }\sum_{\cal G}c(\cal G,E)\int_{\wt\Xi}\prod_{e}\psi_{ a_e}(q_e, E) \dd q_e .$$

\begin{remark}\label{rem FT0}
The equation \eqref{psi_compact} is the only place where the compactly supported condition of $\psi$ is used. If we only assume that $\psi$ is a Schwartz function, then equation \eqref{psi_compact} does not hold exactly, but with an additional error of order $\OO(W^{-D})$ for any large constant $D>0$. Such a small error does not affect our proofs, and we refer the reader to Section \ref{sec pf0} below for the necessary modifications of the proof in the setting of Theorem \ref{main thm0}. 
\end{remark}

\section{Proof of the main results}\label{sec pf0} 


In this section, we complete the proofs of the main results---Theorem \ref{comp_delocal}, Theorem \ref{main thm0}, Theorem \ref{main thm2} and Corollary \ref{main_cor}. First, we prove Theorem \ref{comp_delocal} using the local law \eqref{locallaw}. In fact, we will prove a slightly stronger result in Lemma \ref{lemm_comp_delocal}. For any constants $M, K>1$, we define the following random subset of indices that contains $B_{\gamma, K,\ell}$ as a subset:
$$
\wt{\mathcal B}_{ M, K,\ell} :=\left\{\alpha: \lambda_\alpha \in (-2+\kappa, 2-\kappa) \text{ so that  } \min_{x_0} \sum_x | u_\alpha(x)|^2 \left(\frac{\|x-x_0||_L}{\ell}+1\right)^M \le K\right\}. 
$$
Note that this subset contains all indices associated with bulk eigenvectors that are localized \emph{super-polynomially} in balls of radius $\OO(\ell)$. 

\begin{lemma} \label{lemm_comp_delocal}
	Suppose the assumptions of Theorem \ref{comp_delocal} hold. Fix any constants $c_0>0$ and $M,K>1$. For any $W\le  \ell\le L^{1 - c_0}$, we have that
	\be\label{uinf0} 
	{|\wt{\mathcal B}_{M,K,\ell}|}/{N}\prec \big(  {\ell^{\frac{M}{M+4}}}/{L^{\frac{M-d}{M+4}}}\big)^2 + W^{-d/2}.
	\ee
\end{lemma}

Using Theorem \ref{main thm0} and Lemma \ref{lemm_comp_delocal}, we can easily conclude Theorem \ref{comp_delocal}. 

\begin{proof}[Proof of Theorem \ref{comp_delocal}]
	Since ${\mathcal B}_{\gamma,K,\ell}\subset \wt{\mathcal B}_{M, C_M K,\ell}$ for arbitrarily large $M$ and a constant $C_M>0$, we obtain from \eqref{uinf0} that
	$$ {|\mathcal B_{\gamma,K,\ell}|}/{N}\prec  \left( {\ell}/{L}\right)^2 + W^{-d/2}.$$
	This implies \eqref{uinf_locallength} by Definition \ref{stoch_domination}. To prove \eqref{uinf}, using the spectral decomposition of $G(z)$, we obtain from \eqref{locallaw} that 
	$$| u_\al(x)|^2 \le \eta \im G_{xx}(\lambda_\al +\ii \eta) \prec  W^{2}/L^{2-\e}, $$
	for $\eta = W^{2}/L^{2-\e}$.  Since $\e>0$ is arbitrarily small, we conclude \eqref{uinf}.
\end{proof}

Now we give the proof of Lemma \ref{lemm_comp_delocal} based on Theorem \ref{main thm0}.

\begin{proof}[Proof of Lemma \ref{lemm_comp_delocal}]
	We define the following characteristic function $P_{x,l}$ projecting onto the complement of the $\ell$-neighborhood of $x$: $P_{x,\ell}(y):=\mathbf 1(|y-x|\ge \ell).$ 	Define the following random subset of indices  
	$$\mathcal A_{\delta, \ell}:=\Big\{ \alpha: \lambda_\alpha \in (-2+\kappa, 2-\kappa), \sum_x |u_\alpha(x)|\|P_{x,\ell}\bu_{\al}\|\le \delta \Big\} , $$
	where $\delta \equiv \delta(L)$ may depend on $L$ and is not necessarily a constant. Using Theorem \ref{main thm0}, we get that 
	$$\max_{x,y}|G_{xy}(z)-m(z)\delta_{xy}|\prec W^{-d/2},\quad  \quad \max_{x\in \Z_L^d}\ \eta \sum_{y:|y-x|\le \ell}|G_{xy}|^2 \prec \eta {\ell^2}/{W^2} \le  \ell^2/L^{2-\e},$$
	if we take $\eta = W^{2}/L^{2-\e}$. 	With these estimates, following the proof of Proposition 7.1 of \cite{delocal}, we can obtain that 
	\be\label{uinf1}  {|\mathcal A_{\delta, \ell }| }/{N} \le C\sqrt{\delta} + \OO_\prec \left( \ell^2/L^{2-\e}+W^{-d/2}\right).\ee
	Next we use a similar argument as in the proof of \cite[Corollary 3.4]{ErdKno2011} to derive the estimate \eqref{uinf0} from \eqref{uinf1}. 
	Let $\wt \ell:= \ell L^{c_1}$ for a constant $c_1\in (0,c_0)$. If $\al\in  \wt{\mathcal B}_{M,K,\ell} $, then for any $x_0\in \Z_L^d$ we have that 
	\begin{align*}
		\left(\sum_x |u_\alpha(x)|\|P_{x,\wt\ell}\bu_{\al}\|\right)^2& \le \left[\sum_x |u_\alpha(x)|^2 \left(\frac{\|x-x_0||_L}{\ell}+1\right)^M\right] \left[\sum_x \|P_{x,\wt\ell}\bu_{\al}\|^2  \left(\frac{\|x-x_0||_L}{\ell}+1\right)^{-M}\right] \\
		& \le K \left( \frac{\ell}{\wt\ell}\right)^M\sum_{x,y:\|x-y\|_L \ge \wt\ell} \| u_{\al}(y)\|^2  \left(\frac{\|x-x_0||_L}{\ell}+1\right)^{-M}\left(\frac{\|x-y||_L}{\ell}+1\right)^M \\
		& \le  K\left( \frac{\ell}{\wt\ell}\right)^M\sum_{x,y } \| u_{\al}(y)\|^2   \left(\frac{\|y-x_0||_L}{\ell}+1\right)^M \le  K^2  \left( \frac{\ell}{\wt\ell}\right)^M L^d=:\delta_L^2 .
	\end{align*}
	Thus we have proved that $\wt{\mathcal B}_{ M,K,\ell}\subset \cal A_{\delta_L, \wt\ell}$. Then we get from \eqref{uinf1} that
	$$\left|\wt{\mathcal B}_{M,K,\ell}\right|/N \le CK\left( {\ell}/{\wt\ell}\right)^{M/2} L^{d/2}+\OO_\prec \left(  {\wt\ell^2}/{L^{2-\e}}+W^{-d/2}\right)  \prec L^\e \left(  {\ell^{\frac{M}{M+4}}}/{L^{\frac{M-d}{M+4}}}\right)^2 +W^{-d/2},$$
	where in the second step we minimized the sum over $\wt\ell$. Since $\e$ is arbitrarily small, we conclude \eqref{uinf0}. 
\end{proof}


The proof of Theorem \ref{main thm0} is almost the same as the one for Theorem \ref{main thm}, except for some minor differences regarding the infinite space limits of the $\selfs$. Here we only describe the necessary modifications to the arguments in Section \ref{sec pfmain}, without writing down all the details of the proof of Theorem \ref{main thm0}. First, fix any $n\in \N$, we define the renormalized $\selfs$ $ \Sele_l^{(r)} $, $4\le l \le n$, as follows. 

\begin{definition}[Renormalized $\selfs$] \label{collection elements weak}
Let $z=E+\ii \eta$ with $E\in (-2+\kappa,2-\kappa)$ and $W^{2}/L^{2-\e}\le \eta\le L^{-\e}$ for a small constant $\e>0$. Let $\Sele_{l}$, $4\le l \le n-1$, be a sequence of $\selfs$ satisfying Definition \ref{collection elements}, and $\Sele_n$ be a sum of scaling order $n$ deterministic graphs  constructed in the $n$-th order $\incomp$ in Lemma \ref{incomplete Texp}. 
Then we define the renormalized $\selfs$ inductively as follows. First, we define \smash{$\Sele_4^{(r)}=\Sele_4$} and \smash{$\Sele^{(r)}_{2l+1}=0$}. Suppose we have defined \smash{$\Sele_l^{(r)}$} for all $l\le k-1$. Then we define \smash{$\Sele_k^{(r)}$} as the sum of deterministic graphs obtained from $\Sele_k$ by replacing all the lower order $\selfs$ $ \Sele_l$, $4\le l \le k-1$, in it with 
\be\label{vacuum_renorm} \mathring \Sele_l:= \Sele_l^{(r)} - \Big[\sum_x ( \Sele_l^{(r)})_{0x}\Big]\cdot S .\ee 
%
\end{definition}

By definition, we trivially have that $\mathring \Sele_l$ and its infinite space limit $\mathring \Sele_l^\infty$ satisfy the sum zero properties. In particular, the sum zero properties of \smash{$\mathring \Sele_l^\infty$}, $4\le l \le n-1$, are necessary for \smash{$\Sele_n^{(r),\infty}$}, the infinite space limit of \smash{$\Sele_n^{(r)}$}, to be well-defined. For example, we consider the function in \eqref{FT_finitespace}. In the proof below, we will see that $\sum_{x\in \Z^d}(\Sele^{\infty}_{l})_{0x} = \OO(W^{-D})$ under the assumptions of Theorem \ref{main thm0}, which gives that $\psi_{\Sele_{2k_i}}(Wp,E) = \OO(W^2|p|^2 + W^{-D})$ for $p\ll W^{-1}$. If we do not perform the renormalization in Definition \ref{collection elements weak}, then we have
$$\left|\psi_{ (k; 2k_1, \cdots, 2k_l)}(Wp,E)\right| \lesssim W^{-2(l+1)}|p|^{-2(l+1)} (W^2|p|^2 + W^{-D})^{2l} \quad \text{for $|p|\ll W^{-1}$}.  $$
Thus it may give a non-integrable singularity around $p=0$ in the infinite space limit. (However, notice that $W^{-D}$ is negligible in finite space with $W\ge L^{\e}$ for a constant $\e>0$.) In \eqref{vacuum_renorm}, the matrix $S$ can be replaced by any doubly stochastic matrix whose Fourier transform is a Schwartz function (e.g. $S^k$ for any fixed $k\in \N$).

Using properties \eqref{two_properties0}--\eqref{3rd_property0} for $\Sele_l$, $4\le l \le n-1$, we can obtain the following result. 

\begin{lemma}
	 Under the assumptions of Theorem \ref{main thm0} and in the setting of Definition \ref{collection elements weak}, for $4\le l \le n$, we have that 
	\be\label{wh3rd_property0 diff}
	\left| ( \Selek_{l}^{(r)})_{0x}(z)-( \Sele_{l})_{0x}(z) \right| \le W^{-ld/2} \frac{\eta W^{2d-6}}{\langle x\rangle^{2d-6-\tau}}, \quad \forall \  x\in \Z_L^d ,
	\ee
	and
	\be\label{wh3rd_property0}
	\Big|\sum_{x } (\Selek_{l}^{(r)})_{0x}(z)- \sum_{x } ( \Selek_{l})_{0x}(z)\Big|   \le L^\tau \eta W^{-(l-2)d/2 } ,  
	\ee
	for any constant $\tau>0$.	
\end{lemma}
\begin{proof}	
We prove \eqref{wh3rd_property0 diff} and \eqref{wh3rd_property0} by induction on $l$. First, we trivially have $\Sele_4^{(r)}=\Sele_4$. Then suppose we have shown that \eqref{wh3rd_property0 diff} and \eqref{wh3rd_property0} hold for $\Sele_{l}^{(r)}$ for all $l\le k-1$. Combining this induction hypothesis with \eqref{3rd_property0}, we get that
$$\Big|\sum_{x } (\Selek_{l}^{(r)})_{0x}(z)\Big|   \le L^\tau \eta W^{-(l-2)d/2 } .$$
Then using \eqref{vacuum_renorm} and \eqref{wh3rd_property0 diff}, it is trivial to see that
$$\left| (\mathring\Selek_{l})_{0x}(z)-( \Sele_{l})_{0x}(z) \right| \le W^{-ld/2} \frac{\eta W^{2d-6}}{\langle x\rangle^{2d-6-\tau}}, \quad \forall \  x\in \Z_L^d , \ \ l\le k-1.
$$
Now using the same argument as in the proof of Lemma \ref{lemm V-R},  we can get that \eqref{wh3rd_property0 diff} and \eqref{wh3rd_property0} hold for $\Sele_{k}$. 
\end{proof}

We can obtain the following result on the infinite space limits $ \Sele^{(r),\infty}_{l}$ of $ \Sele_l^{(r)}$, $4\le l \le n$.
\begin{lemma}\label{lem FT1}
Under the assumptions of Theorem \ref{main thm0} and in the setting of Definition \ref{collection elements weak},  for $4\le l \le n$, we have that 
 \be\label{whtwo_properties0V}
 \Sele^{(r),\infty}_{l}(x, x+a) =  \Sele^{(r),\infty}_{l}(0,a), \quad   \Sele^{(r),\infty}_{l}(0, a) = \Sele^{(r),\infty}_{l} (0,-a), \quad \forall \ x,a\in \Z^d,
\ee 
and
\be\label{wh4th_property0V}
\big| \big( \Sele^{(r),\infty}_{l}\big)_{0x} \big| \le W^{-ld/2} \frac{W^{2d-4} }{\langle x\rangle^{2d-4-\tau}}, \quad \forall \ x\in \Z^d, 
\ee
for any constant $\tau>0$. Furthermore, we have that for any fixed $D>0$, 
\be\label{eq FT1}\sum_{x\in \Z^d}(\Sele^{(r),\infty}_{n})_{0x}(m(E), \psi, W) =  W^{-(n-2)d/2} {\Sint}_{n}(m(E), \psi)  +\OO(W^{-D}),\ee
where ${\Sint}_{n}$ is a constant that does not depend on $W$. 
\end{lemma}
\begin{proof}
The property \eqref{whtwo_properties0V} follows from Lemma \ref{lem Rsymm}, and the estimate \eqref{wh4th_property0V} follows from the doubly connected property by Corollary \ref{double inf}. Equation \eqref{eq FT1} can be proved using the same argument as in Section \ref{sec pf FT0}. There is only one difference that has been discussed in Remark \ref{rem FT0}---the equation \eqref{psi_compact} does not hold exactly if $\psi$ is not compactly supported. However, using the fact that $\psi$ is a Schwartz function, we get that \eqref{psi_compact} holds up to a small error $\OO(W^{-D})$, which  leads to the extra $\OO(W^{-D})$ in \eqref{eq FT1}. We omit the details.
\end{proof}

As in Lemma \ref{lemm V-R}, we can bound the difference between  $\Sele^{(r)}_l$ and  $ \Sele^{(r),\infty}_l$.

\begin{lemma}\label{lem V-R weak}
Under the assumptions of Theorem \ref{main thm0} and in the setting of Definition \ref{collection elements weak}, for $4\le l \le n$, we have that for any constant $\tau>0$,
\be\label{V-R1 weak}  
\left|(\Sele_l^{(r)})_{0x}(m(z),\psi, W, L)-( \Sele^{(r),\infty}_l)_{0x}( m(E), \psi, W)\right| \le W^{-ld/2}\frac{\eta W^{2d-6}}{ \langle x\rangle^{2d-6-\tau}},\quad \forall \ x\in \Z_L^d,
\ee
and
\be\label{V-R2 weak}  \Big|\sum_{x\in \Z_L^d} ( \Sele_l^{(r)})_{0x}(m(z),\psi, W, L)-\sum_{x\in \Z^d}( \Sele^{(r),\infty}_l)_{0x}( m(E), \psi, W) \Big| \le L^\tau \eta W^{-(l-2)d/2 }. \ee 
\end{lemma}
\begin{proof}
We prove this lemma by induction on $l$. First, \eqref{V-R1 weak} and \eqref{V-R2 weak} trivially hold for $\Sele_4^{(r)}= \Sele_4$. Now suppose we have shown that \eqref{V-R1 weak} and \eqref{V-R2 weak} hold for $ \Sele_{l}^{(r)}$ for all $l\le k-1$. Then with this induction hypothesis and the same argument as in the proof of Lemma \ref{lemm V-R},  we can show that\eqref{V-R1 weak} and \eqref{V-R2 weak} hold for $\Sele_{k}^{(r)}$. We omit the details.
\end{proof}

Now we are ready to complete the proof of Theorem \ref{main thm0}. 
\begin{proof}[Proof of Theorem \ref{main thm0}]
We repeat the five-step strategy in Section \ref{sec_main_struct}, where the steps 1, 2, 3 and 5 stay the same as in the proof of Theorem \ref{main thm}, because these steps only involve the $\selfs$ $\Sele_l$ but do not use the infinite space limits $\Sele^{\infty}_l(z)$ at any place. Regarding Step 4, we need to prove a counterpart of Lemma \ref{cancellation property} in the setting of Theorem \ref{main thm0}. 
The properties \eqref{two_properties0} and \eqref{4th_property0} follow from Lemma \ref{lem Rsymm} and Corollary \ref{lem Rdouble}. The properties \eqref{two_properties0V}, \eqref{4th_property0V} and \eqref{3rd_property0 diff} will be replaced by \eqref{whtwo_properties0V}, \eqref{wh4th_property0V} and \eqref{V-R1 weak}. It remains to prove the following sum zero properties: 
\be\label{property V-R wh}
\Big|\sum_{x\in \Z_L^d} ( \Sele_{n})_{0x} (m(z),\psi, W, L) \Big|   \le L^\tau \eta W^{-(n-2)d/2 },\quad \forall \ \eta\in[  W^{2}/L^{2-\e} ,L^{-\e}],
\ee
and 
\be\label{property V-R wh2}
\Big|\sum_{x\in \Z^d} (\Sele^{(r),\infty}_{n})_{0x} (m(E),\psi, W) \Big|   \le W^{-D}  ,
\ee
for any constants $\tau, D>0$. By \eqref{wh3rd_property0} and \eqref{V-R2 weak}, we have that
\be\label{V-R2 weak333} \Big|\sum_{x\in \Z_L^d} ( \Sele_n)_{0x}(m(z),\psi, W, L)-\sum_{x\in \Z^d}( \Sele^{(r),\infty}_n)_{0x}( m(E), \psi, W) \Big| \le L^\tau \eta W^{-(n-2)d/2 } .\ee
Hence the estimate \eqref{property V-R wh} is a consequence of \eqref{property V-R wh2}.

The proof of \eqref{property V-R wh2} is similar to the one for Lemma \ref{cancellation property}. Since the proof of Lemma \ref{lem maincancel} does not involve infinite space limits, the estimate \eqref{small SnE0} also holds in the current setting for $L_n$ satisfying \eqref{Lcondition}. Combining this estimate with \eqref{V-R2 weak333} and \eqref{eq FT1}, we can obtain that $ {\Sint}_{n}(m(E), \psi)  =\oo(1)$, which gives ${\Sint}_{n}(m(E), \psi)=0$. Together with \eqref{eq FT1}, it implies \eqref{property V-R wh2}, and hence completes Step 4 of the five-step strategy in Section \ref{sec_main_struct}. Finally, applying the argument in Figure \ref{Fig pfchart1}, we complete the proof of Theorem \ref{main thm0}.
\end{proof}
Finally, we give the proof of Theorem \ref{main thm2} and Corollary \ref{main_cor}.

\begin{proof}[Proof of Theorem \ref{main thm2}]
By Theorem \ref{main thm0}, we know that $G(z)$ satisfies the local law \eqref{locallaw}. Moreover, in the proof of Theorem \ref{main thm0}, we have constructed the $M$-th order $\incomp$ \eqref{mlevelT incomplete} with $n=M$. Setting $\fb_1=\fb_2:=\fb$ in \eqref{mlevelT incomplete}, solving $T_{\fa\fb}$ and taking expectation, we get that
\begin{align}
\E T_{\fa\fb} &=  |m|^2 \left( \frac{1}{1-\Theta \wtSdelta^{(M)}}\Theta\right)_{\fa\fb}   + \sum_x \left( \frac{1}{1-\Theta \wtSdelta^{(M)}}\right)_{\fa x}\E\left[m \Theta_{x \fb}\left(\overline G_{\fb\fb}-\overline m \right)+ (\PIT^{(M)})_{x,\fb \fb}\right] \nonumber\\
	& + \sum_x \left( \frac{1}{1-\Theta \wtSdeltan}\right)_{\fa x} \E \left[(\AIT^{(>M)})_{x,\fb \fb} + (\Err'_{M,D})_{x,\fb\fb}\right]  . \label{solving_T}
\end{align}	
Recall that $(\PIT^{(M)})_{x,\fb \fb}$, $(\AIT^{(>M)})_{x,\fb \fb}$ and $ (\Err'_{M,D})_{x,\fb\fb}$ can be written into the forms in \eqref{form_IT1}. Plugging them into \eqref{solving_T} and using the identity $\Theta^{(M)}=(1-\Theta \wtSdelta^{(M)})^{-1}\Theta$, we obtain that 
\be\label{E_local2}
\E T_{\fa\fb} = |m|^2 \Theta^{(M)}_{\fa\fb} + \big( \Theta^{(M)} \cal G^{(M)} \big)_{\fa\fb}  + \sum_{x}  \Theta^{(M)}_{\fa x} \sum_{\omega}\E\left(\cal G^{err}_\omega \right)_{x\fb},
\ee
where $\cal G^{(M)}$ is defined as 
\be\label{local_R}
\cal G^{(M)}_{x\fb}:= m \delta_{x\fb} \E \left(\overline G_{\fb\fb}-\overline m \right) + \E (\Gamma_R^{(n)})_{x,\fb \fb},
\ee
and $\cal G^{err}_\omega $ are the graphs in $(\Gamma_A^{(>n)})_{x,\fb \fb}$ and $ (\Gamma_{err}^{(n,D)})_{x,\fb \fb}$, i.e. $(\Gamma_A^{(>n)})_{x,\fb \fb}+(\Gamma_{err}^{(n,D)})_{x,\fb \fb}=\sum_{\omega}\left(\cal G^{err}_\omega \right)_{x\fb}. $
To conclude the proof, it remains to prove \eqref{Self_theta}, \eqref{bound_calG} and that 
\be\label{err_Texpand}\sum_{x} \left| \Theta^{(M)}_{\fa x} \E\left(\cal G^{err}_\omega \right)_{x\fb}\right| \le W^{-dM/2} .\ee

First, we can expand $\Theta^{(M)}$ as
\be \label{expand_Theta} \Theta^{(M)}=\left[1-(\Theta\wtSdelta^{(M)})^{K+1}\right]^{-1}\sum_{k=0}^K(\Theta\wtSdelta^{(M)})^{k}\Theta,\ee
for a large constant $K\in \N$. Using \eqref{redundant again} and \eqref{BRB}, we can obtain that 
\be\label{bdd_theta}\left\|\Theta\wtSdelta^{(M)}\right\|_{\ell^\infty \to \ell^\infty} \prec W^{-d}, \quad \text{and}\quad   \left[(\Theta\wtSdelta^{(M)})^{k}\Theta\right]_{xy}  \prec B_{xy}. \ee 
Combining \eqref{expand_Theta} with \eqref{bdd_theta}, we get that 
$$\left|\Theta^{(M)}_{xy}\right|\prec B_{xy} +W^{-(K+1)d} \le 2B_{xy} ,$$
as long as $K$ is chosen to be sufficiently large. This concludes \eqref{Self_theta}. Second, we notice that every graph in $\cal G^{(M)}$ is doubly-connected. Then using \eqref{bound 2net1 strong}, we immediately conclude \eqref{bound_calG}. Finally, we prove \eqref{err_Texpand}. Each graph $\left(\cal G^{err}_\omega \right)_{x\fb}$ can be written into 
\be \label{Ahoforms_add}
\left(\cal G^{err}_\omega \right)_{x\fb} = \sum_{y, y'}(\cal G_0)_{x,yy'} G_{y\fb}\overline G_{y'\fb}, \quad \text{or}\quad  \left(\cal G^{err}_\omega \right)_{x\fb}= \sum_{x,y} (\cal G_0)_{xy}\Theta_{y\fb},\ee
or some forms obtained by setting some indices of $x,y,y'$ to be equal to each other. Without loss of generality, we only consider the first form in \eqref{Ahoforms_add}, while all the other forms are easier to bound. By Definition \ref{def incompgenuni}, the graph $\cal G_0$ is doubly connected. Then we can get the bound
\begin{align*}
 \sum_{x}  \left|\Theta^{(M)}_{\fa x}  \E\left(\cal G^{err}_\omega \right)_{x\fb}\right| & \prec   \sum_{x,y, y'}B_{\fa x} \left|\E (\cal G_0)_{x,yy'} G_{y\fb}\overline G_{y'\fb}\right|\prec  W^{-d/2}\sum_{x,y}B_{\fa x}\E \sum_{y'}\left(\cal G_0^{\abs}\right)_{x,yy'}B^{1/2}_{y\fb}\\
	& \prec  W^{-d/2}W^{-(M-2)d/2}\sum_{x,y}B_{\fa x}B_{x y}^{3/2} B^{1/2}_{y\fb}\prec W^{-(M+1)d/2},
\end{align*}
where in the second step we used \eqref{locallaw} and \eqref{Self_theta}, and in the third step we used that $\sum_{y'} ({\cal G}_0^{\rm{abs}})_{x, y y'}$ is a graph satisfying the assumptions of Lemma \ref{no dot} with two fixed atoms $x$ and $ y$, so that it satisfies \eqref{bound 2net1 strong}. This concludes \eqref{err_Texpand}.
\end{proof}

\begin{proof}[Proof of Corollary \ref{main_cor}]
If $S$ is invertible, then multiplying both sides of \eqref{E_local2} by $|m|^{-2}S^{-1}$, we obtain that 
\begin{align*}
\E |G_{\fa\fb}|^2 & = |m|^2 \left[ \frac{1}{1-\left(1+\wtSdelta^{(M)}\right)|m|^2  S} \right]_{\fa\fb} + \left[  \frac{1}{1-\left(1+\wtSdelta^{(M)} \right)|m|^2  S} \cal G^{(M)} \right]_{\fa\fb} \\
&+  \sum_{x}  \left[  \frac{1}{1-\left(1+\wtSdelta^{(M)} \right)|m|^2  S}\right]_{\fa x} \sum_{\omega}\E\left(\cal G^{err}_\omega \right)_{x\fb}.
\end{align*} 
The last term can be bounded by $\OO(W^{-Md/2})$ using the same argument as the one below \eqref{Ahoforms_add}, and we omit the details. 

On the other hand, if $S$ is singular, we can choose $0<\e_N< L^{-100Md}$ so that $S+\e_N I$ is nonsingular. Then we define another random band matrix $\wt H$ with variance profile $\wt S:=(S+\e_N I)/(1+\e_N)$ and denote its resolvent by $\wt G(z):=(\wt H-z)^{-1}$. The above argument shows that \eqref{EGxy} holds for $\E|\wt G_{xy}|^2$. Moreover, it is easy to show that $\E| G_{xy}|^2=\E|\wt G_{xy}|^2+\OO(W^{-Md/2})$ with a simple perturbation argument. 
\end{proof}

\section{Main ideas for Lemmas \ref{lemma ptree}, \ref{lem normA} and \ref{incomplete Texp}}\label{sec strategy} 
%
%
%
%
%
%
 
 In this section, we discuss some key ideas that will be used in the proofs of three key lemmas, Lemmas \ref{lemma ptree}, \ref{lem normA} and \ref{incomplete Texp}, in \cite{PartII_high}. 

\medskip

\noindent {\underline{\bf Main idea for Lemma \ref{lemma ptree}.}} To prove Lemma \ref{lemma ptree}, it suffices to prove the following self-improving estimate on $T$-variables. If $T_{xy} \prec B_{xy}+\wt\Phi^2$ for a deterministic parameter $\wt\Phi $, then 
\be\label{self_improve1} T_{xy} (z)  \prec B_{xy} +W^{-c_1} \wt\Phi^2 \ee
for a constant $c_1>0$ depending only on $d$ and $c_0$ in \eqref{Lcondition0}. Iterating this estimate for $D/c_1$ many times, we will get that $T_{xy} (z)  \prec B_{xy} +W^{-D}$. This concludes \eqref{pth T} as long as $D$ is large enough.

The estimate \eqref{self_improve1} follows from the high moment bound for any fixed $p\in \N$:
\be\label{locallawptree_intro}
\E T_{\fa\fb} (z) ^p \prec  (B_{\fa\fb} +W^{-c} \wt\Phi^2)^p,\quad \forall \ \fa,\fb\in \Z_L^d.
\ee
We regard $T_{\fa\fb}(z)^p$ as a graph with $p$ copies of $T_{\fa\fb}(z)$. Now we replace one of them with the $n$-th order $T$-expansion. If we replace $T_{\fa\fb}$ with the first two terms on the right-hand side of \eqref{mlevelTgdef}, then using \eqref{thetaxy} and \eqref{intro_redagain} we can bound that
\be\label{Holder1}
\left|\E T_{\fa\fb}^{p-1}m\overline G_{\fb\fb}\left[ \Theta_{\fa\fb} +   \left(\Theta \Sdelta^{(n)} \Theta\right)_{\fa\fb} \right]\right| \prec B_{\fa\fb} \E T_{\fa\fb}^{p-1} .
\ee
Next we replace $T_{\fa\fb}$ with a graph $\cal G_{\fa\fb}$ in $(\PTk)_{\fa,\fb\fb}$, $k\ge 3$. It can be written into the forms in \eqref{twoPks} or some variants of them with $\Theta$ replaced by a labelled $\dashed$ edge. As an example, if $\cal G_{\fa\fb}=\sum_x \Theta_{\fa x}(\cal G_0)_{x \fb}$, then using \eqref{thetaxy} and  \eqref{bound 2net1 cor}, we can bound that
\be\label{PTab1}
|\cal G_{\fa\fb}|\prec W^{ (k - 3) (-d/2+\e_0) + \e_0} \sum_{x} B_{\fa x}  B_{x\fb}\cal A_{x\fb} \prec W^{ (k - 2) (-d/2+\e_0)}  B_{\fa  \fb} ,
\ee
where $\cal A_{x\fb}$ is a variable satisfying $\|\cal A\|_{w;(1,2)}\prec 1$ and in the last step we used \eqref{keyobs3}. So with \eqref{PTab1} and the fact $k\ge 3$, we can bound that 
\be\label{Holder2}
\left|\E T_{\fa\fb}^{p-1} (\PT^{(n)})_{\fa,\fb\fb}   \right|\prec W^{-d/2+\e_0} B_{\fa\fb} \E T_{\fa\fb}^{p-1} . 
\ee
Then we replace  $T_{\fa\fb}$ with a graph $\cal G_{\fa\fb}$ in $(\AT^{(>n)})_{\fa,\fb\fb}$. It can be written into the forms in \eqref{Aho 3forms} or some variants of them. As an example, if $\cal G_{\fa\fb}=\sum_x \Theta_{\fa x}(\cal G_0)_{xy}|G_{y\fb}|^2$, then using \eqref{thetaxy}, \eqref{bound 2net1 cor} and $\ord((\cal G_0)_{xy}) > n$, we can bound that
\begin{align} 
	|\cal G_{\fa\fb}| &\prec  W^{(n-2)(-d/2+\e_0)+\e_0}\sum_{x,y} B_{\fa x} B_{xy}\cal A_{xy}\left(B_{y\fb}+\wt\Phi^2\right) \lesssim W^{(n-1)(-d/2+\e_0)}\sum_{y} B_{\fa y}\left(B_{y\fb}+\wt\Phi^2\right) \nonumber\\
	&\lesssim W^{(n-1)(-d/2+\e_0)}\left( \frac{1}{W^4\langle \fa-\fb\rangle^{d-4}} + \frac{L^2}{W^2}\wt\Phi^2\right)\lesssim W^{-c_0+(n -1)\e_0} \left( B_{\fa\fb} + \wt\Phi^2\right), \label{PTab111}
\end{align}
where we used \eqref{keyobs3} in the second step, $  \sum_{y} B_{\fa y} B_{y\fb}\lesssim W^{-4}\langle \fa-\fb\rangle^{-(d-4)}$ and $\sum_{y} B_{\fa y}\lesssim {L^2}/{W^2}$ in the third step, and $\langle \fa-\fb\rangle\le L$ and \eqref{Lcondition1} in the fourth step. 
With \eqref{PTab111}, we obtain that 
\be\label{Holder3}
\left|\E T_{\fa\fb}^{p-1} (\AT^{(>n)})_{\fa,\fb\fb}\right|   \prec  W^{-c_0+(n-1)\e_0} \left(B_{\fa\fb} + \wt\Phi^2\right)  \E T_{\fa\fb}^{p-1} .
\ee
Finally, if we replace $T_{\fa\fb}$ with $(\QTn)_{\fa,\fb\fb}$, then we apply the $Q$-expansions mentioned in Section \ref{subsec global} to $T_{\fa\fb}^{p-1}(\QT^{(n)})_{\fa\fb}$, and show that 
\be\label{QTab_exp}
\left|\E T_{\fa\fb}^{p-1}(\QT^{(n)})_{\fa,\fb\fb}\right| \prec \sum_{k=2}^{p}\left[W^{-d/4+\e_0/2} \left(B_{\fa\fb}+\wt\Phi^2\right)  \right]^k \E T_{\fa\fb}^{p-k} . 
\ee
The details will be given in \cite{PartII_high}. Combining \eqref{Holder1}, \eqref{Holder2}, \eqref{Holder3} and \eqref{QTab_exp}, and applying H{\"o}lder's inequality and Young's inequality to each term, we obtain that 
\be\label{Holder3.6} \E T_{\fa\fb}^p  \prec W^\e \left( B_{\fa\fb} + W^{-c_1}\wt\Phi^2\right)^p + W^{-\e}\E T_{\fa\fb}^{p}  \quad \Rightarrow \quad \E T_{\fa\fb}^p  \prec W^\e \left( B_{\fa\fb} + W^{-c_1}\wt\Phi^2\right) , \ee
for $c_1:=\min(c_0-(n-1)\e_0, d/4-\e_0/2)$ and any constant $\e>0$. This concludes \eqref{locallawptree_intro} as long as $\e_0$ is sufficiently small.

\medskip

\noindent {\underline{\bf Main idea for Lemma \ref{lem normA}.}} We expand $ \E \tr\left( \cal A^{2p}\right) $ as  
\begin{align}\label{eq_pgons}
	\E \tr\left( \cal A^{2p}\right) = \sum_{x_1,\cdots, x_{2p}\in \cal I}\sum_{a_1, \cdots, a_{2p}} c(a_1, \cdots, a_{2p}) \prod_{i=1}^{2p}G^{a_i}_{x_{i} x_{i+1}},\quad a_i \in \{+,-\},
\end{align}
where we adopted the conventions $x_{2p+1}\equiv x_1$, $G_{xy}^{+}\equiv G_{xy}$ and $G_{xy}^{-}\equiv \overline G_{xy}$, and each $c(a_1, \cdots, a_{2p})$ is a deterministic coefficient of order $\OO(1)$. To conclude \eqref{Moments method}, it suffices to show that  
\be\label{Moments method2}
\sum_{x_1,\cdots, x_{2p}\in \cal I} \E \prod_{i=1}^{2p}G^{a_i}_{x_{i} x_{i+1}} \le K^d\left( W^\e \frac{K^4}{W^4}\right)^{2p-1},\quad \forall \ (a_1, \cdots , a_{2p})\in \{+,- \}^{2p}.
\ee
We regard the above graph as a $2p$-gon graph with $2p$ external vertices $x_1,\cdots, x_{2p}$. We will expand it using the operations defined in Section \ref{sec_basiclocal}, and a similar expansion strategy as the one for Lemma \ref{incomplete Texp} that will be introduced below. In the expansions, we will get internal molecules. Our goal is to expand every $2p$-gon graph into a linear combination of connected deterministic graphs that satisfy a weaker doubly connected property: there exist two disjoint nets $\cal B_{black}$ and $\cal B_{blue}$ of black and blue $\dashed$ edges, so that each internal molecule connects to external molecules through a path of edges in $\cal B_{black}$ and a path of edges in $\cal B_{blue}$.
(Note that if we remove the external molecules from these graphs, the remaining internal molecules do not form  doubly connected graphs, so this new property is weaker than the one in Definition \ref{def 2net}.) Such deterministic graphs will satisfy the bound in \eqref{Moments method2}.

\medskip

\noindent {\underline{\bf Main idea for Lemma \ref{incomplete Texp}.}} The proof of Lemma \ref{incomplete Texp} is based on a carefully designed global expansion strategy. This strategy is also used in the proof of Lemma \ref{lem normA} as discussed above. 

The main difficulty with our expansions is how to maintain the doubly connected structures of the graphs. It is not hard to check that local expansions will not affect the  doubly connected property. However, this is not the case with global expansions introduced in Section \ref{subsec global}, because new molecules created in a global expansion may break the doubly connected property. In fact, a global expansion preserves the  doubly connected structure only when we expand a $T$-variable containing a \emph{redundant} blue solid edge. Here we call a blue solid edge redundant if and only if after removing it, the resulting graph is still doubly connected. 

In \cite{PartII_high}, we will show that the graphs in our expansions actually satisfy a stronger \emph{pre-deterministic property}. Roughly speaking, a doubly connected graph $\cal G$ is said to be pre-deterministic if the following property holds: there exists an order of all the internal blue solid edges 
in $\cal G$, denoted by $b_1\preceq b_2\preceq \cdots$, such that for any $k$, after changing the edges $b_1, \cdots, b_{k-1} $ into $\dashed$ edges, the blue solid edge $b_k$ becomes a redundant edge. We call this order of blue solid edges a \emph{pre-deterministic order}. 

Now the highlight of our expansion strategy is that if we expand the $T$-variable containing the \emph{first redundant edge} in a pre-deterministic order using the global expansion in Section \ref{subsec global}, then the resulting graphs are still pre-deterministic. Then in every new graph, we find the first redundant edge in a pre-deterministic order and expand it further. Continuing in this way, we  finally obtain a linear combination of graphs that can be written into the form \eqref{mlevelT incomplete}. Here we remark that 
after one step of global expansion, we need to apply local expansions to the resulting graphs to turn them into locally standard graphs before we apply the next step of global expansion. In \cite{PartII_high}, we will show that local expansions also do not affect the pre-deterministic property.

Finally, we remark that the above discussion is only for heuristic purpose, and they are not completely rigorous regarding some technical details. In fact, we will use a slightly weaker property, called the \emph{sequentially pre-deterministic property}, instead of the pre-deterministic property. The interested reader can refer to \cite{PartII_high} for more details.

\appendix

\section{Symmetry and translational invariance}

In this section, we record the following simple fact: any deterministic graph with two external atoms satisfies the properties in  \eqref{two_properties0}.

\begin{lemma}\label{lem Rsymm}
Let $\cal M$ be a deterministic matrix in terms of $S$, $S^{\pm}(z)$ and $\Theta(z)$. 
Then we have 
\be\label{two_properties app}\cal M (x, x+a) =  \cal M (0,a), \quad   \cal M (0, a) = \cal M(0,-a), \quad \forall x,a\in \Z_L^d. \ee
\end{lemma}
\begin{proof}
This lemma is a simple consequence of the fact that all the matrices $S$, $S^{\pm}(z)$ and $\Theta(z)$ satisfy the two properties in \eqref{two_properties app}. Suppose $\cal M(x,y)$ can be written into the general form
\begin{align*}
\cal M(x,y)=\sum_{x_1,\cdots, x_\ell}\prod_{(x,x_i)\in \mathbf E}f_{i}(x,x_i)\cdot \prod_{(x_i,x_j)\in \mathbf E}f_{ij}(x_i,x_j) \cdot \prod_{(x_j,y)\in \mathbf E}g_{j}(x_j,y),
\end{align*}
where $\mathbf E$ is the set of all the edges in the graph $\cal M(x,y)$, and $f_{i}$, $f_{ij}$ and $g_j$ satisfy the two properties in \eqref{two_properties app}. Then we have that
\begin{align*}
\cal M(x , x+a) &=\sum_{x_1,\cdots, x_\ell}\prod_{(x,x_i)\in \mathbf E}f_{i}(x,x_i)\cdot \prod_{(x_i,x_j)\in\mathbf E}f_{ij}(x_i,x_j) \cdot \prod_{(x_j,x+a)\in\mathbf E}g_{j}(x_j,x+a) \\
 &=\sum_{x_1,\cdots, x_\ell}\prod_{(0,x_i-x)\in\mathbf E}f_{i}(0,x_i-x)\cdot \prod_{(x_i-x,x_j-x)\in\mathbf E}f_{ij}(x_i-x,x_j-x) \cdot \prod_{(x_j-x,a)\in\mathbf E}g_{j}(x_j-x,a) \\
 &=\sum_{x_1,\cdots, x_\ell}\prod_{(0,x_i)\in\mathbf E}f_{i}(0,x_i)\cdot \prod_{(x_i,x_j)\in\mathbf E}f_{ij}(x_i,x_j) \cdot \prod_{(x_j,a)\in\mathbf E}g_{j}(x_j,a) =\cal M (0 , a) ,
\end{align*}
and 
\begin{align*}
\cal M(0 , a) &=\sum_{x_1,\cdots, x_\ell}\prod_{(0,x_i)\in\mathbf E}f_{i}(0,x_i)\cdot \prod_{(x_i,x_j)\in\mathbf E}f_{ij}(x_i,x_j) \cdot \prod_{(x_j,a)\in\mathbf E}g_{j}(x_j,a) \\
 &=\sum_{x_1,\cdots, x_\ell}\prod_{(0,-x_i)\in\mathbf E}f_{i}(0,-x_i)\cdot \prod_{(-x_i,-x_j)\in\mathbf E}f_{ij}(-x_i,-x_j) \cdot \prod_{(-x_j,-a)\in\mathbf E}g_{j}(-x_j,-a) \\
 &=\sum_{x_1,\cdots, x_\ell}\prod_{(0,x_i)\in\mathbf E}f_{i}(0,x_i)\cdot \prod_{(x_i,x_j)\in\mathbf E}f_{ij}(x_i,x_j) \cdot \prod_{(x_j,-a)\in\mathbf E}g_{j}(x_j,-a) =\cal M (0 , -a) ,
\end{align*}
where we used that for $f$ satisfying \eqref{two_properties app}, $f(-x,-y)=f(0,x-y)=f(0,y-x)=f(x,y).$ 
\end{proof}

\section{Proofs of some deterministic estimates}\label{appd}
In this section, we collect the proofs of some deterministic estimates, including Lemma \ref{lem cancelTheta}, Lemma \ref{lem esthatTheta}, Lemma \ref{lem cancelTheta2} and Claim \ref{claim wtS-S}. We start with the following Taylor expansion:
\begin{equation}\label{taylor}
\Theta = (1-|m|^2 S)^{-1}|m|^2S =  \left( 1-|m|^{2K}S^{K}\right)^{-1}\sum_{k=1}^{K}|m|^{2k}S^{k}.
 \end{equation}
Since $\|S\|_{\ell^\infty \to \ell^\infty} = 1$ and $|m|\le 1 - c\eta$ for some constant $c>0$, 
taking $K=\eta^{-1} W^\tau$ or $K=\eta^{-1} \langle x-y \rangle^\tau$ in \eqref{taylor} for a small constant $\tau>0$, we get that
\be\label{taylor1}
\Theta_{xy} = \sum_{k=1}^{\eta^{-1} W^\tau} |m|^{2k}(S^{k})_{xy} +\OO(e^{-cW^\tau/2}) = \sum_{k=1}^{\eta^{-1} \langle x-y \rangle^\tau} |m|^{2k}(S^{k})_{xy} +\OO(e^{-c\langle x-y \rangle^\tau/2}).
\ee
Since $S$ is a doubly stochastic matrix, $(S^{k})_{xy}$ can be understood through a $k$-step random walk on the torus $\Z_L^d$. We first prove the following lemma for the random walk on $\Z^d$.

\begin{lemma}\label{reg RW2} 
 Let $B_n = \sum_{i=1}^n X_i$ be a random walk on $\mathbb Z^d$ with $i.i.d.$ steps $X_i$ such that \begin{equation*}
\mathbb P(|X_1|=x) = f_{W,L}(x), 
\end{equation*}
for a function $f_{W,L}$ satisfying Assumption \ref{var profile}. 
Let $\cal C$ be the covariance matrix of $X_1$ with $\cal C_{ij}=\mathbb E [ (X_1)_i (X_1)_j ]$. Assume that $ n \ge   W^{c_0}$  for a constant $c_0>0$. Then for any large constant $D>0$, we have that
\be\label{RW_diffusion2}
\mathbb P\left(B_n = x\right) = \frac{1+\oo(1)}{(2\pi n)^{d/2} \sqrt{\det(\cal C)}}e^{-\frac1{2} x^\top (n\cal C)^{-1} x}+\OO(n^{-D}).
\ee
Moreover, suppose $x, a,b\in \Z^d$ satisfy that $|x|\ge W^{1+2\e_0}$ and $|a|\le |b| \le |x|^{1-\e_0}$ for a small constant $\e_0>0$. Then if $n\ge {|x|^{2-2\e_0+\e_1}}/{W^2}$ for a constant $\e_1>0$, we have that
\be\label{RW_diffusion1}
\begin{split}
&\left| \mathbb P\left(B_n = x + a\right) + \mathbb P\left(B_n = x - a \right) - \mathbb P\left(B_n = x + b\right)- \mathbb P\left(B_n = x - b\right) \right| \\
&\le   \frac{| b |^2}{W^2} \frac{n^\tau}{n^{d/2+1}W^d}e^{-\frac1{2} x^\top (n\cal C)^{-1} x}+\OO(|x|^{-D}),
\end{split}
\ee
for any constants $\tau, D>0$.
\end{lemma}

\begin{proof} 
The estimate \eqref{RW_diffusion2} has been proved in Lemma 30 of \cite{PartIII}. We only need to prove \eqref{RW_diffusion1}. By \eqref{subpoly}, we have that for any fixed $\tau, D>0$,
$$ \P(|X_1|\le W^{1+\tau})\ge 1- W^{-D}. $$
Then using a simple Chernoff bound, we can get that for any fixed $\tau, D>0$,
\be\label{largeBn} \P\left(B_n = x\right) = \OO(n^{-D}), \quad \text{for \ \ $|x|\ge Wn^{1/2+\tau}$.}\ee
Thus to prove \eqref{RW_diffusion2} and \eqref{RW_diffusion1}, we only need to focus on the case 
\be\label{restrictx} |x|=\OO(W n^{1/2+\tau_0})\ee 
for a small constant $\tau_0>0$. In the following proof, we always make this assumption.
Using characteristic functions and cumulants, the following estimate has been shown in the proof of Lemma 30 in \cite{PartIII}:
\begin{align} \label{RW000}
\P\left(B_n =x \right)& = \frac{1}{(2\pi)^d}\int\limits_{|p| \le W^{-1}n^{-1/2}|x|^{\tau_0}} \rd p \, e^{-\ii p\cdot x}e^{-\frac12 n p^\top\cal C p}\bigg[1+\sum_{3 \le k \le K_D}\al_k(\hat p)( Wn^{1/2} |p|)^k\bigg] \\
&+\OO(|x|^{-D}), \nonumber
\end{align}
where 
$K_D$ is a fixed integer depending only on $D$, and $\al_k(\hat p)\in \C$ are complex coefficients defined as 
$$ \al_k (\hat p):=\frac{\kappa_k(\hat p)}{k! \cdot W^k} \ii^k n^{1-k/2} .$$
Here $\hat p:= p/|p|$ and $\kappa_k(\hat p)$ denotes the $k$-th cumulant of $\hat p\cdot X_1$. Using \eqref{subpoly}, we can check that
$$ |\kappa_k(\hat p)|\le  C^k k! \cdot W^k, \quad  \forall \ \hat p \in \mathbb S^d ,$$
for a large enough constant $C>0$. Thus we have $\al_k(\hat p)=\OO(n^{1-k/2})$.  Using \eqref{RW000}, we obtain that 
\begin{align} 
&\left| \mathbb P\left(B_n = x + a\right) + \mathbb P\left(B_n = x - a \right) - \mathbb P\left(B_n = x + b\right)- \mathbb P\left(B_n = x - b\right) \right| \label{cumulantFourier2}\\
&\lesssim \bigg|\int\limits_{|p| \le W^{-1}n^{-1/2}|x|^{\tau_0}} \rd p \left[ \cos\left(p\cdot a\right)-\cos\left(p\cdot b\right)\right]  e^{-\ii p\cdot x}e^{-\frac12 n p^\top\cal C p}\bigg[1+\sum_{3 \le k \le K_D}\al_k(\hat p)( Wn^{1/2} |p|)^k\bigg]\bigg|+\OO(|x|^{-D}).\nonumber
\end{align}
For $n\ge {|x|^{2-2\e_0+\e_1}}/{W^2}$, $|a|\le |b| \le |x|^{1-\e_0}$ and $|p| \le W^{-1}n^{-1/2}|x|^{\tau_0}$, we have
 $$|p\cdot a|+ |p\cdot b|\lesssim |x|^{-\e_1/2} |x|^{\tau_0} \le |x|^{-\e_1/4},$$
as long as $\tau_0<\e_1/4$. Thus using the Taylor expansions of $\cos(p\cdot a)$ and $\cos(p\cdot b)$, we can write that
 \be\label{Taylorcos}\cos\left(p\cdot a\right)-\cos\left(p\cdot b\right) = \sum_{k=1}^{K'_D} (-1)^k \frac{(p\cdot a)^{2k}-(p\cdot b)^{2k}}{(2k)!}+\OO(|x|^{-D}),\ee
where $K'_D$ is a fixed integer depending only on $D$ and $\e_1$.  Inserting it into \eqref{cumulantFourier2} and 
bounding each term in the resulting expression, we can obtain \eqref{RW_diffusion1}. 
For example, the leading term is   
 \begin{align} 
&  \int\limits_{|p| \le W^{-1}n^{- 1/2}|x|^{\tau_0}} \rd p \frac{(p\cdot a)^{2}-(p\cdot b)^{2}}{2}  e^{-\ii p\cdot x}e^{-\frac12 n p^\top\cal C p} \nonumber\\ 
& =\frac{1}{ \sqrt{ n^d \det(\cal C)}} \int\limits_{|W\cal C^{-1/2}q|\le |x|^{\tau_0}} \rd q \frac{\left[q\cdot (n\cal C)^{-1/2}a\right]^{2}-\left[q\cdot (n\cal C)^{-1/2}b\right]^{2}}{2}e^{ -\ii q \cdot y}e^{-{q^2}/{2}} 
 +\OO(|x|^{-D}) ,\label{boundstat}
\end{align}
where we used change of variables $ y:=(n\cal C)^{-1/2}x$ and $q:=(n\cal C)^{1/2}p$. Using the conditions in Assumption \ref{var profile}, we can check that
\be\label{operator_sigma}
C^{-1}W^{2}\le \lambda_{\min}(\cal C)\le \lambda_{\max}(\cal C) \le CW^2
\ee
for some large constant $C>0$, where $\lambda_{\min}$ and $\lambda_{\max}$ respectively denote the maximum and minimum eigenvalues of $\cal C$. By \eqref{operator_sigma}, we have that $C^{-1/2}|q|\le |W\cal C^{-1/2}q|\le C^{1/2}|q|$. Then in \eqref{boundstat} we can replace the domain of the integral by \smash{$\int_{q\in \R^d}$}, because the integral over the domain $\{q:|W\cal C^{-1/2}q|>|x|^{\tau_0}\}$ can be bounded by $\OO(|x|^{-D})$ for any fixed $D>0$ due to the term $e^{-q^2/2}$. Hence we can estimate \eqref{boundstat} as
\begin{align*}
 |\eqref{boundstat}| &\lesssim \frac{1}{n^{d/2}W^d} \left|\int_{q\in \R^d} \rd q \frac{[q\cdot (n\cal C)^{-1/2}a]^{2}-[q\cdot (n\cal C)^{-1/2}b]^{2}}{2}  e^{-\ii q\cdot y}e^{-{q^2}/2} \right|+\OO(|x|^{-D})\\
&\lesssim \frac{n^{2\tau_0} }{n^{d/2+1}W^d}\frac{|b|^2}{W^2} e^{-y^2/2} +\OO(|x|^{-D}) 
\end{align*}
where we have bounded the integral using the stationary phase approximation and the fact that $|y|=\OO(n^{\tau_0})$ for $x$ satisfying \eqref{restrictx}. All the other integrals coming from the $k\ge 2$ terms in \eqref{Taylorcos} can be bounded in a similar way, and they all  give sub-leading terms. This proves the bound in \eqref{RW_diffusion1}.  
\end{proof} 

Now we prove Lemma \ref{lem cancelTheta}, Lemma \ref{lem esthatTheta}, Lemma \ref{lem cancelTheta2} and Claim \ref{claim wtS-S} one by one using Lemma \ref{reg RW2}.

\begin{proof}[Proof of Lemma \ref{lem cancelTheta}]
Fix a small constant $\tau>0$. We need to estimate the sum in \eqref{taylor1}. Let $B_n = \sum_{i=1}^n X_i$ be a random walk on $\Z_L^d$ with $i.i.d.$ steps $X_i$, such that $\mathbb P(X_i = y-x)=s_{xy}$. Then we have 
\be\label{ProbRW}(S^k)_{xy}=\mathbb P(B_k = y - x)  .\ee
If $W^{\tau} \le k \le L^{2-\tau}/W^{2}$, using a simple Chernoff bound we can get the large deviation estimate
\be \label{largedeviation}
\P\left(|B_k| \ge k^{1/2+\delta}W \right)  \le k^{-D} 
\ee
for any constants $\delta,D>0$. In particular, it shows that with high probability, $B_k$ can be regarded as a random walk on the full lattice $\Z^d$ if $k \le L^{2-\tau}/W^{2}$, so that both \eqref{RW_diffusion2} and \eqref{RW_diffusion1} can be applied.

Since $g(\cdot)$ is a symmetric function and $\sum_x g(x)=0$, we can write that
\be\label{decomposezero} \sum_{x}\Theta_{0 x}(z) g(x-x_0) =\sum_{a\in\mathfrak A} g(a) \left[  \Theta(0,x_0+a)+ \Theta(0,x_0-a) - \Theta(0,x_0+y_a) -\Theta(0,x_0-y_a)\right],\ee
where $\mathfrak A$ is a subset of $\cal B_K$, and $y_a\in \cal B_K$ depends on $a$ and satisfies $|y_a|\le |a|$. 
By \eqref{taylor1} and \eqref{ProbRW}, we have
 \begin{align}\label{Pk0}
&\left|  \Theta(0,x_0+a)+ \Theta(0,x_0-a) - \Theta(0,x_0+y_a) -\Theta(0,x_0-y_a)\right|\le \sum_{k=1}^{|x_0|^\tau \eta^{-1}}|m|^{2k}  P_k(a,y_a) + |x_0|^{-D} ,
 \end{align} 
for any constant $ D>0$, where 
 $$P_k(a,y_a):= \left| \mathbb P\left(B_k = x_0+a\right) + \mathbb P\left(B_k = x_0-a \right) - \mathbb P\left(B_k = x_0 + y_{a}\right)- \mathbb P\left(B_k = x_0 - y_{a}\right) \right| .$$
Suppose $\tau <\e/2$ so that $\eta\ge W^{2 }/L^{2-2\tau}$. Using the large deviation estimate \eqref{largedeviation}, we can bound that
\be\label{Pk1} 
P_k(a,y_a)=\OO(W^{-D}) \quad \text{ for }\quad 1\le k \le {|x_0|^{2-\tau}}/{W^{2}}.
\ee
For ${|x_0|^{2-\tau}}/{W^{2}} < k\le |x_0|^\tau \eta^{-1}$, $P_k(a,y_a)$ can be bounded using \eqref{RW_diffusion1} as 
\be\label{Pk2}  P_k(a,y_a)\le   \frac{| a |^2}{W^2} \frac{n^\tau}{k^{d/2+1}W^d}e^{-\frac1{2} x_0^\top (k\cal C)^{-1} x_0} .\ee
Plugging \eqref{Pk1} and \eqref{Pk2} into \eqref{Pk0}, we obtain that for any constant $D>0$,
\begin{align*}
|(\ref{Pk0})|&\le \sum_{{|x_0|^{2-\tau}}/{W^{2 }}  \le k\le |x_0|^\tau \eta^{-1}}  \frac{| a |^2}{W^2} \frac{n^\tau}{k^{d/2+1}W^d}e^{-\frac1{2} x_0^\top (k\cal C)^{-1} x_0} + |x_0|^{-D} \\
& \lesssim \frac{|a|^2}{W^2}\frac{n^{\tau} }{\langle x_0 \rangle^{d-d\tau/2}}\mathbf 1_{|x_0|\le  \eta^{-1/2}W^{1+\tau}} +  |x_0|^{-D}.
\end{align*}
Plugging it into \eqref{decomposezero}, we conclude Lemma \ref{lem cancelTheta} since $\tau$ can be arbitrarily small. 
\end{proof}

\begin{proof}[Proof of Lemma \ref{lem esthatTheta}]
Using \eqref{subpoly}, we get that for any fixed $D>0$,
\begin{align}\label{S-S0} 
	\left|(S_{\infty})_{0x} - S_{0x}(L)\right| \le  (|x|+W)^{-D} , \quad \forall \  x\in \Z^d,
\end{align}
where $S(L)$ refers to the variance matrix defined on $\Z_L^d$, and we adopted the convention that $S_{0x}(L)=0$ for $x\notin (-L/2,L/2]^d$. Now we prove the estimates \eqref{S+xy1}--\eqref{Theta-wh} one by one. 
First, by the arguments in the proof of \cite[Lemma 4.2]{PartII}, there exist constants $c_0, c_1>0$ such that  
\be\label{contractionS+}
\left\|\left( \frac{m^2(E)S +c_0}{1+c_0}\right)^2\right\|_{\ell^\infty \to \ell^\infty}\le 1-c_1,\quad \left\|\left( \frac{m^2(E)S_\infty+c_0}{1+c_0}\right)^2\right\|_{\ell^\infty \to \ell^\infty} \le 1-c_1.  
\ee
Then we use \eqref{contractionS+} to estimate the Taylor expansion 
\be\label{expS+}S_\infty^{+}=\frac{m^2(E) S_\infty}{1+c_0}\sum_{k=0}^{\infty}\left(\frac{m^2(E)S_\infty+c_0}{1+c_0}\right)^k.\ee
Using \eqref{contractionS+}, we immediately obtain that $(S^+_{\infty})_{0x}$ exists for any $x\in \Z^d$, $\max_{x}(S^+_{\infty})_{0x}=\OO(W^{-d})$ and 
$$ \sum_{k\ge |x|/W } \left[S_\infty\left(\frac{m^2(E)S_\infty+c_0}{1+c_0}\right)^k\right]_{0x}=\OO\left( W^{-d}(1-c_1)^{|x|/(2W)}\right).$$
On the other hand, when $|x|\ge W^{1+\tau}$, we have that for any fixed $D>0$,
$$\left|\sum_{k < |x|/W} \left[S_\infty\left(\frac{m^2(E)S_\infty+c_0}{1+c_0}\right)^k\right]_{0x}\right| \le |x|^{-D}.$$ 
Here we used a similar large deviation estimate as in \eqref{largedeviation} to derive this estimate. 
Combining the above two estimates, we conclude \eqref{S+xy1}. Subtracting the Taylor expansion of $S^+(z)$ from the expansion \eqref{expS+}, and using $|m(z)-m(E)|=\OO(\eta)$ and \eqref{S-S0}, we can readily conclude \eqref{S+xy2}.

It remains to study $\Theta_\infty$. Suppose we have shown that $(\Theta_\infty)_{0x}$ exists for any $x\in \Z^d$. Then the estimate \eqref{Theta-wh1} follows from \eqref{thetaxy}. Now we prove \eqref{Theta-wh}. For any $\wt L\ge L >  2|x|$, we abbreviate $\wt z:=E+\ii W^2/\wt L^2$, $\wt m\equiv m(\wt z)$, $\wt S\equiv S(\wt L)$ and $\wt\Theta\equiv \Theta(\wt z, \wt L) $. Then using \eqref{taylor1}, we obtain that 
\begin{align}
 \left|\wt\Theta_{0x}  - \Theta_{0x}\right| &\le \sum_{k=1}^{ \langle x\rangle^{2-\tau}/W^2} \left||m|^{2k}S^k_{0x} -|\wt m|^{2k}\wt S^k _{0x}\right| + \sum_{k= \langle x\rangle^{2-\tau}/W^2}^{\langle x\rangle^{-\tau}\eta^{-1}} \left||m|^{2k}S^k_{0x} -|\wt m|^{2k}\wt S^k _{0x}\right|   \nonumber\\
&+ \sum_{k=\langle x\rangle^{-\tau}\eta^{-1}}^{\langle x\rangle^\tau\eta^{-1}} |m|^{2k} S^k_{0x}+ \sum_{k=\langle x\rangle^\tau \eta^{-1}}^{ \langle x\rangle^{-\tau}\wt L^2/W^2} |\wt m|^{2k} \wt S^k_{0x} + \sum_{k=\langle x\rangle^{-\tau}\wt L^2/W^2}^{\langle x\rangle^\tau \wt L^2/W^2} |\wt m|^{2k} \wt S^k_{0x}+ \langle x\rangle^{-D} ,\label{T-T0}
\end{align}
for any constants $\tau,D>0$. 
Using \eqref{S-S0}, $\left||m|^{2k}-|\wt m|^{2k}\right|\lesssim k \eta$ and \eqref{RW_diffusion2}, we can bound the five terms on the right-hand side of \eqref{T-T0} one by one as follows: 
\be \nonumber
\sum_{k=1}^{ \langle x\rangle^{2-\tau}/W^2} \left||m|^{2k}S^k_{0x} -|\wt m|^{2k}\wt S^k _{0x}\right| \le   \langle x\rangle ^{-D},
\ee
\begin{align} \nonumber
&  \sum_{k= \langle x\rangle^{2-\tau}/W^2}^{\langle x\rangle^{-\tau}\eta^{-1}} \left||m|^{2k}S^k_{0x} -|\wt m|^{2k}\wt S^k _{0x}\right|  \lesssim \sum_{k= \langle x\rangle^{2-\tau}/W^2}^{\langle x\rangle^{-\tau}\eta^{-1}} \frac{k\eta}{W^dk^{d/2}}   \le  \frac{\eta  }{W^4\langle x\rangle^{d-4 -d\tau}}  ,
\end{align}
\begin{align} \nonumber
&\sum_{k=\langle x\rangle^{-\tau}\eta^{-1}}^{\langle x\rangle^\tau\eta^{-1}} |m|^{2k} S^k_{0x} \lesssim  \frac{\langle x\rangle^\tau\eta^{-1}\mathbf 1_{|x|\le \langle x\rangle^\tau W \eta^{-1/2}} }{(\langle x\rangle^{-\tau}\eta^{-1})^{d/2}W^d}+ \langle x\rangle^{-D} \le \frac{\eta }{W^4\langle x\rangle^{d-4-2d\tau}}+ \langle x\rangle^{-D} ,
\end{align}
\begin{align*}
 \sum_{k=\langle x\rangle^\tau \eta^{-1}}^{ \langle x\rangle^{-\tau}\wt L^2/W^2} |\wt m|^{2k} \wt S^k_{0x}\lesssim \sum_{k=\max\left(\langle x\rangle^\tau \eta^{-1}, \langle x\rangle^{2-\tau}/W^2\right)}^{\langle x\rangle^{-\tau}\wt L^2/W^2} \frac{1}{W^dk^{d/2}} +  \langle x\rangle^{-D}  \le  \frac{\eta  }{W^4\langle x\rangle^{d-4-d\tau}} +   \langle x\rangle^{-D} ,
 \end{align*}  
\begin{align} \nonumber
\sum_{k=\langle x\rangle^{-\tau}\wt L^2/W^2}^{\langle x\rangle^\tau \wt L^2/W^2} |\wt m|^{2k} \wt S^k_{0x} \lesssim \frac{\langle x\rangle^\tau \wt L^2}{W^2} \frac{1}{(\langle x\rangle^{-\tau}\wt L^2/W^2)^{d/2}W^d} + \langle x\rangle^{-D} \le \frac{ \langle x\rangle^{d\tau}}{W^2\wt L^{d-2} }+ \langle x\rangle^{-D}.
\end{align}
Combining the above estimates and taking $\wt L \to \infty$, we can conclude \eqref{Theta-wh} since $\tau$ is arbitrary. 

Finally, it remains to show that the limit in \eqref{defwhtheta} exists. For any $x\in \Z^d$, we choose $\wt L\ge L \ge L^{c}\langle x\rangle$ for a constant $c>0$, $\eta=W^2/L^2$, $\wt\eta=W^2/\wt L^2$, $m\equiv m(E+\ii \eta)$ and $\wt m\equiv m(E+\ii \wt\eta)$. Again using the Taylor expansion \eqref{taylor1}, we obtain that for any constants $\tau, D>0$,
\begin{align*}
  \left|\wt\Theta_{0x}\left(E+\ii \wt\eta,\wt L\right)  - \Theta_{xy}\left(E+\ii \eta,L\right)\right|&\le \sum_{k=1}^{L^{-\tau}\eta^{-1}} \left||m|^{2k}S^k_{0x}-|\wt m|^{2k}\wt S^k_{0x}\right|  +\sum_{k=L^{-\tau}\eta^{-1}}^{L^\tau \eta^{-1}} |m|^{2k} S^k_{0x}\\
  &+\sum_{k=L^{\tau}\eta^{-1}}^{L^{-\tau} \wt \eta^{-1}} |\wt m|^{2k} \wt S^k_{0x} +\sum_{k=L^{-\tau}\wt\eta^{-1}}^{L^{\tau} \wt \eta^{-1}} |\wt m|^{2k} \wt S^k_{0x} +  L^{-D} .
\end{align*}
Applying similar arguments as above to each term on the right-hand side, we can obtain that
 \begin{align*}
   \left|\wt\Theta_{0x}\left(E+\ii \wt\eta,\wt L\right)  - \Theta_{0x}\left(E+\ii  \eta,L\right)\right|&\le  \frac{\eta L^{2d\tau}}{W^4\langle x\rangle^{d-4}} +  L^{-D} .
 \end{align*}
This shows that for any fixed $x\in \Z^d$, $\wt\Theta_{0x}(E+\ii W^2/\wt L,\wt L)$ is a Cauchy sequence in $\wt L$. Hence the limit in \eqref{defwhtheta} exists. 
\end{proof}

\begin{proof}[Proof of Lemma \ref{lem cancelTheta2}]
 We choose $L$ such that $L\le |x_0|^2 \le 2L$ and $\eta=  W^{2}/L^{2-\tau}$ for a small constant $\tau>0$. By \eqref{Theta-wh}, we have that for any constant $D>0$, 
\be \nonumber 
  |(\Theta_\infty)_{0x}(E)-\Theta_{0x}(E+\ii\eta)|\le \frac{\eta }{W^4\langle x\rangle^{d-4-\tau}}+ \langle x \rangle^{-D} .
\ee
With this estimate, we get that 
\begin{align*}
\left|\sum_{x}\left[(\Theta_\infty)_{0 x}(E)-\Theta_{0x}(E+\ii \eta) \right]g(x-x_0)\right| &\le   \sum_{x\in \cal B_K}\left[ \frac{ |x_0|^\tau}{ |x_0|^{d}}+ |x_0|^{-D} \right] |g(x)|.
\end{align*}
On the other hand, 
by Lemma \ref{lem cancelTheta} we have that
$$\left|\sum_{x}\Theta_{0 x}(E+\ii \eta) g(x-x_0)\right| \le \left(\sum_{x\in \cal B_K}\frac{x^2}{|x_0|^2}|g(x)|\right) \left[  | x_0|^\tau B_{0x_0}    +  |x_0|^{-D} \right].$$
Combining the above two estimates, we conclude the proof.
\end{proof}

\begin{proof}[Proof of Claim \ref{claim wtS-S}]
 Fix any constants $\tau, D>0$, by \eqref{subpoly} we have that 
\begin{align}\label{S-S2} \left|\wt S_{0x} - S_{0x}\right| \le |x|^{-D} \quad \text{for}\quad  |x|\ge  W^{1+\tau}.
\end{align}
We now bound the difference $ \wt S_{0x}-S_{0x} $ for $|x|< W^{1+\tau}$:
\begin{align}\label{S-S3}
	\wt S _{0x}  - S_{0x}  =\frac{\langle x\rangle^2}{L^2} \frac1{L^d}\sum_{p_0\in \mathbb T_L^d}\phi(p_0, x) e^{\ii p_0\cdot x} + \OO(W^{-D})
	,\end{align} 
where we used \eqref{bandcw1} and that 
$$ \int_{|p| > \pi} \psi(Wp)e^{\ii p\cdot x} \dd p =\OO(W^{-D}),$$
because $\psi$ is a Schwartz function.  Moreover, the function $\phi(p_0, x)$ is defined as 
$$\phi(p_0, x):=\frac{L^2}{\langle x\rangle^2} \frac{L^d}{(2\pi)^d}\int_{p\in B(p_0)} \left[ \psi(Wp_0)- \psi(Wp)e^{\ii (p-p_0)\cdot x} \right] \dd p, $$
with $B(p_0)$ being the box centered at $p_0$ with side length $2\pi/L$. It is easy to check that 
$$|\phi(p_0, x)|\lesssim  \sup_{p\in B(p_0)}\left(|\psi(Wp)|+ |\psi'(Wp)|+ |\psi''(Wp)|\right).$$ 
Plugging this estimate into \eqref{S-S3}, we can get that 
\be\label{S-S4}
\left|\wt S_{0x}  - S_{0x} \right| \lesssim \frac{\langle x\rangle^2}{L^2} \frac{1}{W^d}+ \langle x\rangle^{-D},\quad \text{for}\quad |x|\le  W^{1+\tau}.
\ee
Combining \eqref{S-S2} and \eqref{S-S4}, we obtain \eqref{S0xy3}. Then using \eqref{contractionS+}, \eqref{S0xy3} and the Taylor expansions of $S^+$ and $\wt S^+$ as in \eqref{expS+}, we can readily get \eqref{S+xy3}. We omit the details.

It remains to prove \eqref{Theta-wh3}. By \eqref{S0xy3}, we have that for any constants $ \tau,D>0$,
\be\label{T-T3}\left| \wt S^k_{0x} - S^k _{0x}\right| \lesssim \frac{W^{2}}{L^2} \frac{1}{W^d} \sum_{k_1=1}^{k-1} \sum_{|x_1-x_2|\le \langle x\rangle^\tau W }S^{k_1}_{0x_1} \wt S^{k-1-k_1}_{x_2 x} + \langle x\rangle^{-D}.\ee
Moreover, we have a similar inequality as \eqref{T-T0}:
\begin{align}
	|\wt\Theta_{0x}  - \Theta_{0x}| &\le \sum_{k=1}^{ \langle x\rangle^{2-\tau}/W^2} |m|^{2k}\left|S^k_{0x} - \wt S^k _{0x}\right| + \sum_{k= \langle x\rangle^{2-\tau}/W^2}^{\langle x\rangle^{-\tau}\eta^{-1}} |m|^{2k}\left|S^k_{0x} - \wt S^k _{0x}\right|   \nonumber\\
	&+ \sum_{k=\langle x\rangle^{-\tau}\eta^{-1}}^{\langle x\rangle^\tau\eta^{-1}} |m|^{2k} \left(\left|\wt S^k_{0x}\right|+\left|S^k_{0x}\right|\right) + \langle x\rangle^{-D} .\label{T-T1}
\end{align}
Using \eqref{RW_diffusion2} and \eqref{T-T3}, we can bound the three terms on the right-hand side of \eqref{T-T1} one by one as in the estimates below \eqref{T-T0}, which concludes \eqref{Theta-wh3}. We omit the details. 
\end{proof}

\section{Proofs for local expansions}\label{appd localpf}
In this section, we provide the proofs for the lemmas in Section \ref{sec_basiclocal}.
\begin{proof} [Proof of Lemma \ref{expandlabel}]
We first prove that $ {\cal O}_{weight}^{(x)}[\cal G]$ is a canonical local expansion by verifying the properties (i)--(iv) of Definition \ref{def_canonical_local} one by one. To prove property (i), it suffices to show that \eqref{Owx} is an identity in the sense of graph values. This follows from Lemma \ref{ssl} together with the facts that $b_{x\al}=\delta_{x\al} + S^+_{x\al}$ and  $P_\al = 1-Q_\al$. The properties (iii) and (iv) are trivial by definition. 
It remains to prove property (ii). First, it is easy to see that \smash{${\cal O}^{(x),1}_{weight}$} acting on a regular (resp. normal regular) graph gives a linear combination of regular (resp. norma regular) graphs. Second, the $\cal O_{dot}$ in \eqref{eq defOdot} will expand a regular graph into a sum of normal regular graphs by Lemma \ref{lem Odot}. Hence to prove property (ii) of Definition \ref{def_canonical_local}, it suffices to show that \smash{$ {\cal O}^{(x), 2}_{weight}$} acting on a regular graph also gives a linear combination of regular graphs. Now given any regular graph $\cal G$, we need to check the properties (i)--(iii) of Definition \ref{defnlvl0} for the graphs in \smash{${\cal O}^{(x) ,2}_{weight} [\cal G]$}. The properties (i) and (iii) of Definition \ref{defnlvl0} are trivially true, while the property (ii) follows from the fact that the new atoms are connected to $x$ through paths of waved edges. 

In sum, we have shown that  \smash{$ {\cal O}_{weight}^{(x)}[\cal G]$} is a canonical local expansion. Now we prove statements (a) and (b) of Lemma \ref{expandlabel}. If $\cal G$ contains some regular weights $G_{xx}$ or $\overline G_{xx}$ on $x$, then there is a graph in \smash{${\cal O}^{(x),1}_{weight} [\cal G]$} obtained by replacing all these weights by $m$ or $\overline m$, and this graph satisfies (b). All the other graphs in \smash{${\cal O}^{(x),1}_{weight} [\cal G]$} satisfy (a). To conclude the proof, it remains to prove that if $\cal G$ only contains light weights on the atom $x$, then every graph without $Q$-labels 
in \smash{$ {\cal O}_{weight}^{(x)}[\cal G]$} satisfies either (a) or (b). For this purpose, we study the graphs on the right-hand side of \eqref{Owx} one by one.

\begin{itemize}
\item[(1)] The first two graphs on the right-hand side of \eqref{Owx} both have strictly higher scaling orders than $\cal G$ because they contain one more light weight than $\cal G$. 

\item[(2)] We consider any graph, say $\cal G_1$, in ${\cal O}^{(x), 1}_{weight}\circ \cal O_{dot}\left[m  \sum_\al s_{x \al} G_{\al x}\partial_{ h_{\al x}} f(G)\right]$. We have the following cases.
\begin{itemize}
\item  If $\al$ is identified with $x$ or some other atoms in $f(G)$, then we have that
\be\label{larger weight} \text{ord} \left[ \cal G_1 \right] \ge \text{ord} \left[\cal G\right] + 1 ,\ee
because the scaling order of $s_{x\al}$ is larger than $\ord(G_{xx}-m)=1$ and the scaling orders of the graphs in $ \partial_{  h_{\al x}}f(G) $ are $ \ge \text{ord}[f(G) ] $. Hence $\cal G_1$ satisfies (a).


\item If $\al$ is not identified with any other atom, then $G_{\al x}$ is of scaling order 1. Moreover, suppose $\cal G_1$ contains a subgraph in $\partial_{h_{\al x}} f(G)$ obtained through the partial derivatives in \eqref{partialhG}. If $b \ne x$, then the scaling order of $G_{a\al}G_{xb}$ or $\overline G_{bx}\overline G_{\al a}$ is strictly larger than the scaling order of the original component $G_{ab}$, $\overline G_{ba}$, $G_{aa}-m$ or $\overline G_{aa}-\overline m$ (where the last two cases happen if the partial derivative acts on a light weight with $a=b$). Hence $\cal G_1$ satisfies (a).

\item Suppose $\al$ is not identified with any other atom, and $\cal G_1$ contains a subgraph in $\partial_{ h_{ \al x}} f(G)$ obtained through the partial derivatives in \eqref{partialhG}. If $b=x$, then {${\cal O}^{(x), 1}_{weight}$} will expand \eqref{partialhG} into
\be\nonumber
G_{a\al}G_{xx} =m G_{a\al}+G_{a\al}(G_{xx}-m) , \quad  \text{or} \quad \overline G_{xx}\overline G_{\al a}=\overline m\overline G_{\al a} +(\overline G_{xx}-\overline m)\overline G_{\al a}.\ee
If $\cal G_1$ contains $G_{a\al}(G_{xx}-m) $ or $(\overline G_{xx}-\overline m)\overline G_{\al a}$, then it satisfies (a); otherwise, if $\cal G_1$ contains $m G_{a\al}$ or $\overline m\overline G_{\al a}$, then it satisfies (b). 

\end{itemize}
\item[(3)] The graphs in ${\cal O}^{(x), 1}_{weight}\circ\cal O_{dot}\left[m \sum_{\al,\beta} S^{+}_{x\al}s_{\al \beta} G_{\beta \al}\partial_{ h_{ \beta \al}} f (G)\right]$ can be dealt with in the same way as (2).

 
\end{itemize}
Combining the above cases (1)-(3), we conclude the proof of Lemma \ref{expandlabel}.
  \end{proof} 


\begin{proof} [Proof of Lemma \ref{expandmulti}]
First, using Lemma \ref{Oe14} and a similar argument as in the above proof of Lemma \ref{expandlabel}, we can prove that \smash{${\cal O}_{multi-e}^{(x)}[\cal G]$} is a canonical local expansion. To prove the property (a), we study the graphs on the right-hand side of \eqref{Oe1x} one by one. 
 \begin{itemize}

\item[(1)] 
Notice that {$ \cal O_{dot}\left[ \sum_\al s_{x\al }G_{\al y_1} \overline G_{\al y'_i}\right]$} is a sum of subgraphs of scaling orders $\ge \ord(G_{x y_1} \overline G_{xy_i'}) =2$. In addition, there is an extra light weight $\overline G_{xx} -\overline m$ in the third graph on the right-hand side of \eqref{Oe1x}, so it gives graphs of strictly higher scaling orders than $\cal G$ after the dotted edge partition $\cal O_{dot}$. The fourth graph on the right-hand side of \eqref{Oe1x} can be handled in the same way.

\item[(2)] The fifth graph on the right-hand side of \eqref{Oe1x} obviously has strictly higher scaling order than $\cal G$, because it contains one more light weight $G_{\al\al}-m$. 
 
\item[(3)] 
Notice that $ \cal O_{dot}\left[ \sum_\al s_{x\al }G_{x \al}G_{\al y_1}  \right]$ gives subgraphs of scaling orders $\ge 2 > \ord(G_{xy_1})=1$. Hence the dotted edge partition of the sixth graph on the right-hand side of \eqref{Oe1x} gives graphs of strictly higher scaling orders than $\cal G$. The seventh graph on the right-hand side of \eqref{Oe1x} can be handled in the same way.

\item[(4)] Regarding the eighth graph on the right-hand side of \eqref{Oe1x}, we consider any graph, say $\cal G_1$, in $$\cal O_{dot}\left[ \sum_\al s_{x\al } \frac{\cal G}{G_{x y_1}f (G)}G_{\al y_1}\partial_{ h_{ \al x}}f (G)\right].$$
We have the following two cases.
\begin{itemize}
\item If $\al$ is identified with $x$ or some other atoms in $f (G)$, then we have \eqref{larger weight}, because the scaling order of $s_{x\al}$ is larger than $G_{xy_1}$ and the scaling orders of the graphs in $ \partial_{h_{\al x}}f (G) $ are $ \ge \text{ord}[f (G) ] $. 

\item If $\al$ is not identified with any other atom, then $G_{\al y_1}$ is of scaling order 1. Moreover, suppose $\cal G_1$ contains a graph in $\partial_{ h_{\al x}} f (G)$ obtained through the partial derivatives in \eqref{partialhG}. By our assumption on $f (G)$, we must have $a \ne x$ and $b\ne x$. Then the scaling order of $G_{a\al}G_{xb}$ or $\overline G_{bx}\overline G_{\al a}$ is strictly larger than the scaling order of the original component $G_{ab}$, $\overline G_{ba}$, $G_{aa}-m$ or $\overline G_{aa}-\overline m$. Hence 
$\cal G_1$ satisfies \eqref{larger weight}.
\end{itemize}

\item[(5)] Regarding the first graph on the right-hand side of \eqref{Oe1x}. we consider any graph, say $\cal G_1$, in $$\cal O_{dot}\left[  \sum_\al s_{x\al }G_{\al y_1} \overline G_{\al y'_i}\cdot \frac{\cal G_x}{G_{x y_1} \overline G_{xy_i'}} \right].$$
We have the following three cases.
\begin{itemize}
\item If $\al$ is identified with $y_1$ or $y_i'$ and $y_1\ne y_i'$, then the subgraph $s_{xy_1}G_{y_1 y_1} \overline G_{y_1 y'_i}$ or $s_{xy_i'}G_{y_i' y_1} \overline G_{y_i' y'_i}$ is of strictly higher scaling order than $G_{x y_1} \overline G_{xy_i'}$. Hence $\cal G_1$ satisfies \eqref{larger weight}. 

\item If $\al$ is not identified with either $y_1$ or $y_i'$, then $\cal G_1$ satisfies (a.1). 

\item if $\al=y_1=y_i'$, then $\cal G_1$ satisfies (a.2). 
\end{itemize}
The second graph on the right-hand side of \eqref{Oe1x} can be handled in the same way. 
\end{itemize}
Combining the cases (1)--(5), we conclude property (a). 

Finally, if $\deg(x)=1$ or $x$ is connected with exactly two mismatched solid edges in $\cal G$, then one can check that the first two leading terms on the right-hand side of \eqref{Oe1x} vanish, and the above case (5) cannot happen. Hence we get that property (b) holds. 
  \end{proof} 


\begin{proof} [Proof of Lemma \ref{expandG2}]
First, using Lemma \ref{T eq0} and a similar argument as in the proof of Lemma \ref{expandlabel}, we can prove that \smash{${\cal O}_{GG}^{(x)}[\cal G]$} is a canonical local expansion. To prove the properties (a) and (b), we consider the graphs on the right-hand side of \eqref{Oe2x} one by one. 
 
 \begin{itemize}
 \item[(1)] The second to fifth graphs on the right-hand side of \eqref{Oe2x} obviously have strictly higher scaling orders than $\cal G$, because they contain one more light weight than $\cal G$. 
 
 \item[(2)]  With a similar argument as in item (4) of the above proof of Lemma \ref{expandmulti}, we can show that the sixth and seventh graphs on the right-hand side of \eqref{Oe2x} will give graphs satisfying \eqref{larger weight} after the dotted edge partition $\cal O_{dot}$.

\item[(3)] The first graph on the right-hand side of \eqref{Oe2x} satisfies \eqref{larger weight} if $y\ne y'$. Otherwise, if $y=y'$, then the graph is obtained by replacing $G_{xy}   G_{yx }$ with $m S^+_{xy}G_{yy}$. 
\end{itemize}
Combining the cases (1)--(3), we conclude the properties (a) and (b).
\end{proof}

\begin{proof}[Proof of Lemma \ref{expandGGbar}] 
All the statements are corollaries of Lemma \ref{expandmulti}, except for the properties (a) and (b) of $\cal G_1$. We prove these properties by studying the graphs on the right-hand side of \eqref{Oe3x} one by one. 
 \begin{itemize}
 \item[(1)] The second and third graphs on the right-hand side of \eqref{Oe3x} obviously satisfy \eqref{larger weight} because they contain one more light weight than $\cal G$. 
 
\item[(2)] With a similar argument as in item (4) of the above proof of  Lemma \ref{expandmulti}, we can show that the fourth graph on the right-hand side of \eqref{Oe3x} will give graphs satisfying \eqref{larger weight} after the dotted edge partition.

\item[(3)] Any graph in $\cal O_{dot}[\sum_\al s_{x\al }G_{\al y} \overline G_{\al y'} f(G)]$ satisfies \eqref{larger weight} if $\al \in \{y,y'\}$ and $y\ne y'$, satisfies (a) if $\al \notin \{y,y'\}$, and satisfies (b) if $\al = y = y'$. 
\end{itemize}
Combining the cases (1)-(3), we conclude the proof.
\end{proof} 

Finally, we give the proof of Lemma \ref{lvl1 lemma}.
\begin{proof}[Proof of Lemma \ref{lvl1 lemma}]
To conclude the proof, we need to show that the expansion process in Figure \ref{Fig chart1} will finally stop after $\OO(1)$ many iterations. 

\vspace{5pt}

\noindent{\bf The $\cal O_{dot}$--${\cal O}_{weight}$ loop.} First, we prove that the $\cal O_{dot}$--${\cal O}_{weight}$ loop stops after $\OO(1)$ many iterations. By Lemma \ref{expandlabel}, after a weight expansion, every resulting graph satisfies 
at least one of the following conditions: (1) it already satisfies the stopping rules; (2) it has strictly higher scaling order than the input graph; (3) it has strictly fewer weights than the input graph. 
Thus there exists a fixed $k\in \N$ depending on $n$ and the number of weights such that after $k$ iterations of the $\cal O_{dot}$--${\cal O}_{weight}$ loop, every new graph either satisfies the stopping rules already or has no weights in it. A graph in the former case will be sent to the output directly. For a graph in the latter case, ${\cal O}_{weight}$ will be a null operation in the next iteration and this graph exits the $\cal O_{dot}$--${\cal O}_{weight}$ loop successfully.

\vspace{5pt}

\noindent{\bf The $\cal O_{dot}$--${\cal O}_{weight}$--$\cal O_{multi-e}$ loop.} Suppose we apply the $\cal O_{dot}$-${\cal O}_{weight}$-$\cal O_{multi-e}$ iteration once to an input graph, say $\cal G_0$, and get a collection of new graphs, say $\mathscr G_1$. For any new graph in $\mathscr G_1$, if it already satisfies the stopping rules, then we send it to the output directly; otherwise we send it back to the first step $\cal O_{dot}$. For a graph in the latter case, 
if it contains no weights and every atoms in it either has degree 0 or is connected with two matched solid edges, then ${\cal O}_{weight}$ and $\cal O_{multi-e}$ are both null operations and this graph exits the $\cal O_{dot}$--${\cal O}_{weight}$--$\cal O_{multi-e}$ loop successfully. On the other hand, if either ${\cal O}_{weight}$ or $\cal O_{multi-e}$ is a non-trivial operation for a graph $\cal G_1\in \mathscr  G_1$, then we will apply the $\cal O_{dot}$--${\cal O}_{weight}$--$\cal O_{multi-e}$ iteration to it and get a collection of new graphs, say $\mathscr G_2 $. By Lemmas \ref{expandlabel} and \ref{expandmulti}, every graph in $\mathscr G_2 $ either satisfies the stopping rules already, or falls into at least one of the following categories: 
\begin{itemize}
\item[(1)] it has strictly higher scaling order than $\cal G_1$; 
\item[(2)] there is one new atom of degree 2 and one old atom whose degree decreases by 2, while the degree of any other atom stays the same;
\item[(3)] there is no new atom and one old atom whose degree decreases by 2, while the degree of any other atom either stays the same or decreases by 2. 
\end{itemize}
If a graph in $\mathscr G_2$ satisfies the stopping rules, then we send it to the output. Otherwise, we send it back to $\cal O_{dot}$ and apply another $\cal O_{dot}$--${\cal O}_{weight}$--$\cal O_{multi-e}$ iteration to it.

We repeat the above iterations, and construct correspondingly a tree diagram $\cal T$ of graphs as follows. Let $\cal G_0$ be the root, which represents the input graph. Given a graph $\cal G$ represented by a vertex of the tree, its children are the graphs obtained from an $\cal O_{dot}$--${\cal O}_{weight}$--$\cal O_{multi-e}$ iteration acting on $\cal G$. If a graph $\cal G$ satisfies the stopping rules or if ${\cal O}_{weight}$ and $\cal O_{multi-e}$ are null operations for $\cal G$, then $\cal G$ is a leaf of the tree, and it exits the $\cal O_{dot}$--${\cal O}_{weight}$--$\cal O_{multi-e}$ successfully. Let the height of $\cal T$ be the maximum distance between a leaf of $\cal T$ and the root. To show that the $\cal O_{dot}$--${\cal O}_{weight}$--$\cal O_{multi-e}$ loop stops after $\OO(1)$ many iterations, it is equivalent to show that $\cal T$ is a finite tree with height of order $\OO(1)$. 

Let $\cal G_0 \to \cal G_1 \to \cal G_2 \to \cdots \to \cal G_h$ be a self-avoiding path on $\cal T$ from the root to a leaf. We let $k_0=0$. After having defined $k_i$, let $k_{i+1}:=\min\{j> k_i : \text{ord}(\cal G_j) > \text{ord}(\cal G_{k_i})\}$. Then the sequence $\{k_0,k_1, k_2, \cdots\}$ has length $\le n$, since a graph of scaling order $\ge n+1$ already satisfies the stopping rule (S3). Moreover, we claim that $|k_{i+1}-k_i|< n$. In fact, from the above discussion, we see that in order for the scaling order of a child to be the same as the scaling order of its parent, it has to be in category (2) or (3). Note that the total degree of the atoms in $\cal G_{k_i}$ is at most $2n$ (because there are at most $n$ off-diagonal solid edges in it), and each iteration decreases at least one atom's degree by 2. Hence we immediately get that $|k_{i+1}-k_i|\le n$.
The above argument shows that we must have $h\le n^2$, i.e. the height of $\cal T$ is at most $n^2$. 
This means that the $\cal O_{dot}$--${\cal O}_{weight}$--$\cal O_{multi-e}$ loop will stop after at most $n^2$ many iterations.

\vspace{5pt}

Next with Lemmas \ref{expandG2} and \ref{expandGGbar}, using a similar tree diagram argument as above, we can show that both the $\cal O_{dot}$--$\cdots$--$\cal O_{GG}$ and $\cal O_{dot}$--$\cdots$--$\cal O_{G\overline G}$ loops will stop after $\OO(1)$ many iterations. 
In particular, exiting the $\cal O_{dot}$--$\cdots$--$\cal O_{G\overline G}$ loop means that the expansion process in Figure \ref{Fig chart1} is completed successfully, which concludes the proof of Lemma \ref{lvl1 lemma}.
\end{proof}



\end{document}